\documentclass[11pt,reqno]{amsart}
\numberwithin{equation}{section}

\usepackage[tt=false]{libertine}
\usepackage{mathtools}
\usepackage{amssymb,mathrsfs}
\usepackage[varbb]{newpxmath}
\usepackage[normalem]{ulem}
\usepackage[margin=1in]{geometry}
\usepackage{graphicx}
\usepackage{enumerate}
\usepackage{bbm}
\usepackage{hyperref,color}
\hypersetup{
	colorlinks=true,
	pdfpagemode=UseNone,
    citecolor=OliveGreen,
    linkcolor=NavyBlue,
    urlcolor=black,
	pdfstartview=FitW
}
\usepackage{appendix}
\usepackage{tikz}
\usepackage{bm}
\usepackage{mathtools}

\usepackage{amsmath,amsthm,amssymb,amsfonts}
\numberwithin{equation}{section}
\theoremstyle{plain}
\newtheorem{thm}{Theorem}[section]
\newtheorem{lem}[thm]{Lemma}

\newtheorem{cor}[thm]{Corollary}

\newtheorem{as}[thm]{Assumption}
\newtheorem*{theorem*}{\textbf{Theorem}}
\newtheorem{definition}[thm]{Definition}

\newtheorem{lemma}[thm]{Lemma}

\theoremstyle{definition}

\theoremstyle{remark}
\newtheorem{remark}[thm]{Remark}

\usepackage{graphicx}
\usepackage[all,arc]{xy}
\usepackage{tikz}
\usepackage{svg}

\usepackage{mdframed}
\usepackage{tcolorbox}
\usepackage{capt-of}
\usepackage{enumitem}
\usepackage{algorithm}
\usepackage{algorithmic}
\usepackage{adjustbox}

\usepackage{hyperref}
\hypersetup{
  colorlinks,
  linkcolor=red,
  filecolor=red,
  citecolor=red,
  urlcolor=red
}
\usepackage{url}

\usepackage{autonum}

\allowdisplaybreaks

\makeatletter
\@tfor\next:=abcdefghijklmnopqrstuvwxyzABCDEFGHIJKLMNOPQRSTUVWXYZ\do{%
  \def\command@factory#1{%
    \expandafter\def\csname #1b\endcsname{{\mathbf{#1}}}
  }
 \expandafter\command@factory\next
}
\makeatother

\makeatletter
\@tfor\next:=acdeghijklnopqrstuvwxyzABCDEFGHIJKLMNOPQRSTUVWXYZ\do{%
  \def\command@factory#1{%
    \expandafter\def\csname b#1\endcsname{{\mathbf{#1}}}
  }
 \expandafter\command@factory\next
}
\makeatother

\makeatletter
\@tfor\next:=ABCDEFGHIJKLMNOPQRSTUVWXYZ\do{%
  \def\command@factory#1{%
    \expandafter\def\csname c#1\endcsname{{\mathcal{#1}}}
  }
 \expandafter\command@factory\next
}
\makeatother

\makeatletter
\@tfor\next:=ABCDEFGHIJKLMNOPQRTUVWXYZ\do{%
  \def\command@factory#1{%
    \expandafter\def\csname #1#1\endcsname{{\mathbb{#1}}}
  }
 \expandafter\command@factory\next
}
\makeatother

\makeatletter
\def\greekvectors#1{%
 \@for\next:=#1\do{%
    \def\X##1;{%
     \expandafter\def\csname b##1\endcsname{{\bm{\csname##1\endcsname}}}
     }
   \expandafter\X\next;
  }
}
\greekvectors{Alpha,Beta,Gamma,Delta,Theta,Lambda,Xi,Pi,Sigma,Upsilon,Phi,Chi,Psi,Omega,alpha,beta,gamma,delta,epsilon,zeta,theta,iota,kappa,lambda,mu,nu,xi,pi,rho,sigma,tau,upsilon,phi,chi,psi,omega,varphi,varepsilon,varrho,vartheta}
\makeatother

\makeatletter
\@tfor\next:=abcdefghijklmnpqstuvwxyz\do{%
  \def\command@factory#1{%
    \expandafter\def\csname t#1\endcsname{{\tilde{#1}}}
  }
 \expandafter\command@factory\next
}
\makeatother

\makeatletter
\@tfor\next:=ABCDEFGHIJKLMNOPQRSTUVWXYZ\do{%
  \def\command@factory#1{%
    \expandafter\def\csname t#1\endcsname{{\widetilde{#1}}}
  }
 \expandafter\command@factory\next
}
\makeatother

\makeatletter
\def\greekvectors#1{%
 \@for\next:=#1\do{%
    \def\X##1;{%
     \expandafter\def\csname t##1\endcsname{{\tilde{\csname##1\endcsname}}}
     }
   \expandafter\X\next;
  }
}
\greekvectors{Alpha,Beta,Gamma,Delta,Theta,Lambda,Xi,Pi,Sigma,Upsilon,Phi,Chi,Psi,Omega,alpha,beta,gamma,delta,epsilon,zeta,eta,theta,iota,kappa,lambda,mu,nu,xi,pi,rho,sigma,tau,upsilon,phi,chi,psi,omega,varphi,varepsilon,varrho,vartheta}
\makeatother

\makeatletter
\@tfor\next:=abcdefghijklmnopqrstuvwxyz\do{%
  \def\command@factory#1{%
    \expandafter\def\csname hat#1\endcsname{{\hat{#1}}}
  }
 \expandafter\command@factory\next
}
\makeatother

\makeatletter
\@tfor\next:=ABCDEFGHIJKLMNOPQRSTUVWXYZ\do{%
  \def\command@factory#1{%
    \expandafter\def\csname hat#1\endcsname{{\widehat{#1}}}
  }
 \expandafter\command@factory\next
}
\makeatother

\makeatletter
\def\greekvectors#1{%
 \@for\next:=#1\do{%
    \def\X##1;{%
     \expandafter\def\csname hat##1\endcsname{{\hat{\csname##1\endcsname}}}
     }
   \expandafter\X\next;
  }
}
\greekvectors{alpha,beta,gamma,delta,epsilon,zeta,eta,theta,iota,kappa,lambda,mu,nu,xi,pi,rho,sigma,tau,upsilon,phi,chi,psi,omega,varphi,varepsilon,varrho,vartheta}
\makeatother

\makeatletter
\def\greekvectors#1{%
 \@for\next:=#1\do{%
    \def\X##1;{%
     \expandafter\def\csname hat##1\endcsname{{\widehat{\csname##1\endcsname}}}
     }
   \expandafter\X\next;
  }
}
\greekvectors{Alpha,Beta,Gamma,Delta,Theta,Lambda,Xi,Pi,Sigma,Upsilon,Phi,Chi,Psi,Omega}
\makeatother

\makeatletter
\@tfor\next:=abcdefghijklmnopqrstuvwxyz\do{%
  \def\command@factory#1{%
    \expandafter\def\csname bar#1\endcsname{{\bar{#1}}}
  }
 \expandafter\command@factory\next
}
\makeatother

\makeatletter
\@tfor\next:=ABCDEFGHIJKLMNOPQRSTUVWXYZ\do{%
  \def\command@factory#1{%
    \expandafter\def\csname bar#1\endcsname{{\overline{#1}}}
  }
 \expandafter\command@factory\next
}
\makeatother

\makeatletter
\def\greekvectors#1{%
 \@for\next:=#1\do{%
    \def\X##1;{%
     \expandafter\def\csname bar##1\endcsname{{\bar{\csname##1\endcsname}}}
     }
   \expandafter\X\next;
  }
}
\greekvectors{alpha,beta,gamma,delta,epsilon,zeta,eta,theta,iota,kappa,lambda,mu,nu,xi,pi,rho,sigma,tau,upsilon,phi,chi,psi,omega,varphi,varepsilon,varrho,vartheta}
\makeatother

\makeatletter
\def\greekvectors#1{%
 \@for\next:=#1\do{%
    \def\X##1;{%
     \expandafter\def\csname bar##1\endcsname{{\overline{\csname##1\endcsname}}}
     }
   \expandafter\X\next;
  }
}
\greekvectors{Alpha,Beta,Gamma,Delta,Theta,Lambda,Xi,Pi,Sigma,Upsilon,Phi,Chi,Psi,Omega}
\makeatother

\makeatletter
\@tfor\next:=abcdefghijklmnpqstuvwxyz\do{%
  \def\command@factory#1{%
    \expandafter\def\csname bre#1\endcsname{{\breve{#1}}}
  }
 \expandafter\command@factory\next
}
\makeatother

\makeatletter
\@tfor\next:=ABCDEFGHIJKLMNOPQRSTUVWXYZ\do{%
  \def\command@factory#1{%
    \expandafter\def\csname bre#1\endcsname{{\breve{#1}}}
  }
 \expandafter\command@factory\next
}
\makeatother

\makeatletter
\def\greekvectors#1{%
 \@for\next:=#1\do{%
    \def\X##1;{%
     \expandafter\def\csname bre##1\endcsname{{\breve{\csname##1\endcsname}}}
     }
   \expandafter\X\next;
  }
}
\greekvectors{Alpha,Beta,Gamma,Delta,Theta,Lambda,Xi,Pi,Sigma,Upsilon,Phi,Chi,Psi,Omega,alpha,beta,gamma,delta,epsilon,zeta,eta,theta,iota,kappa,lambda,mu,nu,xi,pi,rho,sigma,tau,upsilon,phi,chi,psi,omega,varphi,varepsilon,varrho,vartheta}
\makeatother

\newcounter{Ccnt}
\makeatletter
\newcommand\Co[1]{%
\@ifundefined{C-#1}%
  {\stepcounter{Ccnt}\expandafter\xdef\csname C-#1\endcsname{\arabic{Ccnt}}}%
  {}%
\kappa_{\csname C-#1\endcsname}}
\makeatother

\definecolor{goldenrod}{rgb}{0.85, 0.65, 0.13}

\renewcommand{\[}{\begin{equation}}
\renewcommand{\]}{\end{equation}}

\mathcode`l="8000
\begingroup
\makeatletter
\lccode`\~=`\l
\DeclareMathSymbol{\lsb@l}{\mathalpha}{letters}{`l}
\lowercase{\gdef~{\ifnum\the\mathgroup=\m@ne \ell \else \lsb@l \fi}}%
\endgroup

\newcommand{\supp}{\text{supp}}
\newcommand{\stleq}{\preceq_{\rm st}}
\newcommand{\stlow}{\preceq_{\rm st}}
\newcommand{\lalg}{\ensuremath{\sqrt{\log(n)}\lambda_{\rm ALG}}}

\newcommand{\bone}{\mathbf{1}}
\newcommand{\Id}{{\rm Id}}

\newcommand{\norm}[1]{\lVert #1 \rVert}
\newcommand{\iidsim}{\overset{\mathrm{iid}}{\sim}}
\newcommand{\laweq}{\overset{\cL}{=}}

\newcommand{\conv}[1]{\overset{#1}{\rightarrow}}

\newcommand{\argmin}{\operatorname*{arg\,min}}

\newcommand{\betaConst}{\frac{\lambda}{k^{r/2}n^{(r-1)/2}\sqrt{M_1}}}

\newcommand{\setlister}{(\Delta_t)}
\newcommand{\runtime}{M}

\newcommand{\parity}{\wp}
\newcommand{\NIQ}{\cS\cG\cC}
\newcommand{\CHW}{C_{\mathrm{HW}}}

\renewcommand{\1}{1\!\!1}
\newcommand{\plog}{\mathrm{poly}(\log(n))}
\newcommand{\bSignal}{f_{\theta, \lambda}(S_t, p_t, q_t)}
\newcommand{\bNoise}{g_W(S_t, p_t, q_t)}
\newcommand{\bReg}{h_\gamma(S_t, p_t, q_t)}
\newcommand{\bCorrect}{K_{p_t, t_{p_t}, q_t}}
\newcommand{\fConst}{\frac{f_{\theta, \lambda}(S_t, p_{t},
-(S_t)_{p_t})}{V_n \sqrt{\runtime_1}}}
\newcommand{\fConstSmall}{[f_{\theta, \lambda}(S_t, p_{t},
-(S_t)_{p_t})/(V_n \sqrt{\runtime_1})]}

\newcommand{\fuConst}{\frac{2C_{\rm Signal}\frac{\lambda}{k^{r/2}}}{V_n\sqrt{M_1}}(\langle
  S_t, \theta \rangle - 2)^{r-1}}
  \newcommand{\fuConsto}{\frac{2C_{\rm Signal}\frac{\lambda}{k^{r/2}}}{V_n\sqrt{M_1}}(\langle
  S_{t_0}, \theta \rangle - 2)^{r-1}}

\newcommand{\prefactor}{2C_{\rm Signal}\frac{\lambda}{k^{r/2}V_n\sqrt{M_1}}}
\newcommand{\Variance}{V(S_t, p_t, q_t)}
\newcommand{\proposal}{S_t + e_{p_t}(q_t - (S_t)_{p_t})}

\newcommand{\fPrime}[2]{\frac{\tilde f_{\theta, \lambda}(#1, #2)}{V_n
\sqrt{\runtime_1}}}

\newcommand{\finalIter}{T^*}

\makeatletter
\def\@tocline#1#2#3#4#5#6#7{\relax
  \ifnum #1>\c@tocdepth 
  \else
    \par \addpenalty\@secpenalty\addvspace{#2}%
    \begingroup \hyphenpenalty\@M
    \@ifempty{#4}{%
      \@tempdima\csname r@tocindent\number#1\endcsname\relax
    }{%
      \@tempdima#4\relax
    }%
    \parindent\z@ \leftskip#3\relax \advance\leftskip\@tempdima\relax
    \rightskip\@pnumwidth plus4em \parfillskip-\@pnumwidth
    #5\leavevmode\hskip-\@tempdima
      \ifcase #1
       \or\or \hskip 1em \or \hskip 2em \else \hskip 3em \fi%
      #6\nobreak\relax
    \hfill\hbox to\@pnumwidth{\@tocpagenum{#7}}\par
    \nobreak
    \endgroup
  \fi}
\makeatother

\usepackage{ragged2e}

\newcommand{\miniFigSideLeftArxiv}[3]{
  \refstepcounter{figure}
  \begin{minipage}[t]{0.48\textwidth}
    \centering
    \rotatebox{90}{\qquad\qquad$|\cos(S_t, \theta)|$}
    \adjustbox{trim= .4cm .3cm 1.8cm 0cm, clip}{
    \includegraphics[width=1\columnwidth]{#1}
    }

    \vspace{-3pt}
    \hspace{1em}{\small Runtime}
    
    \vspace{0.5ex}
    \noindent \small{\justifying{Figure \arabic{figure}: #2}}\label{fig:#3}
  \end{minipage}%
}
\newcommand{\miniFigSideLeftArxivSpecial}[3]{
  \refstepcounter{figure}
  \begin{minipage}[t]{0.48\textwidth}
    \centering
    \rotatebox{90}{\qquad\qquad$|\cos(S_t, \theta)|$}
    \adjustbox{trim= .2cm 0cm 1.1cm 0cm, clip}{
    \includegraphics[width=1\columnwidth]{#1}
    }

    \vspace{-3pt}
    \hspace{1em}{\small Runtime}
    
    \vspace{0.5ex}
    \noindent \small{\justifying{Figure \arabic{figure}: #2}}\label{fig:#3}
  \end{minipage}%
}
\newcommand{\miniFigSide}[3]{
  \refstepcounter{figure}
  \begin{minipage}[t]{0.48\textwidth}
    \centering
    \adjustbox{trim= .4cm .3cm 0cm 0cm, clip}{
    \includegraphics[width=1\columnwidth]{#1}
    }

    \vspace{-3pt}
    \hspace{-4em}{\small Runtime}
    
    \vspace{0.5ex}
    \noindent \small{\justifying{Figure \arabic{figure}: #2}}\label{fig:#3}
  \end{minipage}%
}

\tcbset{
  algstyle/.style={
    enhanced,
    sharp corners,
    colback=white,
    colframe=black,
    boxrule=.4pt,
    left=2pt,right=2pt,top=2pt,bottom=2pt,  
    before skip=0pt,after skip=0pt,
  }
}

\newcommand{\nocontentsline}[3]{}
\let\origcontentsline\addcontentsline
\newcommand\stoptoc{\let\addcontentsline\nocontentsline}
\newcommand\resumetoc{\let\addcontentsline\origcontentsline}

\usepackage{pifont}

\begin{document}

\title[Almost-Optimal Local-Search Methods for Sparse Tensor PCA]{Almost-Optimal Local-Search Methods\\ for Sparse Tensor PCA}

\author[M. Lovig, C. Sheehan, K. Tsirkas, I. Zadik]{Max Lovig$^\star$, Conor Sheehan$^\star$, Konstantinos Tsirkas$^\star$ and Ilias Zadik$^\star$}
\thanks{
$^\star$Department of Statistics and Data Science, Yale University,\\
Email: \texttt{\{max.lovig,conor.sheehan,kostas.tsirkas,ilias.zadik\}@yale.edu}}

\date{
\today}

\begin{abstract}
   Local-search methods are widely employed in statistical applications, yet interestingly, their theoretical foundations remain rather underexplored, compared to other classes of estimators such as low-degree polynomials and spectral methods. Of note, among the few existing results recent studies have revealed a significant ``local-computational'' gap in the context of a well-studied sparse tensor principal component analysis (PCA), where a broad class of local Markov chain methods exhibits a notable underperformance relative to other polynomial-time algorithms. In this work, we propose a series of local-search methods that provably ``close'' this gap to the best known polynomial-time procedures in multiple regimes of the model, including and going beyond the previously studied regimes in which the broad family of local Markov chain methods underperforms. Our framework includes: (1) standard greedy and randomized greedy algorithms applied to the (regularized) posterior of the model; and (2) novel \emph{random-threshold} variants, in which the randomized greedy algorithm accepts a proposed transition if and only if the corresponding change in the Hamiltonian exceeds a random Gaussian threshold—rather that if and only if it is positive, as is customary. The introduction of the random thresholds enables a tight mathematical analysis of the randomized greedy algorithm's trajectory by crucially breaking the dependencies between the iterations, and could be of independent interest to the community.
\end{abstract}

\maketitle

\newpage

\tableofcontents

\newpage

\setcounter{page}{1}

\section{Introduction}

In recent decades, a significant body of research has focused on uncovering the "computational limits" of various statistical estimation models. More specifically, the primary aim of such works is to characterize the parameter regimes where successful estimation is achievable by polynomial-time estimators. Noteworthy examples of statistical estimation frameworks analyzed under this lens includes multiple variants of principal component analysis (PCA) models, such as tensor PCA \cite{montanari2014statistical, arous2020algorithmic, perryTensor} and sparse tensor PCA \cite{chen2024lowtemperaturemcmcthresholdcases, luo2022tensor, niles_weedAllOrNothing}, sparse regression models \cite{gamarnik2022sparse, yang2016computational, chen2024lowtemperaturemcmcthresholdcases}, and community detection models like the celebrated planted clique model \cite{jerrum1992large, brennan2018reducibility, gamarnik2024landscape, chen2025almost,gheissari2023findingplantedcliquesusing}, among others.

 To explore these computational boundaries, many studies have rigorously analyzed the performance of large classes of polynomial-time estimators within established models. Among these, low-degree polynomials have emerged as one of the most well-studied classes (see e.g., \cite{kunisky2019notes,schramm2022computational} and references therein), with extensive results documenting their efficiency across various estimation tasks. Furthermore, an intriguing conjecture within hypothesis testing posits, for many ``noisy and symmetric'' statistical models, $O(\log n)$-degree polynomials match or surpass the performance achieved by any polynomial-time estimator. Other notable classes of estimators that have received a similar level of attention include Statistical Query methods \cite{feldman2017statistical, diakonikolas2017statistical}, spectral techniques \cite{hopkins2019robust, dhara2023power}, and message passing algorithms \cite{donoho2009message, feng2022unifying, montanariEquivalence}.

Interestingly, despite the popularity and utility of the class of ``local'' methods (such as local Markov Chain Monte Carlo (MCMC) methods) in applied statistics and machine learning (see e.g., \cite{o2009review} for a survey on variable selection), our mathematical understanding of these iterative algorithms in high-dimensional parameter estimation problems remains comparatively underdeveloped. This discrepancy is frequently attributed to the inherent complexities of analyzing these iterative algorithms within probabilistic frameworks. Nevertheless, recent works have begun to elucidate the performance of ``natural'' classes of local MCMC methods in specific canonical contexts, with some findings indicating these methods underperform relative to other polynomial-time procedures. Two notable instances occurred recently in the planted clique model, specifically the Metropolis processes \cite{gamarnik2024landscape,chen2025almost}, and in a classical sparse tensor PCA model for a much larger class of ``low-temperature'' MCMC methods \cite{chen2024lowtemperaturemcmcthresholdcases}. It is natural to question whether this underperformance stems from difficulty in accurately designing and analyzing the ``correct'' local method or if it indicates a more fundamental limitation of these methods. In this work, we achieve a deeper understanding of this issue by concentrating on a slight generalization of the well-studied sparse tensor PCA model studied in \cite{chen2024lowtemperaturemcmcthresholdcases, niles_weedAllOrNothing, caiSubmatrix, luo2022tensor}.

\subsection{Sparse Tensor PCA: definition and literature}\label{sec:lit}
\textbf{Model} To define the sparse tensor PCA model, let $r, k,n \in \mathbb{N}$ where $n$ is the ``growing'' parameter, $k=k_n \leq n$ with $k=n^{\Omega(1)}$, and $r \geq 2$ constant with respect to $n$, i.e., $2 \leq r=\Theta(1)$\footnote{We use throughout the standard asymptotic notation $O,\Omega,\Theta,o, \omega,\Tilde O,\Tilde \Omega,\Tilde \Theta,\tilde{o},\tilde{\omega}$ for the underlying relation to the growing $n$.}. Then we consider the ``signal'' $\theta$ to be drawn uniformly at random for either (1) among all binary $k$-sparse vectors, i.e. $\theta \sim \mathrm{Unif}\{\Theta_k\}$ for $\Theta_k:=\{v \in \{0,1\}^n: \|v\|_0=k\}$ or (2) among all Rademacher $k$-sparse vectors, i.e. $\theta \sim \mathrm{Unif}\{\widehat{\Theta}_k\}$ for $\widehat{\Theta}_k:=\{v \in \{-1,0,1\}^n: \|v\|_0=k\}.$ Given $\theta,$ the statistician observes the $r$-tensor
\begin{align}\label{eq:model}
    Y=\frac{\lambda}{k^{r/2}} \theta^{\otimes r} + W,
\end{align}
where $W$ is a $r$-tensor with i.i.d. $\mathcal{N}(0,1)$ entries. The goal of the statistician is to design an estimator using $Y$ which outputs exactly $\theta$ (or an element in $\{-\theta,\theta\}$ by sign symmetry, if $r$ is even) with high probability (w.h.p.) as $n \rightarrow +\infty$ \footnote{Throughout this work, by saying an event happens with high probability (w.h.p.) as $n \rightarrow +\infty$, we mean it happens with probability $1-o(1)$.}.

\textbf{Computational-Statistical trade-off} Using standard Bayesian results the optimal estimator for this exact recovery task is the Maximum-a-posterior (MAP) estimator, which in this setting corresponds to solving the optimization problem $\max_{v \in P_k} \langle v^{\otimes t}, Y \rangle,$ where $P_k \in \{\Theta_k,\widehat{\Theta}_k\}$ depending on the choice of prior on $\theta$. A folklore analysis implies MAP (and hence any estimator), in both aforementioned priors, can identify the true parameter $\theta$ w.h.p.\! as $n \rightarrow +\infty$ if and only if $\lambda=\omega(\lambda_{\mathrm{IT}})$ for some $\lambda_{\mathrm{IT}}=\tilde{\Theta}(\sqrt{k})$ (see e.g., \cite{luo2022tensor}). Yet, via reductions from variants of the planted clique problem, the class of polynomial-time estimators is conjectured to succeed (1) in the binary case $\theta \sim \mathrm{Unif}\{\Theta_k\}$ only when $\lambda=\omega(\lambda^{(B)}_{\mathrm{ALG}})$ for some $\lambda^{(B)}_{\mathrm{ALG}}=\tilde{\Theta}\left(\min \left\{k^{r/2}, \frac{n^{(r-1)/2}}{k^{r/2-1}}\right\}\right)$ and (2) in the Rademacher case $\theta \sim \mathrm{Unif}\{\Theta_k\}$ only when $\lambda=\omega(\lambda^{(R)}_{\mathrm{ALG}})$ for some $\lambda^{(R)}_{\mathrm{ALG}}=\tilde{\Theta}(\min \{k^{r/2}, n^{r/4}\})$  \cite{luo2022tensor}. It is easily checked that $\min\{\lambda^{(B)}_{\mathrm{ALG}},\lambda^{(R)}_{\mathrm{ALG}} \}=n^{\Omega(1)} \lambda_{\mathrm{IT}}$ and hence the model under both priors is conjectured to exhibit a ``computational-statistical'' gap. Notably in both cases---the conjectured optimal polynomial-time estimators are very simple; when $\lambda=\tilde{\omega}(\lambda^{(Q)}_{\mathrm{ALG}})$ for any $Q \in \{B,R\}$, and $k=O(\sqrt{n})$, thresholding the diagonal of $Y$ succeeds \cite{amini2008high,johnstone2009consistency}, while if $k=\Omega(\sqrt{n})$, unfolding the tensor $Y$ combined with a simple spectral estimator succeeds \cite{luo2022tensor}.

\textbf{Underperformance of Local methods} In terms of ``local" methods,  the only relevant works we are aware of are \cite{arous2023free,chen2024lowtemperaturemcmcthresholdcases} which study a natural class of MCMC methods defined for sparse tensor PCA \emph{under the binary prior} $\theta \sim \mathrm{Unif}(\Theta_k)$. They study the performance of Markov chains where the stationary measure is the temperature-misparametrized posterior \[\mu_{\beta}(v) \propto \exp(\beta \langle v^{\otimes r}, Y \rangle), v \in \Theta_k,\]where $\beta>0$ is a parameter the statistician can tune. The locality on the chain is imposed on the Markov chain by allowing transitions from any binary $k$-sparse vector only to binary $k$-sparse vectors of Hamming distance two (i.e., vectors that are one ``swap'' away in terms of their support). Importantly, for large enough $\beta>0$, any such Markov chain run for sufficient time to reach stationarity will (approximately) solve MAP, and hence, recover $\theta$ if $\lambda=\omega(\lambda_{\mathrm{IT}})$. Of course, the question of whether any such local MCMC method would work in polynomial-time is of a different nature, as ``bottlenecks'' could be present \cite{levin2017markov}.  We also highlight that the mentioned class of MCMC estimators includes several natural algorithmic schemes; for example, if one chooses the Metropolis process as the Markov Chain, then for $\beta=+\infty$, the MCMC algorithm reduces to running the so-called ``randomized greedy'' algorithm with the objective being the log-posterior $\langle v^{\otimes t}, Y \rangle$. In each round, the randomized greedy algorithm proposes a move to a Hamming-distance-2 neighbor of the current binary $k$-sparse vector which is ``accepted'' if the objective value $\langle v^{\otimes t}, Y \rangle$ is strictly increasing---perhaps one of the simplest local algorithms to solve for the optimal MAP estimator.

Interestingly, using a bottleneck ``Overlap Gap Property'' approach proposed in \cite{gamarnik2022sparse}, the authors of  \cite{chen2024lowtemperaturemcmcthresholdcases} establish for any $\beta>0$ large enough that any such local Markov chain find $\theta$ in polynomial-time only when $\lambda=\omega(\lambda^{(B)}_{\mathrm{MCMC}})$ for some $\lambda^{(B)}_{\mathrm{MCMC}}=\tilde{\Theta}\left(\min (k^{r - 1/2}, n^{r-1}/k^{r - 3/2})\right)$. A quick check implies in all regimes of $k=o(n)$, that $\lambda^{(B)}_{\mathrm{MCMC}}/\lambda^{(B)}_{\mathrm{ALG}}=n^{\Omega(1)}$, thus a large underperformance of these local MCMC methods compared to other (very simple) polynomial-time estimators succeeding whenever $\lambda=\omega(\lambda^{(B)}_{\mathrm{ALG}}).$ Strangely, this underperformance appears also to be ``structural", satisfying the following also ``geometric rule" identity $\lambda^{(B)}_{\mathrm{MCMC}}/\lambda^{(B)}_{\mathrm{ALG}}=\tilde{\Theta}(\lambda^{(B)}_{\mathrm{ALG}}/\lambda_{\mathrm{IT}})$ in all regimes of $k,r$ (see also \cite{chen2024lowtemperaturemcmcthresholdcases} for a discussion).

\subsection{Main Contribution} In light of the negative findings reported by \cite{chen2024lowtemperaturemcmcthresholdcases}, a natural question arises: can another local method, operating in a natural geometry (such as proxies on the parameter space endowed with Hamming distance), recover the signal $\theta$ in polynomial-time whenever $\lambda=\omega(\lambda^{(B)}_{\mathrm{ALG}})$? Similarly, can local methods succeed when $\lambda=\omega(\lambda^{(R)}_{\mathrm{ALG}})$ in the less studied case of the Rademacher sparse prior?

In this work, we design natural local optimization methods which \emph{provably succeed in polynomial-time to recover the signal whenever $\lambda=\tilde{\omega}(\lambda^{(Q)}_{\mathrm{ALG}})$ for both cases $Q \in \{B,R\}.$} Our methods succeed in all regimes of the sparsity $n^{\Omega(1)}=k \leq n$, any tensor power $r \geq 2$ in the binary case, and any odd $r \geq 3$ in the Rademacher case. We highlight that, in light of the lower bounds of \cite{chen2024lowtemperaturemcmcthresholdcases}, constructing such local methods is a rather non-trivial algorithmic goal, which requires us to propose delicate algorithmic and mathematical innovations in their design and analysis, which could be of potential independent interest.

While our local algorithms differ depending on the parameter regime of $(k,r)$ and the choice of prior they share some \emph{common characteristics} which we describe now. First, each operates on an extended state space; in the binary case the parameter space $\Theta_k$ is replaced with the binary cube $\Theta=\{0,1\}^p$ and in the Rademacher case the parameter space $\widehat{\Theta}_k$ is replaced with the trinary cube $\widehat{\Theta}=\{-1,0,1\}^p.$ We note that this ``lifting'' is a common practice in the literature of improving local methods (see e.g., \cite{yang2016computational, gheissari2023findingplantedcliquesusing} for similar ideas).  We then impose sparsity implicitly by regularizing the log-posterior in varying ``strength'' depending on the regime of interest, i.e., we focus on optimizing 
\begin{align}\label{eq:Ham} H_{\beta,\gamma}(\sigma)= \langle \sigma^{\otimes r}, Y \rangle-\gamma  \|\sigma\|^{\beta}_0.
\end{align}Here $\gamma,\beta>0$ are parameters which we tune depending on the desired intensity of the regularizer. Finally, we endow the binary cube $\Theta$ and the trinary cube $\widehat{\Theta}$ with simple neighborhood graph structures $\mathcal{G}, \widehat{\mathcal{G}}$ where two vectors are connected if and only if they have Hamming distance one. Our local methods will then operate under this natural geometry implied by $\mathcal{G}, \widehat{\mathcal{G}}$; i.e., at each transition, the method will only move between neighboring states in $\mathcal{G}$ or $ \widehat{\mathcal{G}}$. 

Two algorithmic schemes we frequently employ in our methods are (1) a \emph{greedy local search method} on $H_{\beta,\gamma}(\sigma)$, which at every iteration chooses the neighbor of the current state in $\mathcal{G}$ (or $ \widehat{\mathcal{G}}$) that maximizes $H_{\beta,\gamma}(\sigma)$, and (2) a \emph{randomized greedy local search method} on $H_{\beta,\gamma}(\sigma)$, which at every iteration chooses a neighbor of the current state in $\mathcal{G}$ (or $ \widehat{\mathcal{G}}$) uniformly at random and accepts the transition only if the objective value $H_{\beta,\gamma}(\sigma)$ increases along with it.

Our algorithmic results are then as follows.

\textbf{The ``sparser'' case: $k=O(\sqrt{n})$, any $r \geq 2$, and both Binary and Rademacher priors.} In this regime, we prove for $\gamma=\Theta(\sqrt{\log n})$ and $\beta=r$ and any $Q \in \{B,R\}$, when $\lambda=\tilde{\omega}(\lambda^{(Q)}_{\mathrm{ALG}})$  both greedy and randomized greedy local searches optimizing $H_{\beta,\gamma}(\sigma)$ on $ \widehat{\mathcal{G}}$ output the signal in polynomial-time w.h.p.\! as $n \rightarrow +\infty$ (see Theorem \ref{thm:main_sparsegreedy}). Importantly, our results hold only if the methods are initialized carefully by a state $\sigma$ containing only two non-zero non-zero coordinates that both agree with the signal's coordinate, i.e., $\|\sigma\|_0=2, \langle \sigma, \theta \rangle=2$. 
Importantly, such an initialization can be found in $O(n^2)$-time via brute-force search over all $2$-sparse vectors in $\widehat{\Theta}=\{-1,0,1\}^n$. The proof of success for this method follows by a series of careful first moment methods that ensures the dynamics stay in a part of the state space that has no bad local maxima.

\textbf{The ``denser'' case, part 1: $k=\Omega(\sqrt{n})$, any $r \geq 2$, Binary prior} In this more delicate regime, we split two cases to explain our results. If $k=\Theta(n)$, notice that using the all-one vector, $\sigma=\mathbf{1}$ one can achieve ``weak correlation" with $\theta$, i.e. $\langle \sigma, \theta \rangle/(\|\sigma\|_2 \|\theta\|_2)=\Omega(1).$ This is useful for local methods on these values of $k$. We prove that for $\gamma=\Theta(\sqrt{\log n})$ and $\beta=(r+1)/2$, if $\lambda=\tilde{\omega}(\lambda^{(B)}_{\mathrm{ALG}})$, then randomized greedy optimizing $H_{\beta,\gamma}(\sigma)$ on $\mathcal{G}$, initialized at $\sigma=\mathbf{1}$, outputs $\theta$ in polynomial-time w.h.p.\! as $n \rightarrow +\infty$ (see Theorem \ref{thm:DenseBinaryMainTheorem}). Notice, that $\beta$ is chosen so that the regularizer in $H_{\beta,\gamma}(\sigma)$ is ``softer'' than the previous sparser case---this is crucial for our analysis. 

Unfortunately, when $k=o(n)$, choosing $\sigma=\mathbf{1}$ does not achieve weak correlation anymore. For this reason, we instead propose a two-stage local search method in that regime: one local method to achieve weak correlation (see Theorem \ref{thm:GoodInit1}), and one to boost it to exact recovery (see Theorem \ref{thm:DenseBinaryMainTheorem}). Specifically, we first show that if $\lambda=\tilde{\omega}(\lambda^{(B)}_{\mathrm{ALG}})$ the (early stopped) greedy dynamics optimizing $H_{\beta,\gamma}(\sigma)$ on $\mathcal{G}$ with infinite regularizer (e.g., $\gamma=\beta=\infty$), when initialized at $\sigma=\mathbf{1}$, output in polynomial-time a $\sigma'$ that achieves weak correlation $\langle \sigma', \theta \rangle/(\|\sigma'\|_2 \|\theta\|_2)=\Omega(1)$ w.h.p.\! as $n \rightarrow +\infty.$ Then, we prove that if $\lambda=\tilde{\omega}(\lambda^{(B)}_{\mathrm{ALG}})$, the randomized greedy approach described above ($\gamma=\Theta(\sqrt{\log n})$ and $\beta=(r+1)/2$) with initialization $\sigma'$ now recovers $\theta$  w.h.p.\! as $n \rightarrow +\infty.$ 

The proof that the randomized greedy ``boosts" weak correlation to exact recovery follows by an application of a canonical paths argument introduced in \cite{chen2024lowtemperaturemcmcthresholdcases}, combined with a delicate application of the Hanson-Wright inequality. The proof for greedy dynamics, initialized with $\sigma=\mathbf{1}$, achieving weak correlation is based on a coupling argument that carefully leverages and adapt some beautiful ideas used in \cite{FeigeRon, gheissari2023findingplantedcliquesusing} to analyze greedy procedures for the planted clique problem to the Gaussian tensor case.

\textbf{The ``denser" case, part 2: $k=\Omega(\sqrt{n})$, any odd $r \geq 3$, Rademacher prior} For this final regime, we propose a novel two-stage local-search method that provably outputs $\theta$ in almost-linear time when $\lambda=\tilde{\omega}(\lambda^{(R)}_{\mathrm{ALG}})$ for all $k=\Omega(\sqrt{n})$, and for all odd $r \geq 3$ (see Theorem \ref{thm:Trinary}). The first stage is a variant of randomized greedy on $H_{\beta, 0}$ initialized in $\bar \Theta = \{-1,1\}^n$, with $\beta = (r+1)/2$, restricted to transitions of Hamming distance one in $\bar \Theta$. The output of this algorithm is then fed into the second stage algorithm, again a variant of randomized greedy local search on $H_{\beta, \gamma}$ on $\widehat{\mathcal{G}}$ for $\gamma = \Theta(\log n)$ and $\beta$ defined as above. In both stages, the different to ``classical'' randomized greedy methods, is that our proposed Algorithms ``accept'' a transition if the difference in $H_{\beta, \gamma}$ or $H_{\beta,0}$ is larger than a mean-zero \emph{Gaussian random threshold} of appropriate variance, distinct from the ``classical" threshold of zero of folklore randomized greedy. Of additional interest, the first stage requires a key ``homotopy initialization'' $\sigma^0 \in \{-1,1\}^n$ introduced in \cite{pmlr-v65-anandkumar17a}, providing a slightly improved correlation with the signal over a random initialization exactly when the tensor power $r$ is odd---it satisfies $\mathrm{cos}(\sigma^0,\theta) \geq n^{-1/4}.$

We believe that the introduction of these random thresholds per iteration is a novel contribution to the literature of local-search methods in average-case environments, significantly easing the analysis of the dynamics. The key mathematical reason is that with this random thresholds the ``acceptance" step of each iterate becomes independent via a Gaussian cloning argument. 

\begin{remark}One may wonder why we do not offer guarantees for our local search methods in the ``denser" case with Rademacher prior and even $r$. The only reason is that in this case we are not aware of any similar ``warm" and simple initialization like the Homotopy initialization in the case $r$ is odd. We remark though that our analysis implies that for any even $r$, our two-stage local search method with input any a slightly warmer-than-random initialization (i.e., any $\sigma^0 \in \{-1,1\}^n$ with $\mathrm{cos}(\sigma^0,\theta) \geq n^{-1/4}$) outputs $\theta$ in almost linear time.
\end{remark}

\subsection{Further Comparison with Previous Work}
On top of the relevant works mentioned above, we expand on four more. The sparse tensor PCA model we study, when the prior is Bernoulli sparse, $r=2$ and $\lambda=k$ can be interpreted as a ``Gaussian" version of the planted clique model introduced in \cite{jerrum1992large}. Indeed, the major difference between the models is that planted clique has an adjacency matrix containing Bernoulli noise entries, while in the sparse matrix PCA setting, the noise entries are Gaussian. For the planted clique model, two greedy methods have been analyzed in \cite{FeigeRon} and \cite{gheissari2023findingplantedcliquesusing} that remarkably work down to the computational threshold which is $k=\Theta(\sqrt{n})$. In \cite{FeigeRon}, the authors analyze a two-stage method, where the first part is a greedy local search and the second is a combinatorial post-processing step, while in \cite{gheissari2023findingplantedcliquesusing} the authors prove that a slightly modified greedy local procedure suffices by itself without a second step. Our results for a binary signal with $k=\Omega(\sqrt{n})$, specifically Algorithm \ref{alg:GreedyPeeling1} and its analysis, are built upon the aforementioned works and generalize them to provide local methods that work for the sparse tensor PCA case down to its computational threshold for any $r \geq 2$ and for any $k=\Omega(\sqrt{n})$.

For the Rademacher sparse prior, much less is known even for specific subcases. The only relevant work on local methods we are aware applies in fact only to the non-sparse $k=n$ and matrix case $r=2$ where \cite{liuLocallyStationary} used stochastic localization (SL) ideas to prove that Glauber dynamics on the hypercube $\{-1,1\}^n$ sampling from $\mu_{\beta}$ for appropriate $\beta>0$ (weakly) recovers the signal in polynomial-time at the computational threshold of the model. While a very interesting work, unfortunately no extension of the used SL techniques is currently known to either the sparse or tensor cases. We highlight that our work instead considers all odd $r\geq 3$, and any sparsity level $k\geq\sqrt{n}$ (including $k=n$), and we propose a two-stage randomized greedy procedure (with random thresholds) which succeeds down to the computational threshold for all such parameters. 

Moreover, as we mentioned above, we are not aware of any local method in the literature which accepts possible transitions based on a random threshold---hence, we consider this an algorithmic novelty of the present work. Yet, we remark that the origin of our suggestion lies on an existing concept in the literature to ``inject" noise in the iterates of iterative algorithms to break the dependencies between them. This idea has been for example applied in the analysis of power method for tensor PCA in \cite{pmlr-v65-anandkumar17a}, which importantly without this noise injection step underperforms. That being said, our scheme for injecting noise to the iterates is rather elaborate to accommodate the coordinate-by-coordinate changes of our local search methods. For this reason, we hope that it will be useful in the analysis of other local methods in the statistics and average-case analysis literature.

\subsection{Notation}
We write $d_{1}(x,y)=\lVert x - y\rVert_{1}$ for the
$\ell_{1}$--distance. The support size of a vector $v$ is $\|v\|_{0}$, i.e.\ the
number of its nonzero entries.  For $v\in\mathbb R^{n}$ we denote by $v^{\otimes
r}\in\mathbb R^{n^{r}} = \RR^{n^{\otimes r}}$ the $r$--fold tensor power, whose
entries satisfy $ (v^{\otimes r})_{i_{1}\dots i_{r}} =v_{i_{1}}\cdots
v_{i_{r}}$. We use $\lceil x\rceil$ and $\lfloor x\rfloor$ for the usual ceiling
and floor operation.  The notation $\plog$ stands for an arbitrary polylogarithmic factor in
$n$.  We employ the standard asymptotic notation $O(\cdot)$, $o(\cdot)$,
$\Theta(\cdot)$, $\Omega(\cdot)$, and $\omega(\cdot)$ as $n\to\infty$. 
We further write $g(n) = \Tilde O(f(n))$, if $g(n)=O\left(f(n)\plog\right)$ for some $\plog$ factor, and similarly for $\Tilde o$, $\Tilde \Theta$, $\Tilde \Omega$, and $\Tilde \omega$. Given $v\in\mathbb R^{d}$ and
$\mathcal S\subseteq[d]$, we denote by $v_{\mathcal S}=(v_{i})_{i\in\mathcal S}$
the corresponding subvector.  The indicator of an event $\cA$ is $\1\{\cA\}$.
We write $\Id_{n}$ for the $n\times n$ identity matrix. For two vectors
$v,v'\in\mathbb R^{n}$, their cosine similarity is
$\cos(v,v')=\frac{v^{\top}v'}{\|v\|_{2}\,\|v'\|_{2}}$. The notation $\laweq$
denotes equality in law, and $\stlow$ for stochastic domination in the upper tail (i.e. for any random variables $X_1,X_2$, we say $X_1\stleq X_2$ if for all $x\in\mathbb{R}$, $\mathbb{P}(X_1\geq x)\leq \mathbb{P}(X_2\geq x)$).  Finally,
$\Phi(x)=\frac1{\sqrt{2\pi}}\int_{-\infty}^{x}e^{-t^{2}/2}\,dt$ denotes the
standard Gaussian CDF.  

\section{Main Results}\label{sec:MainResults}
 In this section, we formally present our results.

\subsection{Sparser \texorpdfstring{$k=O(\sqrt{n})$}{k=O(sqrt(n))}, all tensor powers and both priors regime}
In this section we consider the model \eqref{eq:model} under the assumption that $\omega(1) = k \leq C\sqrt{n}$ for some (arbitrary) constant $C>0$ and any tensor power $r \geq 2$. We also propose two local search algorithms that succeed for any (worst-case) signal value $\theta \in \widehat{\Theta}_k$; in particular they apply under both cases of the binary sparse prior $\theta \sim \mathrm{Unif}(\Theta_k)$ and the Rademacher sparse prior $\theta \sim \mathrm{Unif}(\widehat{\Theta}_k)$. For a fixed value of $\gamma > 0$, we consider the objective $H_{r, \gamma},$ i.e., the one defined in \eqref{eq:Ham} for $\beta=r$.
We then study greedy local search (Algorithm \ref{alg:SparseGreedy2}), and randomized greedy local search (Algorithm \ref{alg:SparserandrandGreedy}), with objective value $H_{r,\gamma}$ on the neighboring graph $\widehat{G}$ (i.e., neighbors are of Hamming distance 1) defined on $\widehat{\Theta}=\{-1,0,1\}^n$.  

\begin{algorithm}
\caption{Greedy Local Search Algorithm for $H_{r,\gamma}$}\label{alg:SparseGreedy2}
 \begin{algorithmic}[1]
    \REQUIRE $Y\in\mathbb{R}^{n^{\otimes r}}$, $S^0$, $\gamma>0$
    \REPEAT
      \STATE $\mathcal{N}(S^t)\gets\{\sigma':d_H(\sigma,\sigma')=1\}$
      \STATE $\sigma_{\text{loc}}\gets\arg\max_{\sigma\in\mathcal{N}(S^t)}H_{r,\gamma}(\sigma)$
      \IF{$H_{r,\gamma}(\sigma_{\text{loc}})>H_{r,\gamma}(S^t)$}
        \STATE $S^{t+1}\gets\sigma_{\text{loc}}$
      \ELSE
        \STATE \textbf{terminate} and output $S^t$
      \ENDIF
    \UNTIL{terminated}
    \end{algorithmic}
\end{algorithm}

\begin{algorithm}
  \caption{Randomized Greedy Local Search Algorithm for $H_{r, \gamma}$}\label{alg:SparserandrandGreedy}
  \begin{algorithmic}[1]
    \REQUIRE $Y\in\mathbb{R}^{n^{\otimes r}}$, $\sigma_0$, $\gamma>0$, $M\in\mathbb{N}$
    \STATE $t\gets1$, $S^1\gets\sigma_0$
    \WHILE{$t<M$}
      \STATE $\mathcal{N}(S^t)\gets\{\sigma':d_H(\sigma,\sigma')=1\}$
      \STATE Pick $\sigma'$ uniformly at random from $\mathcal{N}(S^t)$
      \IF{$H_{r,\gamma}(\sigma')>H_{r,\gamma}(S^t)$}
        \STATE $S^{t+1}\gets\sigma'$, $t\gets t+1$
      \ENDIF
    \ENDWHILE
    \RETURN $S^{t+1}$
  \end{algorithmic}
\end{algorithm}
\noindent Our main result for this section is as follows.

\begin{thm}\label{thm:main_sparsegreedy}
 Consider an arbitrary $\theta \in \{-1,0,1\}^n, \|\theta\|_0=k$ and assume $\omega(1)=k\leq C\sqrt{n}$ for some constant $C>0$ and any tensor power $r \geq 2.$ Moreover, assume that $S^0 \in \{-1,0,1\}^n$ satisfies $\|S^0\|_0=\langle S^0, \theta \rangle =2$.

 Then for any $\lambda =\omega( k^{r/2} \sqrt{\log n})$ if we set $\gamma=C_\gamma\sqrt{\log n}$ for some sufficiently large constant $C_\gamma=C_{\gamma}(r)>0,$ the following event holds with probability at least $1-n^{-1/2}$. Algorithm \ref{alg:SparseGreedy2} with input $(Y,S^0)$ outputs $\theta$ in $k-2$ iterations and Algorithm \ref{alg:SparserandrandGreedy} with input $(Y,S^0,M)   $ for $M=\Theta (n\log n)$ outputs $\theta$.
 \end{thm} 
Using the formulas in Section \ref{sec:lit} note that if $k=O(\sqrt{n})$ both $\lambda^{(R)}_{\mathrm{ALG}}$ and $\lambda^{(B)}_{\mathrm{ALG}}$ are of the order $\tilde{\Theta}(k^{r/2}).$ Hence, Theorem \ref{thm:main_sparsegreedy} implies that down to the conjectured computational threshold (up to polylogarithmic terms) for both priors, as long as we can find an $S^0 \in \{-1,0,1\}^n$ satisfying $\|S^0\|_0=\langle S^0, \theta \rangle =2$, both the greedy and randomized greedy local search method optimizing $H_{r,\Theta(\sqrt{\log n})}$ initialized with $S^0,$ find the signal $\theta$, w.h.p. as $n \rightarrow +\infty.$ 
    
We now remark on why a desired $S^0$ can be easily found in polynomial-time and via local methods.

\begin{remark}\label{rem:findinggoodinit}["Finding the correct initialization"]
Recall we need to find an $S^0$ with $\|S^0\|_0=\langle S^0, \theta \rangle =2$ to obtain guarantees after we run either of the algorithms. Note that this can be circumvented by enumerating all $O(n^2)$ vectors $\sigma \in \{-1,0,1\}^n$ with $\|\sigma\|_0=2$ and running $O(n^2)$ times either of the algorithms, each time with $S^0=\sigma$. Theorem \ref{thm:main_sparsegreedy} guarantees that one of these $O(n^2)$ runs outputs $\theta$ w.h.p. as $n \rightarrow +\infty$. It is also easy to check in polynomial-time and only using ``local information" whether the output of each algorithm equals $\theta$. This is because $\theta$ is the unique local maximum that either algorithm can produce (see Section \ref{sec:sparsecase} for a detailed explanation). Combining the above leads to a polynomial-time local method at the conjectured optimal $\lambda$, under both priors when $k=O(\sqrt{n})$.
\end{remark}

\subsection{Denser \texorpdfstring{$k=\Omega(\sqrt{n})$}{k=Omega(sqrt(n))}, all tensor powers, and Binary prior regime.}
Here we focus on the binary sparse prior $\theta \sim \mathrm{Unif}(\Theta_k)$, under the assumption $k=\Omega\left(\sqrt{n}\right) $ and any $r \geq 2$. Again from Section \ref{sec:lit}, in this regime the conjectured computational threshold is $\lambda^{(B)}_{\mathrm{ALG}}= \Tilde \Theta\left( n^{\frac{r-1}{2}}/k^{\frac{r}{2}-1}\right).$ As mentioned above, we propose a two-stage local-search method to find $\theta$, the first to obtain weak correlation\footnote{Defined as finding an $\sigma$ with $\langle \sigma,\theta \rangle/(\|\sigma\|_2 \|\theta\|_2)=\Omega(1).$}, and the second to boost weak correlation to exact recovery.

Notice that if $k=\Theta(n)$ it is trivial to achieve weak correlation by just outputting the all-one vector $\mathbf{1}$. Unfortunately, that ceases to be true when $k=o(n)$. Our first algorithm shows that running a (very aggressive) greedy local search process respecting the geometry of the graph $\mathcal{G}$ (i.e., Hamming distance one neighbors) on $\Theta=\{0,1\}^n,$ initialized with $\mathbf{1}$ allows us to still obtain weak correlation w.h.p. as $n \rightarrow +\infty.$

\begin{algorithm}
\caption{Greedy~Local Search}\label{alg:GreedyPeeling1}
 \begin{algorithmic}[1]
    \REQUIRE $Y \in \mathbb{R}^{n^{\otimes r}}$
\STATE $Q\gets \max\{Y,0\}$ (elementwise)
\STATE Let $P^0 \gets \mathbf{1}\in\mathbb{R}^n$
\FOR{$t = 0$ to $n-\frac{3}{2}k-1$}
    \STATE \makebox[\linewidth][l]{Find $\ell~\gets~\argmin\limits_{i: P^t_i = 1}~\langle~e_i~\otimes (P^t)^{\otimes r-1},~Q~\rangle$}
    \STATE Set $P^{t+1}_\ell = 0$
\ENDFOR
\RETURN $P^{n - \frac{3}{2}k}$
    \end{algorithmic}
\end{algorithm}
\begin{remark}
We note that $\langle e_i \otimes (P^t)^{\otimes r-1}, Q \rangle$ in line 4 of Algorithm \ref{alg:GreedyPeeling1} is a first-order approximation to $\langle (P^t)^{\otimes r},Y\rangle-\langle (P^t-e_i)^{\otimes r},Y\rangle$. Thus, Algorithm \ref{alg:GreedyPeeling1} should be understood as a proxy of Greedy Local Search on optimizing $H_{\infty,\infty}(\sigma)$ defined in \eqref{eq:Ham} when $Y$ is replaced by $Q=\max\{Y,0\}$, i.e., truncated at zero. We remark that we do both the first-order approximation and the truncation at zero for technical reasons and we don't think either are fundamental. 
\end{remark}

\begin{algorithm}
\caption{\rlap{Norm-constrained Randomized Greedy Local Search on $H_{(r+1)/2, \gamma}$}}\label{alg:RandGreedy2}
 \begin{algorithmic}[1]
    \REQUIRE $Y \in \mathbb{R}^{n^{\otimes r}}$, $\sigma_0$ with $\norm{\sigma_0}_0\leq \frac{3}{2}k$, $M\in\NN$.
    
    \STATE Let $t \gets 1$ and $S_1 \gets \sigma_0$
    \WHILE{$t<M$}
    \STATE Let $\mathcal{N}(S_t) = \{\sigma' \mid d_H(S_t, \sigma') = 1\}$\\ \hspace{3cm}$
    \cap \{\sigma\in\{0,1\}^n:\norm{\sigma}_0\leq \frac{3}{2}k\}$
        \STATE Choose a random element $\sigma' \in \mathcal{N}(S_t)$ uniformly
        \IF{$H_{\frac{r+1}{2},\gamma}(\sigma') - H_{\frac{r+1}{2},\gamma}(S_t) > 0$} 
            \STATE $S_{t+1} \gets \sigma'$
            \STATE $t \gets t + 1$
        \ENDIF
    \ENDWHILE
    \STATE \textbf{return} $S_M$
    \end{algorithmic}
\end{algorithm}
When~$k~\leq~2n/3$~we~obtain~the~following~weak~correlation~guarantee~for~Algorithm~\ref{alg:GreedyPeeling1}~when~$\lambda=\tilde{\omega}(\lambda^{(B)}_{\mathrm{ALG}}).$

\begin{thm} \label{thm:GoodInit1} Suppose $\theta \sim \mathrm{Unif}(\Theta_k)$, $\Omega\left(\sqrt{n}\right)= k\leq \frac{2}{3}n$ and any $r\geq 2$. If $\lambda=\tilde{\Omega}\left(n^{\frac{r-1}{2}}/k^{\frac{r}{2}-1}\right)$
then, w.h.p. as $n \rightarrow +\infty$, Algorithm \ref{alg:GreedyPeeling1} with input $Y$ outputs a vector $\sigma\in\{0,1\}^n$ satisfying $\norm{\sigma}_0\leq \frac{3}{2}k$ and $\langle \sigma,\theta\rangle \geq\frac{1}{8}k$.
\end{thm}

Now, we propose Algorithm \ref{alg:RandGreedy2} which is a local algorithm that can boost weak correlation to exact recovery, again down to the conjectured computational threshold. The procedure of this algorithm is simply to run a norm-constrained randomized greedy algorithm on $H_{\beta, \gamma}$, defined in \eqref{eq:Ham}, for $\beta=(r+1)/2$. The dynamics are again under the same geometry $\mathcal{G}$ on $\Theta=\{0,1\}^n.$ We refer to it as norm-constrained because it restricts the 0-norm of the iterates $S_t$ to be at most $\frac{3}{2}k$ at all times. 

We obtain the following theorem for its performance. 
\begin{thm}\label{thm:DenseBinaryMainTheorem}
    Assume that $\Omega(\sqrt{n})\leq k\leq n$, any $r\geq 2$ and $\theta \sim \mathrm{Unif}(\Theta_k)$. Suppose $\sigma_0\in\{0,1\}^n$ satisfies $\langle \sigma_0,\theta\rangle\geq \frac{1}{8}k$ and $\norm{\sigma}_0\leq \frac{3}{2}k$. Then for some $\gamma {=}\Theta\left(\sqrt{\log n}\right)$ if $M{=}\Theta\left(nk^{(r+1)/2}\right)$ and $\lambda{=}\tilde{\omega}\left( n^{\frac{r-1}{2}}/k^{\frac{r}{2}-1}\right)$, Algorithm \ref{alg:RandGreedy2} with input $(Y,\sigma_0,M)$ outputs $\theta$ with probability at least $2/3$.
\end{thm}
We note that the above algorithm holds with constant probability $2/3$, which is chosen arbitrarily, and can be made arbitrarily close to $1$.

 \subsection{Denser \texorpdfstring{$k=\Omega(\sqrt{n})$}{k=Omega(sqrt(n))}, odd \texorpdfstring{$r \geq 3$}{r >= 3} tensor power, and Rademacher prior regime}\label{sec:Rez2}
 Here we focus on the Rademacher sparse prior $\theta \sim \mathrm{Unif}(\widehat{\Theta}_k)$ under the assumption $k=\Omega\left(\sqrt{n}\right)$ and $r \geq 3$ odd. Based on Section \ref{sec:lit}, the conjectured computational threshold is $\lambda^{(R)}_{\mathrm{ALG}}=\tilde{\Theta}(n^{r/4})$ for this regime.
Below, we present two randomized greedy algorithms with our novel suggestion of random thresholds, Algorithm
\ref{alg:A0} and Algorithm \ref{alg:B0}, using the Hamiltonian $H_{\frac{r+1}{2}, \gamma}$ and $H_{\frac{r+1}{2}, 0}$ respectively.

\begin{algorithm}
\caption{Sub-routine Of Generating Random Thresholds}\label{alg:Inject_RT}
 \begin{algorithmic}[1]
    \REQUIRE $M \in \NN$
\FOR{\(i = 1 \) to \(n\)}
    \STATE Sample \(G_i^1,\dots,G_i^{\runtime}\) i.i.d. from $\cN(0,\runtime)$.
    \STATE \(\displaystyle \bar{G}_i \gets \frac{1}{M}\sum_{j=1}^{M} G_i^j\).
    \FOR{\(t = 1\) to \(M \)}
    \STATE \(\displaystyle Z_{i,t} \gets G^t_i-\bar{G}_i\).
    \ENDFOR
\ENDFOR
\RETURN $((Z_i^t)_{i \in [n], t \in [M]})$
    \end{algorithmic}
\end{algorithm}

\begin{algorithm}
\caption{Randomized Greedy Local Search on $H_{(r+1)/2, 0}$ with random thresholds}\label{alg:B0}
 \begin{algorithmic}[1]
    \REQUIRE \(Y \in \mathbb{R}^{n^{\otimes r}}\), \(S_0 \in \{-1,1\}^n\), $M\in\mathbb{N}$
\STATE \(t \gets 0\) and \(t_i \gets 0\) for each \(i \in [n]\).\\
\STATE \makebox[\linewidth][l]{Let $((Z^t_i)_{i \in [n], t \in [M]})$ be the output of Algorithm \ref{alg:Inject_RT} with input $M$}
\vspace{-1em}

\WHILE{$\max_{i\in [n]} t_{i} \leq M$}
\STATE $\mathcal{N}(S_t) \gets \{\sigma' \in \{-1,1\}^n : d_H(S_t, \sigma') = 1\}$
\STATE Choose a uniform random $\sigma' \in \mathcal{N}(S_t)$
\STATE Let $i \in [n]$ be such that $\sigma'_i \neq (S_t)_i$.
\STATE \(t_{i} \gets t_{i} + 1\), \(t \gets t + 1\)
    \IF{\(H_{\frac{r+1}{2},0}(\sigma') - H_{\frac{r+1}{2},0}(S_t) > 
      \|(\sigma')^{\otimes r} - (S_t)^{\otimes
      r}\|_F Z_i^{t_i}\) }
        \STATE \(S_{t+1} \gets \sigma'\) 
    \ENDIF
\ENDWHILE

\RETURN \(S_{t + 1}\)
    \end{algorithmic}
\end{algorithm}

\begin{algorithm}
    \caption{Randomized Greedy Local Search on $H_{(r+1)/2, \gamma}$ with random thresholds}\label{alg:A0}
    \begin{algorithmic}[1]
    \REQUIRE \(Y \in \mathbb{R}^{n^{\otimes r}}\), \(S_0 \in \{-1,0,1\}^n\), $\gamma \geq 0, M \in \mathbb{N}$
    \STATE \(t \gets 0\) and \(t_i \gets 0\) for each \(i \in [n]\).
    \STATE Let $((Z^t_i)_{i \in [n], t \in [M]})$ \text{be the output of Algorithm \ref{alg:Inject_RT} with input $M$} 
    \WHILE{$t \leq M$}
    \STATE Let $\mathcal{N}(S_t) = \{\sigma' \in \{-1,0,1\}^n \mid d_H(S_t, \sigma') = 1\}$
        \STATE Choose a random element $\sigma' \in \mathcal{N}(S_t)$ uniformly
         \STATE Let $i \in [n]$ be such that $\sigma'_i \neq (S_t)_i$.
         \STATE \(t_{i} \gets t_{i} + 1\), \(t \gets t + 1\)
        \IF{$H_{\frac{r+1}{2}, \gamma}(\sigma') - H_{\frac{r+1}{2}, \gamma}(S_t) > \|(\sigma')^{\otimes r}-(S_t)^{\otimes r}\|_F Z^{t_i}_i$} 
            \STATE $S_{t+1} \gets \sigma'$
        \ENDIF
    \ENDWHILE
    \STATE \textbf{return} $S_{M+1}$
    \end{algorithmic}
\end{algorithm}

We also also employ the ``warm'' Homotopy
initialization $\sigma^{\rm HOM}$ previously used in
\cite{pmlr-v65-anandkumar17a}. Due to technical reasons, this initialization is
only implementable when $r \geq 3$ is odd, and this is the reason we restrict to this case when $\theta \in \{-1,0,1\}^n$ with $k = \Omega(\sqrt{n})$. This analysis of this initialization is vital to proving the main result of this Section.

\begin{definition}\label{def:HOM} For odd $r \in \NN$, let $S_{\rm HOM} \in
    \{-1,1\}^n$ be
    \[(S_{\rm HOM})_i = 2\cdot \1\left\{\sum\nolimits_{j_1, \dots, j_{(r-1)/2} = 1}^n
    Y_{i, j_1, j_1, \dots, j_{(r-1)/2}, j_{(r-1)/2}} \geq 0\right\} - 1.\]
\end{definition}

Our guarantee for our two-stage algorithm with random thresholds is as follows. We remark that the algorithm provably finds the signal down to the computational threshold, in an almost-linear time.
\begin{thm}\label{thm:Trinary} 
Consider an arbitrary $\theta {\in} \{-1,0,1\}^n$, $\|\theta\|{=}k$. Assume $k {\geq} \sqrt{n}$ and let $r {\geq} 3$ be any odd tensor power. Let $S_0 {=} S_{\rm HOM}$.
If $M_1 {=} \Theta(\log^{4} n)$, $M_2 {=} \Theta(\log n)$, $\gamma_{M_2}=\gamma=\Theta\left(\log n\right)$  and $\lambda_\gamma = \lambda = \Tilde
        \Omega(n^{r/4})$, then w.h.p. as $n \conv{} +\infty$, if $S_1$ is the output of Algorithm \ref{alg:B0} with input $(Y, S_0, M_1)$, then Algorithm \ref{alg:A0} with input $(Y, S_1, \gamma, M_2)$ outputs $\theta$ in $\Theta(n \log^4 n)$ iterations.
\end{thm}

\section{Simulations}\label{sec:simul}
\subsection{Simulation Specifications}\label{sec:simul_spec}
In this section, we provide simulations verifying for small $n$ a subset of our positive results, but also discussing possible extensions. Given the simplicity of our successful algorithms and their guarantees in the case $k=O(\sqrt{n})$, we stick to the case $k=\Omega(\sqrt{n}).$ 
\subsection{Binary Signal: \texorpdfstring{$\theta\in\{0,1\}^n$}{theta in \{0,1\}n} and \texorpdfstring{$k=\Omega(\sqrt{n})$}{k=Omega(sqrt(n))}}\label{sec:BinaryRandGreedy}
We focus on the randomized greedy approach described in Algorithm \ref{alg:RandGreedy2}. However, we remove the constraint on the norm (i.e. we no longer enforce that $\norm{S_t}_0\leq \frac{3}{2}k$ at all times $t$) in order to demonstrate our conjecture that the norm constraint is merely necessary as a proof device. We therefore simulate  Algorithm \ref{alg:RandGreedyBinaryUnconstrained}, which is identical to Algorithm \ref{alg:RandGreedy2} but with the constraint on the norm removed.

We focus on two initializations: the all ones vector, and the uniform initialization over the $k$-sparse binary vectors, in order to demonstrate that the initialization makes a crucial difference. In each case, we simulate for various different values of $\lambda$, and track the cosine of the angle of our current state with $\theta$ over time for each choice of $\lambda$. We choose the order of the tensor $r$ to be 3, $n=150$, $k=22\approx n^{0.6}$, and $\gamma=\sqrt{\log n}$. In this case, the algorithmic threshold is predicted to be $\lambda^{(B)}_{\mathrm{ALG}}=\Tilde \Theta\left(n^{(r-1)/2}/k^{r/2-1}\right)=\Tilde \Theta\left(n/\sqrt{k}\right)\approx n^{0.7}$. Indeed, in Figure \ref{fig:BinaryTensorRandGreedyOnesInit} where we use the all ones initialization, we see the success of the algorithm down to $\lambda \approx n^{0.7}$. We remark that this is the same threshold as in our Theorem \ref{thm:DenseBinaryMainTheorem}, which requires a warm initialization and a norm constraint on the algorithm for the theoretical guarantees.

Now, the randomized greedy algorithm with a uniform $k$-sparse initialization resembles a ``local" MCMC method run on the (non-extended) parameter space of $k$-sparse vectors which given the results of \cite{chen2024lowtemperaturemcmcthresholdcases} is conjectured to work only if $\lambda>\lambda^{(B)}_{\mathrm{MCMC}}=\Tilde\Theta\left(n^{r-1}/k^{r-3/2}\right)=\Tilde\Theta\left(n^2/k^{1.5}\right)\approx n^{1.1}$. Satisfyingly, in Figure \ref{fig:BinaryTensorRandGreedySparseInit} we see exactly this suboptimal performance under this initialization. 

\begin{figure}[ht]
\miniFigSideLeftArxiv{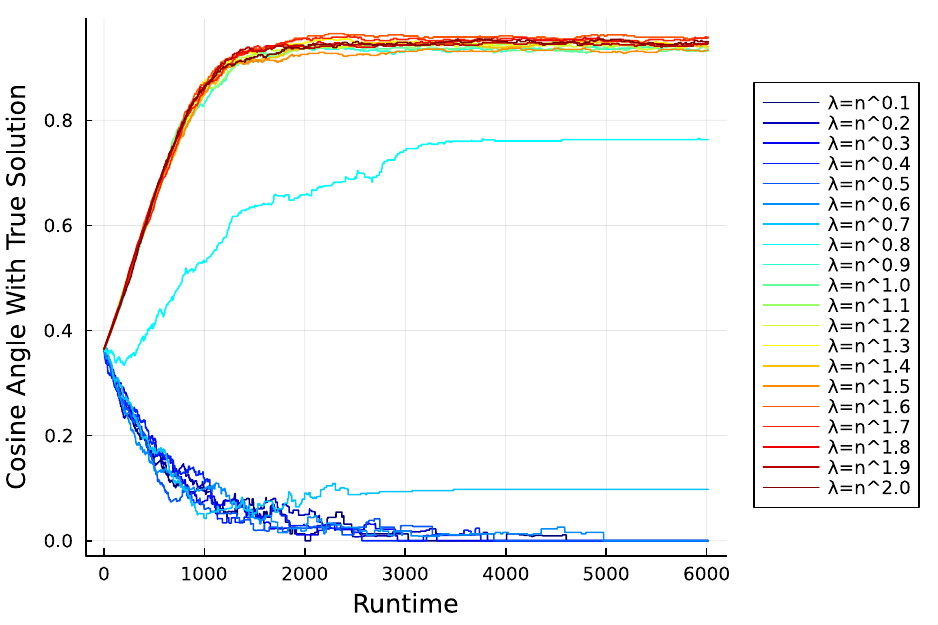}{Randomized greedy for Sparse 3-Tensor PCA when $\theta\in\{0,1\}^n$. Mean angle vs time for $\lambda = n^{\alpha}$, initialized at the all ones vector. We predict that $\alpha=0.7$ is the threshold for fast recovery. Here $n=150$, $k=22\approx n^{0.6}$, and $\gamma =\sqrt{\log n}$.}{BinaryTensorRandGreedyOnesInit}
\hfill
\miniFigSide{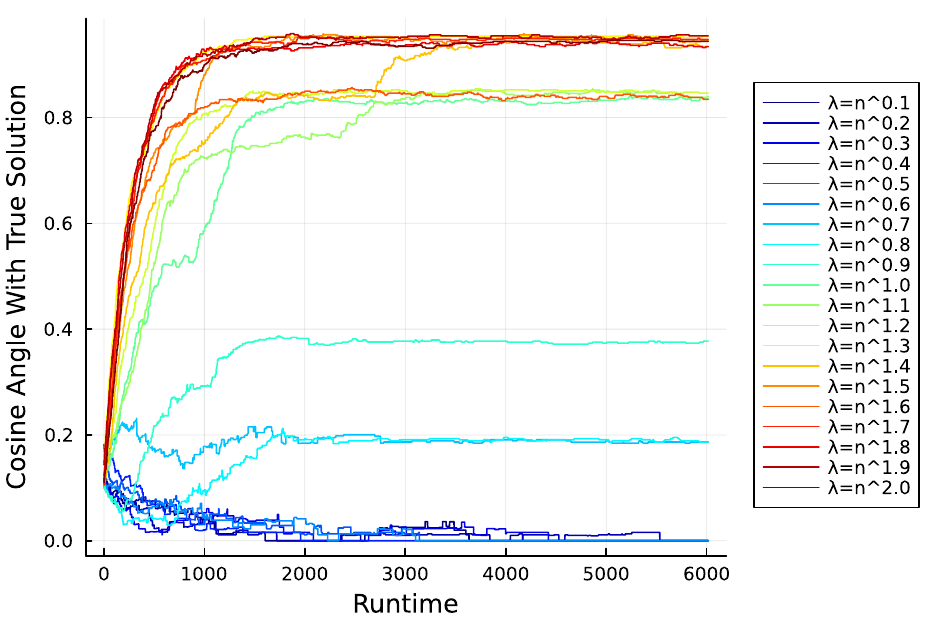}{Randomized greedy for Sparse 3-Tensor PCA when $\theta\in\{0,1\}^n$. Mean angle vs time for $\lambda = n^{\alpha}$, initialized at a random $k$-sparse binary vector. We predict that $\alpha=1.1$ is the threshold for fast recovery. Here $n=2000$, $k=96\approx n^{0.6}$, and $\gamma =\sqrt{\log n}$.}{BinaryTensorRandGreedySparseInit}
\end{figure}

\subsection{Trinary Signal: \texorpdfstring{$\theta\in\{-1,0,1\}^n$}{theta in
\{-1,0,1\}n} and \texorpdfstring{$k=\Omega(\sqrt{n})$}{k=Omega(sqrt(n))}}\label{sec:simulTer}
Firstly, we simulate the randomized greedy approach specified in Algorithm \ref{alg:RandGreedyTrinaryUnconstrained}. It is identical to Algorithm \ref{alg:A0}, but with the random thresholds instead being set to zero, since we consider the random thresholds to be a proof device in this case. We use two different initializations: a uniform
initialization over all $\{-1,0,1\}^n$, and the homotopy initialization
$S_{\mathrm{HOM}}$ defined in Definition \ref{def:HOM}. We again choose the order of the tensor $r$ to be 3, and set $n=150$ and $k=56\approx n^{0.8}$, with $\gamma= \log n$. In this parameter regime, the algorithmic threshold is predicted to be
$\lambda^{(R)}_{\mathrm{ALG}}=\Tilde \Theta\left(n^{r/4}\right)\approx
n^{0.75}.$  In Figure \ref{fig:TrinaryTensorRandGreedySparseInit} we see that
the threshold of randomized greedy for $\lambda$ appears to be $n^{1.4}$ for the
uniform initialization, and in Figure
\ref{fig:TrinaryTensorRandGreedyHomotopyInit} to be $n^{0.9}$ for the homotopy
initialization. While we again see a clear boost from choosing a warmer
initialization, we see that vanilla randomized greedy fails to match the
computational threshold. Hence, a more proficient local method seems to be
needed in this case.

This further motivates our two-stage procedure for this
case, which runs Algorithm \ref{alg:RandGreedyTrinaryUnconstrained} after a run of Algorithm \ref{alg:RandGreedyStage1Appendix},
which is a randomized greedy approach that only flips signs of the coordinates
along its trajectory. Algorithm \ref{alg:RandGreedyStage1Appendix} functions identically to Algorithm \ref{alg:B0}, but without the random thresholds. We initialize this two-stage simulation at $S_{\mathrm{HOM}}$ and see in Figure
\ref{fig:TrinaryTensorTwoStageHomotopyInit} that the two-stage method does
appear to succeed indeed down to the computational threshold $\lambda  \approx
n^{0.75}.$ This makes a strong case for the algorithmic value of the proposed
two-stage method.

Additionally, the results (and proofs) of Lemma \ref{lem:boost03} and Lemma \ref{lem:boost01} suggest an interesting two phase evolution for the trajectory of Algorithm \ref{alg:B0} when $\lambda$ is near the algorithmic threshold, i.e., when $\lambda=\tilde{\Theta}(n^{r/4})$: (1) first there is an oscillatory phase where the cosine of the angle between the iterates and the signal moves as a function of time like a biased random walk with (an increasingly) positive bias and (2) second, there is  ``monotonic" phase where the cosine of the angle becomes an non-decreasing function of the time. Based on our Lemmas, the first phase should hold exactly until the cosine of the angle hits a specific threshold, which for $\lambda=\tilde{\Theta}(n^{r/4})$ becomes $\cos(S_t, \theta) \leq f(\alpha)=n^{(2\alpha - 3)/(4(r-1))}$ for $k = \Theta(n^\alpha)$. Strikingly, when we simulated 400 times the randomized greedy algorithm but \emph{without }the random thresholds (i.e., Algorithm \ref{alg:RandGreedyStage1Appendix}) at the algorithmic threshold $\lambda=\tilde{\Theta}(n^{r/4})$ we see qualitatively the same two-phase behavior even for $r=3,n=150, k=116 \approx n^{0.95}$: over 380 of our 400 simulations of the Algorithm \ref{alg:RandGreedyStage1Appendix} for these parameters values exhibit first an oscillatory phase, followed by a monotonic phase. In Figure \ref{fig:TrinaryTensorOneStage}, we plot 29 distinct simulations (filtered from the 400 total simulations) that exhibit the ``most extreme" oscillatory phase so that it is visible at that level of granularity. The dashed black line represents the value of $f(\alpha) \approx 0.502$ for these values of parameters that the phase transition between the two phases is expected to take place for large $n$ by analogy with our analysis of Algorithm \ref{alg:B0}.

\begin{figure}[ht]
\miniFigSideLeftArxiv{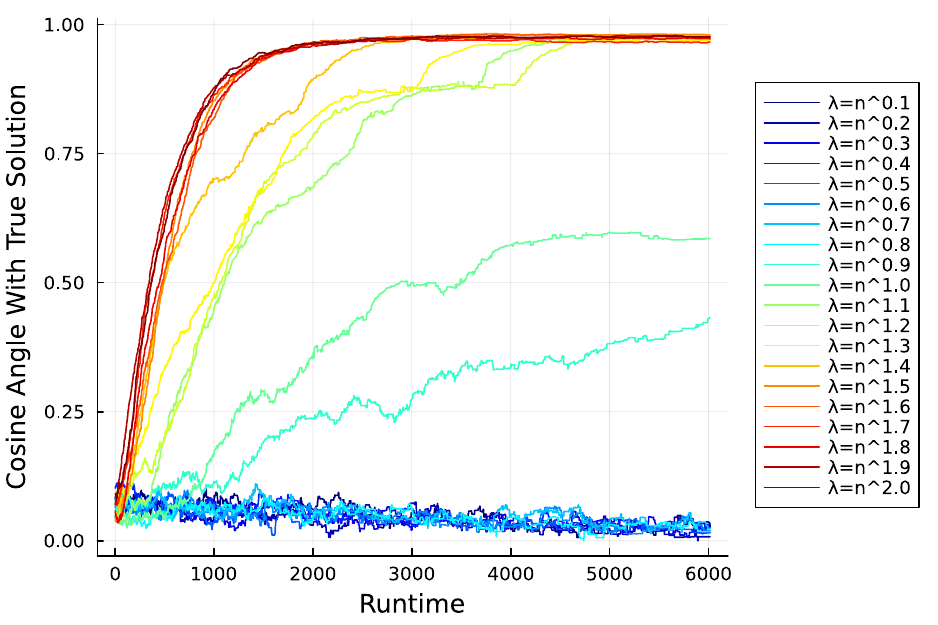}{Randomized greedy for Sparse 3-Tensor PCA when $\theta\in\{-1,0,1\}^n$. Mean absolute angle vs time for $\lambda = n^{\alpha}$, initialized at a uniform random trinary vector. We predict that $\alpha=1.4$ is the threshold for fast recovery. Here $n=150$, $k=56\approx n^{0.8}$, and $\gamma =\log n$.}{TrinaryTensorRandGreedySparseInit}
\hfill
\miniFigSide{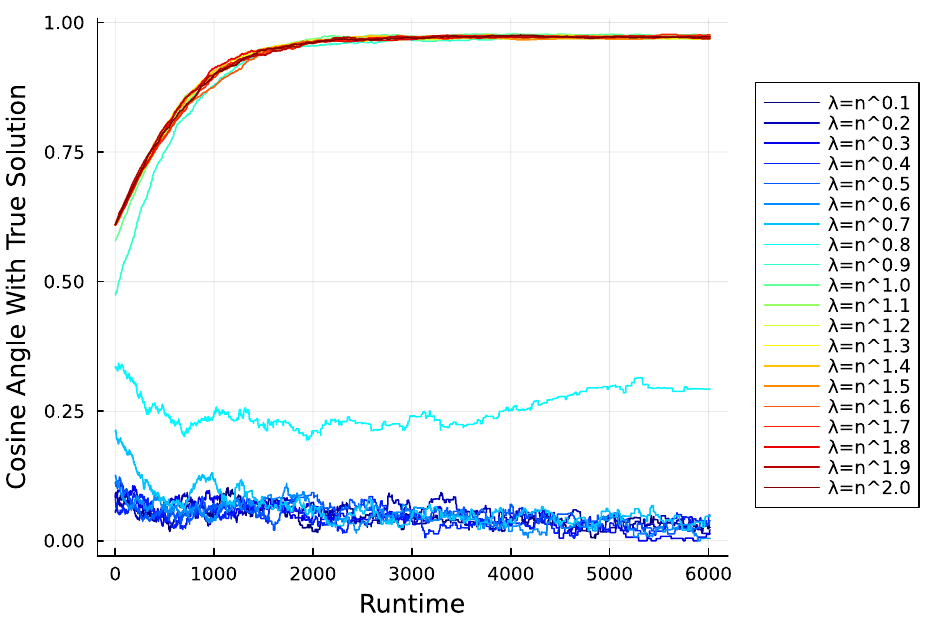}{Randomized greedy for Sparse 3-Tensor PCA when $\theta\in\{-1,0,1\}^n$. Mean absolute angle vs time for $\lambda = n^{\alpha}$, initialized at $S_{\mathrm{HOM}}$. We predict that $\alpha=0.9$ is the threshold for fast recovery. Here $n=150$, $k=56\approx n^{0.8}$, and $\gamma =\log n$.}{TrinaryTensorRandGreedyHomotopyInit}
\end{figure}

\begin{figure}[ht]

  \miniFigSide
      {n150reps10alphaPoint8TwoHOM.pdf}
      {Two-stage algorithm for sparse 3-tensor PCA when
       $\theta\in\{-1,0,1\}^n$. Mean absolute angle vs.\ time for
       $\lambda=n^{\alpha}$, initialised at $S_{\mathrm{HOM}}$.
       We predict that $\alpha=0.75$ is the threshold for fast recovery.
       Here $n=150$, $k=56\approx n^{0.8}$, and $\gamma=\log n$.}
{TrinaryTensorTwoStageHomotopyInit}\hfill
    \miniFigSideLeftArxivSpecial
      {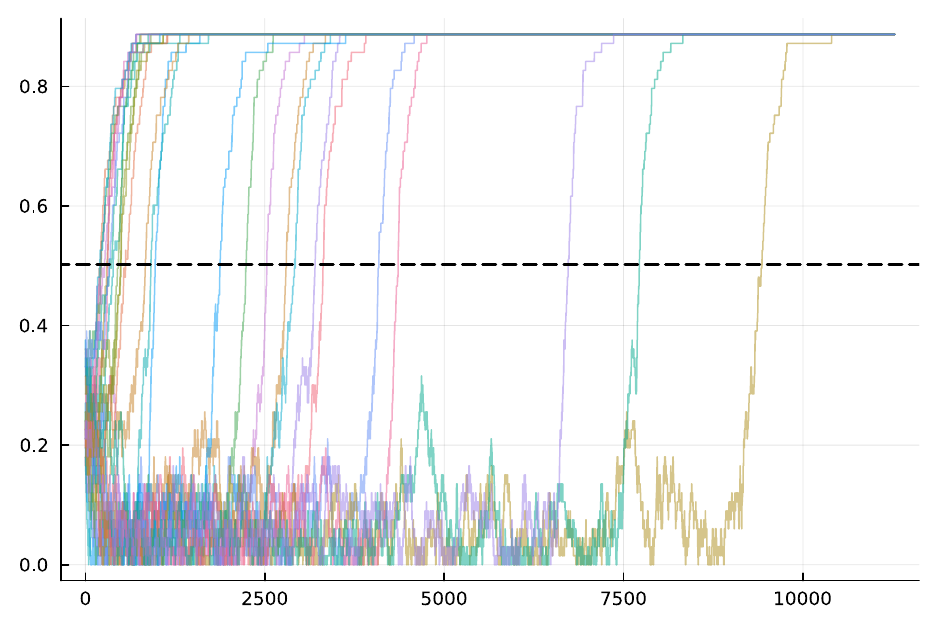}
      {Stage one of the two-stage algorithm for sparse 3-tensor PCA when
       $\theta\in\{-1,0,1\}^n$. Absolute angle vs.\ time for
       $\lambda=\Theta(n^{3/4})$, initialised at $S_{\mathrm{HOM}}$.
       Here $n=150$, $k= 116 \approx n^{0.95}$, and we plot 29 simulations (filtered from 400 total simulations) which exhibit the most visible oscillatory phase. The dashed black line (where $|\cos(S_t, \theta)| \approx .502$) is a prediction for when this stage transitions from an oscillatory phase (below the line) to a monotonically increasing phase (above the line). Color is only used for visual clarity.}
      {TrinaryTensorOneStage}
     
\end{figure}

\begin{algorithm}
    \caption{Binary Randomized Greedy Local Search on $H_{(r+1)/2, \gamma}$}\label{alg:RandGreedyBinaryUnconstrained}
    \begin{algorithmic}[1]
    \REQUIRE $Y \in \mathbb{R}^{n^{\otimes r}}$, $\sigma_0 \in \{0,1\}^n$, $\gamma\geq 0$, $M\in\NN$.
    
    \STATE Let $t \gets 1$ and $S_1 \gets \sigma_0$
    \WHILE{$t<M$}
    \STATE Let $\mathcal{N}(S_t) = \{\sigma'\in\{0,1\}^n \mid d_H(S_t, \sigma') = 1\}$
        \STATE Choose a random element $\sigma' \in \mathcal{N}(S_t)$ uniformly
        \IF{$H_{(r+1)/2,\gamma}(\sigma') - H_{(r+1)/2,\gamma}(S_t) > 0$} 
            \STATE $S_{t+1} \gets \sigma'$
            \STATE $t \gets t + 1$
        \ENDIF
    \ENDWHILE
    \STATE \textbf{return} $S_M$
    \end{algorithmic}
\end{algorithm}

\begin{algorithm}
    \caption{Trinary Randomized Greedy Local Search on $H_{(r+1)/2, \gamma}$}\label{alg:RandGreedyTrinaryUnconstrained}
    \begin{algorithmic}[1]
    \REQUIRE $Y \in \mathbb{R}^{n^{\otimes r}}$, $\sigma_0 \in \{-1,0,1\}^n$, $\gamma\geq 0$, $M\in\NN$.
    
    \STATE Let $t \gets 1$ and $S_1 \gets \sigma_0$
    \WHILE{$t<M$}
    \STATE Let $\mathcal{N}(S_t) = \{\sigma'\in\{-1,0,1\}^n \mid d_H(S_t, \sigma') = 1\}$
        \STATE Choose a random element $\sigma' \in \mathcal{N}(S_t)$ uniformly
        \IF{$H_{(r+1)/2,\gamma}(\sigma') - H_{(r+1)/2,\gamma}(S_t) > 0$} 
            \STATE $S_{t+1} \gets \sigma'$
            \STATE $t \gets t + 1$
        \ENDIF
    \ENDWHILE
    \STATE \textbf{return} $S_M$
    \end{algorithmic}
\end{algorithm}

\begin{algorithm}[h]
\caption{Randomized Greedy Local Search on $H_{(r+1)/2, 0}$}
\label{alg:RandGreedyStage1Appendix}
\begin{algorithmic}[1]
\REQUIRE \(Y \in \mathbb{R}^{n^{\otimes r}}\), \(S_0 \in \{-1,1\}^n\), $M\in\mathbb{N}$
\STATE \(t \gets 0\).
\WHILE{$t<M$}
\STATE $\mathcal{N}(S_t) \gets \{\sigma' \in \{-1,1\}^n \mid d_H(S_t, \sigma') = 1\}$
\STATE Choose a random element $\sigma' \in \mathcal{N}(S_t)$ uniformly
\STATE Let $i \in [n]$ be such that $\sigma'_i \neq (S_t)_i$.
    \IF{\(H_{\frac{r+1}{2},0}(\sigma') - H_{\frac{r+1}{2},0}(S_t) > 0\) }
        \STATE \(S_{t+1} \gets \sigma'\) 
    \ENDIF
\ENDWHILE

\RETURN \(S_{M}\)
\end{algorithmic}
\end{algorithm}

\subsection{On The Predicted Thresholds Given In Section \ref{sec:simul_spec}}
The five figures given in Section \ref{sec:simul_spec} contain an asymptotic ``predicted" threshold for the value of $\alpha$, recall $\lambda = n^\alpha$, where Algorithms \ref{alg:RandGreedyBinaryUnconstrained} and \ref{alg:RandGreedyTrinaryUnconstrained} succeed in finding $\theta$. While the predictions are made via an analytic formula, they seem to be in (approximate) agreement with what we observe in simulations when $n$ is relatively small ($n=150$). Given this interesting agreement, we elaborate more on this predictive rule as it could be of independent interest and future research. 

The prediction is made from our analysis of the variants of Algorithms \ref{alg:RandGreedyBinaryUnconstrained} and \ref{alg:RandGreedyTrinaryUnconstrained} (analyzed when $r$ is odd and $k=\Omega(\sqrt{n})$), where, at each iteration, the algorithm accepts a move based on a random threshold. For Figures 1--4,  we calculate the critical $\lambda=n^{\alpha}$ based on the following rule:
randomized greedy (on the binary or trinary cube) outputs $\theta$ in polynomial-time, given initialization $S_1$, if and only if
\[\lambda =\Tilde \Omega\left( \frac{\sqrt{k}}{\cos(S_1, \theta)^{r-1}}\right)\label{eq:initSigTrade}.\] 
The following table gives the calculation for each figure's stated asymptotic thresholds (based on their initialization):
\begin{table}[H]
\begin{tabular}{l|llllll}
         &$\theta$& $k$ & $S_1$ & $\langle S_1, \theta \rangle$ & $\cos(S_1, \theta)$ & $\sqrt{k} / \cos(S_1, \theta)^2$ \\ \hline
Figure 1 & Binary&  $n^{0.6}$   &      All Ones &  $k$ &  $\sqrt{k/n}$  & $n / \sqrt{k} = n^{0.7}$ \\
Figure 2 & Binary&   $n^{0.6}$  &  Random $k$-Sparse   &   $k^2/n$  &  $k/n$  & $n^2 / k^{3/2} = n^{1.1}$                                \\
Figure 3 &Trinary&   $n^{0.8}$  &  Random Trinary  &  $\sqrt{k}$ & $1/\sqrt{n}$ & $n\sqrt{k} = n^{1.4}$                           \\
Figure 4 &Trinary&  $n^{0.8}$   &  Homotopy  &   $\sqrt{k}n^{1/4}$ & $1/n^{r/4}$ &  $\sqrt{nk} = n^{0.9}$                            
\end{tabular}
\end{table}
For more information on how we arrive at this initialization-SNR trade-off, see Section \ref{sec:ReducProofBinary}. The asymptotic threshold in Figure 5 (where a two stage randomized greedy is executed) is $\alpha = 3/4$ which is the algorithm threshold of its setting. Our prediction for its optimality comes from our Theorem \ref{thm:Trinary} which proves the optimality of the random threshold variants.

It is also interesting to see that our rule \eqref{eq:initSigTrade} also accurately captures the success of randomized greedy even when $k=O( \sqrt{n})$ as considered by Theorem \ref{thm:main_sparsegreedy}, i.e., in the relatively sparser case. In this case, our initialization is a vector $e_i+e_j$ where $\theta_i =\theta_j= 1$. It is an easy calculation to see that $\cos(e_i+e_j, \theta) = \Theta(1/\sqrt{k})$, so \eqref{eq:initSigTrade} suggests that the threshold for success of randomized greedy is $\lambda =\Tilde \Omega\left( k^{r/2}\right)$, which is exactly the algorithmic threshold $\lambda^{(B)}_{\rm ALG}$ when $k=o\left(\sqrt{n}\right)$ as proven Theorem \ref{thm:main_sparsegreedy}.

\section{Proof of Theorem \texorpdfstring{\ref{thm:main_sparsegreedy}}{thm:main sparsegreedy}}\label{sec:sparsecase}

In this section, we prove Theorem \ref{thm:main_sparsegreedy}. In order to do this, we prove the (slightly) more explicit result Theorem \ref{thm:sparseGreedyConvergence}, which directly implies Theorem \ref{thm:main_sparsegreedy}. Note that the only difference is that we make the dependence on constants explicit. For this reason let $\gamma=C_{\gamma}\sqrt{\log n}$, where $C_{\gamma}>0$ is a constant we will pick later, and define 
\begin{align}\label{def:C'main}
&C_r:=2\cdot 3^r(\sqrt{2r}+1), \qquad C':=C_{\gamma}+C_r+1.
\end{align}
Then---recalling that $\lambda^{(R)}_{\rm ALG} =\tilde \Theta (k^{r/2})$ 
when $k \leq \sqrt{n}$---we have the following result
\begin{thm}\label{thm:sparseGreedyConvergence}
   Suppose $k\leq C\sqrt{n}$ for some constant $C>0$, $\gamma=C_\gamma\sqrt{\log n}$, $\lambda \geq C_\lambda k^{r/2} \sqrt{\log n} $ and suppose $C_\gamma\geq 1+C_r$, $C_\lambda\geq C'$ where $C_r, C'$ were defined in \eqref{def:C'main}. Moreover, assume that $S^0 \in \{-1,0,1\}^n$ satisfies $\|S^0\|_0=\langle S^0, \theta \rangle =2$. Then, with probability at least $1-n^{-1/2}$, Algorithm \ref{alg:SparseGreedy2}  with input $(Y,S^0)$ outputs $\theta$ in $k-2$ iterations and Algorithm \ref{alg:SparserandrandGreedy} with input $(Y,S^0, M)$, where $M=\lceil 6n\log(3n)\rceil$, outputs $\theta$.   
 \end{thm} 
We will also show the following corollary which in particular confirms the Remark \ref{rem:findinggoodinit}, which describes how one can resolve the initialization assumption of Theorem \ref{thm:sparseGreedyConvergence}.
 \begin{cor}\label{cor:localmaxthetainA}
Let $k\leq C\sqrt{n}$ for some constant $C>0$, $\gamma=C_\gamma\sqrt{\log n}$, $\lambda \geq C_\lambda k^{r/2} \sqrt{\log n}$ and suppose that $C_\gamma\geq 1+C_r$, $C_\lambda\geq C' $ where $C_r, C'$ were defined in \ref{def:C'main}. Then, it holds with probability at least $1-n^{-1/2}$ that the only local maximum of $H_{\gamma, r}$ in  $ \mathcal{A}$ is $\theta$.
\end{cor}
In the next subsection we state and prove the key Lemmas we will use to prove Theorem \ref{thm:sparseGreedyConvergence}.

\subsection{Proofs of Key Lemmas}

Our results indicate a ``nice" subset of the state space $\widehat{\Theta}$ under which we can guarantee convergence of our algorithms. This subset is defined as follows
\begin{align}\label{def:ofA}
\mathcal{A}=\left\{ \sigma \in \{-1,0,1\}^n : \frac{\lambda}{k^{r/2}}(\langle \sigma, \theta \rangle-1)^{r-1} \geq C'\sqrt{\log n} (\|\sigma \|_0-1)^{r-1} \ \text{and} \quad \|\sigma\|_0\geq 2, \langle \sigma, \theta \rangle \geq 2\right\}
\end{align}
and we will use it significantly going forward.

\begin{lem}\label{lem:eisigmathetaprodPlusMinus}
    Let $\sigma,\theta\in\mathbb{R}^n$. Then for every $i\in[n]$, every $j\in[r]$ and $s\in \mathbb{R}$ it holds that
    \begin{align}
        \langle (se_i)^{\otimes j} \otimes \sigma^{\otimes (r-j)},\theta^{\otimes r}\rangle = (s\theta_i)^j \langle \sigma, \theta \rangle ^{r-j}.
    \end{align}
\end{lem}
\begin{proof}
By using standard properties of inner products on tensor products we can write
\[
\langle (se_i)^{\otimes j} \otimes \sigma^{\otimes (r-j)},\theta^{\otimes r}\rangle=\langle (se_i)^{\otimes j} , \theta^{\otimes j}\rangle \langle \sigma^{\otimes (r-j)},\theta^{\otimes (r-j)}\rangle
\]
Notice though that 
\[
\langle (se_i)^{\otimes j} , \theta^{\otimes j}\rangle= \langle se_i, \theta \rangle ^j=(s\theta _i)^j
\]
and 
\[
 \langle \sigma^{\otimes (r-j)} ,\theta^{\otimes (r-j)}\rangle =\langle \sigma, \theta \rangle ^{r-j}. 
\]
Combining the equations above we get the desired result. 
\end{proof}

\begin{lem}\label{lem:eisigmaWprodPlusMinus}
    Let $W\in \mathbb{R}^{n^{\otimes r}}$ be a symmetric $r$-tensor with i.i.d. standard Gaussian entries. It holds with probability at least $1-n^{-1/2}$ that for every $\sigma\in\{-1,0,1\}^n\setminus \{\textbf{0}\}$, any $s \in \{-2,-1,1,2\}$ and every $i\in[n]$ that
    \begin{align}
        \left|\sum_{j=1}^r\binom{r}{j}\langle (se_i)^{\otimes j} \otimes \sigma^{\otimes (r-j)},W\rangle \right| \leq 3^r(\sqrt{2r}+1)\norm{\sigma}_0^{r-1}\sqrt{\log n}.
    \end{align}
\end{lem}
\begin{proof}
    $W$ has $n^r$ entries each distributed as $\mathcal{N}(0,1)$. It therefore holds by Lemmas \ref{lem:MaxGaussianMean} and \ref{lem:MaxGaussianDeviation} (with $u = \sqrt{\log n}$) that with probability at least $1-n^{-1/2}$, 
    \[\max_{h\in[n]^r}\left|W_{h}\right|\leq (\sqrt{2r}+1)\sqrt{\log n}.\]
    Next, note that for each $j$, $\langle (se_i)^{\otimes j} \otimes \sigma^{\otimes (r-j)},W\rangle$ is the signed sum of at most $\norm{\sigma}_0^{r-j}$ entries of $W$, and therefore $\sum_{j=1}^r\binom{r}{j}2^j\langle (se_i)^{\otimes j} \otimes \sigma^{\otimes (r-j)},W\rangle$ is the signed sum of at most $\sum_{j=1}^r\binom{r}{j}2^j\norm{\sigma}_0^{r-j}\leq 3^r \norm{\sigma}_0^{r-1}$ entries of $W$, which gives the desired result. 
\end{proof}

\begin{lem}\label{lem:HamilDiffIdentsPlusMinus}
For any $\sigma, \theta\in\{-1,0,1\}^n$, we have the following identities:
\begin{itemize}
        \item For any $i$ such that $\sigma_i=0$,  $s\in\{-1,1\}$,
        \begin{align}\label{eq:HamilDiffAddIdenPlusMinus}
            &H(\sigma+se_i)-H(\sigma)\\
            &=\frac{\lambda }{k^{r/2}}\sum_{j=1}^r\binom{r}{j}(s\theta_i)^{j}\langle \sigma, \theta \rangle ^{r-j}+\sum_{j=1}^r\binom{r}{j}(s\theta_i)^j\langle e_i^{\otimes j} \otimes \sigma^{\otimes (r-j)}, W\rangle -\gamma \left(\norm{\sigma}_0+1\right)^{r}+\gamma \norm{\sigma}_0^r
        \end{align}
        \item For any $i$ such that $\sigma_i\neq 0 $ and $s\in \{1,2\}$,
        \begin{align}\label{eq:HamilDiffRemoveIdenPlusMinus}
             &H(\sigma-\sigma _ise_i) - H(\sigma)\\
             &=\frac{\lambda}{k^{r/2}}\sum_{j=1}^r\binom{r}{j}(-\sigma _is\theta_i)^{j}\langle \sigma,\theta\rangle^{r-j} +\sum_{j=1}^r\binom{r}{j}(-\sigma _is\theta_i)^j \langle e_i^{\otimes j} \otimes \sigma^{\otimes (r-j)}, W\rangle\\
             &-\gamma \norm{\sigma_0-\sigma _i se_i}^{r} + \gamma \norm{\sigma}_0^r.
        \end{align}
    \end{itemize}

\end{lem}
\begin{proof}
First we prove \eqref{eq:HamilDiffAddIdenPlusMinus}. Let $\sigma\in\{-1,0,1\}^n$ with $\sigma_i = 0$ and let $s\in\{-1,1\}$. Then by the binomial theorem we can decompose $H(\sigma+se_i)$ as 
    \begin{align}
        H(\sigma+se_i) &= \langle \left(\sigma+se_i\right)^{\otimes r},Y\rangle-\gamma \left(\norm{\sigma}_0+1\right)^{r}\\
        &= \sum_{j=0}^r\binom{r}{j}\langle (se_i)^{\otimes j} \otimes \sigma^{\otimes (r-j)},Y\rangle -\gamma \left(\norm{\sigma}_0+1\right)^{r}\\
        &= \langle \sigma^{\otimes r},Y\rangle +\sum_{j=1}^r\binom{r}{j}\langle (se_i)^{\otimes j} \otimes \sigma^{\otimes (r-j)},Y\rangle -\gamma \left(\norm{\sigma}_0+1\right)^{r}\\
        &= \langle \sigma^{\otimes r},Y\rangle +\sum_{j=1}^r\binom{r}{j}\langle (se_i)^{\otimes j} \otimes \sigma^{\otimes (r-j)},\frac{\lambda}{k^{r/2}}\theta^{\otimes r}+W\rangle -\gamma \left(\norm{\sigma}_0+1\right)^{r}\\
        &=\langle \sigma^{\otimes r},Y\rangle +\frac{\lambda}{k^{r/2}}\sum_{j=1}^r\binom{r}{j}\langle (se_i)^{\otimes j} \otimes \sigma^{\otimes r-j}, \theta^{\otimes r}\rangle\\
        &+\sum_{j=1}^r\binom{r}{j}\langle e_i^{\otimes j} \otimes \sigma^{\otimes (r-j)}, W\rangle -\gamma \left(\norm{\sigma}_0+1\right)^{r}\\
        \intertext{and by Lemma \ref{lem:eisigmathetaprodPlusMinus} this is equal to}
        &=\langle \sigma^{\otimes r},Y\rangle +\frac{\lambda }{k^{r/2}}\sum_{j=1}^r\binom{r}{j}(s\theta_i)^j\langle \sigma, \theta \rangle ^{r-j}+\sum_{j=1}^r\binom{r}{j}(s\theta_i)^j\langle e_i^{\otimes j} \otimes \sigma^{\otimes r-j}, W\rangle -\gamma \left(\norm{\sigma}_0+s\right)^{r}
    \end{align}
     and therefore \eqref{eq:HamilDiffAddIdenPlusMinus} follows, since $H(\sigma)=\langle \sigma^{\otimes r}, Y\rangle-\gamma\|\sigma\|^r_0$. Next we prove \eqref{eq:HamilDiffRemoveIdenPlusMinus}. Again, decompose $H(\sigma-\sigma_i se_i)$ as 
    \begin{align}
        H(\sigma-\sigma_i se_i) &= \langle \left(\sigma-\sigma_i se_i\right)^{\otimes r},Y\rangle-\gamma \norm{\sigma_0-\sigma _i se_i}^{r}\\
        &= \sum_{j=0}^r\binom{r}{j}(-1)^j \langle (se_i)^{\otimes j} \otimes \sigma^{\otimes (r-j)},Y\rangle -\gamma \norm{\sigma_0-\sigma _i se_i}^{r}\\
        &= \langle \sigma^{\otimes r},Y\rangle +\sum_{j=1}^r\binom{r}{j}(-1)^j \langle (\sigma_i se_i)^{\otimes j} \otimes \sigma^{\otimes (r-j)},Y\rangle -\gamma \norm{\sigma_0-\sigma _i se_i}^{r}\\
        &= \langle \sigma^{\otimes r},Y\rangle +\sum_{j=1}^r\binom{r}{j}(-1)^j \langle (\sigma_i se_i)^{\otimes j} \otimes \sigma^{\otimes (r-j)}, \frac{\lambda}{k^{r/2}}\theta^{\otimes r}+W\rangle -\gamma \norm{\sigma_0-\sigma _i se_i}^{r}\\
        &= \langle \sigma^{\otimes r},Y\rangle +\frac{\lambda}{k^{r/2}}\sum_{j=1}^r\binom{r}{j}(-1)^j \langle (\sigma_i se_i)^{\otimes j} \otimes \sigma^{\otimes (r-j)}, \theta^{\otimes r}\rangle \\
        &+\sum_{j=1}^r\binom{r}{j}(-1)^j \langle (\sigma_i se_i)^{\otimes j} \otimes \sigma^{\otimes (r-j)}, W\rangle-\gamma \norm{\sigma_0-\sigma _i se_i}^{r}
        \intertext{and by Lemma \ref{lem:eisigmathetaprodPlusMinus} this is equal to}
        &=\langle \sigma^{\otimes r},Y\rangle +\frac{\lambda }{k^{r/2}}\sum_{j=1}^r\binom{r}{j}(\sigma_i s\theta_i)^j(-1)^j \langle \sigma,\theta\rangle^{r-j} +\sum_{j=1}^r\binom{r}{j}  (-\sigma_is\theta_i)^j \langle e_i^{\otimes j} \otimes \sigma^{\otimes (r-j)}, W\rangle\\
        &-\gamma \norm{\sigma_0-\sigma _i se_i}^{r}.
    \end{align}
    Therefore, \eqref{eq:HamilDiffRemoveIdenPlusMinus} follows using $H(\sigma)=\langle \sigma^{\otimes r}, Y\rangle-\gamma\|\sigma\|^r_0$. 
\end{proof}
The Lemma below also implies that for any coordinate $i\in [n]$ such that $\sigma_i=\theta_i=0$ it holds that   $H(\sigma+s_i)<H(\sigma)$ for any $s_i\in \{e_i, -e_i\}$. 
\begin{lem}\label{lem:HboostRemoveNP}[``It is always improving to zero-out a coordinate not in $\mathrm{Support}(\theta)$"]
    There exists a sufficiently large constant $C_{\gamma}=C_\gamma(r)>0$ such that if $\gamma = C_\gamma\sqrt{\log n}$, then the following holds with probability at least $1-n^{-1/2}$: for any $\sigma\neq \textbf{0}$, and any $i\in [n]$ such that $\sigma_i\neq 0$ and $\theta_i=0$, let $\sigma'\in\{-1,0,1\}^n$ be defined as $\sigma'_i=\theta_i$ and $d_H(\sigma, \sigma')=1$. Then, 
\begin{align}
H(\sigma')\geq \max_{s\in \{1,-1\}}H(\sigma'+se_i)+(\norm{\sigma}_0-1)^{r-1}\sqrt{\log n}.
\end{align}
\end{lem}
\begin{proof}
Notice first that we defined $\sigma'=\sigma-\sigma_ie_i$. Based on that, to prove our result we will prove the following two inequalities:
 \begin{align}
        H(\sigma-\sigma_ie_i) - H(\sigma) \geq \|\sigma \|_0^{r-1}\sqrt{\log n}
    \end{align}
and 
\begin{align}
        H(\sigma-\sigma_ie_i) - H(\sigma-2\sigma_ie_i) \geq \|\sigma \|_0^{r-1}\sqrt{\log n}.
    \end{align}
We start with the first one.    By \eqref{eq:HamilDiffRemoveIdenPlusMinus}, we have
    \begin{align}
        H(\sigma-\sigma_ie_i) - H(\sigma) = \sum_{j=1}^r\binom{r}{j}(-\sigma_i)^j \langle e_i^{\otimes j} \otimes \sigma^{\otimes (r-j)}, W\rangle-\gamma \left(\norm{\sigma}_0-1\right)^{r} + \gamma \norm{\sigma}_0^r
\end{align}
    and using Lemma \ref{lem:eisigmaWprodPlusMinus}
    \[\sum_{j=1}^r\binom{r}{j}(-\sigma_i)^j\langle e_i^{\otimes j} \otimes \sigma^{\otimes (r-j)}, W\rangle\geq -3^r(\sqrt{2r}+1) \norm{\sigma}_0^{r-1}\sqrt{\log n}.\]
 Then, we get the following lower bound
    \begin{align}
        H(\sigma-\sigma_ie_i) - H(\sigma) \geq -3^r(\sqrt{2r}+1) \norm{\sigma}_0^{r-1}\sqrt{\log n} -\gamma \left(\norm{\sigma}_0-1\right)^{r} + \gamma \norm{\sigma}_0^r
    \end{align}
    Next, note that we can lower bound $- \left(\norm{\sigma}_0-1\right)^{r} + \norm{\sigma}_0^r$ by $\norm{\sigma}_0^{r-1}$. Indeed, 
    \[
    - \left(\norm{\sigma}_0-1\right)^{r} + \norm{\sigma}_0^r=\sum_{k=0}^{r-1}\norm \sigma_0^{r-1-k}(\norm\sigma_0 -1)^k\geq \norm \sigma_0^{r-1}.
    \]
    Therefore, we conclude since $\gamma=C_{\gamma}\sqrt{\log n}$
    \begin{align}
        H(\sigma-\sigma_ie_i) - H(\sigma) &\geq -3^r(\sqrt{2r}+1) \norm{\sigma}_0^{r-1}\sqrt{\log n} -\gamma \left(\norm{\sigma}_0-1\right)^{r} + \gamma \norm{\sigma}_0^r\\
        &\geq  -3^r(\sqrt{2r}+1) \norm{\sigma}_0^{r-1}\sqrt{\log n} + C_\gamma \norm{\sigma}_0^{r-1}\sqrt{\log n}.
    \end{align}
The result then follows by choosing $C_\gamma\geq 1+3^r(\sqrt{2r}+1)$. For the second part of the Lemma notice that 
 By \eqref{eq:HamilDiffRemoveIdenPlusMinus}, we have
    \begin{align}
        &H(\sigma-\sigma_ie_i) - H(\sigma-2\sigma_ie_i)\\
        &=(H(\sigma-\sigma_ie_i) - H(\sigma))-(H(\sigma-2\sigma_ie_i) - H(\sigma)) \\
        &= \sum_{j=1}^r\binom{r}{j}(-\sigma_i)^j \langle e_i^{\otimes j} \otimes \sigma^{\otimes (r-j)}, W\rangle-\gamma \left(\norm{\sigma}_0-1\right)^{r} + \gamma \norm{\sigma}_0^r-\sum_{j=1}^r\binom{r}{j}(-2 \sigma_i)^j \langle e_i^{\otimes j} \otimes \sigma^{\otimes (r-j)}, W\rangle\\
        &\geq -2\cdot3^r(\sqrt{2r}+1) \norm{\sigma}_0^{r-1}\sqrt{\log n}+ C_\gamma \norm{\sigma}_0^{r-1}\sqrt{\log n}
\end{align}
where for the last inequality we used $- \left(\norm{\sigma}_0-1\right)^{r} + \norm{\sigma}_0^r\geq \norm{\sigma}_0^{r-1}$ and Lemma \ref{lem:eisigmaWprodPlusMinus}. Now choosing $C_\gamma\geq 1+2\cdot 3^r(\sqrt{2r}+1)$ we get our desired result.
\end{proof}

\begin{lem}\label{lem:signaddremoveplanted}[``In $\mathcal{A},$ it is always improving to fix any coordinate in $\mathrm{Support}(\theta)$. "]
Let $k\leq C\sqrt{n}$ for some constant $C>0$, $\gamma=C_\gamma\sqrt{\log n}$, $\lambda \geq C_\lambda k^{r/2} \sqrt{\log n}$ and suppose that $C_\gamma\geq 1+C_r$, $C_\lambda\geq C' $ where $C_r, C'$ were defined in \ref{def:C'main}. 

Then, the following statement holds with probability at least $1-n^{-1/2}$ for any $\sigma \in \mathcal{A}$, defined in \ref{def:ofA},:
For any $i\in [n]$ such that $\theta_i\neq 0$ and $\sigma_i \neq \theta_i$ let $\sigma'\in\{-1,0,1\}^n$ be defined as $\sigma'_i=\theta_i$ and $d_H(\sigma, \sigma')=1$. Then, 
\begin{align}
H(\sigma')\geq \max_{s\in \{1,2\}}H(\sigma'-s\theta_ie_i)+(\norm{\sigma}_0-1)^{r-1}\sqrt{\log n}.
\end{align}
\end{lem}
\begin{proof} 
Fix a $\sigma \in \mathcal{A}$. We will break down the proof of the Lemma in the three inequalities below. Firstly, we prove that
for any $i\in [n]$ such that $\theta_i\neq 0$ and $\sigma_i \neq \theta_i$ 
\[H(\sigma+(\theta_i-\sigma_i)e_i)-H(\sigma) \geq (\|\sigma \|_0-1)^{r-1}\sqrt{\log n}.\]
This, in particular, implies that moving to $\sigma'$ (defined in the statement of the Lemma) always increases $H$. Secondly, we show that $\sigma'$ is the maximum between the three possible vectors that have Hamming distance at most one from $\sigma$ and have all coordinates $j\in [n]$, $j\neq i$ same with $\sigma$. This breaks down to the following two inequalities: for any $i\in [n]$ such that $\theta_i\neq 0$ and $\sigma_i =0$ 
\begin{align}
H(\sigma)\geq H(\sigma-\theta_ie_i)+\norm{\sigma}_0^{r-1}\sqrt{\log n}.
\end{align}
For any $i\in [n]$ with $\theta_i\neq 0$ and $\sigma_i=-\theta_i$
\[
H(\sigma+2\theta_ie_i)-H(\sigma+\theta_ie_i)\geq\|\sigma \|_0^{r-1}\sqrt{\log n}.
\] These three inequalities combined together will give us the desired result. We start by proving the first inequality. First, let's consider the case where for some $i\in [n]$ $\sigma_i\neq 1$ and $\theta_i=1$. Then, we have
    \begin{align}
        &H(\sigma+(1-\sigma_i)e_i)-H(\sigma)\\
        &= \frac{\lambda}{k^{r/2}}\sum_{j=1}^r \binom{r}{j}(1-\sigma_i)^j\langle \sigma,\theta\rangle^{r-j}+\sum_{j=1}^r\binom{r}{j}\langle e_i^{\otimes j} \otimes \sigma^{\otimes r-j}, W\rangle -\gamma \norm{\sigma_0+(1-\sigma_i)e_i}^{r}+\gamma \norm{\sigma}_0^r\\
        &= \frac{\lambda}{k^{r/2}}\sum_{j=1}^r \binom{r}{j}(1-\sigma_i)^j\langle \sigma,\theta\rangle^{r-j}+\sum_{j=1}^r\binom{r}{j}\langle e_i^{\otimes j} \otimes \sigma^{\otimes r-j}, W\rangle -\gamma \left(\norm{\sigma}_0+1+\sigma_i\right)^{r}+\gamma \norm{\sigma}_0^r\\
        &\geq  \frac{\lambda}{k^{r/2}}\sum_{j=1}^r \binom{r}{j}(1-\sigma_i)^j\langle \sigma,\theta\rangle^{r-j}+\sum_{j=1}^r\binom{r}{j}\langle e_i^{\otimes j} \otimes \sigma^{\otimes r-j}, W\rangle -\gamma \left(\norm{\sigma}_0+1-\sigma_i\right)^{r}+\gamma \norm{\sigma}_0^r\\
        &\geq \sum_{j=1}^r \binom{r}{j}(1-\sigma_i)^j\left(\frac{\lambda}{k^{r/2}}\langle \sigma,\theta\rangle^{r-j}-\gamma\|\sigma\|_0^{r-j}\right)+\sum_{j=1}^r\binom{r}{j}(1-\sigma_i)^j\langle e_i^{\otimes j} \otimes \sigma^{\otimes (r-j)}, W\rangle
    \end{align}
     where we used that $1-\sigma_i\geq 1+\sigma_i$ since $\sigma_i\in \{0,-1\}$. By Lemma \ref{lem:eisigmaWprodPlusMinus}, $\sum_{j=1}^r\binom{r}{j}(1-\sigma_i)^j\langle e_i^{\otimes j} \otimes \sigma^{\otimes (r-j)}, W\rangle\geq -3^r(\sqrt{2r}+1)\norm{\sigma}_0^{r-1}\sqrt{\log n}$, which means we can lower bound this by
    \begin{align}
        &H(\sigma+(1-\sigma_i)e_i)-H(\sigma)\\
        &\geq \sum_{j=1}^r \binom{r}{j}(1-\sigma_i)^j\left(\frac{\lambda}{k^{r/2}}\langle \sigma,\theta\rangle^{r-j}-\gamma\|\sigma\|_0^{r-j}\right)-3^r(\sqrt{2r}+1)\norm{\sigma}_0^{r-1}\sqrt{\log n}.
    \end{align}
   Since $\sigma \in \mathcal{A}$ we know that $\|\sigma\|_0\geq 2, \langle \sigma, \theta\rangle \geq 2$ and using $\langle \sigma, \theta \rangle \leq \|\sigma\|_0$ we get
    \begin{align}\label{eq:belonginA}
     \left(\frac{\langle \sigma , \theta \rangle}{\|\sigma\|_0}\right)^{r-1}\geq \left(\frac{\langle \sigma , \theta \rangle-1}{\|\sigma\|_0-1}\right)^{r-1}\geq \frac{k^{r/2}(C_{\gamma}+C_r+1)\sqrt{\log n}}{\lambda}\geq \frac{k^{r/2}\gamma}{\lambda}   
    \end{align}
    and using again $\langle \sigma, \theta \rangle \leq \|\sigma\|_0$ gives us
    \begin{align}\label{eq:belonginA2}
    \left(\frac{\langle \sigma , \theta \rangle}{\|\sigma\|_0}\right)^{r-j}\geq  \frac{k^{r/2}\gamma}{\lambda} \quad \text{and} \quad \left(\frac{\langle \sigma , \theta \rangle-1}{\|\sigma\|_0-1}\right)^{r-j}\geq  \frac{k^{r/2}\gamma}{\lambda}
    \end{align}
     for any $r-1\geq j\geq 1$. Therefore,
\begin{align}
        &H(\sigma+(1-\sigma_i)e_i)-H(\sigma) \\
        &\geq \sum_{j=1}^r \binom{r}{j}(1-\sigma_i)^j\left(\frac{\lambda}{k^{r/2}}\langle \sigma,\theta\rangle^{r-j}-\gamma\|\sigma\|_0^{r-j}\right)-3^r(\sqrt{2r}+1)\norm{\sigma}_0^{r-1}\sqrt{\log n}\\
        &\geq \sum_{j=2}^r \binom{r}{j}(1-\sigma_i)^j\underbrace{\left(\frac{\lambda}{k^{r/2}}\langle \sigma,\theta\rangle^{r-j}-\gamma\|\sigma\|_0^{r-j}\right)}_{\geq 0} \quad \text{ by \eqref{eq:belonginA2}}\\
        &\qquad+r(1-\sigma_i)^{r-1}\left(\frac{\lambda}{k^{r/2}}\langle \sigma,\theta\rangle^{r-1}-\gamma\|\sigma\|_0^{r-1}\right)-C_r\norm{\sigma}_0^{r-1}\sqrt{\log n}\\
        &\geq \frac{\lambda}{k^{r/2}}\langle \sigma,\theta\rangle^{r-1}-\gamma\|\sigma\|_0^{r-1}-C_r\norm{\sigma}_0^{r-1}\sqrt{\log n}\\
        &\geq \norm{\sigma}_0^{r-1}\sqrt{\log n}
    \end{align} where for the last inequality we used \eqref{eq:belonginA}. Now we prove the second case for the first inequality where $\theta_i=-1$ and $\sigma_i\neq -1$. We start by decomposing the following using \eqref{eq:HamilDiffAddIdenPlusMinus}
    \begin{align}
    &H(\sigma -(1+\sigma_i)e_i)-H(\sigma)\\
    &=\frac{\lambda}{k^{r/2}}\sum_{j=1}^r\binom{r}{j}(1+\sigma_i)^j\langle \sigma, \theta \rangle ^{r-j}-\sum_{j=1}^r\binom{r}{j}(1+\sigma_i)^j\langle e_i^{\otimes j} \otimes \sigma^{\otimes (r-j)}, W\rangle\\
    &-\gamma \norm{\sigma}^{r}+\gamma \norm{\sigma-(1+\sigma_i)e_i}_0^r\\
    &\geq \frac{\lambda}{k^{r/2}}\sum_{j=1}^r\binom{r}{j}\langle \sigma, \theta \rangle ^{r-j} -\gamma \norm{\sigma}^{r}+\gamma (\norm{\sigma}_0-1)^r-3^r(\sqrt{2r}+1)\norm{\sigma}_0^{r-1}\sqrt{\log n}\\
    &\text{where for the last step we used Lemma \ref{lem:eisigmaWprodPlusMinus}, the fact that} \ \norm{\sigma-(1+\sigma_i)e_i}_0\geq  \norm{\sigma-1}_0 \ . \\
    &\text{We continue by bounding the above by} \\
    &\geq \sum_{j=1}^r\binom{r}{j}\left(\frac{\lambda}{k^{r/2}}\langle \sigma, \theta \rangle ^{r-j}-\gamma (\|\sigma\|_0-1)^{r-j} \right)-3^r(\sqrt{2r}+1)\norm{\sigma}_0^{r-1}\sqrt{\log n}\\
    & \geq \sum_{j=1}^r\binom{r}{j}\left(\frac{\lambda}{k^{r/2}}\langle \sigma, \theta \rangle ^{r-j} -\gamma \|\sigma\|_0^{r-j}\right)-3^r(\sqrt{2r}+1)\norm{\sigma}_0^{r-1}\sqrt{\log n}\\
    &= \sum_{j=2}^r\binom{r}{j}\underbrace{\left(\frac{\lambda}{k^{r/2}}\langle \sigma, \theta \rangle ^{r-j}-\gamma \|\sigma\|_0^{r-j} \right)}_{\geq 0} \quad \text{by \eqref{eq:belonginA2}}\\
    &+r\left(\gamma \|\sigma\|_0^{r-1}-\frac{\lambda}{k^{r/2}}\langle \sigma, \theta \rangle^{r-1}\right)-3^r(\sqrt{2r}+1)\norm{\sigma}_0^{r-1}\sqrt{\log n}\\
    &\geq \gamma \|\sigma\|_0^{r-1}-\frac{\lambda}{k^{r/2}}\langle \sigma, \theta \rangle ^{r-1}
    +3^r(\sqrt{2r}+1)(\norm{\sigma}_0-1)^{r-1}\sqrt{\log n}\\
    &\geq \|\sigma\|_0^{r-1}\sqrt{\log n}
    \end{align}
    where for the last inequality we used  \eqref{eq:belonginA} and we conclude the proof for the first part. For the second part of the Lemma notice that 
    using Lemma \ref{lem:HamilDiffIdentsPlusMinus} we get 
\begin{align}
&H(\sigma)-H(\sigma-\theta_ie_i)\\
&=-(H(\sigma-\theta_ie_i)-H(\sigma))\\
&=-\frac{\lambda(-\theta_i^2) }{k^{r/2}}\sum_{j=1}^r\binom{r}{j}\langle \sigma, \theta \rangle ^{r-j}-(-\theta_i)\sum_{j=1}^r\binom{r}{j}\langle e_i^{\otimes j} \otimes \sigma^{\otimes (r-j)}, W\rangle +\gamma \left(\norm{\sigma}_0+1\right)^{r}-\gamma \norm{\sigma}_0^r\\
&\geq\sum_{j=1}^r\binom{r}{j}\frac{\lambda}{k^{r/2}}\langle \sigma, \theta \rangle ^{r-j}+\theta_i\sum_{j=1}^r\binom{r}{j}\langle e_i^{\otimes j} \otimes \sigma^{\otimes (r-j)}, W\rangle\\
&\geq \frac{\lambda}{k^{r/2}}\langle \sigma,\theta\rangle^{r-1}-C_r\norm{\sigma}_0^{r-1}\sqrt{\log n}\\
&\geq \norm{\sigma}_0^{r-1}\sqrt{\log n}
\end{align}
where for the last inequality we used \eqref{eq:belonginA} that holds for all $\sigma\in \mathcal{A}$. It now only remains to show the final part of the Lemma. Again we decompose the following expression using  \ref{lem:HamilDiffIdentsPlusMinus},
\begin{align}
&H(\sigma+2\theta_ie_i)-H(\sigma+\theta_ie_i)\\
&=(H(\sigma+2\theta_ie_i)-H(\sigma))-(H(\sigma+\theta_ie_i)-H(\sigma))\\
&=\frac{\lambda }{k^{r/2}}\sum_{j=1}^r\binom{r}{j}2^j\langle \sigma, \theta \rangle ^{r-j}+\sum_{j=1}^r\binom{r}{j}2^j\langle e_i^{\otimes j} \otimes \sigma^{\otimes (r-j)}, W\rangle -\gamma \norm{\sigma}_0^{r}+\gamma \norm{\sigma}_0^r\\
&-\frac{\lambda }{k^{r/2}}\sum_{j=1}^r\binom{r}{j}\langle \sigma, \theta \rangle ^{r-j}-\sum_{j=1}^r\binom{r}{j}\langle e_i^{\otimes j} \otimes \sigma^{\otimes (r-j)}, W\rangle +\gamma \left(\norm{\sigma}_0+1\right)^{r}-\gamma \norm{\sigma}_0^r\\
&\geq \sum_{j=1}^r\binom{r}{j}(2^j-1)\frac{\lambda}{k^{r/2}}\langle \sigma, \theta \rangle ^{r-j}+\sum_{j=1}^r\binom{r}{j}(2^j-1)\langle e_i^{\otimes j} \otimes \sigma^{\otimes (r-j)}, W\rangle\\
& \geq \sum_{j=1}^r\binom{r}{j}\frac{\lambda}{k^{r/2}}\langle \sigma, \theta \rangle ^{r-j}-3^r(\sqrt{2r}+1)\norm{\sigma}_0^{r-1}\sqrt{\log n} \quad \text{by Lemma \ref{lem:eisigmaWprodPlusMinus}}\\
& \geq r\frac{\lambda}{k^{r/2}}\langle \sigma, \theta \rangle ^{r-1}-3^r(\sqrt{2r}+1)\norm{\sigma}_0^{r-1}\sqrt{\log n}\\
&\geq \|\sigma\|_0^{r-1}\sqrt{\log n}
\end{align}
where for the last inequality we used again the inequality \eqref{eq:belonginA} that holds for all $\sigma \in \mathcal{A}$.
\end{proof}  
  
\begin{lem}\label{lem:fromplatendtosthelse}[``In $\mathcal{A}$, correct coordinates don't change'']
 Let $k\leq C\sqrt{n}$ for some constant $C>0$, $\gamma=C_\gamma\sqrt{\log n}$, $\lambda \geq C_\lambda \lalg$ and suppose that $C_\gamma\geq 1+C_r$, $C_\lambda\geq C' $ where $C_r, C'$ were defined in \ref{def:C'main}.
Then, it holds with probability at least $1-n^{-1/2}$ that for any \ $\sigma \in \mathcal{A}$ which was defined in Lemma \ref{lem:signaddremoveplanted} and for any $i\in [n]$ such that $\sigma_i=\theta_i$ moving to a vector $\sigma'\in\{-1,0,1\}^n$ with $d_H(\sigma, \sigma')=1$ and $\sigma_i\neq \sigma_i'$ will not increase $H$, i.e. 
\[
H(\sigma')\leq H(\sigma).
\]
\end{lem}

\begin{proof} 
Let's first assume that for some coordinate $i\in [n]$ it holds that $\sigma_i=\theta_i=0$. Then, substituting this coordinate to $\sigma_i\neq0$ cannot result in an increase in $H$. Indeed, this is just the contrapositive of what we proved in Lemma \ref{lem:HboostRemoveNP}. A similar argument applies to the case where we substitute a coordinate $i$ from $\sigma_i=\theta_i$ and $\theta_i\neq 0$ to $\sigma_i\neq \theta_i$. We continue by proving this statement. Let's first prove the case where we move to $\sigma_i=0$.
For any $i\in [n]$ s.t. $\sigma_i=\theta_i \neq 0$ using \eqref{eq:HamilDiffAddIdenPlusMinus} we calculate
\begin{align}
&H(\sigma -\theta _ie_i)-H(\sigma)=-(H(\sigma)-H(\sigma -\theta _ie_i))\\
&=-\frac{\lambda \theta_i^2}{k^{r/2}}\sum_{j=1}^r\binom{r}{j}\langle \sigma-\theta_i e_i, \theta \rangle ^{r-j}-\sum_{j=1}^r\binom{r}{j}\langle e_i^{\otimes j} \otimes (\sigma-\theta_i e_i)^{\otimes (r-j)}, W\rangle\\
&+\gamma \left(\norm{\sigma-\theta_ie_i}_0+1\right)^{r}-\gamma \norm{\sigma-\theta_ie_i}_0^r\\
&\leq -\frac{\lambda}{k^{r/2}}\sum_{j=1}^r\binom{r}{j}(\langle \sigma, \theta \rangle-1) ^{r-j} +\gamma \norm{\sigma}^{r}-\gamma (\norm{\sigma}_0-1)^r+3^r(\sqrt{2r}+1)(\norm{\sigma}_0-1)^{r-1}\sqrt{\log n}\\
&\leq \sum_{j=1}^r\binom{r}{j}\left(\gamma (\|\sigma\|_0-1)^{r-j}-\frac{\lambda}{k^{r/2}}(\langle \sigma, \theta \rangle-1) ^{r-j} \right)+3^r(\sqrt{2r}+1)(\norm{\sigma}_0-1)^{r-1}\sqrt{\log n}\\
&\leq -(\|\sigma\|_0-1)^{r-1}\sqrt{\log n}
\end{align}
where we used for the last step $\sigma \in \mathcal{A}$ and in particular we used the Equations \eqref{eq:belonginA}, \eqref{eq:belonginA2} that hold for all $\sigma \in \mathcal{A}$. The expression we arrived at is negative, proving our claim. Finally, consider the case where $\sigma_i=\theta_i$, and $\theta_i\neq 0$ and we want to substitute the $i$-th coordinate to $\sigma_i=-\theta_i$. Then, using \eqref{eq:HamilDiffRemoveIdenPlusMinus}
\begin{align}
&H(\sigma -2\theta _ie_i)-H(\sigma)=\\
&=\frac{\lambda}{k^{r/2}}\sum_{j=1}^r\binom{r}{j}(-2)^j\langle \sigma, \theta \rangle ^{r-j}-\sum_{j=1}^r\binom{r}{j}\langle e_i^{\otimes j} \otimes \sigma^{\otimes (r-j)}, W\rangle\\
&-\gamma \norm{\sigma-2\theta_ie_i}_0^{r}+\gamma \norm{\sigma}_0^r\\
&\leq \frac{\lambda}{k^{r/2}}\sum_{j=1}^r\binom{r}{j}(-2)^j\langle \sigma, \theta \rangle ^{r-j} +3^r(\sqrt{2r}+1)(\norm{\sigma}_0-1)^{r-1}\sqrt{\log n}\\
&\leq \frac{\lambda}{k^{r/2}}(\langle \sigma, \theta \rangle -2)^r-\frac{\lambda}{k^{r/2}}\langle \sigma, \theta \rangle^r+3^r(\sqrt{2r}+1)\norm{\sigma}_0^{r-1}\sqrt{\log n}\\
&\leq \frac{\lambda}{k^{r/2}}(\langle \sigma, \theta \rangle -2)^r-\frac{\lambda}{k^{r/2}}\langle \sigma, \theta \rangle^r+\frac{\lambda}{2k^{r/2}}\langle \sigma, \theta \rangle^{r-1} \quad \text{by Lemma \eqref{eq:belonginA}}\\
& \leq \left(1-\frac{2}{\langle \sigma, \theta \rangle}\right)^r-1+\frac{1}{2\langle \sigma, \theta \rangle} \\
& \leq 1-\frac{2}{\langle \sigma, \theta \rangle}-1+\frac{2}{\langle \sigma, \theta \rangle}\\
&=0, \quad \text{because} \ \sigma \in \mathcal{A} \ \text{implies} \ \langle \sigma, \theta \rangle \geq 2
\end{align}
and we have proven our desired result.  
\end{proof}

\subsection{The proof of Theorem \ref{thm:sparseGreedyConvergence} and Corollary \ref{cor:localmaxthetainA}}
\begin{proof}[Proof of Theorem \ref{thm:sparseGreedyConvergence}]
First we prove the statement about Algorithm \ref{alg:SparseGreedy2}. The statement about Algorithm \ref{alg:SparserandrandGreedy} will follow easily afterwards. The proof for  Algorithm \ref{alg:SparseGreedy2} will proceed in the following steps. We call $S^t$ the iterates of the Algorithm, for $t=0,1,\ldots.$
\begin{itemize}
\item \textbf{Step 1:}  We argue that w.h.p. for all $t=0,1,\ldots$ if $S^t\in \mathcal{A}\setminus \{\theta\}$, $d_H(S^{t+1}, \theta)=d_H(S^t, \theta)-1$. 
\item \textbf{Step 2:}  Then, we prove that w.h.p. for all $t=1,2,\ldots$ if $S^t\in \mathcal{A}$ for some $t\in \mathbb{N}$, then $S^{t+1}\in \mathcal{A}$. Together with step 1, this establishes the desired result.
\end{itemize} 

\textbf{Step 1.} In this step we prove that whp for all $t$ if $S^t\in \mathcal{A} \setminus \{\theta\}$, then $d_H(S^{t+1}, \theta)=d_H(S^t, \theta)-1$. Indeed, by Lemmas \ref{lem:HboostRemoveNP},  \ref{lem:signaddremoveplanted}, \ref{lem:fromplatendtosthelse} we have w.h.p. that the configuration \( S^{t+1} \) is obtained from $S^{t}$ by choosing a coordinate $\sigma_i$ of $S^{t}$ that satisfies $\sigma_i\neq \theta_i$ and replacing it with $\theta_i$; in particular $d_H(S^{t+1}, \theta)=d_H(S^t, \theta)-1$. To see this, assume the algorithm is to modify the coordinate $i \in [n]$. If $\theta_i=0$ and $\sigma_i \not =0 $ then according to Lemma \ref{lem:HboostRemoveNP} the algorithm will change the value of $\sigma_i$ to $\theta_i=0$ w.h.p. If $\theta_i \not =0$ and $\sigma_i\neq 0$ then according to Lemma \ref{lem:signaddremoveplanted}, the algorithm will change the value of $\sigma_i$ to $\theta_i$ w.h.p. Finally, note that as long as $S^t\in \mathcal{A}\setminus \{\theta\}$, there will always exist a coordinate $i\in [n]$ such that $\sigma_i\neq \theta_i$. Equivalently, using the Lemmas mentioned above, there will always exist a vector $S'\in \mathcal{N}(S^t)$ such that $H(S')>H(S^t)$ and therefore our algorithm will only terminate once $d_H(S^t, \theta)=0$,
concluding the proof of our claim.

\textbf{Step 2.} Assume now that for some $t,$ \( \sigma=S^{t} \in \mathcal{A} \) and that $i$-th is the coordinate changed in the $t$-iteration. Firstly, suppose $\theta_i\neq 0$, so we substitute a coordinate from $\sigma_i\neq \theta_i$ to $\theta_i$ and therefore $ \langle S^{t+1},\theta\rangle \geq \langle S^t,\theta\rangle+1$ and $\norm{S^{t+1}}_0\leq\norm{S^{t}}_0+1$. Hence,
    \begin{align}
    \langle S^{t+1},\theta\rangle-1 \geq \langle S^t,\theta\rangle &\geq \left(\frac{ k^{r/2}C'\sqrt{\log n}}{\lambda}\right)^{\frac{1}{r-1}}(\norm{S^t}_0-1) +1\\
    &\geq \left(\frac{ k^{r/2}C'\sqrt{\log n}}{\lambda}\right)^{\frac{1}{r-1}}(\norm{S^{t+1}}_0-2) +1\\
    &= \left(\frac{ k^{r/2}C'\sqrt{\log n}}{\lambda}\right)^{\frac{1}{r-1}}(\norm{S^{t+1}}_0-1) +1-\left(\frac{ k^{r/2}C'\sqrt{\log n}}{\lambda}\right)^{\frac{1}{r-1}}\\
    &\geq \left(\frac{ k^{r/2}C'\sqrt{\log n}}{\lambda}\right)^{\frac{1}{r-1}}(\norm{S^{t+1}}_0-1)
    \end{align}
      where the last inequality holds by choosing $C_\lambda\geq C'$  to make $\frac{k^{r/2}C'\sqrt{\log n}}{\lambda}\leq 1$. In the second case, where $S^{t+1}$ is $S^t$ with one coordinate s.t. $\theta_i=0$ substituted from $\sigma_i\neq 0$ to $\sigma_i=0$, 
    \begin{align}
        \langle S^{t+1},\theta\rangle-1 = \langle S^t,\theta\rangle-1 &\geq \left(\frac{k^{r/2}C'\sqrt{\log n}}{\lambda}\right)^{\frac{1}{r-1}}(\norm{S^t}_0-1) \\
        &= \left(\frac{k^{r/2}C'\sqrt{\log n}}{\lambda}\right)^{\frac{1}{r-1}}\norm{S^{t+1}}_0 \\
        &\geq \left(\frac{k^{r/2}C'\sqrt{\log n}}{\lambda}\right)^{\frac{1}{r-1}}(\norm{S^{t+1}}_0-1).
    \end{align}
    So in both cases $S^{t+1}\in \mathcal{A}$ which completes the proof of the claim. The above also implies that once $(\langle S^t,\theta\rangle-1)^{r-1}\geq \frac{k^{r/2}C'\sqrt{\log n}}{\lambda}(\norm{S^t}_0-1)^{r-1}$,  $S^t$ will converge to $\theta$ monotonically in Hamming distance meaning that w.h.p. it will only move to vectors that at each step reduce the Hamming distance $d_H(S^t, \theta)$, and so reaches $\theta$ in $d_H(S^0,\theta)=k-2$ further iterations. 
    
    Finally, we prove our result for Algorithm \ref{alg:SparserandrandGreedy}. Notice that by Lemma \ref{lem:fromplatendtosthelse} we know that if at some time $t\in [\lceil 6n\log (3n) \rceil]$ 
    a move is proposed that substitutes the $i$-th coordinate from $\sigma_i=\theta_i$ to $\sigma_i\neq \theta_i$ then our algorithm will w.h.p. reject that move. Therefore, we only need to show that Algorithm \ref{alg:SparserandrandGreedy} will reach for every coordinate $i\in [n]$ for some time $t_i\leq \lceil 6n\log (3n)\rceil$ a vector $S^{t_i}$ with $\sigma_{i}=\theta_{i}$. If that happens, then using what we said above we are guaranteed to output $\theta$ when Algorithm \ref{alg:SparserandrandGreedy} is terminated. From Lemmas \ref{lem:HboostRemoveNP},  \ref{lem:signaddremoveplanted} though, we know that if a move is proposed that substitutes a coordinate $i\in[n]$ from $\sigma_i\neq \theta_i$ to $\sigma_i=\theta_i$ then w.h.p. our Algorithm will accept it. For any vector $\sigma$ there exists a total of $3n$ possible moves. Using Coupon collector, which is presented in Lemma \ref{lem:coupon}, we know that with probability at least $1-n^{-1}$, in $\lceil 6n\log (3n)\rceil$ steps, all $3n$ possible moves will be proposed at least once. In particular, for all coordinates $i\in [n]$ such that $\sigma_i\neq \theta_i$ the move that proposes substituting the $i$-th coordinate to $\sigma_i=\theta_i$ will be proposed at least once with probability at least $1-o(1)$. Putting these together we see that Algorithm \ref{alg:SparserandrandGreedy} will indeed converge in $\lceil 6n\log(3n)\rceil$ steps proving our desired result.
\end{proof}

\begin{proof}[Proof of Corollary \ref{cor:localmaxthetainA}]
The proof follows easily from Lemmas \ref{lem:signaddremoveplanted}, \ref{lem:HboostRemoveNP} 
Indeed, suppose that there exists a local maximum  $t\in \mathcal{A}$ such that $t\neq \theta$. Then, for any $t'\in \mathcal{A}$ such that $d_H(t,t')=1$ it should hold that $H(t')\leq H(t)$. At the same time though, since $t\neq \theta$ there exists $i\in [n]$ such that $t_i\neq \theta_i$. If $\theta_i=0$ then moving to the vector $t^*$ that has $t^*_j=t_j$ for any $j\neq i$ and $t^*_i=0$ will increase $H$ according to Lemma \ref{lem:HboostRemoveNP} and similarly if $\theta_i\neq 0$ moving to the vector $t^*$ that has $t^*_j=t_j$ for any $j\neq i$ and $t^*_i=\theta_i$ will increase $H$ according to Lemma \ref{lem:signaddremoveplanted}. In both cases this implies that there exists a vector $t^*\in \mathcal{A}$  with $d_H(t^*, t)=1$ and $H(t^*)>H(t)$. 
Combining the above, we have reached a contradiction and therefore completed our proof.
\end{proof}

\section{Proof of Theorem \texorpdfstring{\ref{thm:GoodInit1}}{thm:GoodInit1}}
In this section, we prove Theorem \ref{thm:GoodInit1}. In order to do this, we prove the (slightly) more explicit result Theorem \ref{thm:Feige_analysis}, which directly implies Theorem \ref{thm:GoodInit1}. Note that the only difference is that here we assume the exact logarithmic factor of $\sqrt{\log n}$ on $\lambda$, and make the dependence on the constant $C_\lambda$ explicit.

The main technical achievement of the present proof is to adapt the so-called peeling process argument of \cite{FeigeRon} and \cite{gheissari2023findingplantedcliquesusing} to the case of Gaussian noise with tensor elements rather than the Bernoulli noise and matrix elements setting of the previous works.
\begin{thm} \label{thm:Feige_analysis} Suppose $\theta \sim \mathrm{Unif}(\Theta_k)$ and that $\Omega\left(\sqrt{n}\right)= k\leq \frac{2}{3}n$. There exists a sufficiently large constant $C_\lambda = C_\lambda(r)>0$ such that if
\[\lambda \geq   C_\lambda  \frac{n^{\frac{r-1}{2}}}{k^{\frac{r}{2}-1}}\sqrt{\log n}\]
then, with high probability as $n\rightarrow +\infty$, Algorithm \ref{alg:GreedyPeeling1} with input $Y$ outputs a vector $\sigma\in\{0,1\}^n$ satisfying $\norm{\sigma}_0\leq \frac{3}{2}k$ and $\langle \sigma,\theta\rangle \geq\frac{1}{8}k$.
\end{thm}

\subsection{Notation}
Here we introduce some notation used throughout this section. Let
\[Q=\max\{Y,0\}\]
with the $\max$ operation applied elementwise. Let $(P^t)_{t\geq 0}$ denote the iterates of Algorithm \ref{alg:GreedyPeeling1} with input $Y$. Define for each $0\leq t\leq n-\frac{3}{2}k$ and each $i\in[n]$ the indicator variables
\[\cD_t(i) = \mathbf{1}\left(P^t_i=0\text{ and } \theta_i=0\right)\]
and 
\[\cD_t^*(i) = \mathbf{1}\left(P^t_i=0\text{ and } \theta_i=1\right)\]
We denote by $\cD_t$ the vector $(\cD_t(1),\dots,\cD_t(n))$, and similarly for $\cD_t^*$. Define for all $\lambda>0$ and for all $k$, 
\[m_{\lambda,k} = \frac{\lambda}{k^{r/2}}\Phi\left(\frac{\lambda}{k^{r/2}}\right)+\phi\left(\frac{\lambda}{k^{r/2}}\right)\]
where $\Phi$ and $\phi$ are the standard normal c.d.f. and p.d.f. respectively. Note that $m_{\lambda,k}=\mathbb{E}\left[\max\left\{X,0\right\}\right]$ where $X\sim\cN\left(\frac{\lambda}{k^{r/2}},1\right)$.

Without loss of generality and for ease of notation, we will assume that all of $\theta$'s nonzero coordinates lie in its first $k$ entries, i.e. that $\{i\in[n]:\theta_i=1\}=[k]$. For each $i\in[k]$, let
\begin{align}
\label{eq:t_i}
T_i = \inf\{t:P^t_i=0\} 
\end{align} be the time that $i$ is first removed by Algorithm \ref{alg:GreedyPeeling1}. For any two random variables $A$ and $B$, we write $A\stleq B$ if for all $t$ it holds that $\mathbb{P}(A\geq t)\leq \mathbb{P}(B\geq t)$.

Finally, when $r\geq 3$, we will define the tensor $H_t\in\{0,1\}^{n^{\otimes r-1}}$ for each time $t$ as
\begin{align}\label{eq:HtensorDef}
    H_t:=\sum_{\stackrel{j_2,\dots,j_r\in[n]^{r-1}}{\text{indices span $\geq 2$ of $\supp(P^t),\supp(\cD_t),\supp(\cD^*_t)$}}}e_{j_2}\otimes\dots\otimes e_{j_r}
\end{align}
Equivalently, we can define $H_t$ as
\begin{align}
    H_t:=\bone^{\otimes r-1}-(P^t)^{\otimes r-1}-(\cD_t)^{\otimes r-1}-(\cD_t^*)^{\otimes r-1}
\end{align}
(i.e. the tensor product consisting of cross terms of $P^t$, $\cD_t$, and $\cD_t^*$) so that we may use the following identity for any $i\in [k]$ (which holds true because $P^t+\cD_t+\cD_t^*=\bone$ by definition):
\begin{align}\label{eq:HtensorIdentity}
    &\langle e_i\otimes (P^t)^{\otimes r-1},Q\rangle\\
        &= \langle e_i\otimes \bone ^{\otimes r-1},Q\rangle -\langle e_i\otimes (\cD^t)^{\otimes r-1},Q\rangle - \langle e_i\otimes (\cD_*^t)^{\otimes r-1},Q\rangle-\langle e_i\otimes H_t,Q\rangle
\end{align}
When $r=2$, we define $H_t=\mathbf{0}\in\mathbb{R}^n$.
\subsection{Proof of Theorem \texorpdfstring{\ref{thm:Feige_analysis}}{thm:Feige analysis}}
In order to prove Theorem \ref{thm:Feige_analysis}, we will use four key results, Lemmas \ref{lem:SizeOfAgeqK}, \ref{lem:CorrWithOnesVecLB}, \ref{lem:mlamKlarge}, and \ref{lem:goodEvent2}. Their proofs are given in Section \ref{sec:peelingdensekeyproofs}.

\begin{lem}\label{lem:SizeOfAgeqK}
     Fix any $\varepsilon\in(0,\frac{1}{2})$. For any $\kappa>0$ (where we allow $\kappa$ to grow with $n$), define $\cA(\kappa)$ as the following random subset of $[k]$:
     \[
     \label{eq:setATensor1}\begin{split} \mathcal{A}(\kappa) = \bigg\{i\in [k]: &\langle e_i\otimes (P^t)^{\otimes r-1},Q\rangle\\
     \geq&\langle e_i\otimes \bone^{\otimes r-1},Q\rangle -  (2\pi)^{-1/2}\left(\norm{\cD_t}_0^{r-1}+\cH_{1,t}\right) -m_{\lambda,k}\left(\norm{\cD_t^*}_0^{r-1}+\cH_{2,t}\right)\\
        &-\frac{\kappa}{\sqrt{\varepsilon}}\left(2n^{(r-1)/2}+k^{(r-1)/2}\right)\\ &\text{ for all } t < T_i\bigg\},\end{split}
     \]
     where $T_i$ is defined in \eqref{eq:t_i}, and
    \begin{align}\label{eq:H1def}
        \cH_{1,t}:=n^{r-1}-k^{r-1}-\norm{\cD_t}_0^{r-1}-\norm{P_t}_0^{r-1}+(k-\norm{\cD_t^*}_0)^{r-1}
    \end{align}
    counts the number of choices of $(j_2,\dots,j_r)$ such that at least one coordinate is not in the support of $\theta$ \emph{and} each $(j_2,\dots,j_r)$ spans at least two of $(P^t,\cD_t,\cD_t^*)$. Further,
    \begin{align}\label{eq:H2def}
        \cH_{2,t}:=k^{r-1}-(k-\norm{\cD_t^*}_0)^{r-1}-\norm{\cD_t^*}_0^{r-1}
    \end{align}
    counts the number of choices of $(j_2,\dots,j_r)$ such that all coordinates are in the support of $\theta$ \emph{and} each $(j_2,\dots,j_r)$ spans both $P^t$ and $\cD_t^*$. Then,
     \[
        \PP\left(|\cA(\kappa)|\geq (1-\varepsilon) k\right) \geq 1-\frac{3}{\kappa^2}
     \]

\end{lem}
\begin{lem}\label{lem:CorrWithOnesVecLB}
    The following holds with probability at least $1-\frac{4k}{n^{2}}$: every  $i \in[k]$ satisfies
    \begin{align}
        \langle e_i\otimes \bone^{\otimes r-1},Q\rangle \geq m_{\lambda,k}k^{r-1}+(2\pi)^{-1/2}\left(n^{r-1}-k^{r-1}\right)-2\sqrt{\frac{2}{c}}(\log n)^{1/2}n^{(r-1)/2}
    \end{align}
    where $c>0$ is a universal constant.
\end{lem}
\begin{lem}\label{lem:mlamKlarge}
    Let $C>0$ be any constant. Then there exists a large enough constant $C_\lambda>0$ such that if $\lambda \geq  C_\lambda  \frac{n^{\frac{r-1}{2}}}{k^{\frac{r}{2}-1}}\sqrt{\log n}$, then 
    \begin{align}
        \left(m_{\lambda,k}-(2\pi)^{-1/2}\right)k^{r-1}\geq C\sqrt{\log n}n^{(r-1)/2}
    \end{align}
\end{lem}

\begin{lem}\label{lem:goodEvent2}
    There exists a constant $C=C(r)>0$ such that with probability at least $1-e^{-n}$, the following holds for every $t$ such that $0\leq t\leq n-\frac{3}{2}k-1$:
     \begin{align}
    \min_{i: (P^{t})_i = 1, \theta_i = 0} \langle e_i\otimes (P^t)^{\otimes r-1},Q\rangle \leq (2\pi)^{-1/2}\left(n-\norm{\cD_t}_0+\norm{\cD_t^*}_0\right)^{r-1}+Cn^{(r-1)/2}\\
    \end{align}
\end{lem}

We are now ready to prove Theorem \ref{thm:Feige_analysis}.
\begin{proof}[Proof of Theorem \ref{thm:Feige_analysis}]
    We will use Lemmas \ref{lem:SizeOfAgeqK}, \ref{lem:CorrWithOnesVecLB} and \ref{lem:mlamKlarge} to obtain a relatively large lower bound on the value of $\langle e_i\otimes (P^t)^{\otimes r-1},Q\rangle$ for most coordinates $i$ in the support of $\theta$, at all times before they are removed. 

    We will then use Lemma \ref{lem:goodEvent2} to upper bound the value of $\langle e_i\otimes (P^t)^{\otimes r-1},Q\rangle$ when $i$ is not in the support of $\theta$. Consequently, we will find that the minimizer of $\langle e_i\otimes (P^t)^{\otimes r-1},Q\rangle$ is ``usually" outside the support of $\theta$, so most of the ``good" coordinates are retained by Algorithm \ref{alg:GreedyPeeling1}.
    
    Let $\varepsilon\in(0,\frac{1}{2})$ be a constant,  and let $\kappa>0$ be a parameter to be determined (possibly growing to infinity). Using Lemmas \ref{lem:SizeOfAgeqK}, \ref{lem:CorrWithOnesVecLB} and \ref{lem:goodEvent2}, we have via a union bound that the following events hold simultaneously with probability at least $1-3/\kappa^2-4k/n^2-e^{-n}$: first,
    \[|\cA(\kappa)|\geq (1-\varepsilon)k\label{eq:FeigeEvent1}\]
    where $\cA(\kappa)$ is defined in \eqref{eq:setATensor1}. Secondly, every $i\in[k]$ satisfies 
    \begin{align}\label{eq:FeigeEvent2}
        \langle e_i\otimes \bone^{\otimes r-1},Q\rangle \geq m_{\lambda,k}k^{r-1}+(2\pi)^{-1/2}\left(n^{r-1}-k^{r-1}\right)-2\sqrt{\frac{2}{c}}(\log n)^{1/2}n^{(r-1)/2},
    \end{align}
    where $c>0$ is the universal constant in \eqref{lem:CorrWithOnesVecLB}. Finally, for every $0\leq t\leq n-\frac{3}{2}k-1$,
    \begin{align}\label{eq:FeigeEvent3}
        \min_{i: P^{t}_i = 1, \theta_i = 0} \langle e_i\otimes (P^t)^{\otimes r-1},Q\rangle \leq (2\pi)^{-1/2}\left(n-\norm{\cD_t}_0+\norm{\cD_t^*}_0\right)^{r-1}+Cn^{(r-1)/2}
    \end{align}
    where $C=C(r)>0$ is the sufficiently large constant in Lemma \ref{lem:goodEvent2}.
    Conditional on the events \eqref{eq:FeigeEvent1}, \eqref{eq:FeigeEvent2} and \eqref{eq:FeigeEvent3}, we will prove inductively that with probability 1, $\cA(\kappa) \subseteq \supp (P^t)$ for every $t\leq n-\frac{3}{2}k$, and therefore, for a sufficiently small choice of $\varepsilon$, we have that $P^{n-\frac{3}{2}k}$ satisfies $\langle P^{n-\frac{3}{2}k},\theta\rangle \geq \frac{1}{8}k$, as desired. If we then choose $\kappa=\frac{1}{3}(\varepsilon\log n)^{1/2}$ so that $1-3/\kappa^2-4k/n^2-e^{-n}=1-o(1)$, the theorem then follows. 
    
    Clearly we have $\cA(\kappa)\subseteq P^0 = \mathbf{1}$. Assume that $\cA(\kappa) \subseteq \supp(P^{t})$. Then 
    by definition of $\cA(\kappa)$ (see \eqref{eq:setATensor1}), it holds for every $i\in\cA(\kappa)$ that
    \begin{align}\label{eq:FeigeIneq7}
        &\langle e_i\otimes (P^t)^{\otimes r-1},Q\rangle\\
        &\geq \langle e_i\otimes \bone^{\otimes r-1},Q\rangle -  (2\pi)^{-1/2}\left(\norm{\cD_t}_0^{r-1}+\cH_{1,t}\right) -m_{\lambda,k}\left(\norm{\cD_t^*}_0^{r-1}+\cH_{2,t}\right)\\
        &-\frac{\kappa}{\sqrt{\varepsilon}}\left(2n^{(r-1)/2}+k^{(r-1)/2}\right)\\
        \intertext{and by \eqref{eq:FeigeEvent2}, we can lower bound this by}
        &\geq m_{\lambda,k}k^{r-1}+(2\pi)^{-1/2}\left(n^{r-1}-k^{r-1}\right)-2\sqrt{\frac{2}{c}}(\log n)^{1/2}n^{(r-1)/2}\\
        &- (2\pi)^{-1/2}\left(\norm{\cD_t}_0^{r-1}+\cH_{1,t}\right)-m_{\lambda,k}\left(\norm{\cD_t^*}_0^{r-1}+\cH_{2,t}\right)\\
        &-\frac{\kappa}{\sqrt{\varepsilon}}\left(2n^{(r-1)/2}+k^{(r-1)/2}\right)\\
        \intertext{Next apply the trivial upper bound that $2n^{(r-1)/2}+k^{(r-1)/2}\leq 3n^{(r-1)/2}$, so this is at least}
        &\geq m_{\lambda,k}k^{r-1}+(2\pi)^{-1/2}\left(n^{r-1}-k^{r-1}\right)-2\sqrt{\frac{2}{c}}(\log n)^{1/2}n^{(r-1)/2}\\
        &-  (2\pi)^{-1/2}\left(\norm{\cD_t}_0^{r-1}+\cH_{1,t}\right)-m_{\lambda,k}\left(\norm{\cD_t^*}_0^{r-1}+\cH_{2,t}\right) -\frac{3\kappa n^{(r-1)/2}}{\sqrt{\varepsilon}}\\
        \intertext{Applying the definition of $\cH_{1,t}$ and $\cH_{2,t}$ and rearranging yields that this is equal to}
        &= \left(m_{\lambda,k}-(2\pi)^{-1/2}\right)\left(k-\norm{\cD_t^*}_0\right)^{r-1}+\left(2\pi\right)^{-1/2}\left(n-\norm{\cD_t}_0-\norm{\cD_t^*}_0\right)^{r-1}\\
        &-2\sqrt{\frac{2}{c}}(\log n)^{1/2}n^{(r-1)/2} -\frac{3\kappa n^{(r-1)/2}}{\sqrt{\varepsilon}}\\
        \intertext{Recall that $\cA(\kappa)\subseteq \supp(P^{t})$ by inductive hypothesis and $|\cA(\kappa)|\geq (1-\varepsilon)k$ by \eqref{eq:FeigeEvent1}, and so $\norm{\cD_t^*}_0\leq \varepsilon k$, and thus this is at least}
        &=(1-\varepsilon)^{r-1} \left(m_{\lambda,k}-(2\pi)^{-1/2}\right)k^{r-1}+\left(2\pi\right)^{-1/2}\left(n-\norm{\cD_t}_0-\norm{\cD_t^*}_0\right)^{r-1}\\
        &-2\sqrt{\frac{2}{c}}(\log n)^{1/2}n^{(r-1)/2} -\frac{3\kappa n^{(r-1)/2}}{\sqrt{\varepsilon}}\\
        \intertext{Now we apply Lemma \ref{lem:mlamKlarge} with $C=(2+2\sqrt{\frac{2}{c}})2^{-(r-1)}$ and the bound $\varepsilon<\frac{1}{2}$, to get a bound of}
        &\geq 2\sqrt{\log n}n^{(r-1)/2}+\left(2\pi\right)^{-1/2}\left(n-\norm{\cD_t}_0-\norm{\cD_t^*}_0\right)^{r-1} -\frac{3\kappa n^{(r-1)/2}}{\sqrt{\varepsilon}}\\
        \intertext{Finally, with the choice $\kappa=\frac{1}{3}(\varepsilon\log n)^{1/2}$, since $\varepsilon$ is constant, we see that for large enough $n$ this is lower bounded by}
        &\geq \sqrt{\log n}n^{(r-1)/2}+\left(2\pi\right)^{-1/2}\left(n-\norm{\cD_t}_0-\norm{\cD_t^*}_0\right)^{r-1} \\
    \end{align}
    Combining all the above inequalities, we therefore have that for every $i\in\cA(\kappa)$,

    \begin{align}
         \sqrt{\log n}n^{(r-1)/2}+\left(2\pi\right)^{-1/2}\left(n-\norm{\cD_t}_0-\norm{\cD_t^*}_0\right)^{r-1}\leq \langle e_i\otimes (P^t)^{\otimes r-1},Q\rangle\label{eq:FeigeIneq9}
    \end{align}

    By \eqref{eq:FeigeEvent3}, we have the upper bound
    \begin{align}
    \min_{i: P^{t}_i = 1, \theta_i = 0} \langle e_i\otimes (P^{t})^{\otimes r-1},Q\rangle  &\leq (2\pi)^{-1/2}\left(n-\norm{\cD_t}_0+\norm{\cD_t^*}_0\right)^{r-1}+Cn^{(r-1)/2}\\
    &< \left(2\pi\right)^{-1/2}\left(n-\norm{\cD_t}_0-\norm{\cD_t^*}_0\right)^{r-1}+\sqrt{\log n}n^{(r-1)/2}\\
    &\leq \min_{i\in\cA(\kappa)}\langle e_i\otimes (P^{t})^{\otimes r-1},Q\rangle
    \end{align}
    where the second inequality holds for large enough $n$ and the last inequality holds by \eqref{eq:FeigeIneq9}. Thus, $\cA(\kappa)\subseteq \supp (P^{t+1})$, as desired, completing the proof. 
    
\end{proof}

\subsection{Proofs of Key Lemmas}\label{sec:peelingdensekeyproofs}
Towards proving Lemma \ref{lem:SizeOfAgeqK}, we introduce Lemmas \ref{lem:StochDomProcess} and \ref{lem:bprocess_meanvar}. Their proofs are given in Section \ref{sec:peelingdenseauxproofs}. The arguments used are inspired by those in \cite{gheissari2023findingplantedcliquesusing} and \cite{FeigeRon} for the setting of detecting a planted clique in a random graph.
\begin{lem}\label{lem:StochDomProcess}
    For every $i\in[k]$, there exist stochastic processes $\{b_{t,1}(i)\}_{t\leq n}$ and $\{b_{t,2}(i)\}_{t\leq n}$ such that for each $t<T_i$ (where $T_i$ is defined in \eqref{eq:t_i}), we have
    \[\langle e_i \otimes (\cD_t)^{\otimes r-1},Q\rangle \stleq b_{t,1}(i),\]
    \[\langle e_i \otimes (\cD_t^*)^{\otimes r-1},Q\rangle \stleq b_{t,2}(i),\]
    and
    \[
        \langle e_i \otimes H_t,Q\rangle \stleq b_{t,3}(i),
    \]
    where $H_t$ is defined in \eqref{eq:HtensorDef} (recall that $\stleq$ refers to stochastic domination). The process $\{b_{t,1}(i)\}_{t\leq n}$ is given by
    \[
        b_{t,1}(i) = \sum_{j=1}^{\norm{\cD_t}_0^{r-1}}\eta_{j,1}(i)
    \]
    with initial value $\eta_{0,1}(i)=0$, where $\{\eta_{j,1}(i)\}_{j\geq 1}$ are independent random variables such that for each $j$, $\PP(\eta_{j,1}(i)=0)=\frac{1}{2}$, and with probability $\frac{1}{2}$, $\eta_{j,1}(i)$ follows a $\cN(0,1)$ distribution truncated to the positive half of the real line.
     
     In addition, $\{b_{t,2}(i)\}_{t<T_i}$ is given by
    \[
       b_{t,2}(i) = \sum_{j=1}^{\norm{\cD^*_t}_0^{r-1}}\eta_{j,2}(i)
    \]
     with initial value $\eta_{0,2}(i)=0$, where $\{\eta_{j,2}(i)\}_{j\geq 1}$ are independent random variables such that for each $j$, $\PP(\eta_{j,2}(i)=0)=\Phi(-\frac{\lambda}{k^{r/2}})$, and with probability $1-\Phi\left(-\frac{\lambda}{k^{r/2}}\right)$, $\eta_{j,2}(i)$ follows a $\cN\left(\frac{\lambda}{k^{r/2}},1\right)$ distribution truncated to the positive half of the real line.

    Finally, $\{b_{t,3}(i)\}_{t<T_i}$ is given by 
    \[ 
        b_{t,3}(i) = \sum_{j=1}^{\cH_{1,t}}\eta_{j,3}(i)+\sum_{j=1}^{\cH_{2,t}}\eta_{j,4}(i)
    \]
    with initial values $\eta_{0,3}(i)=\eta_{0,4}(i)=0$, where $\cH_{1,t}$ and $\cH_{2,t}$ are defined in \eqref{eq:H1def} and \eqref{eq:H2def} respectively.  Finally, $\{\eta_{j,3}(i)\}_{j\geq 1}$ are i.i.d. with the same distribution as $\{\eta_{j,1}(i)\}_{j\geq 1}$, and $\{\eta_{j,4}(i)\}_{j\geq 1}$ are i.i.d. with the same distribution as $\{\eta_{j,2}(i)\}_{j\geq 1}$.
\end{lem}

    \begin{lem}\label{lem:bprocess_meanvar}
    For each $i\in[k]$, let $b_{t,1}(i)$, $b_{t,2}(i)$ and $b_{t,3}(i)$ be the processes defined in Lemma \ref{lem:StochDomProcess}. Then for each $t$, 
    \[\EE(b_{t,1}(i)|\norm{\cD_t}_0)=(2\pi)^{-1/2}\norm{\cD_t}_0^{r-1},\]
    \[\EE(b_{t,2}(i)|\norm{\cD_t^*}_0)=m_{\lambda,k}\norm{\cD_t^*}_0^{r-1} \]
    and 
    \[\EE(b_{t,3}(i)|\cH_{1,t},\cH_{2,t})=(2\pi)^{-1/2}\cH_{1,t}+m_{\lambda,k}\cH_{2,t}\]
    It further holds that for each $t$ that,
    \[\mathrm{Var}(b_{t,1}(i))\leq n^{r-1},\]
    \[\mathrm{Var}(b_{t,2}(i))\leq k^{r-1},\]
    and
    \[\mathrm{Var}(b_{t,3}(i))\leq n^{r-1}.\]
\end{lem}

\begin{proof}[Proof of Lemma \ref{lem:SizeOfAgeqK}]
    Fix $i\in[k]$ and $\kappa>0$. Let $b_{t,1}(i)$ be the process in Lemma \ref{lem:StochDomProcess} that stochastically dominates $\langle e_i \otimes (\cD_t)^{\otimes r-1},Q\rangle$. Observe that
    \[b_{t,1}(i)-\frac{\norm{\cD_t}_0^{r-1}}{\sqrt{2\pi}}\]
    is a martingale, and $T_i\leq n$ is a bounded stopping time. For the remainder of this proof, we consider all probabilistic statements as conditional on the values of $\{\cD_t\}_{t<T_i}$ and $\{\cD_t^*\}_{t<T_i}$. Thus we have by Lemma \ref{lem:StochDomProcess} and Doob's martingale inequality that
    \begin{align}\label{eq:SizeOfAgeqK1}
    &\PP\left(\sup_{t<T_i}\left(\langle e_i \otimes (\cD^t)^{\otimes r-1},Q\rangle-\frac{\norm{\cD_t}_0^{r-1}}{\sqrt{2\pi}}\right)\geq \kappa\sqrt{\frac{n^{r-1}}{\varepsilon}}\right)\\
    \stackrel{\text{Lemma \ref{lem:StochDomProcess}}}{\leq}&\PP\left(\sup_{t<T_i}\left(b_{t,1}(i)-\frac{\norm{\cD_t}_0^{r-1}}{\sqrt{2\pi}}\right)\geq \kappa\sqrt{\frac{n^{r-1}}{\varepsilon}}\right)\\
    \leq &\PP\left(\sup_{t\leq n-\frac{3}{2}k}\left(b_{t,1}(i)-\frac{\norm{\cD_t}_0^{r-1}}{\sqrt{2\pi}}\right)\geq \kappa\sqrt{\frac{n^{r-1}}{\varepsilon}}\right)\\
    \leq&\PP\left(\sup_{t\leq n-\frac{3}{2}k}\left(b_{t,1}(i)-\frac{\norm{\cD_t}_0^{r-1}}{\sqrt{2\pi}}\right)^2\geq \left(\kappa\sqrt{\frac{n^{r-1}}{\varepsilon}}\right)^2\right)\\
    \stackrel{\text{Doob}}{\leq}&\frac{\EE\left(\left(b_{n,1}(i)-\frac{\norm{\cD_{n}}_0^{r-1}}{\sqrt{2\pi}}\right)^2\right)}{\left(\kappa\sqrt{\frac{n^{r-1}}{\varepsilon}}\right)^2}\\
    \intertext{and since the numerator is just the variance of $\mathrm{Var}\left(b_{n-\frac{3}{2}k,1}(i)\right)\leq n^{r-1}$ by Lemma \ref{lem:bprocess_meanvar}, we get that this is at most}
    &\leq\frac{n^{r-1}}{\left(\kappa\sqrt{\frac{n^{r-1}}{\varepsilon}}\right)^2} = \frac{\varepsilon}{\kappa^2} \label{eq:ABound1}
    \end{align}
    Next, let $b_{t,2}(i)$ be the process in Lemma \ref{lem:StochDomProcess} that stochastically dominates $\langle e_i \otimes (\cD_t^*)^{\otimes r-1},Q\rangle$. Recall that $m_{\lambda,k} = \frac{\lambda}{k^{r/2}}\Phi\left(\frac{\lambda}{k^{r/2}}\right)+\phi\left(\frac{\lambda}{k^{r/2}}\right)$, and notice that
    \begin{align}
        &b_{t,2}(i)-m_{\lambda,k}\norm{\cD_t^*}_0^{r-1}
    \end{align}
    is also a martingale. Then by an identical argument, we have that 
    \begin{align}
         &\PP\left(\sup_{t<T_i}\left\langle e_i \otimes (\cD^*_t)^{\otimes r-1},Q\right\rangle-m_{\lambda,k}\norm{\cD_t^*}_0^{r-1}\geq  \kappa\sqrt{\frac{k^{r-1}}{\varepsilon}} \right)\\
          &\leq \frac{\EE\left(\left(b_{n-\frac{3}{2}k,2}(i)-m_{\lambda,k}\norm{\cD_{n-\frac{3}{2}k}^*}_0^{r-1}\right)^2\right)}{\left(\kappa\sqrt{\frac{k^{r-1}}{\varepsilon}} \right)^2}\\
    \end{align}
    Again since the numerator is $\mathrm{Var}\left(b_{n-\frac{3}{2}k,2}(i)\right)\leq  k^{r-1}$ by Lemma \ref{lem:bprocess_meanvar}, we get that
    \begin{align}
         &\PP\left(\sup_{t<T_i}\left\langle e_i \otimes (\cD^*_t)^{\otimes r-1},Q\right\rangle-m_{\lambda,k}\norm{\cD_{t}^*}_0^{r-1}\geq  \kappa\sqrt{\frac{k^{r-1}}{\varepsilon}} \right)\\
          &\leq \frac{\varepsilon}{\kappa^2}\label{eq:ABound2}
    \end{align}

    Now let $b_{t,3}(i)$ be the process in Lemma \ref{lem:StochDomProcess} that stochastically dominates $\langle e_i\otimes  H_t,Q\rangle$. Then notice that
    \[
        b_{t,3}(i)-(2\pi)^{-1/2}\cH_{1,t}-m_{\lambda,k}\cH_{2,t}
    \]
    is also a martingale, and so by an identical argument to the above two cases, we have
    \begin{align}
        &\mathbb{P}\left(\sup_{t<T_i}\left\langle e_i \otimes H_t,Q\right\rangle -(2\pi)^{-1/2}\cH_{1,t}-m_{\lambda,k}\cH_{2,t} \geq \kappa\sqrt{\frac{n^{r-1}}{\varepsilon}}\right)\\
        &\leq \frac{\varepsilon}{\kappa^2}\label{eq:ABound3}
    \end{align}
    Next, we note that by \eqref{eq:HtensorIdentity}, we have the identity
    \begin{align}
        \langle e_i\otimes (P^t)^{\otimes r-1},Q\rangle =  \langle e_i\otimes \bone ^{\otimes r-1},Q\rangle -\langle e_i\otimes (\cD^t)^{\otimes r-1},Q\rangle - \langle e_i\otimes (\cD_*^t)^{\otimes r-1},Q\rangle-\langle e_i\otimes H_t,Q\rangle
    \end{align}
    Thus by \eqref{eq:ABound1}, \eqref{eq:ABound2}, and \eqref{eq:ABound3}, we get that with probability at most $\frac{3\varepsilon}{\kappa^2}$, it holds that 
    \begin{align}
        &\inf_{t<T_i}\langle e_i\otimes (P^t)^{\otimes r-1},Q\rangle \\
        &< \langle e_i\otimes \bone^{\otimes r-1},Q\rangle -  (2\pi)^{-1/2}\left(\norm{\cD_t}_0^{r-1}+\cH_{1,t}\right) -m_{\lambda,k}\left(\norm{\cD_t^*}_0^{r-1}+\cH_{2,t}\right)\\
        &-\frac{\kappa}{\sqrt{\varepsilon}}\left(2n^{(r-1)/2}+k^{(r-1)/2}\right)
    \end{align}
    This holds for each $i\in[k]$, and therefore by Markov's inequality we have
    \begin{align}
        \mathbb{P}\left(|[k]\setminus\cA(\kappa)|\geq \varepsilon k\right)\leq \frac{k\left(\frac{3\varepsilon}{\kappa^2}\right)}{\varepsilon k}=\frac{3}{\kappa^2}
    \end{align}
    as desired.
\end{proof}

    \begin{proof}[Proof of Lemma \ref{lem:CorrWithOnesVecLB}]
    Fix $i\in[k]$. Note that we can decompose $\langle e_i\otimes \bone^{\otimes r-1},Q\rangle$ as
    \begin{align}\label{eq:eiDecomp}
        \langle e_i\otimes \bone^{\otimes r-1},Q\rangle = \sum_{\underline{h}\in[k]^{r-1}}\max\left\{\frac{\lambda}{k^{r/2}}+W_{i\underline{h}},0\right\}+\sum_{\underline{h}\in[n]^{r-1}\setminus[k]^{r-1}]}\max\left\{W_{i\underline{h}},0\right\}
    \end{align}
    For each $\underline{h}\in[k]^{r-1}$, $\max\left\{\frac{\lambda}{k^{r/2}}+W_{i\underline{h}},0\right\}$ is a subgaussian random variable with mean $m_{\lambda,k}$ and variance at most $1$. This can be seen by noting that $\frac{\lambda}{k^{r/2}}+W_{i\underline{h}}\sim\mathcal{N}\left(\frac{\lambda}{k^{r/2}},0\right)$, and then applying the formula for the mean and variance of a truncated Gaussian random variable (see e.g. page 156 of \cite{JohnsonKotzBalakrishnan1994}). Then by Hoeffding's inequality, it holds for some universal constant $c>0$ that for any $t>0$,
    \begin{align}
        \PP\left(\left|\sum_{\underline{h}\in[k]^{r-1}}\left(\max\left\{\frac{\lambda}{k^{r/2}}+W_{i\underline{h}},0\right\}-m_{\lambda,k}\right)\right|\geq \sqrt{t} \right)\leq 2\exp\left(\frac{-ct}{k^{r-1}}\right)
    \end{align}
    and so by taking $t=\frac{2}{c}k^{r-1}\log n$, we have that with probability at least $1-2n^{-2}$,
    \begin{align}
        \sum_{\underline{h}\in[k]^{r-1}}\max\left\{\frac{\lambda}{k^{r/2}}+W_{i\underline{h}},0\right\}\geq m_{\lambda,k}k^{r-1}-\sqrt{\frac{2}{c}}k^{(r-1)/2}(\log n)^{1/2}.
    \end{align}
    By a similar argument with $t=\frac{2}{c}(n^{r-1}-k^{r-1})\log n$, it holds that 
    \begin{align}
        \sum_{\underline{h}\in[n]^{r-1}\setminus[k]^{r-1}]}\max\left\{W_{i\underline{h}},0\right\}&\geq (2\pi)^{-1/2}(n^{r-1}-k^{r-1})-\sqrt{\frac{2}{c}}\left(n^{r-1}-k^{r-1}\right)^{1/2}(\log n)^{1/2}
    \end{align}
    with probability at least $1-2n^{-2}$. We therefore see by \eqref{eq:eiDecomp} that with probability at least $1-4n^{-2}$,
    \begin{align}
        \langle e_i\otimes \bone^{\otimes r-1},Q\rangle &\geq m_{\lambda,k}k^{r-1}+(2\pi)^{-1/2}\left(n^{r-1}-k^{r-1}\right)-\sqrt{\frac{2}{c}}(\log n)^{1/2}\left(k^{(r-1)/2}+\left(n^{r-1}-k^{r-1}\right)^{1/2}\right)\\
        &\geq m_{\lambda,k}k^{r-1}+(2\pi)^{-1/2}\left(n^{r-1}-k^{r-1}\right)-2\sqrt{\frac{2}{c}}(\log n)^{1/2}n^{(r-1)/2}
    \end{align}
    Then by taking a union bound over all $i\in[k]$, the probability that this inequality holds for all $i$ is at least
    \[1-\frac{4k}{n^{2}}\]
    as desired
    \end{proof}

\begin{proof}[Proof of Lemma \ref{lem:mlamKlarge}]
    Define the function $f(x) = \frac{1}{2}x+\phi(x)-(2\pi)^{-1/2}$. Then we claim that first that for all $x\geq 0$, it holds that $f(x) \geq \frac{1}{4}x$. To see this, note we have equality at 0, and that $f'(x) =\frac{1}{2}-\frac{xe^{-x^2/2}}{\sqrt{2\pi}}\geq \frac{1}{4}$ for all $x>0$. 
    We have $\Phi(x)\geq \frac{1}{2}$ for all $x>0$, and therefore it holds that
    \begin{align}
        m_{\lambda,k}-(2\pi)^{-1/2}&\geq \frac{\lambda}{2k^{r/2}}+\phi\left(\frac{\lambda}{k^{r/2}}\right)-(2\pi)^{-1/2}\\
        &=f\left(\frac{\lambda}{k^{r/2}}\right) \geq \frac{\lambda}{4k^{r/2}}
    \end{align}
    and therefore
    \begin{align}
        \left(m_{\lambda,k}-(2\pi)^{-1/2}\right)k^{r-1}&\geq \frac{1}{4}\lambda k^{\frac{r}{2}-1}\\
        &\geq \frac{1}{4}C_\lambda \sqrt{\log n}n^{(r-1)/2}\\
        &\geq C\sqrt{\log n}n^{(r-1)/2}
    \end{align}
    for a large enough choice of $C_\lambda$.
    
\end{proof}

We introduce Lemma \ref{lem:goodEvent0} towards proving Lemma \ref{lem:goodEvent2}.
\begin{lem}\label{lem:goodEvent0}
Let $n\geq 7$. Let $G$ be a tensor in $\RR^{n^{\otimes r}}$ with independent subgaussian entries, each with variance at most 1. For each $\sigma\in B_\sigma:=\left\{\sigma\in\{0,1\}^n:|i
    \in [n]: \sigma_i = 1, \theta_i = 0| \geq \frac{1}{2}k\right\}$, define the random variable $Z_{\sigma}$: 
    \[Z_{\sigma} =
    \frac{1}{|\{i \in [n]: \sigma_i = 1, \theta_i = 0\}|} \sum_{i: \sigma_i = 1,
    \theta_i = 0} \langle e_i \otimes \sigma^{\otimes r-1}, G\rangle\]
    Then, there exists a constant $C=C(r) > 0$ depending only on $r$ such that the following holds with probability at least $1-e^{-n}$ for large enough $n$: 
    \[\label{eq:eventGood} \min_{i:\sigma_i=1,\theta_i=0} \langle e_i\otimes \sigma^{\otimes r-1},G\rangle  \leq \EE(Z_\sigma)+Cn^{(r-1)/2} \]
\end{lem}
\begin{proof}[Proof of Lemma \ref{lem:goodEvent0}]
    For each $\sigma\in\{0,1\}^n$, let $b_\sigma = |\{i \in [n]: \sigma_i = 1, \theta_i = 0\}|$. Consider a choice of
    $\sigma$ with $b_\sigma \geq k/2$. Observe that $Z_\sigma$ is a subgaussian random variable, whose variance is upper bounded by $ \frac{(b_\sigma+k)^{r-1}}{b_\sigma}$. This is because the inner product  $ \langle e_i \otimes \sigma^{\otimes r-1}, G\rangle$ contains at most $\norm{\sigma}_0^{r-1}\leq (b_\sigma+k)^{r-1}$ nonzero terms, each of variance at most 1. Therefore, $\sum_{i: \sigma_i = 1,
    \theta_i = 0} \langle e_i \otimes \sigma^{\otimes r-1}, G\rangle$ has variance at most $b_\sigma(b_\sigma+k)^{r-1}$, so $Z_\sigma$ has variance at most $\frac{(b_\sigma+k)^{r-1}}{b_\sigma}$. It therefore holds by Hoeffding's inequality that for some universal constant $c>0$,
    \begin{align}\PP(\left|Z_{\sigma}(j)-\EE(Z_\sigma(j))\right| \geq \sqrt{t}) \leq 2&\exp\left(-\frac{c b_\sigma t}{(b_\sigma + k)^{r-1}}\right)\\
    \leq 2 &\exp\left(-\frac{ct}{(3^{r-1})b_\sigma^{r-2}}\right)
    \end{align}
    where the second inequality holds because $b_\sigma\geq \frac{1}{2}k$ by assumption. Let $t=\frac{3^{r-1}}{c}n^{r-1}$ so that this probability is at most $2e^{-2n}$ (because $b_\sigma\leq n$). As there are trivially less than $2^n$ such choices of $\sigma$, by a union bound, the complement of event \eqref{eq:eventGood} has probability at most
    \[2\cdot 2^n \cdot e^{-2n} \leq e^{-n}\]
    for $n\geq 7$, completing the proof.
\end{proof}
\begin{proof}[Proof of Lemma \ref{lem:goodEvent2}]
    We first note that since $t\leq n-\frac{3}{2}k$, the conditions of Lemma \ref{lem:goodEvent0} are satisfied with $G=Q$. It therefore holds that for some $r$-dependent constant $C>0$ that, for every $t$, with probability at least $1-e^{-n}$ that
    \begin{align} 
        &\min_{i: (P^{t})_i = 1, \theta_i = 0} \langle e_i\otimes (P^t)^{\otimes r-1},Q\rangle\\
        &\leq\mathbb{E}\left(\frac{1}{|\{i \in [n]: P^t_i = 1, \theta_i = 0\}|} \sum_{i: P^t_i = 1, \theta_i = 0} \langle e_i \otimes (P^t)^{\otimes r-1}, Q\rangle\right)+Cn^{(r-1)/2}\\
        &=\mathbb{E}\left(\langle e_i \otimes (P^t)^{\otimes r-1}, Q\rangle\right)+Cn^{(r-1)/2}\\
    \end{align}
    where $i$ is any coordinate such that $P^t_i=1$ and $\theta_i=0$. The inner product contains at most $\norm{P^t}_0^{r-1}=(n-\norm{\cD_t}_0-\norm{\cD_t^*}_0)^{r-1}$ nonzero terms. Since $\theta_i=0$, each of these terms follows the distribution $X=\max\{Z,0\}$, where $Z\sim\cN(0,1)$. A simple calculation yields that $\mathbb{E}(X)=(2\pi)^{-1/2}$, completing the proof.
\end{proof}

\subsection{Proofs of Auxiliary Lemmas}\label{sec:peelingdenseauxproofs}
Here we prove Lemmas \ref{lem:StochDomProcess} and \ref{lem:bprocess_meanvar}.
\begin{proof}[Proof of Lemma \ref{lem:StochDomProcess}]
       Here we follow the edge-exposure argument of \cite{gheissari2023findingplantedcliquesusing} and \cite{FeigeRon}, but applied to the case of sparse tensor PCA instead of planted clique. We demonstrate the construction of $b_{t,1}(i)$, as the arguments for $b_{t,2}(i)$ and $b_{t,3}(i)$ are similar, but with a shift in the mean to account for those coordinates lying in the support of $\theta$. 

        We introduce Algorithm \ref{alg:TentativePeeling} as a proof device. We briefly give an informal description of its operation. We generate $Y$, and apply the operation $\max\{\cdot,0\}$ elementwise, similarly to Algorithm \ref{alg:GreedyPeeling1}. We then generate a masked tensor $\Gamma$ where we conceal what we call the ``connectivity" of every coordinate to $i$ (where this terminology is used because these arguments are inspired by graph-theoretic results). We do this by erasing the coordinates of $Q$ involving $i$ in their first index. 
        
        For the current revealed state of the tensor $\Gamma^t$, and the current set of coordinates being considered $M^t$, Algorithm \ref{alg:TentativePeeling} then proceeds in a similar manner to Algorithm \ref{alg:GreedyPeeling1}, by finding the coordinate $s$ that minimizes $\langle e_s \otimes (M^t)^{\otimes r-1}, \Gamma^t \rangle$, which we can view as as proxy for the ``connectivity" of a coordinate $s$ to the current set $M^t$. If at time $t$, coordinate $s$ minimizes $\langle e_s \otimes (M^t)^{\otimes r-1}, \Gamma^t \rangle$, Algorithm \ref{alg:TentativePeeling} then updates the revealed information in the tensor $\Gamma$ by revealing all the coordinates of $Q$ that both have $i$ as the first index and involve $s$. It then recalculates the minimizer. If a coordinate $s$ is still the minimizer after having its ``connectivity" revealed, the algorithm removes it from consideration by subtracting it from $M$. If not, it keeps looping until such a minimizer is found. After such a minimizer is found,  Algorithm \ref{alg:TentativePeeling} then repeats this process until $n-\frac{3}{2}k$ coordinates have been removed from $M$, similar to Algorithm \ref{alg:GreedyPeeling1}. 

        Our result is based on the intuitive idea that the revealed values $Q$ for coordinates that are dropped should be smaller than the values for those that are not dropped, because the dropped coordinates were still the minimizers even after their connectivity was revealed.
    \begin{algorithm}[h]
        \caption{Tentative Peeling Process for coordinate $i$.}\label{alg:TentativePeeling}
    \begin{algorithmic}[1]
        \REQUIRE $Y \in \mathbb{R}^{n^{\otimes r}}$, $i\in[k]$
    \STATE $Q\gets \max\{Y,0\}$ (elementwise)
    \STATE Let $\Gamma^0 \in\mathbb{R}^{n^{\otimes r}}$.
    \FOR{$(j_2,\dots,j_r)\in[n]^{r-1}$}
            \STATE $\Gamma^0_{i,j_2,\dots,j_r}\gets0$
    \ENDFOR
    \STATE Let $M^0 \gets \mathbf{1}\in\mathbb{R}^n$
    \FOR{$t = 0$ to $n-3k/2-1$}
        \STATE  $\ell \gets \argmin\limits_{s: (M^t)_s = 1} \langle e_s \otimes (M^t)^{\otimes r-1}, \Gamma^t \rangle$
        \STATE $\Tilde\Gamma^{t+1}\gets \Gamma^t$
        \WHILE{$\exists (j_2,\dots,j_r)$ containing $\ell$ such that $\Tilde \Gamma^{t+1}_{i,j_2,\dots,j_r}\neq Q_{i,j_2,\dots,j_r}$}
            \FOR{$(j_2,\dots,j_r)\in[n]^{r-1}$}
                \STATE $\Tilde \Gamma^{t+1}_{i,j_2,\dots,j_r}\gets Q_{i,j_2,\dots,j_r}$
            \ENDFOR
            \STATE  $\ell \gets \argmin\limits_{s: (M^t)_s = 1} \langle e_s \otimes (M^t)^{\otimes r-1}, \Tilde \Gamma^{t+1} \rangle$
            
        \ENDWHILE
        \STATE $\Gamma^{t+1}\gets \Tilde\Gamma^{t+1}$
        \STATE  $\ell \gets \argmin\limits_{s: (M^t)_s = 1} \langle e_s \otimes (M^t)^{\otimes r-1}, \Gamma^{t+1} \rangle$
        \STATE $M^{t+1}\gets M^t-e_\ell$
    \ENDFOR
    \RETURN $M^{n - \frac{3}{2}k}$
    \end{algorithmic}
    \end{algorithm}
    We now proceed with the proof. Fix a choice of $i\in[k]$. Let $\Gamma^t$ and $M^t$ be the iterates of Algorithm \ref{alg:TentativePeeling} with input $Y$ and $i\in[k]$, and let $P^t$ be the iterates of Algorithm \ref{alg:GreedyPeeling1} with input $Y$. The first step of our proof is to show that for each $t<T_i$, $M^t=P^t$ with probability 1, where $T_i=\min\{t:P^t_i=0\}$ is the time that $i$ is first dropped by Algorithm \ref{alg:GreedyPeeling1}. 
    
    We proceed by induction. The base case is trivial as $M^0=P^0=\mathbf{1}$. Suppose that for some $t$, $M^t=P^t$. Let $\ell\in[n]$ be the coordinate removed from $M$ by Algorithm \ref{alg:TentativePeeling} at time $t$, i.e. we have $M^t_\ell=1$ and $M^{t+1}_\ell=0$. By the condition in the while loop of Algorithm \ref{alg:TentativePeeling}, this implies that the connectivity of $i$ and $\ell$ has been revealed at this time (i.e. $\Gamma^{t+1}_{i,j_2,\dots,j_r}=Q_{1,j_2,\dots,j_r}$ for every $(j_2,\dots,j_r)$ containing $\ell$). Thus,
    \begin{align}\label{eq:truncAlgeq1}
        \langle e_\ell \otimes (M^t)^{\otimes r-1}, \Gamma^{t+1} \rangle=\langle e_\ell \otimes (M^t)^{\otimes r-1}, Q \rangle=\langle e_\ell \otimes (P^t)^{\otimes r-1}, Q \rangle
    \end{align}
    Next, we note that by construction, for every $(j_1,\dots,j_r)$, it holds that $\Gamma^{t+1}_{j_1,\dots,j_r}\leq Q_{j_1,\dots,j_r}$. Therefore, since $M^t=P^t$ by inductive hypothesis, we have for every $s\in[n]$ that 
    \begin{align}\label{eq:truncAlgeq2}
        \langle e_s \otimes (M^t)^{\otimes r-1}, \Gamma^{t+1} \rangle\leq\langle e_s \otimes (P^t)^{\otimes r-1}, Q \rangle
    \end{align}
    We then note that since $\ell= \argmin\limits_{s: (M^t)_s = 1} \langle e_s \otimes (M^t)^{\otimes r-1}, \Gamma^{t+1}\rangle $ by definition (and as $Q$'s elements are truncated normal random variables, the $\argmin$ is unique with probability 1), it therefore holds by \eqref{eq:truncAlgeq1} and \eqref{eq:truncAlgeq2} that
    \[\ell= \argmin\limits_{s: (M^t)_s = 1} \langle e_s \otimes (M^t)^{\otimes r-1}, \Gamma^{t+1}\rangle= \argmin\limits_{s: (P^t)_s = 1} \langle e_s \otimes (P^t)^{\otimes r-1}, Q\rangle\]
    (informally, if connectivity within $M$ lower bounds connectivity within $P$, and $\ell$ is the minimizer for $M$, and connectivity for $\ell$ in $M$ equals that for $\ell$ in $P$, then $\ell$ must also be the minimizer for $P$). Therefore, $M^{t+1}=P^{t+1}$ with probability 1, which completes the induction. 
    
    Next, define the following two sets for each $t<T_i$:
    \begin{align}
        \cR_{t,i}:=\{s\in[n]\setminus[k]:\Gamma^{t+1}_{i,j_2,\dots,j_r}=Q_{i,j_2,\dots,j_r}\text{ for all }(j_2,\dots,j_r)\text{ containing $s$}\}
    \end{align}
    and 
    \begin{align}
        \cR^*_{t,i}:=\{s\in[k]:\Gamma^{t+1}_{i,j_2,\dots,j_r}=Q_{i,j_2,\dots,j_r}\text{ for all }(j_2,\dots,j_r)\text{ containing $s$}\}
    \end{align}
    $\cR_{t,i}$ can be viewed as the set of coordinates outside the support of $\theta$ whose connectivity to the coordinate $i\in[k]=\supp(\theta)$ has been revealed by time $t$. $\cR^*_{t,i}$ is similar but for coordinates in $\supp(\theta)$. For any choice of $t$, $i$ and $s$, we claim that for each $(j_2,\dots,j_r)$ containing $s$, the law of $Q_{i,j_2,\dots,j_r}$ is independent of the event $\{s\in\cR_{t,i}\}$. This is because $Q$ is not changed at any stage by Algorithm \ref{alg:TentativePeeling}, and because the decision to reveal $s$'s connectivity (i.e. to set $\Gamma^{t+1}_{i,j_2,\dots,j_r}=Q_{i,j_2,\dots,j_r}$) is taken before observing $Q_{i,j_2,\dots,j_r}$ . Therefore, conditional that $s\in\cR_{t,i}$, $\Gamma^{t+1}_{i,j_2,\dots,j_r}$ for any $(j_2,\dots,j_r)$ containing $s$ must have the same law as the marginal law of $Q_{i,j_2,\dots,j_r}$ - that is, $0$ with probability $\frac{1}{2}$, and with probability $\frac{1}{2}$ following a $\cN(0,1)$ distribution truncated to be nonnegative.
    
    A similar argument applies to $\cR^*_{t,i}$. Conditional that $s\in\cR^*_{t,i}$, $\Gamma^{t+1}_{i,j_2,\dots,j_r}$ for any $(j_2,\dots,j_r)$ containing $s$ must have the same law as the marginal law of $Q_{i,j_2,\dots,j_r}$. When at least one of $(j_2,\dots,j_r)$ lies outside the support of $\theta$ that marginal law is as above, $0$ with probability $\frac{1}{2}$, and with probability $\frac{1}{2}$ following a $\cN(0,1)$ distribution truncated to be nonnegative. On the other hand, when all of $(j_2,\dots,j_r)$ are in $\supp(\theta)$, then the marginal distribution is 0 with probability $\Phi(-\frac{\lambda}{k^{r/2}})$, and with probability $1-\Phi\left(-\frac{\lambda}{k^{r/2}}\right)$, follows a $\cN\left(\frac{\lambda}{k^{r/2}},1\right)$ distribution truncated to the positive half of the real line.

    Now, suppose that $s\in[n]$ has been dropped by time $t$ (i.e. that $M^t_s=P^t_s=0$). Conditional on this event, we claim that for any $(j_2,\dots,j_r)$ containing $s$, the distribution of $Q_{i,j_2,\dots,j_r}$ is stochastically dominated by its marginal law. That is, we claim that for any $x$ it holds that
    \begin{align}\label{eq:stochasticDomInequality}
       \mathbb{P}\left(Q_{i,j_2,\dots,j_r}\geq x|M^t_s=0\right)\leq \mathbb{P}\left(Q_{i,j_2,\dots,j_r}\geq x\right)
    \end{align}

    To see this, fix any $x\in\mathbb{R}$, any $s\in[n]$, and any $(j_2,\dots,j_r)$ containing $s$. For any random variable $X$ defined on the same probability space as $Q$ and $M^t$, define
    \[f(X)=\bone\left\{X\geq x\right\}\]
    Further, define the function $g$ as
    \[g(X)=\mathbb{P}\left(M_s^t=0 | X\geq x\right).\] 
    Then we note that $f(Q_{i,j_2,\dots,j_r})$ is increasing with $Q_{i,j_2,\dots,j_r}$, while $g(Q_{i,j_2,\dots,j_r})$ is decreasing with $Q_{i,j_2,\dots,j_r}$ (because if $Q_{i,j_2,\dots,j_r}$ is larger, the probability that $s$ will be dropped is lower). Therefore, by Chebyshev's sum inequality (see e.g. page 43 of \cite{hardy1934inequalities}), we have 
    \[\mathbb{E}\left[f(Q_{s,j_2,\dots,j_r})g(Q_{s,j_2,\dots,j_r})\right]\leq \mathbb{E}\left[f(Q_{s,j_2,\dots,j_r})\right]\mathbb{E}\left[g(Q_{s,j_2,\dots,j_r})\right]\]
    or equivalently
    \begin{align} 
        &\mathbb{E}\left[\bone\left\{Q_{i,j_2,\dots,j_r}\geq x\right\}\mathbb{P}\left(M_s^t=0 | Q_{i,j_2,\dots,j_r}\geq x\right)\right]\\
        &\leq \mathbb{E}\left[\bone\left\{Q_{i,j_2,\dots,j_r}\geq x\right\}\right]\mathbb{E}\left[\mathbb{P}\left(M_s^t=0 | Q_{i,j_2,\dots,j_r}\geq x\right)\right]
    \end{align}
    and therefore by the law of iterated expectations,
    \begin{align}
        \mathbb{P}\left[Q_{i,j_2,\dots,j_r}\geq x \text{ and } M_s^t=0\right]\leq \mathbb{P}\left[Q_{i,j_2,\dots,j_r}\geq x\right]\mathbb{P}\left[M_s^t=0\right]
    \end{align}
    which implies that
    \begin{align}
       \mathbb{P}\left(Q_{i,j_2,\dots,j_r}\geq x|M^t_s=0\right)\leq \mathbb{P}\left(Q_{i,j_2,\dots,j_r}\geq x\right)
    \end{align}
    and we obtain \eqref{eq:stochasticDomInequality} as desired.

    Now let $\{\eta_{j,1}(i)\}_{j\geq 1}$ be a sequence of i.i.d. random variables such that for each $j$, $\PP(\eta_{j,1}(i)=0)=\frac{1}{2}$, and with probability $\frac{1}{2}$, $\eta_{j,1}(i)$ follows a $\cN(0,1)$ distribution truncated to the positive half of the real line. By the above argument, it therefore holds that
    \begin{align}
        &\langle e_i \otimes (\cD_t)^{\otimes r-1},Q\rangle\\
            =&\sum_{(j_2,\dots,j_r)} \Gamma^{t+1}_{i,j_2,\dots,j_r}\cD_t(j_2)\dots\cD_t(j_r)\\
            \stleq &\sum_{j=1}^{\norm{\cD_t}_0^{r-1}}\eta_{j,1}(i)
    \end{align}
    where the last line holds because there are $\norm{\cD_t}_0^{r-1}$ nonzero terms in $\sum_{(j_2,\dots,j_r} \Gamma^{t+1}_{i,j_2,\dots,j_r}\cD_t(j_2)\dots\cD_t(j_r)$. Finally, taking
    \[b_{t,1}(i) = \sum_{j=1}^{\norm{\cD_t}_0^{r-1}}\eta_{j,1}(i)\] completes the proof.

    Next, we construct $b_{t,2}(i)$. Let $\{\eta_{j,2}(i)\}_{j\geq 1}$ be a sequence of i.i.d. random variables such that for each $j$, $\PP(\eta_{j,2}(i)=0)=\Phi(-\frac{\lambda}{k^{r/2}})$, and with probability $1-\Phi(-\frac{\lambda}{k^{r/2}})$, $\eta_{j,2}(i)$ follows a $\cN\left(\frac{\lambda}{k^{r/2}},1\right)$ distribution truncated to the positive half of the real line. By the above argument and \eqref{eq:stochasticDomInequality}, it thus holds that
     \begin{align}
     &\langle e_i \otimes (\cD^*_t)^{\otimes r-1},Q\rangle\\
        =&\sum_{(j_2,\dots,j_r)} \Gamma^{t+1}_{i,j_2,\dots,j_r}\cD_t^*(j_2)\dots\cD^*_t(j_r)\\
         \stleq &\sum_{j=1}^{\norm{\cD^*_t}_0^{r-1}}\eta_{j,2}(i)
     \end{align}
     We then let
     \[b_{t,2}(i) = \sum_{j=1}^{\norm{\cD_t}_0^{r-1}}\eta_{j,1}(i)\]
     Finally, we construct $b_{t,3}(i)$. Let $\{\eta_{j,3}(i)\}_{j\geq 1}$ be a sequence i.i.d. random variables with the same distribution as $\{\eta_{j,1}(i)\}_{j\geq 1}$, and let $\{\eta_{j,4}(i)\}_{j\geq 1}$ be i.i.d. with the same distribution as $\{\eta_{j,2}(i)\}_{j\geq 1}$. Then, the inner product $\langle e_i \otimes H_t,Q\rangle$, where $H_t$ is defined in \eqref{eq:HtensorDef} will have $\cH_{1,t}$ terms where at least one coordinate involved is not in the support of $\theta$, and $\cH_{2,t}$ terms where all coordinates are in the support of $\theta$. Then, by the above argument and \eqref{eq:stochasticDomInequality},
     \begin{align}
         &\langle e_i \otimes H_t,Q\rangle\\
         &=\sum_{\stackrel{j_2,\dots,j_r\in[n]^{r-1}}{\text{indices span $\geq 2$ of $\supp(P^t),\supp(\cD_t),\supp(\cD^*_t)$, not all in $\supp(\theta)$}}}Q_{i,j_2,\dots,j_r}\\
         &+\sum_{\stackrel{j_2,\dots,j_r\in[n]^{r-1}}{\text{indices span $\geq 2$ of $\supp(P^t),\supp(\cD_t),\supp(\cD^*_t)$, all in $\supp(\theta)$}}}Q_{i,j_2,\dots,j_r}\\
         &\stleq \sum_{j=1}^{\cH_{1,t}}\eta_{j,3}(i)+\sum_{j=1}^{\cH_{2,t}}\eta_{j,4}(i)
     \end{align}
     Finally, we let 
     \[b_{t,3}(i)=\sum_{j=1}^{\cH_{1,t}}\eta_{j,3}(i)+\sum_{j=1}^{\cH_{2,t}}\eta_{j,4}(i).\]
    \end{proof}

\begin{proof}[Proof of Lemma \ref{lem:bprocess_meanvar}]
    The mean follows directly by construction of $b_{t,1}(i)$, $b_{t,2}(i)$, and $b_{t,3}(i)$. The variance bound follows for fixed $t$ because, for every $j$ and every $i$, $\eta_{j,1}(i)$, $\eta_{j,2}(i)$, $\eta_{j,3}(i)$, and $\eta_{j,4}(i)$ all have variance at most $1$.
\end{proof}

\section{Proof of Theorem \texorpdfstring{\ref{thm:DenseBinaryMainTheorem}}{thm:DenseBinaryMainTheorem}}
In this section we prove Theorem \ref{thm:DenseBinaryMainTheorem1}, which directly implies Theorem \ref{thm:DenseBinaryMainTheorem}, but makes the dependence on constants and polylogarithmic factors more explicit.
\begin{thm}\label{thm:DenseBinaryMainTheorem1}
    Assume that $\Omega(\sqrt{n})= k=o(n)$, any $r\geq 2$ and $\theta \sim \mathrm{Unif}(\Theta_k)$. Suppose $\sigma_0\in\{0,1\}^n$ satisfies $\langle \sigma_0,\theta\rangle\geq \frac{1}{8}k$ and $\norm{\sigma}_0\leq \frac{3}{2}k$. There exist sufficiently large constants  $C_\lambda(r),C_\gamma(r)>0$ such that if all the following hold:
    \begin{itemize}
        \item $\gamma=C_\gamma\sqrt{\log n}$
        \item $\lambda \geq   C_\lambda  \frac{n^{\frac{r-1}{2}}}{k^{\frac{r}{2}-1}}\sqrt{\log n}$
        \item $M=\Omega\left(nk^{(r+1)/2}\right)$
        \item $\sigma_0\in\{0,1\}^n$ satisfies $\langle \sigma_0,\theta\rangle\geq \frac{1}{8}k$ and $\norm{\sigma}_0\leq \frac{3}{2}k$
    \end{itemize}
    then, with probability at least $\frac{2}{3}-o(1)$, Algorithm \ref{alg:RandGreedy2} with input $(Y,\sigma_0,M)$ outputs $\theta$.
\end{thm}

\subsection{Proof of Theorem \texorpdfstring{\ref{thm:DenseBinaryMainTheorem1}}{thm:DenseBinaryMainTheorem1}}
In order to prove Theorem \ref{thm:DenseBinaryMainTheorem1}, we will need the following key lemmas. Their proofs are given in Section \ref{sec:randrestrictkeyproofs}. Lemma \ref{lem:randRestrictWork01} is the main technical contribution of this section.
\begin{lem}\label{lem:randRestrictWork01}
    Assume the conditions in Theorem \ref{thm:DenseBinaryMainTheorem1} hold. For any constant $C>0$ define the set 
    \[\Tilde \Omega(C)=\left\{\sigma\in\{0,1\}^n\setminus\{\textbf{0}\}:\frac{\langle \sigma,\theta\rangle}{\norm{\sigma}\norm{\theta}}\geq C\left(\frac{\sqrt{k\log n}}{\lambda}\right)^{1/(r-1)}\text{, }\norm{\sigma}_0\geq 2\right\} \label{eq:tildeOmegaC}\]
    Then there exists a large enough choice of $C$ such that the following statements hold simultaneously with probability at least $1-3n^{-1}$:
    \begin{enumerate}[label=(\Alph*), ref=\Alph*]
    \item  \label{lem:randRestrictWork01:itemA} For any $\sigma \in \Tilde \Omega(C)\setminus\{\theta\}$ there exists a $\sigma' \in \Tilde \Omega(C)$ satisfying $d_H(\sigma', \sigma) = 1$ and
    \begin{align}
        H_{(r+1)/2, \gamma}(\sigma') - H_{(r+1)/2, \gamma}(\sigma)  \geq  \|\sigma\|^{(r-1)/2}_0 \sqrt{\log(n)} \label{eq:inc1}
    \end{align}

    \item \label{lem:randRestrictWork01:itemB}Let $\hat \Omega(C) = \Tilde \Omega(C) \cap \{\sigma \in \{0,1\}^n: \|\sigma\|_0 \leq 3k/2, \langle \sigma, \theta \rangle \geq k/10\}$. Then, for each $\sigma \in \hat \Omega(C)$ there are at least $\lfloor 2/3(k-\langle \sigma,\theta\rangle)\rfloor$ distinct $\sigma' \in  \hat \Omega(C)$ such that $d_H(\sigma', \sigma) = 1, d_H(\sigma',\theta)=d_H(\sigma',\theta)-1$ and 
    \[ H_{(r+1)/2, \gamma}(\sigma') - H_{(r+1)/2, \gamma}(\sigma)  \geq \frac{r}{2\cdot 10^{r-1}}\lambda k^{\frac{r}{2}-1}\]
\end{enumerate}

\end{lem}
We will also need the following two lemmas concerning the landscape of $H_{(r+1)/2, \gamma}$:
\begin{lem}\label{lem:GamilRange}
    The range $R_{H_{(r+1)/2, \gamma}}$ of $H_{(r+1)/2, \gamma}$ on the set $\{\sigma\in\{0,1\}^n:\norm{\sigma}_0\leq \frac{3}{2}k\}$ satisfies
    \[R_{H_{(r+1)/2, \gamma}}=O\left( k^{\frac{r}{2}}\left(\lambda+k^{\frac{r}{2}}\sqrt{\log n}\right)\right)\]
    with probability at least $1-n^{-1/2}$.
\end{lem}
\begin{lem}\label{lem:GamilDiffIdents}
    For any $\sigma\in\{0,1\}^n$, we have the following identities:
    \begin{itemize}
        \item For any $i$ such that $\sigma_i=0$,
        \begin{align}\label{eq:GamilDiffAddIden}
            &H_{(r+1)/2, \gamma}(\sigma+e_i)-H_{(r+1)/2, \gamma}(\sigma)\\
            &= \frac{\lambda\theta_i }{k^{r/2}}\sum_{j=1}^r \binom{r}{j}\langle \sigma,\theta\rangle^{r-j}+\sum_{j=1}^r\binom{r}{j}\langle e_i^{\otimes j} \otimes \sigma^{\otimes r-j}, W\rangle\\
            &-\gamma \left(\norm{\sigma}_0+1\right)^{\frac{r+1}{2}}+\gamma \norm{\sigma}_0^{\frac{r+1}{2}}
        \end{align}
        \item For any $i$ such that $\sigma_i=1$,
        \begin{align}\label{eq:GamilDiffRemoveIden}
             &H_{(r+1)/2, \gamma}(\sigma-e_i) - H_{(r+1)/2, \gamma}(\sigma)\\
             &=\frac{\lambda \theta_i}{k^{r/2}}\sum_{j=1}^r\binom{r}{j}(-1)^j \langle \sigma,\theta\rangle^{r-j} +\sum_{j=1}^r\binom{r}{j}(-1)^j \langle e_i^{\otimes j} \otimes \sigma^{\otimes r-j}, W\rangle\\
             &-\gamma \left(\norm{\sigma}_0-1\right)^{\frac{r+1}{2}} + \gamma \norm{\sigma}_0^{\frac{r+1}{2}}
        \end{align}
    \end{itemize}
\end{lem}

\begin{proof}[Proof of Theorem \ref{thm:DenseBinaryMainTheorem1}]
    We will first prove the statement in Theorem \ref{thm:DenseBinaryMainTheorem1} in the case where
    \[\lambda=C_\lambda  \frac{n^{\frac{r-1}{2}}}{k^{\frac{r}{2}-1}}\sqrt{\log n}\]
    for a large enough choice of constant $C_\lambda>0$. Later, we will use a coupling argument to strengthen the result to all $\lambda$ larger than this.
    
    Let $S_t$ denote the iterates of Algorithm \ref{alg:RandGreedy2} with input $(Y,\sigma_0,M)$. We first define another process $\Tilde S_t$ in the following way: for $1\leq t\leq M$, let $\Tilde S_t$ be the iterates of Algorithm \ref{alg:RandGreedyProofDevice} with input $(Y,\sigma_0,M)$. Note that Algorithm \ref{alg:RandGreedyProofDevice} is identical to Algorithm \ref{alg:RandGreedy2}, except that only transitions that maintain $\langle \Tilde S_t,\theta\rangle\geq \frac{k}{10}$ are considered. This is not an implementable algorithm, because we do not have direct access to $\langle \Tilde S_t,\theta\rangle$. However, we will show that with probability at least $1-e^{-0.01k}$, $S_t=\Tilde S_t$ for the first $\Omega(e^k)$ many steps. We will also show that $\Tilde S_M=\theta$ with probability at least $\frac{2}{3}-o(1)$. Morever, since $\theta$ is a local maximum of $H_{(r+1)/2, \gamma}$ with probability $1-o(1)$, once $S_t$ hits it, it is absorbed, completing the proof.
    \begin{algorithm}
    \caption{Norm-constrained Randomized Greedy Local Search on $H_{(r+1)/2, \gamma}$ with high overlap maintained}\label{alg:RandGreedyProofDevice}
    \begin{algorithmic}[1]
        \REQUIRE $Y \in \mathbb{R}^{n^{\otimes r}}$, $\sigma_0 \in \{\sigma\in\{0,1\}^n:\norm{\sigma}_0\leq \frac{3}{2}k\}$, $M\in\NN$.
        \STATE Let $t \gets 1$ and $S_1 \gets \sigma_0$
        \WHILE{$t<M$}
        \STATE Let $\mathcal{N}(\sigma) = \{\sigma' \mid d_H(S_t, \sigma') = 1\} \cap \{\sigma:\langle\sigma,\theta\rangle\geq \frac{k}{10},\norm{\sigma}_0\leq \frac{3}{2}k\}$
        \STATE Choose a random element $\sigma' \in \mathcal{N}(\sigma)$ uniformly
        \IF{$H_{(r+1)/2, \gamma}(\sigma') - H_{(r+1)/2, \gamma}(S_t) > 0$} 
            \STATE $S_{t+1} \gets \sigma'$
            \STATE $t \gets t + 1$
        \ENDIF
        \ENDWHILE
        \STATE \textbf{return} $S_M$
        \end{algorithmic}
    \end{algorithm}
    
    First, to show that $\Tilde S_M=\theta$, we will apply Theorem 2.3 of \cite{chen2024lowtemperaturemcmcthresholdcases}. This theorem says that for a randomized greedy algorithm on a graph $G$ with Hamiltonian $H$ defined on $G$, if the maximum degree of $G$ is at most $\Delta$, and if there exists a spanning tree of $G$ rooted at a vertex $v^*$ such that (i) $v^*$ is the global maximizer of $H$, and (ii) moving up one step in the spanning tree increases $H$ by at least $\delta$, then the randomized greedy algorithm from any initialization outputs $v^*$ after $O\left(\frac{\Delta R_H}{\delta}\right)$ iterations, where $R_H$ is the range of $H$.

    Consider the natural graph on $\{\sigma\in\{0,1\}^n:\langle\sigma,\theta\rangle\geq \frac{k}{10},\norm{\sigma}_0\leq \frac{3}{2}k\}$ where two elements $\sigma,\sigma'$ are connected if they have $d_H(\sigma,\sigma')=1$. We will construct a spanning tree on this graph, rooted at $\theta$. Let $C>0$ be a constant. We claim that whenever $\langle \sigma, \theta\rangle \geq \frac{1}{10}k$ and $\norm{\sigma}_0\leq \frac{3}{2}k$, it holds that $\sigma\in\Tilde\Omega(C)$, where $\Tilde\Omega(C)$ is defined in \eqref{eq:tildeOmegaC}. This holds because
    \begin{align}
        \frac{\langle \sigma, \theta\rangle}{\norm{\sigma}_2\norm{\theta}_2}\geq \frac{\frac{k}{10}}{\sqrt{\frac{3k}{2}}\sqrt{k}} = \frac{\sqrt{2}}{10\sqrt{3}}
    \end{align}
    while since $\lambda \geq C_\lambda  \frac{n^{\frac{r-1}{2}}}{k^{\frac{r}{2}-1}}\sqrt{\log n}$, we have
    \begin{align}
        C\left(\frac{\sqrt{k\log n}}{\lambda}\right)^{1/(r-1)}\leq C\left(\frac{\sqrt{k\log n}}{C_\lambda\frac{n^{(r-1)/2}}{k^{r/2-1}}\sqrt{\log n}}\right)^{1/(r-1)} &= C\left(\frac{k^{(r-1)/2}}{C_\lambda n^{(r-1)/2}}\right)^{1/(r-1)}\\
        &=\frac{C}{C_\lambda^{1/(r-1)}}\sqrt{\frac{k}{n}}\leq \frac{\sqrt{2}}{10\sqrt{3}}
    \end{align}
    for a large enough choice of $C_\lambda$, since $k\leq n$. Thus, $\sigma\in \Tilde\Omega(C)$. Then by Lemma \ref{lem:randRestrictWork01}(A), if $\sigma\neq \theta$, then $\sigma$ has a neighbor $\sigma'$ with  $H_{(r+1)/2, \gamma}(\sigma') - H_{(r+1)/2, \gamma}(\sigma)  \geq  \|\sigma\|^{(r-1)/2}_0 \sqrt{\log(n)}=\Theta\left(k^{(r-1)/2}\sqrt{\log n}\right)$. Choose $\sigma'$ to be the parent of $\sigma$ in the tree, and we obtain a spanning tree rooted at $\theta$ with $\delta=\Theta\left(k^{(r-1)/2}\sqrt{\log n}\right)$.

    Next, we note that the maximum degree of the graph $\Delta=n$, as each $\sigma$ has exactly $n$ neighbors.
    Therefore, the assumptions of Theorem 2.3 of \cite{chen2024lowtemperaturemcmcthresholdcases} are satisfied, with Hamiltonian $H_{(r+1)/2, \gamma}$, $\Delta = n$, $\delta =  k^{\frac{r-1}{2}}\sqrt{\log n}$ and $\beta = \infty$. By Lemma \ref{lem:GamilRange}, the range of $H_{(r+1)/2, \gamma}$, $R_{H_{(r+1)/2, \gamma}}$, satisfies $R_{H_{(r+1)/2, \gamma}} = O\left( k^{\frac{r}{2}}\left(\lambda+k^{\frac{r}{2}}\sqrt{\log n}\right)\right)$ with high probability. Recall that for this part of the proof we assume that
    \[\lambda = \Theta\left(\frac{\sqrt{\log n}n^{\frac{r-1}{2}}}{k^{\frac{r}{2}-1}}\right)\]
    Therefore, Algorithm \ref{alg:RandGreedyProofDevice} with
    \begin{align}
        M&=\Omega\left(\frac{n k^{\frac{r}{2}}\left(\lambda+k^{\frac{r}{2}}\sqrt{\log n}\right)}{k^{\frac{r-1}{2}}\sqrt{\log n}}\right)\\
        &= \Omega\left(\frac{nk^{1/2}\lambda}{\sqrt{\log n}}+nk^{\frac{r+1}{2}}\right)\\
        &= \Omega\left(\frac{n^{\frac{r+1}{2}}}{k^{\frac{r-3}{2}}}+nk^{\frac{r+1}{2}}\right)\\
        &=\Omega\left(nk^{\frac{r+1}{2}}\right)
    \end{align}
    outputs $\theta$ with probability at least $2/3-o(1)$.
    
    Next, we show that $S_t$ and $\Tilde S_t$ are coupled up until this time. To do this, it suffices to show that $S_t$ has $\langle S_t,\theta\rangle \geq k/10$ for all $t\leq M$ with probability $1-o(1)$. Suppose that $\langle S_t,\theta\rangle\in[\frac{k}{10},\frac{k}{8}]$. We first have the trivial upper bound
    \begin{align}
        \PP\left(\langle S_{t+1}, \theta \rangle = \langle S_t, \theta \rangle - 1| \langle S_t, \theta \rangle \in [k/10, k/8]\right) \leq \frac{\langle S_t, \theta \rangle}{n}\leq \frac{k}{8n}
    \end{align}
    because each of the $\langle S_t, \theta \rangle$ correct coordinates has probability $\frac{1}{n}$ of being proposed to drop at each time $t$. Meanwhile, by Lemma \ref{lem:randRestrictWork01} (\ref{lem:randRestrictWork01:itemB}), we have the lower bound 
    \begin{align}
        \PP\left(\langle S_{t+1}, \theta \rangle = \langle S_t, \theta \rangle + 1| \langle S_t, \theta \rangle \in [k/10, k/8]\right) \geq \frac{2}{3}\frac{k-\langle S_t, \theta \rangle}{n}\geq \frac{2}{3}\frac{7k}{8n}=\frac{7k}{12n}
    \end{align}
    Define a lazy random walk on the integers as follows: Let $a$ be a positive integer, and let $X_0 = \lfloor \frac{k}{8}\rfloor$. Then for all $t\geq 0$, let
    \begin{align}
        \PP(X_{t+1} = X_t-1) = \frac{1}{8}\\
        \PP(X_{t+1} = X_t+1) = \begin{cases}
            \frac{14}{24} \text{ if $X_t<a$}\\
            0 \text{ if $X_t = a$}
        \end{cases}\\
        \PP(X_{t+1} = X_t) = \begin{cases}
            \frac{7}{24} \text{ if $X_t<a$}\\
            \frac{7}{8} \text{ if $X_t = a$}
        \end{cases}
    \end{align}
    i.e. a simple asymmetric random walk, but reflected downwards at $a$. Then with the choice of $a=\lfloor \frac{k}{8}\rfloor$, the random walk $\{X_t\}_{t\geq 0}$ is stochastically dominated by the process $\langle S_t,\theta\rangle$. We will show that with probability at least $1-e^{-0.001k}$, the first time that $X_t=\lceil\frac{k}{10}\rceil$ is at least $\Omega(e^k)$, and thus with the same probability $\langle S_t,\theta\rangle\geq \frac{k}{10}$ for at least the first $\Omega(e^k)$ steps. Define $f:(-\infty,a]\rightarrow \RR$ as
    \begin{align}
        f(x)=\begin{cases}
            \left(\frac{3}{14}\right)^{x}\text{ if $x<a$}\\
            \left(\frac{3}{14}\right)^{a-1}\text{ if $x=a$}
        \end{cases}
    \end{align}
    Then we claim that $M_t:=f(X_t)$ is a martingale. To see this, if $X_t<a$, we have 
    \begin{align}
        \EE\left(M_{t+1}|X_t\right) &= \frac{1}{8}\left(\frac{3}{14}\right)^{X_t-1}+\frac{14}{24}\left(\frac{3}{14}\right)^{X_t+1}+\frac{7}{24}\left(\frac{3}{14}\right)^{X_t}\\
        &=\left(\frac{3}{14}\right)^{X_t} = M_t
    \end{align}
    While if $X_t = a$:
    \begin{align}
        \EE\left(M_{t+1}|X_t\right) &= \frac{1}{8}\left(\frac{3}{14}\right)^{a-1}+\frac{7}{8}\left(\frac{3}{14}\right)^{a-1}\\
        &=\left(\frac{3}{14}\right)^{a-1} = M_t
    \end{align}
    so $M_t$ is a martingale. Let $T$ be any positive integer, let $\tau_1 = \inf \left\{t>0: X_t = \lceil\frac{k}{10}\rceil\right\}$, and let $\tau =\min\left\{\tau_1,T\right\}$. Then $\tau$ is a bounded stopping time, and $M_t$ is a non-negative martingale, so by the Optional Stopping Theorem, it holds that 
    \[\EE(M_\tau) = \EE(M_0) = \left(\frac{3}{14}\right)^{\lfloor \frac{k}{8}\rfloor-1}\]
    Decompose $\EE(M_\tau)$ as 
    \begin{align}
        \EE(M_\tau) = \EE(M_\tau|\tau_1\leq T)\PP(\tau_1\leq T)+\EE(M_\tau|\tau_1>T)\PP(\tau_1>T)=\left(\frac{3}{14}\right)^{\lfloor \frac{k}{8}\rfloor-1}
    \end{align}
    $\EE(M_\tau|\tau_1>T)\PP(\tau_1>T)\geq 0$ because $M_t$ is non-negative, which gives us
    \begin{align}
        \EE(M_\tau|\tau_1\leq T)\PP(\tau_1\leq T)\leq \left(\frac{3}{14}\right)^{\lfloor \frac{k}{8}\rfloor-1}
    \end{align}
    Note that on the event that $\{\tau_1\leq T\}$, it holds that $M_\tau = M_{\tau_1} = \left(\frac{3}{14}\right)^{\lceil \frac{k}{10}\rceil}$, and therefore
    \begin{align}
        \PP(\tau_1\leq T)\leq \left(\frac{3}{14}\right)^{\lfloor \frac{k}{8}\rfloor-1-\lceil \frac{k}{10}\rceil}\leq e^{-0.001k}
    \end{align}
    for all $T>0$, where the last inequality holds for large enough $k$. In particular, we choose $T=\Omega(e^k)$ to get that with probability at least $1-e^{-0.001k}$, $X_t$ does not reach $\lceil\frac{k}{10}\rceil$ within the first $\Omega(e^k)$ steps, and therefore since $\langle S_t,\theta\rangle$ has an initial value of at least $\frac{k}{8}$, with the same probability it does not decrease to $\frac{k}{10}$ within the first $\Omega(e^k)$ iterations.

    It remains to consider the case where $\lambda\geq C_\lambda  \frac{n^{\frac{r-1}{2}}}{k^{\frac{r}{2}-1}}\sqrt{\log n}$, rather than $\lambda=C_\lambda  \frac{n^{\frac{r-1}{2}}}{k^{\frac{r}{2}-1}}\sqrt{\log n}$. Consider two instances of the problem with the following coupling: 
    \begin{align}
        Y_1 = \frac{\lambda_1}{k^{r/2}}\theta^{\otimes r}+W_1
    \end{align}
    and 
    \begin{align}
        Y_2 = \frac{\lambda_2}{k^{r/2}}\theta^{\otimes r}+W_2
    \end{align}
    where we couple $W_1=W_2$ with probability 1, so that the only difference is in the signal-to-noise parameters $\lambda_1,\lambda_2$. Assume that $\lambda_2\geq \lambda_1 = C_\lambda  \frac{n^{\frac{r-1}{2}}}{k^{\frac{r}{2}-1}}\sqrt{\log n}$, where $C_\lambda$ is the sufficiently large constant chosen above. Suppose we run Algorithm \ref{alg:RandGreedyProofDevice} twice, once with input $(Y_1,\sigma_0,M)$ and once with $(Y_2,\sigma_0,M)$, but condition that the same bit flips are proposed at each iteration. As proven above, the instance with $Y_1$ outputs $\theta$ with probability at least $2/3$. We claim that the same is true for the instance with $Y_2$.  To see this, first note that by Lemma \ref{lem:GamilDiffIdents}, if the proposed bit flip at a particular iteration is for a coordinate outside the support of $\theta$, then the change in the value of $H_{(r+1)/2, \gamma}$ is independent of $\lambda$, so the instances couple perfectly (due to the noise coupling). Secondly, when adding a coordinate in the support of $\theta$, if $H_{(r+1)/2, \gamma}(\sigma+e_i)-H_{(r+1)/2, \gamma}(\sigma)>0$ with $\lambda=\lambda_1$, then it trivially remains positive when $\lambda=\lambda_2\geq \lambda_1$ by \eqref{eq:GamilDiffAddIden}, so the instance with $\lambda_2$ adds every coordinate in $\supp(\theta)$ that the instance with $\lambda_1$ adds. Finally, when the proposed bit flip is to remove a coordinate in $\supp(\theta)$, even if the instance with $Y_1$ removes it, the instance with $Y_2$ may not if $\lambda_2$ is large enough. It therefore holds that the iterates of the instance with $Y_2$, $ S_t^2$, maintain a higher overlap with $\theta$ at every step than the iterates of the instance with $Y_1$, $ S_t^1$ (i.e. $\langle S_t^2,\theta\rangle \geq \langle  S_t^1,\theta\rangle$ for each $t$). Therefore, $S_t^2=\theta$ after at most $M$ iterations with probability at least $2/3-o(1)$. This completes the proof of Theorem \ref{thm:DenseBinaryMainTheorem1}.
\end{proof}

\subsection{Proofs of Key Lemmas}\label{sec:randrestrictkeyproofs}
Here we prove Lemmas \ref{lem:randRestrictWork01}, \ref{lem:GamilRange}, and \ref{lem:GamilDiffIdents}.

In order to prove Lemma \ref{lem:randRestrictWork01}, we will need the following auxiliary technical lemma, whose proof is given in Section \ref{sec:randrestrictauxproofs}.
\begin{lem}\label{lem:HanWrit1} Let $m,T\leq n$ be integers with $m\geq 2$. Let $\cS_m \subseteq \{\sigma \in \Omega: \|\sigma\|_0 = m\}$. Let $\CHW$ be the universal constant in Lemma \ref{lem:RefHanWrit} and let $1\leq j\leq r$. Then, for every $\sigma \in \cS_m$ and every $i \in [T]$, there exists a constant $C=C(r)>0$ such that for all $a>0$ satisfying
    \begin{align}\label{eq:HWMainStatement1}\min\left\{\frac{a^2}{T} , a\right\} \geq \CHW \left(1+\frac{T}{m}\right) \log(n^d|\cS_m|),\end{align}
    and any $d \in \NN$ the following statement holds with
    probability at least $1-n^{-d}$:
    
    For every $\sigma \in \cS_m$ and $q \in [T]$, there are at least $T-q$ indices $i \in [T]$
    satisfying
    \[\left|\langle e_i^{\otimes j} \otimes \sigma^{\otimes (r-j)}, W\rangle\right| \leq C\sqrt{\frac{m^{(r-j)}(a + T)}{q}}\]
\end{lem}
\begin{proof}[Proof of Lemma \ref{lem:randRestrictWork01}]
    Throughout this proof, we will use the notation that for all $\sigma \in \{0,1\}^n,$ $\ell_\sigma = \langle \sigma,\theta\rangle$ and $m_\sigma=\norm{\sigma}_0$. First we prove that (\ref{lem:randRestrictWork01:itemA}) holds for all $\sigma\in\Tilde\Omega(C)$ satisfying $k-\ell_{\sigma} \leq m_{\sigma}/2$, and then we show that (\ref{lem:randRestrictWork01:itemA}) holds for all $\sigma\in\Tilde\Omega(C)$ with $k-\ell_{\sigma} > m_{\sigma}/2$. 

    Finally, we show that (\ref{lem:randRestrictWork01:itemB}) is satisfied.

    We start with proving the first claim. 
    
    \textbf{Case 1: proving that (\ref{lem:randRestrictWork01:itemA}) holds for $\sigma\in \Tilde \Omega_1(C):=\left\{\sigma\in\Tilde\Omega(C): k- \ell_\sigma \leq m_{\sigma}/2 \right\}$.}

  For any $\sigma\in\Tilde\Omega_1(C)$, let
  \[\cS_1(\sigma)=\left\{\sigma'\in\Tilde\Omega_1(C):\ell_{\sigma'}=\ell_\sigma,m_{\sigma'}=m_\sigma\right\}\]
  Then $|\cS_1(\sigma)|\leq \binom{n}{m_\sigma-\ell_\sigma}\binom{k}{k-\ell_\sigma}\leq n^{m_\sigma+k-2\ell_\sigma}$. We consider the set of ``neighbors'' of $\sigma$, i.e., $\sigma' \in \{0,1\}^n$ with $d_H(\sigma,\sigma')=1$, that either remove a coordinate outside the support of $\theta$ (i.e. transitions from $\sigma$ to $\sigma'=\sigma - e_i$, where $i$ satisfies $\sigma_i=1$ and $\theta_i=0$) or which add a coordinate in the support of $\theta$ (i.e. $\sigma$ to $\sigma'=\sigma+e_i$, where $\sigma_i = 0$ and $\theta_i = 1$).  The number of possible such transitions is given by $T_\sigma=m_\sigma+k-2\ell_\sigma$. Let $a=16C_{\mathrm{HW}}(m_\sigma+k-2\ell_\sigma)\log n$. Observe that $T_\sigma \leq m_\sigma$ and $m_{\sigma}+k-2\ell_{\sigma}\geq 1$, and thus
  \[\left(1+\frac{T_\sigma}{m_\sigma}\right)\log(n^3|\cS_1(\sigma)|) \leq 4(3+m_{\sigma}+k-2\ell_{\sigma})\log n \leq 16(m_{\sigma}+k-2\ell_{\sigma})\log n\]
  It is clear that $a^2/T_\sigma \geq a$. Then with $q=T_\sigma/3$, $d=3$, $T=T_\sigma$, and $m=m_\sigma$, the assumptions of Lemma \ref{lem:HanWrit1} are therefore satisfied. Hence, with probability at least $1-n^{-3}$, it holds that for every $\sigma'\in\cS_1(\sigma)$ and for every $1\leq j\leq r$,, at least two thirds of the $(m_\sigma+k-2\ell_\sigma)$ transitions that either remove a coordinate $i\not\in\supp(\theta)$ or add a coordinate $i\in\supp(\theta)$ satisfy
    \begin{align}
        &\left|\langle e_i^{\otimes j}\otimes (\sigma')^{\otimes r-j},W\rangle\right|\\
        \leq &C_1 \left(\frac{m_\sigma^{r-j}(a+T_\sigma)}{q_\sigma}\right)^{1/2}\\
        =&C_1\left(\frac{m_\sigma^{r-j}(16C_{\mathrm{HW}}(m_\sigma+k-2\ell_\sigma)\log n+m_\sigma+k-2l_\sigma)}{\frac{1}{3}(m_\sigma+k-2l_\sigma)}\right)^{1/2}\\
        =&C_1\left(48\CHW m_\sigma^{r-j}\log n+3m_\sigma^{r-j}\right)^{1/2}\\
        \leq &C_2 m_\sigma^{(r-j)/2} \sqrt{\log n}
    \end{align}
    where $C_1=C_1(r)>0$, $C_2=C_2(r)>0$ are sufficiently large constants depending only on $r$. Thus for a large enough choice of constant $C_3(r)>0$, at least two thirds of the transitions satisfy
    \begin{align}
        \left|\sum_{j=1}^r\binom{r}{j}(-1)^j\langle e_i^{\otimes j}\otimes (\sigma')^{\otimes r-j},W\rangle\right|\leq C_3 m_\sigma ^{(r-1)/2}\sqrt{\log n}\label{eq:Omega1Bound1}
    \end{align}
    In particular, by a union bound over the $n$ possible values of $m_\sigma$ and the $k$ possible values of $l_\sigma$, it holds with probability at least $1-n^{-1}$ that for every $\sigma\in\Tilde\Omega_1(C)$ that at least two thirds of the transitions that add a coordinate in $\supp(\theta)$ or remove a coordinate outside $\supp(\theta)$ satisfy \eqref{eq:Omega1Bound1}.

    Now consider any such transition from $\sigma$ to a $\sigma'$ satisfying \eqref{eq:Omega1Bound1}. In the case where $\sigma'=\sigma-e_i$ for a coordinate $i\not\in\supp(\theta)$ of $\sigma$ where $\sigma_i=1$, we therefore see by Lemma \ref{lem:GamilDiffIdents} that
    \begin{align}
        &H_{(r+1)/2, \gamma}(\sigma-e_i)-H_{(r+1)/2, \gamma}(\sigma)\\
        &=\frac{\lambda \theta_i}{k^{r/2}}\sum_{j=1}^r\binom{r}{j}(-1)^j \langle \sigma,\theta\rangle^{r-j} +\sum_{j=1}^r\binom{r}{j}(-1)^j \langle e_i^{\otimes j} \otimes \sigma^{\otimes r-j}, W\rangle\\
        &-\gamma \left(\norm{\sigma}_0-1\right)^{\frac{r+1}{2}} + \gamma \norm{\sigma}_0^{\frac{r+1}{2}}\\
        &=\sum_{j=1}^r\binom{r}{j}(-1)^j \langle e_i^{\otimes j} \otimes \sigma^{\otimes r-j}, W\rangle-\gamma \left(\norm{\sigma}_0-1\right)^{\frac{r+1}{2}} + \gamma \norm{\sigma}_0^{\frac{r+1}{2}}\\
        &\geq -C_3 m_\sigma ^{(r-1)/2}\sqrt{\log n}-C_\gamma \sqrt{\log n}\left(m_\sigma-1\right)^{\frac{r+1}{2}} + C_\gamma \sqrt{\log n}m_\sigma^{\frac{r+1}{2}}\\
        &\geq -C_3 m_\sigma ^{(r-1)/2}\sqrt{\log n}+C_\gamma \sqrt{\log n}\left(\frac{r+1}{2}\right)(m_\sigma-1)^{(r-1)/2}\\
        &\geq m_\sigma^{(r-1)/2}\sqrt{\log n}
    \end{align}
    for a large enough choice of $C_\gamma$, where the next to last inequality holds by the Mean Value Theorem.

    Next, consider the case of a transition from $\sigma$ to $\sigma'$ where $\sigma'=\sigma+e_i$ for some coordinate $i\in\supp\theta)$, such that \eqref{eq:Omega1Bound1} is satisfied. Then, by Lemma \ref{lem:GamilDiffIdents},
    \begin{align}
            &H_{(r+1)/2, \gamma}(\sigma+e_i)-H_{(r+1)/2, \gamma}(\sigma)\\
            &= \frac{\lambda\theta_i }{k^{r/2}}\sum_{j=1}^r \binom{r}{j}\langle \sigma,\theta\rangle^{r-j}+\sum_{j=1}^r\binom{r}{j}\langle e_i^{\otimes j} \otimes \sigma^{\otimes r-j}, W\rangle -\gamma \left(\norm{\sigma}_0+1\right)^{\frac{r+1}{2}}+\gamma \norm{\sigma}_0^{\frac{r+1}{2}}\\
            &\geq \frac{\lambda}{k^{r/2}}\sum_{j=1}^r \binom{r}{j}\langle \sigma,\theta\rangle^{r-j}-C_3m_\sigma^{(r-1)/2}\sqrt{\log n} -C_\gamma\sqrt{\log n} \left(m_\sigma+1\right)^{\frac{r+1}{2}}+C_\gamma\sqrt{\log n} m_\sigma^{\frac{r+1}{2}}\\
            &\geq \frac{\lambda}{k^{r/2}}rl_\sigma^{r-1}-C_3m_\sigma^{(r-1)/2}\sqrt{\log n} -C_\gamma\sqrt{\log n} \left(m_\sigma+1\right)^{\frac{r+1}{2}}+C_\gamma\sqrt{\log n} m_\sigma^{\frac{r+1}{2}}\\
            \intertext{Since $\sigma\in\Tilde\Omega(C)$, we have $l_\sigma^{r-1}\geq C^{r-1}\left(\frac{\sqrt{k\log n}}{\lambda}\right)k^{(r-1)/2}m_\sigma^{(r-1)/2}$, and therefore this is at least}\\
            &\geq C^{r-1}r\sqrt{\log n}m_\sigma^{(r-1)/2}-C_3m_\sigma^{(r-1)/2}\sqrt{\log n} -C_\gamma\sqrt{\log n} \left(m_\sigma+1\right)^{\frac{r+1}{2}}+C_\gamma\sqrt{\log n} m_\sigma^{\frac{r+1}{2}}\\
            &\geq m_\sigma^{(r-1)/2}\sqrt{\log n}
        \end{align}
    for a large enough choice of $C$. 
    Thus, \eqref{lem:randRestrictWork01:itemA} is satisfied for every $\sigma\in\Tilde\Omega_1(C)$ with probability at least $1-n^{-1}$.
    
    \textbf{Case 2: proving that (\ref{lem:randRestrictWork01:itemA}) holds for $\sigma\in \Tilde \Omega_2(C):=\left\{\sigma\in\Tilde\Omega(C): k- \ell_\sigma > m_{\sigma}/2 \right\}$.}
    Again, for any $\sigma\in\Tilde\Omega_2(C)$, and let
    \[\cS_2(\sigma)=\left\{\sigma'\in\Tilde\Omega_2(C):l_{\sigma'}=l_\sigma,m_{\sigma'}=m_\sigma\right\}\]
    Then $|\cS_2(\sigma)|\leq \binom{n}{m_\sigma}\leq n^{m_\sigma}$. Consider the set of ``neighbors'' of $\sigma$, i.e., $\sigma' \in \{0,1\}^n$ with $d_H(\sigma,\sigma')=1$, that add a  coordinate in $\supp(\theta)$ (i.e. $\sigma$ to $\sigma'=\sigma+e_i$, where $\sigma_i = 0$ and $\theta_i = 1$). The number of possible such transitions is given by $T_\sigma=k-l_\sigma$. Let $a=5\CHW(k-l_\sigma)\log n$. Observe that 
    \[\left(1+\frac{T_\sigma}{m_\sigma}\right)\log\left(n^3|\cS_2(\sigma)|\right)\leq\left(1+\frac{k-l_\sigma}{m_\sigma}\right)(3+m_\sigma)\log n \leq 5(k-l_\sigma)\log n\]
    Note also that $a^2/T_\sigma\geq a$. Then with $q=T_\sigma/3$, $d=3$, $T=T_\sigma$, and $m=m_\sigma$, the assumptions of Lemma \ref{lem:HanWrit1} are therefore satisfied. It therefore holds with probability at least $1-n^{-3}$ that for every $\sigma'\in\cS_2(\sigma)$, and for every $1\leq j\leq r$, at least two thirds of the $k-l_\sigma$ transitions that add a  coordinate $i$ in $\supp(\theta)$ satisfy
    \begin{align}
        &\left|\langle e_i^{\otimes j}\otimes (\sigma')^{\otimes r-j},W\rangle\right|\\
        \leq &C_1 \left(\frac{m_\sigma^{r-j}(a+T_\sigma)}{q_\sigma}\right)^{1/2}\\
        =&C_1\left(\frac{m_\sigma^{r-j}(5C_{\mathrm{HW}}(k-l_\sigma)\log n+k-l_\sigma)}{\frac{1}{3}(k-l_\sigma)}\right)^{1/2}\\
        =&C_1\left(15\CHW m_\sigma^{r-j}\log n+3m_\sigma^{r-j}\right)^{1/2}\\
        \leq &C_2 m_\sigma^{(r-j)/2} \sqrt{\log n}
    \end{align}
    where again $C_1=C_1(r)>0$, $C_2=C_2(r)>0$ are sufficiently large constants depending only on $r$. Thus for a large enough choice of constant $C_3>0$, at least two thirds of the transitions satisfy
    \begin{align}
        \left|\sum_{j=1}^r\binom{r}{j}(-1)^j\langle e_i^{\otimes j}\otimes (\sigma')^{\otimes r-j},W\rangle\right|\leq C_3 m_\sigma ^{(r-1)/2}\sqrt{\log n}\label{eq:Omega1Bound2}
    \end{align}
    By a union bound again, it holds with probability at least $1-n^{-1}$ that for every $\sigma\in\Tilde\Omega_2(C)$ that at least two thirds of the transitions that add a coordinate in $\supp(\theta)$ satisfy \eqref{eq:Omega1Bound2}. For such a $\sigma$, and such a transition $\sigma'=\sigma+e_i$, we have by Lemma \ref{lem:GamilDiffIdents} that
    \begin{align}
            &H_{(r+1)/2, \gamma}(\sigma+e_i)-H_{(r+1)/2, \gamma}(\sigma)\\
            &= \frac{\lambda\theta_i }{k^{r/2}}\sum_{j=1}^r \binom{r}{j}\langle \sigma,\theta\rangle^{r-j}+\sum_{j=1}^r\binom{r}{j}\langle e_i^{\otimes j} \otimes \sigma^{\otimes r-j}, W\rangle -\gamma \left(\norm{\sigma}_0+1\right)^{\frac{r+1}{2}}+\gamma \norm{\sigma}_0^{\frac{r+1}{2}}\\
            &\geq \frac{\lambda}{k^{r/2}}\sum_{j=1}^r \binom{r}{j}\langle \sigma,\theta\rangle^{r-j}-C_3m_\sigma^{(r-1)/2}\sqrt{\log n} -C_\gamma\sqrt{\log n} \left(m_\sigma+1\right)^{\frac{r+1}{2}}+C_\gamma\sqrt{\log n} m_\sigma^{\frac{r+1}{2}}\\
            &\geq \frac{\lambda}{k^{r/2}}rl_\sigma^{r-1}-C_3m_\sigma^{(r-1)/2}\sqrt{\log n} -C_\gamma\sqrt{\log n} \left(m_\sigma+1\right)^{\frac{r+1}{2}}+C_\gamma\sqrt{\log n} m_\sigma^{\frac{r+1}{2}}\label{eq:thm:randRestrictWorkBranch}\\
            \intertext{Since $\sigma\in\Tilde\Omega(C)$, we have $l_\sigma^{r-1}\geq C^{r-1}\left(\frac{\sqrt{k\log n}}{\lambda}\right)k^{(r-1)/2}m_\sigma^{(r-1)/2}$, and therefore this is at least}\\
            &\geq C^{r-1}r\sqrt{\log n}m_\sigma^{(r-1)/2}-C_3m_\sigma^{(r-1)/2}\sqrt{\log n} -C_\gamma\sqrt{\log n} \left(m_\sigma+1\right)^{\frac{r+1}{2}}+C_\gamma\sqrt{\log n} m_\sigma^{\frac{r+1}{2}}\\
            &\geq m_\sigma^{(r-1)/2}\sqrt{\log n}
        \end{align}
        for a large enough choice of $C$, and therefore (\ref{lem:randRestrictWork01:itemA}) is satisfied with probability at least $1-n^{-1}$.

        \textbf{Case 3: proving (\ref{lem:randRestrictWork01:itemB})}
        Note that by an identical argument as in Case 2 as far as \eqref{eq:thm:randRestrictWorkBranch}, we have with probability at least $1-n^{-1}$ that
        for every $\sigma\in\hat\Omega(C)$, there exist at least $\lfloor \frac{2}{3}(k-l_\sigma)\rfloor$ transitions that add a coordinate $i$ in $\supp(\theta)$ and that satisfy
        \begin{align}
            &H_{(r+1)/2, \gamma}(\sigma+e_i)-H_{(r+1)/2, \gamma}(\sigma)\\
            &\geq \frac{\lambda}{k^{r/2}}rl_\sigma^{r-1}-C_3m_\sigma^{(r-1)/2}\sqrt{\log n} -C_\gamma\sqrt{\log n} \left(m_\sigma+1\right)^{\frac{r+1}{2}}+C_\gamma\sqrt{\log n} m_\sigma^{\frac{r+1}{2}}
            \intertext{By assumption that $\sigma\in\hat\Omega(C)$, we have $l_\sigma\geq \frac{k}{10}$, so this is at least}
            &\geq \frac{\lambda}{k^{r/2}}r\left(\frac{k}{10}\right)^{r-1}-C_3m_\sigma^{(r-1)/2}\sqrt{\log n} -C_\gamma\sqrt{\log n} \left(m_\sigma+1\right)^{\frac{r+1}{2}}+C_\gamma\sqrt{\log n} m_\sigma^{\frac{r+1}{2}}\\
            \intertext{and for large enough $C_\gamma$, this is at least}
            &\geq \frac{\lambda}{k^{r/2}}r\left(\frac{k}{10}\right)^{r-1}-2rC_\gamma\sqrt{\log n} m_\sigma^{\frac{r-1}{2}}\\
            \intertext{by assumption, $m_\sigma\leq \frac{3k}{2}$, so we lower bound this by}
            &\geq \frac{\lambda}{k^{r/2}}r\left(\frac{k}{10}\right)^{r-1}-2rC_\gamma\sqrt{\log n} \left(\frac{3k}{2}\right)^{\frac{r-1}{2}}\label{eq:lem:randRestrictWork01LambdaLB}\\
            &\geq \frac{r}{2\cdot 10^{r-1}}\lambda k^{\frac{r}{2}-1}
        \end{align}
        where the last inequality holds for a large enough choice of $C_\lambda$ because the first term in \ref{eq:lem:randRestrictWork01LambdaLB} is dominating. To see this, simply apply the assumption $\lambda \geq C_\lambda  \frac{n^{\frac{r-1}{2}}}{k^{\frac{r}{2}-1}}\sqrt{\log n}$. Combining this with the result for part (\ref{lem:randRestrictWork01:itemA}), we see that (\ref{lem:randRestrictWork01:itemA}) and (\ref{lem:randRestrictWork01:itemB}) hold simultaneously with probability at least $1-3n^{-1}$, as desired.
\end{proof}
We now prove Lemma \ref{lem:GamilRange}.

\begin{proof}[Proof of Lemma \ref{lem:GamilRange}]
    Consider
    \[H_{(r+1)/2, \gamma}(\sigma) = \langle \sigma^{\otimes r},Y\rangle -\gamma \norm{\sigma}_0^{(r+1)/2} = \frac{\lambda}{k^{r/2}}\langle \sigma,\theta\rangle^{r}+\langle \sigma^{\otimes r},W\rangle-C_\gamma \sqrt{\log n}\norm{\sigma}_0^{(r+1)/2}.\]
    It is easy to see that since $0\leq \langle \sigma,\theta\rangle\leq k$, it holds that
    \[0\leq \frac{\lambda}{k^{r/2}}\langle \sigma,\theta\rangle^{r}\leq \lambda k^{r/2}\]
    Next, $W$ has $n^r$ entries each distributed as $\mathcal{N}(0,1)$. It therefore holds by Lemmas \ref{lem:MaxGaussianMean} and \ref{lem:MaxGaussianDeviation} (with $u = \sqrt{\log n}$) that with probability at least $1-n^{-1/2}$, 
    \[\max_{\underline{h}\in[n]^r}\left|W_{\underline{h}}\right|\leq (\sqrt{2r}+1)\sqrt{\log n}\]
    The sum $\langle \sigma^{\otimes r},W\rangle$ has $\norm{\sigma}_0^r\leq \left(\frac{3}{2}k\right)^r$ nonzero terms, and so we have
    \[-(\sqrt{2r}+1)\sqrt{\log n}\left(\frac{3}{2}k\right)^r\leq \langle \sigma^{\otimes r},W\rangle\leq (\sqrt{2r}+1)\sqrt{\log n} \left(\frac{3}{2}k\right)^r\]
    Combining the two bounds, we obtain the desired result.
\end{proof}
\begin{proof}[Proof of Lemma \ref{lem:GamilDiffIdents}]
    The proof is identical to the proof of Lemma \ref{lem:HamilDiffIdentsPlusMinus}, but applied only to the case of $\sigma\in\{0,1\}^n$, and with a different exponent on $\norm{\sigma}_0$.
\end{proof}

\subsection{Proofs of Auxiliary Lemmas}\label{sec:randrestrictauxproofs}

In order to prove Lemma \ref{lem:HanWrit1}, we will need the following two technical results.
\begin{lem}\label{lem:EigenvaluesStochDom}
    Suppose $A\sim\cN(0,\Sigma_A)$ and $B\sim\cN(0,\Sigma_B)$, where $A,B\in\RR^d$. If $\Sigma_B-\Sigma_A$ is positive semidefinite, then 
    \[\norm{A}_2^2 \stleq \norm{B}_2^2\]
\end{lem}
\begin{proof}[Proof of Lemma \ref{lem:EigenvaluesStochDom}]
    Since $A\sim\cN(0,\Sigma_A)$, write $A=\Sigma_A^{1/2}Z_A$, where $Z_A\sim\cN(0,I)$. We therefore see that
    \[\norm{A}_2^2 = (\Sigma_A^{1/2} Z_A)^\top (\Sigma_A^{1/2}Z_A)=Z_A^\top \Sigma_A Z_A\]
    Next, since $\Sigma_A$ is symmetric and positive semidefinite, there exists an orthonormal matrix $U_A$ and a diagonal matrix $\Lambda_A=\mathrm{diag}(\lambda_1(A),\dots,\lambda_d(A))$ such that
    \[\Sigma_A=U_A \Lambda_A U_A^\top\]
    where $(\lambda_1(A),\dots,\lambda_d(A))$ are the (ordered) eigenvalues of $\Sigma_A$. Define $\Tilde Z_A=U_A^\top Z$. Note that since $Z_A\sim\cN(0,I)$ and $U_A$ is orthonormal, $\Tilde Z_A\sim\cN(0,I)$ as well. Thus:
    \begin{align}
        \norm{A}_2^2 =Z_A^\top \Sigma_A Z_A &= Z_A^\top U_A \Lambda_A U_A^\top Z_A\\
        &=\Tilde Z_A^\top \Lambda_A \Tilde Z_A\\
        &= \sum_{i=1}^d \lambda_i(A)(\Tilde Z_A)_i^2
    \end{align}
    Similarly for $B\sim\cN(0,\Sigma_B)$, we get
    \begin{align}
        \norm{B}_2^2 = \sum_{i=1}^d \lambda_i(B)(\Tilde Z_B)_i^2
    \end{align}
    where $(\lambda_1(B),\dots,\lambda_d(B))$ are the eigenvalues of $\Sigma_B$ and $\Tilde Z_B\sim\cN(0,I)$. Since $\Sigma_B-\Sigma_A$ is assumed positive semidefinite, $\lambda_i(B)\geq \lambda_i(A)$ for each $i$. Note also that for each $i$, $(\Tilde Z_A)_i^2$ and $(\Tilde Z_A)_i^2$ both follow a $\chi^2_1$ distribution. It therefore holds that for any $t>0$, 
    \[\PP(\norm{A}_2^2\geq t)\leq \PP(\norm{B}_2^2\geq t)\]
    as desired. 
\end{proof}

\begin{lem}[\cite{RudelsonVershyninHanWrit}, Theorem 1.1]\label{lem:RefHanWrit} Let $Z \sim
\cN(0,I_n)$ be a standard Gaussian random vector. Let $A \in \RR^{m \times n}$
and $X = AZ \in \RR^m$. Let $K = AA^\top$ be the covariance matrix of $X$. There
exists a universal finite constant $\CHW > 1$ such that, for any $a > 0$, it holds, 
\[\PP(\|X\|_2^2 \geq \EE[\|X\|_2^2] + a) \leq
\exp\left(-\frac{1}{\CHW}\min\left(\frac{a^2}{\|K\|_{F}^2},\frac{a}{\|K\|_2^2} \right)\right)\]
\end{lem}

We are now ready to prove Lemma \ref{lem:HanWrit1}.
\begin{proof}[Proof of Lemma \ref{lem:HanWrit1}]
    Throughout this proof, we will use the shorthand 
    \[Z^\sigma_i(j) = \langle e_i^{\otimes j} \otimes \sigma^{\otimes (r-j)}, W\rangle\]
    for any $1\leq j\leq r$, any $\sigma\in\cS_m$ and any $i\in[T]$. We will further define
    \[\Tilde Z^\sigma_i(j) = \frac{Z^\sigma_i(j)}{\sqrt{r!r^2m^{r-j}}}\]
    and the vector $\Tilde Z^\sigma(j) := (\Tilde Z_i^\sigma(j))_{i \in [T]}$. We first bound the second moments of each element of $\Tilde Z^\sigma(j)$:
    \begin{align}
        \EE[(\Tilde Z^\sigma_i(j))^2] &= \frac{1}{r!r^2m^{r-j}}\EE[ (Z^\sigma_i(j))^2] \\
        &= \frac{1}{r!r^2m^{r-j}}\sum_{\substack{h_{j+1}, \dots, h_r \in [m]\\ h'_{j+1}, \dots, h'_r \in [m]}} \sigma_{h_{j+1}}\cdots\sigma_{h_r}\sigma_{h'_{j+1}}\cdots\sigma_{h'_r}\EE[W_{i,h_{j+1},\dots,h_r} W_{i,h'_{j+1},\dots,h'_r}] \\
        &= \frac{1}{r!r^2m^{r-j}}\sum_{\substack{h_{j+1}, \dots, h_r \in [m]\\ h'_{j+1}, \dots, h'_r \in[m]}}\sigma_{h_{j+1}}\cdots\sigma_{h_r}\sigma_{h'_{j+1}}\cdots\sigma_{h'_r}\mathbf{1}\{(i, h_{j+1}, \dots, h_r) = (i, h'_{j+1}, \dots, h'_r)\}\\
        &\leq \frac{r!m^{r-j}}{r!r^2m^{r-j}} \leq \frac{1}{4} < 1
    \end{align}
    where the next-to-last inequality holds because $r\geq 2$. Similarly, for each $i,l \in [T]$ with $i \not = l$, we have
    \begin{align}
        \EE[\Tilde Z^\sigma_i(j) \Tilde Z^\sigma_l(j)] &= \frac{1}{r!r^2m^{r-j}}\EE[ (Z^\sigma_i(j))(Z^\sigma_l(j)) ]\\
        &= \frac{1}{r!r^2m^{r-j}}\EE[\langle e_i^{\otimes j} \otimes \sigma^{\otimes r-j}, W\rangle \langle e_l^{\otimes j} \otimes \sigma^{\otimes r-j}, W\rangle]\\
        &= \frac{1}{r!r^2m^{r-j}}\sum_{\substack{h_{j+1}, \dots, h_r \in [m]\\ h'_{j+1}, \dots, h'_r \in [m]}} \sigma_{h_{j+1}}\cdots\sigma_{h_r}\sigma_{h'_{j+1}}\cdots\sigma_{h'_r}\EE[W_{i,h_{j+1},\dots,h_r} W_{l,h'_{j+1},\dots,h'_r}]\\
        &=  \frac{1}{r!r^2m^{r-j}}\sum_{\substack{h_{j+1}, \dots, h_r \in [m]\\ h'_{j+1}, \dots, h'_r \in [m]}} \sigma_{h_{j+1}}\cdots\sigma_{h_r}\sigma_{h'_{j+1}}\cdots\sigma_{h'_r}\mathbf{1}\{(i, h_{j+1}, \dots, h_r) = (l, h'_{j+1}, \dots, h'_r)\}\\
        &\leq \frac{r!(r-j)^2 m^{r-j-1}}{r!r^2m^{r-j}} \leq \frac{1}{m}.
    \end{align}
    Letting $\Sigma$ be the covariance matrix of $\Tilde Z^{\sigma}(j)$, we therefore see that the diagonal elements of $\Sigma$ are in $[0,\frac{1}{4}]$, while the off diagonal elements are in $[0,\frac{1}{m}]$. Next, define the $T\times T$ matrix $K = (1-\frac{1}{m})\Id + \frac{1}{m}\bone\bone^\top$. We claim that $K-\Sigma$ is positive semidefinite. To see this, let $x\in\RR^T\setminus\{\mathbf{0}\}$. Then by the above bounds on $\Sigma$, we have
    \begin{align}
        x^\top \Sigma x = \sum_{i=1}^T x_i (\Sigma x)_i &= \sum_{i=1}^T x_i \left(\sum_{j=1}^T \Sigma_{ij} x_j\right)\\
        &\leq \sum_{i=1}^T x_i \left(\frac{1}{4}x_i + \frac{1}{m}\sum_{j\neq i}^T x_j\right)\\
        &\leq \frac{1}{4}\norm{x}_2^2+ \frac{1}{m}\left(x^\top \bone \right)^2
    \end{align}
    It is easily checked that $x^\top K x = (1-\frac{1}{m})\norm{x}_2^2+\frac{1}{m}\left(x^\top \bone \right)^2$, and therefore since $m\geq 2$,
    \[x^\top (K-\Sigma)x\geq 0\]
    which proves the claim. Letting $Z\sim\cN(0,K)$, it thus holds by Lemma \ref{lem:EigenvaluesStochDom} that 
    \[\norm{\Tilde Z^\sigma(j)}_2\stleq \norm{Z}_2\]
    Next, note that $\norm{K}_2\leq \norm{K}_\infty \leq 1+\frac{T}{m}$, so apply Lemma \ref{lem:RefHanWrit} to see that for some universal constant $\CHW>1$,
    \begin{align}
        &\PP\left(\norm{\Tilde Z^\sigma(j)}_2^2\geq \EE(\norm{Z}_2^2)+a\right)\\
        \leq &\PP\left(\norm{Z}_2^2\geq \EE(\norm{Z}_2^2)+a\right)\\
        \leq &\exp\left(-\frac{1}{\CHW}\min\left\{\frac{a^2}{T+\frac{T^2}{m^2}},\frac{a}{1+\frac{T}{m}}\right\}\right)\\
        \leq &\frac{1}{n^d|\cS|}
    \end{align}
    where the last inequality holds by \eqref{eq:HWMainStatement1}. Taking a union bound over the elements of $\cS_m$, it therefore holds with probability at least $1-n^{-d}$ that for all $\sigma\in\cS_m$, 
    \begin{align}
        \norm{Z^\sigma(j)}_2^2 &\leq r! r^2 m^{r-j}\left(a+\EE(\norm{Z}_2^2)\right)\\
        &=r! r^2 m^{r-j}\left(a+T)\right)
    \end{align}
    where for the last equality we use $\EE(\norm{Z}_2^2)=T$. Hence for any $q \in [T],$ by Markov's inequality there at most $q$ coordinates of $Z^\sigma(j)$ with $(Z^\sigma_i(j))^2 \geq \frac{1}{q}r!r^2m^{r-j} (a + T)$. Choosing $C=C(r)=\sqrt{r!}r$ therefore completes the proof of the lemma. 
\end{proof}

\section{Subset Gaussian Cloning}\label{sec:NIQ}

\subsection{Intuition}

Iterative algorithms with input $Y$ are hard to analyze because updates reuses the realization $W$, inducing correlations between iterates. Ideally, we hope to ignore these correlations and assume the $t$-th iteration ``sees'' the random variable $\tilde Y_t = \frac{\lambda}{k^{r/2}}\theta^{\otimes r} + \tilde W_t$ where $\tilde W_t$ is a Gaussian tensor independent of everything else. 

Although this behavior is not true for general iterative algorithms, we design a ``corrected'' algorithm where this ``fresh noise'' behavior holds true. This idea underlies the random thresholds in Algorithms \ref{alg:B0} and \ref{alg:A0}
from Section \ref{sec:\runtime ainResults}. Indeed, these thresholds ``noise'' the observation $H_{\frac{r + 1}{2}, \gamma}(\sigma') - H_{\frac{r+1}{2}}(S_t)$ from line 8 of both Algorithms \ref{alg:B0} and \ref{alg:A0}. Under this noise, each $t$ iteration effectively uses the tensor $\tilde Y_t$ mentioned above where $\tilde W$ has (a bit) inflated variance compared to $W$. We present an abstraction of this technique, deemed \emph{Subset Gaussian Cloning}, which can be applied to general Gaussian additive models.

\subsection{The \texorpdfstring{$\NIQ$}{SGC} Algorithm}

\begin{definition}\label{def:NIQ} We consider a Gaussian additive model defined
as follows. Let $N \in \NN$, ``signal'' $\mu^* \in \RR^N$, ``noise'' $Z \sim
\cN(0, \Id)$, and a sequence of observations $(Y_i)_{i \in [N]} = (\mu^*_i +
Z_i)_{i \in [N]}$.

\end{definition}

Now, consider $\runtime \in \NN$ and a (possibly random, independent of
everything else) sequence of sets $\setlister = \setlister_{t \in \NN} \subseteq
[N]^\NN$. Then, \emph{Subset Gaussian Cloning}, abbreviated as $\NIQ(\mu^*,
\setlister, \runtime)$, is defined in Algorithm \ref{alg:SCG}. It takes an observation $Y\in\mathbb{R}^N$, and a sequence of subsets $(\Delta_t)_{t\in\mathbb{N}}$ of $[N]$, and returns for each subset $\Delta_t$ the noised observation $X_{\Delta_t}$, which is simply $Y_{\Delta_t}$ with an appropriate amount of correlated Gaussian noise added ($Y_{\Delta_t}\in\mathbb{R}^{|\Delta_t|}$ is defined in the obvious way). 
\begin{algorithm}
\caption{Subset Gaussian Cloning}
\label{alg:SCG}
\begin{algorithmic}[1]
\REQUIRE Observation \(Y\in\RR^N\), parameter \(\runtime \in \NN\), coordinate subsets \(\setlister = (\Delta_t)_{t \in \NN} \subseteq
[N]^\NN\)
\FOR{\(i = 1\) to \(N\)} 
\STATE Sample $G_i^1, \dots,
G_i^{\runtime}$ i.i.d. from $\cN(0,\runtime)$. 
\STATE \(\displaystyle \bar{G}_i
\gets \frac{1}{\runtime} \sum_{j=1}^{\runtime} G_i^j\) 
\STATE \(b_i^0 \gets 0\)
\ENDFOR

\STATE $t \gets 0$
\STATE $T_\runtime  \gets 0$

\WHILE{$\max_{i \in [N]} \sum_{t' \leq t} \1\{i \in \Delta_{t'}\} \leq \runtime $\footnotemark}
\STATE $t \gets t + 1$
    \STATE $T_\runtime  \gets t$
        \FOR{$i\in[N]$}
        \IF{\(i \in \Delta_t\)}
            \STATE \(b_i^t \gets b_i^{t-1} + 1\)
        \ELSE
            \STATE \(b_i^t \gets b_i^{t-1}\)
        \ENDIF
        \ENDFOR
    \STATE \[X_{\Delta_t} \gets Y_{\Delta_t} + \left( G^{b^t_{\Delta_t}}_{\Delta_t} -
    \bar{G}_{\Delta_t} \right) \]\vspace{-15pt}
    \ENDWHILE
    \RETURN $(X_{\Delta_t})_{t \in [T_\runtime]}$
\end{algorithmic}
\end{algorithm}

Notably (as we see in a moment), the algorithm $\NIQ(\mu^*, \setlister, \runtime)$ which has input $Y_i, i \in [N]$ has output \[(X_t)_{t
\in [T_\runtime]} = \left(\mu^*_{\Delta_t} + (G')_{\Delta_t}^{b^t_{\Delta_t}}\right)_{t
\in [T_\runtime ]},\] where each $(G')_i^j \iidsim \cN(0,\runtime)$ for $i \in [N]$ and
$j \in [\runtime ].$ 

Observe, the final value of $T_\runtime $ from Algorithm \ref{alg:SCG} is equivalent to
the runtime of Algorithm \ref{alg:B0}. That is, $T_\runtime $ at the time
of termination is a random function in the probability space of $(\Delta_t)_{t \in
\NN}$ and can be equivalently defined as
\[T_\runtime (\setlister) = T_\runtime  =
\inf\left\{t \in \NN: \max_{i \in [N]} \sum_{t' \leq t} \1\{i \in \Delta_t \} =
\runtime  + 1\right\}.\]

This ``noising" technique is an extension of Gaussian cloning (an implication of
the infinite divisibility of Gaussians \cite[Chapter 2]{Sato99}) inspired by the
noise injection methods of \cite[Algorithm 2]{pmlr-v65-anandkumar17a}.
\footnotetext{Note that this condition ensures that each coordinate $i\in[N]$ is updated at most $M$ times.}
\subsection{\texorpdfstring{$\NIQ$}{SGC} Has An Independent Gaussian Law}

\begin{thm}\label{thm:NIQ} Let $T_\runtime$ and $(X_{\Delta_t})_{t \in
[T_\runtime]}$ be the runtime and output of Algorithm \ref{alg:SCG}
 with input $(Y,\runtime , \setlister)$ where $Y$ is from Definition \ref{def:NIQ} for some signal $\mu^*$. Then, almost surely over the
random function $T_\runtime $,
$(X_{\Delta_t})_{t \in [T_\runtime ]}$ is a sequence of independent Gaussians with each satisfying 
$X_{\Delta_t} \laweq \mu^*_{\Delta_t} + Z'_{\Delta_t}$ where $Z'_{\Delta_t} \sim
\cN(0, \runtime  \cdot \Id_{|\Delta_t|})$ independent of $\setlister$. 
\end{thm}
\begin{proof}
The random variable $T_\runtime $ has support $\NN \cap [\runtime ,\runtime N]$. \runtime oreover, it is
important to note that random variable $T_\runtime $ is independent of both the Gaussian
additive model $Y = \mu^* + Z$ from Definition \ref{def:NIQ} and the sampled
Gaussians $G^j_i$ for each $i \in [N]$ and $j \in [\runtime ]$ as $\setlister$ is
independent of both random variables. Hence, the following proof is given conditional on a fixed $T_\runtime  \in
(\NN \cap [\runtime ,\runtime N])$. 

The proof follows from calculating the mean and covariance of
$X_t = Y_{\Delta_t} + (G^{b^t_{\Delta_t}}_{\Delta_t} - \bar G_{\Delta_t})$ and
$X_{t'}$ for two times $t, t' \in [T_\runtime]$. As $(X_t)_{t \in
[T_\runtime]}$ is a sum of Gaussian vectors, then $(X_t)_{t \in [T_\runtime]}$
is itself a Gaussian vector, concluding the proof. For notational simplicity, we drop the superscript
$t$ on $b^t_{\Delta_t}$ as it is clear from the context.

First, we check the mean. This follows from the fact that $Y_{\Delta_t} =
\mu^*_{\Delta_t} + Z_{\Delta_t}$ where $Z_{\Delta_t} \iidsim \cN(0,
\Id_{|\Delta_t|})$ and the fact that both $G^{b_{\Delta_t}}_{\Delta_t}$ and $\bar
G_{\Delta_t}$ are centered Gaussians. Thus, we have the trivial calculation that
    $\EE[X_{\Delta_t}] = \mu^*_{\Delta_t}$---the desired mean.

Second, we check the covariance. We calculate this value on the centered random
variable $\tilde X_t = Z_{\Delta_t} - G^{b_{\Delta_t}}_{\Delta_t} + \bar
G_{\Delta_t}$ as it has the same covariance structure as $X_t$. As $\tilde X_t$
is centered, we need just calculate the expected product of each $\tilde X_{t,i}$ and
$\tilde X_{t', i'}$. For concreteness, let $t,t' \in [T_\runtime]$, $i, 
i' \in \Delta_{t'}$. \runtime oreover, we let $j_t = b_{\Delta_t}$ and $j_{t'} =
b_{\Delta_{t'}}$. Calculate,
\begin{align}
\EE[\tilde X_{t,i} \tilde X_{t', i'}] 
&= \EE[(Z_{i} + G^{j_t}_{i} - \bar G_{i}) (Z_{i'} + G^{j_{t'}}_{i'} - \bar G_{i'})] \\
&= \EE[Z_{i}Z_{i'}] + \EE[G^{j_t}_{i}G^{j_{t'}}_{i'}] - \EE[G^{j_t}_{i}\bar G_{i'}] - \EE[\bar G_{i}G^{j_{t'}}_{i'}] + \EE[\bar G_i \bar G_{i'}].
\label{eq:NIQcov}
\end{align}
We calculate the above summation term by term.
\begin{enumerate}
\item $\EE[Z_{i}Z_{i'}]$: This term is zero if $i \neq i'$ and one if $i =
i'$ by the assumption on $Z_i$ in Definition~\ref{def:NIQ}.
\item $\EE[G^{j_t}_{i}G^{j_{t'}}_{i'}]$: As we assume each $G_i^j$ are 
independent Gaussians for each $i \in [N]$ and $j \in [\runtime]$.
This term is $\runtime$ if $t = t'$ and $i = i'$ and zero otherwise.
\item $\EE[G^{j_t}_{i}\bar G_{i'}]$: By the definition of $\bar G_i$ from
Definition \ref{def:NIQ}, $\EE[G^{j_t}_i \frac{1}{\runtime} \sum_{j = 1}^\runtime
G_{i'}^j] = \frac{1}{\runtime}\EE[G^{j_t}_i G^{j_{t}}_{i'}]$,
where the last equality is because $G^{j_t}_i$ and $G^{j_{t'}}_{i'}$ 
are independent if $j_{t'} \ne j_t$. It then follows that 
$\EE[G^{j_t}_{i}\bar G_{i'}] = 1$ if and only if $i = i'$.
\item $\EE[\bar G_{i}G^{j_{t'}}_{i'}]$: By symmetry, this term is also one if
and only if $i = i'$.
\item $\EE[\bar G_i \bar G_{i'}]$: We have that $\EE[\bar G_i \bar G_{i'}] =
\frac{1}{\runtime^2} \sum_{j = 1}^\runtime \sum_{j' = 1}^\runtime \EE[G_i^j
G_{i'}^{j'}]$. This is equal to \newline $\frac{1}{\runtime^2} \sum_{j = 1}^\runtime
\runtime = 1$ if $i = i'$ and zero otherwise.
\end{enumerate}
Plugging each term into \eqref{eq:NIQcov} gives, 
\[
\EE[\tilde X_{t,i} \tilde X_{t', i'}] = \1\{i = i'\} + \runtime \cdot \1\{t = t', i = i'\} -2 \cdot \1\{i = i'\} + \1\{i = i'\} = \runtime \cdot \1\{t = t', i = i'\}.
\]
This is the desired covariance, proving the equivalence in law for the sequence
$(X_{\Delta_t})_{t \in [T_\runtime]}$. 

As we have proven this statement for any element in the support of the random
variable $T_\runtime$, then the statement holds almost surely over $T_\runtime$.
\end{proof}

This theorem prescribes that each returned value of $X_{\Delta_t} = Y_{\Delta_t}
+ (G^{b_{\Delta_t}}_{\Delta_t} - \bar G_{\Delta_t})$ is equivalent to viewing
$\mu^*_{\Delta_t} + {\rm noise}$ where each ``noise'' term is an independent
Gaussian sample. The cost of $\NIQ$'s Gaussian law is from
inflating the variance of the noise by a factor of $\runtime$. To prove the
results from Section \ref{sec:Rez2}, we choose $\runtime = \plog$. We will
see in the subsequent sections that by introducing some auxiliary variables, each algorithm defined in Section \ref{sec:Rez2} is a
special case of $\NIQ$ with $n$ Gaussian observations (See Lemma
\ref{lem:equiv03} and Lemma \ref{lem:AlgBTrans} for details). This will be of significant help in our analysis of the algorithms.

\section{Proof of Theorem \texorpdfstring{\ref{thm:Trinary}}{thm:Trinary}: The Trajectory of Randomized Greedy with Random Thresholds}\label{sec:NIM}
In this section we prove Theorem \ref{thm:Trinary}.

\subsection{Converting Algorithm \ref{alg:A0}, Algorithm \ref{alg:B0} to lazy versions}

First, to simplify the following analysis, we define an almost equivalent ``lazy" version of
Algorithm \ref{alg:A0}, Algorithm \ref{alg:A}, where we have included the sub-routine Algorithm
\ref{alg:Inject_RT} and have allowed self-transitions. The purpose of Algorithm \ref{alg:A} is to relate Algorithm
\ref{alg:A0} to subset Gaussian cloning (presented above in Section \ref{sec:NIQ})
and to introduce some notation (e.g., the $p_t,q_t$ notation) to explicitly represent which transitions are
proposed on the neighborhood graph. We also introduce a similar slight modification of Algorithm
\ref{alg:B0}. 

\begin{algorithm}[ht]
    \caption{A lazy version of Algorithm \ref{alg:A0}: Randomized Greedy Local Search on $H_{(r+1)/2, \gamma}$ with random thresholds}\label{alg:A}
\begin{algorithmic}[1]
\REQUIRE Tensor observation \(Y \in \mathbb{R}^{n^{\otimes r}}\), initialization \(S_1 \in \{-1,0,1\}^n\), sparsity penalization $\gamma
\in\RR_{\geq 0}$, parameter $ \runtime  \in \NN$.
\STATE \(t \gets 0\) and \(t_i \gets 0\) for each \(i \in [n]\).
\FOR{\(i = 1 \) to \(n\)}
    \STATE Sample \(G_i^1,\dots,G_i^{\runtime}\) i.i.d. from $\cN(0,\runtime)$.
    \STATE \(\displaystyle \bar{G}_i \gets \frac{1}{\runtime}\sum_{j=1}^{\runtime} G_i^j\).
\ENDFOR

\STATE Sample \((p_1,\dots,p_{\lceil\runtime n/2\rceil})\) i.i.d. uniform over \([n]\) (i.e. the coordinate to change at step \(t\))
\STATE Sample \((q_1,\dots,q_{\lceil\runtime n/2\rceil})\) i.i.d.\ uniform over
    \(\{-1,0,1\}\) (i.e. the proposed assignment at step \(t\)).

\FOR{\(t = 1\) to \(\lceil \tfrac{\runtime\,n}{2}\rceil \)} \STATE
    \(t_{p_t} \gets t_{p_t} + 1\) 
    \STATE $S' \gets S_t + e_{p_t}(q_t - (S_t)_{p_t})$.
    \[D_t \gets H_{\frac{r+1}{2}, \gamma}(S') - H_{\frac{r+1}{2}, \gamma}(S_t)  - \bigl(G_{p_t}^{\,t_{p_t}} - \bar G_{p_t}\bigr) \|S'^{\otimes r} - S_t^{\otimes r}\|_F
    \label{eq:DSparse}\]\vspace{-15pt} \IF{\(D_t > 0\)} \STATE
    \((S_{t+1})_{p_t} \gets q_t\) \quad and \quad \((S_{t+1})_j \gets
    (S_t)_k\) for \(j \neq p_t\). \ELSE \STATE \(S_{t+1} \gets
    S_t\). \ENDIF \ENDFOR

\RETURN \(S_{\,\lceil \runtime\,n/2\rceil + 1}\)
\end{algorithmic}
\end{algorithm}

\begin{algorithm}[ht]
    \caption{A lazy version of Algorithm \ref{alg:B0}: Randomized Greedy Local Search on $H_{(r+1)/2, 0}$ with random thresholds}\label{alg:B}
\begin{algorithmic}[1]
\REQUIRE Tensor observation \(Y \in \mathbb{R}^{n^{\otimes r}}\), initialization \(S_1 \in \{-1,1\}^n\), parameter $\runtime  \in \NN$.
\STATE \(t \gets 0\) and \(t_i \gets 0\) for each \(i \in [n]\).
\FOR{\(i = 1\) to \(n\)}
    \STATE Sample \(G_i^1,\dots,G_i^{\runtime}\) i.i.d. from $\cN(0,\runtime)$.
    \STATE \(\displaystyle \bar{G}_i \gets \frac{1}{\runtime}\sum_{j=1}^{\runtime} G_i^j\)
\ENDFOR

\STATE Sample \(p_1,\dots,p_{n\runtime }\) i.i.d.\ uniform over \([n]\) (i.e. the coordinate at step \(t\) whose sign we propose to flip).

\WHILE{$\max_{i\in [n]} t_{i} \leq \runtime$}
\STATE $t \gets t+1$,
    \STATE \(t_{p_t} \gets t_{p_t} + 1,\)
    \STATE $S' \gets S_t - 2e_{p_t}(S_t)_{p_t}$
    \[\label{eq:DAlign}
        D'_t \;\gets\; H_{\frac{r+1}{2}, 0}(S')
      - H_{\frac{r+1}{2}, \gamma} (S_t)
      - \bigl(G_{p_t}^{\,t_{p_t}} 
        - \bar G_{p_t}\bigr) 
        \|S'^{\otimes r} - S_t^{\otimes
      r}\|_F 
    \]\vspace{-15pt}
    \IF{\(D'_t > 0\)}
        \STATE \((S_{t+1})_{p_t} \gets -\,(S_t)_{p_t}^t\) \quad and \quad  \((S_{t+1})_j^{t+1} \gets (S_t)_k^t\) for \(j \neq p_t\)
    \ELSE
        \STATE \(S_{t+1} \gets S_t\)
    \ENDIF
\ENDWHILE

\STATE $\finalIter \gets t$

\RETURN \(S_{\finalIter + 1}\)

\end{algorithmic}
\end{algorithm}

Under these restated algorithms, we provide a more precise version of Theorem
\ref{thm:Trinary}.

\begin{thm}[A stronger statement to Theorem \ref{thm:Trinary}]\label{thm:TrinaryExt}
    Consider an arbitrary $\theta \in \{-1,0,1\}^n$, $\|\theta\| = k$. Assume $k=\Omega( \sqrt{n})$ and let $r \geq 3$ be any odd tensor power. Let $S_0 = S_{\rm
    HOM}$. If $\runtime _1 = \lceil \log^4(n) \rceil$, $\runtime _2 = \lceil 25 \log(3n)
    \rceil$, $\gamma = 2(\runtime_2\log(\runtime_2 n))^{1/2} +
    1$ and $\lambda = \Tilde \Omega(n^{r/4})$ (dependent on $\runtime_2$ and $\gamma$),
    then w.h.p. as $n \conv{} +\infty$, if $S_1$ is the output of Algorithm
    \ref{alg:B} with input $(Y, S_0, \runtime_1)$, then Algorithm \ref{alg:A} with
    input $(Y, S_1, \gamma, \runtime_2)$ outputs $\theta$. Note that both algorithms terminate in $O(n \log^4 n)$
    iterations.
    \end{thm}

We reiterate that Algorithms \ref{alg:A} and \ref{alg:B} differ from Algorithms \ref{alg:A0} and \ref{alg:B0} only in that they can admit the proposal $q_t = (S_t)_{p_t}$ where it is irrelevant if such a transition is accepted or not. This change to a ``lazy" randomized greedy scheme is made to simplify the analysis. To prove
Theorem \ref{thm:Trinary} from Theorem \ref{thm:TrinaryExt}, consider the
subset of steps $\{t:q_t \ne (S_t)_{p_t}\}$; this set corresponds to the iterative
steps of Algorithms \ref{alg:A0} and \ref{alg:B0}. 
Hence, the obvious coupling which neglects the ``lazy" steps where $q_t = (S_t)_{p_t}$ implies that Algorithms \ref{alg:A} and \ref{alg:B} succeed whenever Algorithms \ref{alg:A0} and \ref{alg:B0} in at most the same time. Thus Theorem \ref{thm:Trinary} holds when Theorem \ref{thm:TrinaryExt}
holds. 

The remainder of this subsection is dedicated to proving Theorem
\ref{thm:TrinaryExt}, and thus Theorem \ref{thm:Trinary}. More precise
conditions on $\lambda$ are given in Assumption \ref{as:tau}.

\subsection{Useful Definitions And Constants}
Recall from Theorem \ref{thm:TrinaryExt} we defined $\runtime_1 = \lceil
\log^4(n) \rceil$ and $\runtime_2 = \lceil 25 \log(3n) \rceil$. Throughout this
section, $S_t\in\{-1,0,1\}^n$ denotes the iterates of either Algorithm
\ref{alg:A} or Algorithm \ref{alg:B}, depending on context. Note, for both
Algorithms \ref{alg:A} and \ref{alg:B}, we have introduced random variables
$p_t$ and $q_t$ (for Algorithm \ref{alg:B} we fix $q_t = -(S_t)_{p_t}$) representing the proposal to change coordinate $(S_t)_{p_t}$ to
$q_t$ at time $t$. We also defined $\finalIter$ as the final time step in Algorithm \ref{alg:B}.
\emph{Finally, we define three constants we use repeatedly in what follows. Let,}
\[
C_{\rm Noise} := \left(\sum_{j = 1}^r \binom{r}{j}^2 2^{2j}\right)^{1/2}\!\!\!\!\!\!,\qquad
C_{\mathrm{Signal}}:=r,\qquad
C_{\mathrm{Reg}}:=\frac{r+1}{2}.\]

Before we prove Theorem \ref{thm:TrinaryExt} (And thus Theorem
\ref{thm:Trinary}), we provide a decomposition of both $D_t$ and $D'_t$ used
in Algorithm \ref{alg:A} and Algorithm \ref{alg:B} respectively. Through this
decomposition, three terms representing the contribution of the signal, noise,
and regularization are identified.

\begin{definition}\label{def:AnalysisA} Define the following five functions
    dependent on either Algorithm \ref{alg:A} or Algorithm \ref{alg:B}'s iterate
    $S_t$ and its the assignment of $p_t$ and $q_t$ inducing the proposal $S' = S_t + e_{p_t}(q_t - (S_t)_{p_t})$ (note for Algorithm
    \ref{alg:B}, $q_t = -(S_t)_{p_t}$): 
    \begin{align}
        \bSignal   &= \frac{\lambda}{k^{r/2}}\langle (S')^{\otimes r},\,\theta^{\otimes r}\rangle - \langle (S_t)^{\otimes r},\,\theta^{\otimes r}\rangle,\\
        \bNoise    &= \bigl\langle(S')^{\otimes r},\,W\bigr\rangle
                      - \langle(S_t)^{\otimes r},\,W\rangle,\\
        \bReg      &= \gamma\left[\Bigl(\bigl\|S_t\bigr\|_0 + \bigl|q_t\bigr| - \bigl|(S_t)_{p_t}^t\bigr|\Bigr)^{\frac{r+1}{2}}
                      - \,\bigl\|S_t\bigr\|_0^{\frac{r+1}{2}}\right],\\
        \Variance &= \|(S')^{\otimes r} - (S_t)^{\otimes r}\|_F\\
        \bCorrect  &= \bigl(G_{p_t}^{\,t_{p_t}} - \bar G_{p_t}\bigr) V
        \end{align}
 \end{definition}
     Above, the term $\bCorrect$ uses the Gaussian random variable $G_i^j$ for
     $i \in [n]$ and $j \in [\runtime]$ utilized in Algorithms \ref{alg:A} and \ref{alg:B}. The usefulness of the mentioned functions is discussed in the following remark.

\begin{remark}[Decomposing The Difference of Hamiltonian $H_{\frac{r+1}{2},
    \gamma}$]\label{rm:AlgA} Consider the Hamiltonian $H_{\frac{r+1}{2}, \gamma}$ from
    \eqref{eq:Ham}, $S_t \in \{-1,0,1\}^n$ and proposals $p_t, q_t
    \in [n] \times \{-1,0,1\}$. Denote the corresponding proposed transition by
    $S' = \proposal$. Then, we decompose the Hamiltonian difference as
    follows:
    \begin{align}
        H_{\frac{r+1}{2}, \gamma}(S') - H_{\frac{r+1}{2}, \gamma}(S_t) 
        &= \langle S'^{\otimes r}, Y\rangle - \gamma\|S'\|_0^{\frac{r+1}{2}} -  \langle (S_t)^{\otimes r}, Y \rangle + \gamma\|S_t\|_0^{\frac{r+1}{2}}\\
        &= \langle (\proposal)^{\otimes r}, Y\rangle - \gamma\|\proposal\|_0^{\frac{r+1}{2}}\\
        &\qquad - \langle (S_t)^{\otimes r}, Y \rangle + \gamma\|S_t\|_0^{\frac{r+1}{2}}\\
        &= \frac{\lambda}{k^{r/2}}\langle (\proposal)^{\otimes r}, \theta^{\otimes r}\rangle - \langle (S_t)^{\otimes r}, \theta^{\otimes r}\rangle\\
        &\qquad + \langle(\proposal)^{\otimes r} , W\rangle - \langle (S_t)^{\otimes r}, W \rangle\\
        &\qquad - \gamma\left((\|S_t\|_0 + (|q_t| - |(S_t)_{p_t}^t|))^{\frac{r+1}{2}} - \|S_t\|_0^{\frac{r+1}{2}}\right),\\
        &= \bSignal + \bNoise - \bReg
    \end{align}
\end{remark}

Now, given this decomposition, to analyze important quantities of our algorithm such as $D_t, D'_t$ we can focus on bounding $\bSignal$, $\bNoise$,
and $\bReg$. We accomplish this by the auxiliary results, Lemmas
\ref{lem:NoiseEquiv}, \ref{lem:VarLawBound}, \ref{lem:bSignalControl} and
\ref{lem:bRegControl} and collect many useful bounds for these terms
in Corollary \ref{cor:bounds}; we use them extensively. 

\subsection{An Initialization-Signal Trade-off For Algorithm \ref{alg:A0}}\label{sec:ReducProofBinary}

On the path to proving Theorem \ref{thm:Trinary}, we establish a trade-off
between the quality of an initialization $S_1$ and the minimal value of
$\lambda$ for which Algorithm \ref{alg:A0} recovers $\theta$. This
trade-off has the following informal description: If 
\[\lambda = \Tilde \Omega\left(\frac{\sqrt{k}}{\cos(S_1,
\theta)^{r-1}}\right)\label{eq:tradeoff},\] then Algorithm \ref{alg:A0}
``monotonically" converges to $\theta$, in the sense that at every change we strictly decrease the Hamming distance to $\theta$. This result can be (intuitively)
interpreted as a condition for local convexity about the true solution $\theta$
for the Hamiltonian $H_{\frac{r+1}{2}, \gamma}$ for a tuned $\gamma$ and
$\lambda$ large enough. The condition
\eqref{eq:tradeoff} is codified by the following Lemma.

\begin{lemma}\label{lem:success2}For any $\runtime \in \NN$, let $C^*_n =
    2\runtime^{1/2}\sqrt{\log(\runtime n)}$ and $\gamma = C^*_nC_{\rm
    Noise}/C_{\rm Reg}+1$. For all $t \in [T^*_\runtime]$, where
    \[T^*_\runtime = \min\left\{t \in \NN: \max_{i \in [n]} \sum_{t' \leq t}
    \1\{p_{t'} = i\} = \runtime + 1\right\}.\] 
    Then, the following holds with probability $1-o(1)$: 
    
    If $S_1\in \{-1,0,1\}^n$, $\lambda \in \RR_+$, satisfy $\langle S_1, \theta \rangle \geq 2$ and \footnote{As any vector $v \in \{-1,0,1\}^n$ has $\|v\|_2^2 = \|v\|_0$.} 
    \[\left(\frac{\langle S_1, \theta\rangle - 2}{\|\theta\|_2(\|S_1\|_2 + 1)}\right)^{r-1} =\left(\frac{\langle S_1, \theta \rangle - 2}{k^{1/2}(\|S_1\|_0 +
            1)^{1/2}}\right)^{r-1} \geq \frac{2\gamma}{C_{\rm Signal}}
            \frac{\sqrt{k}}{\lambda},\] 
   then, for each $t\in [T_\runtime^*]$, 
    \begin{enumerate}
        \item When $p_{t}, q_t$ are such that $\theta_{p_t} = (S_t)_{p_t}$ and 
        $q_t \neq \theta_{p_t}$, then $D_t < 0$, leading to rejection.
    \item When $p_{t}, q_t$ are such that $\theta_{p_t} \ne (S_t)_{p_t}$ and 
        $q_t = \theta_{p_t}$, then $D_t > 0$, leading to acceptance.
    \end{enumerate}
    \end{lemma}
The proof of Lemma \ref{lem:success2} is given in Section \ref{sec:NoiseInjectKeyLemsProofs}. 

\subsection{Proving Theorem \ref{thm:TrinaryExt}/\ref{thm:Trinary}}\label{sec:ProofTrinary}
Now, we turn to the proof of Theorem \ref{thm:TrinaryExt}. Assume that $k \geq 
C_k\sqrt{n}$ for some positive constant $C_k$. Before continuing, we make a quantitative assumption on $\lambda$ to
make the following analysis more explicit.

\begin{remark}\label{rm:tau} For convenience, parameterize $\lambda = \tau
    \sqrt{\runtime_1 n^{r/4}}$, recalling that $\runtime_1 = \lceil \log^4(n)
    \rceil$ from Theorem
    \ref{thm:TrinaryExt}; $\tau$ may scale with $n$. As $\runtime_1 = \plog$, it
    suffices to prove the statement in Theorem \ref{thm:TrinaryExt} for all
    $\tau$ larger than a $\plog$ factor, as then $\tau \sqrt{\runtime_1}$
    remains $\plog$.
    \end{remark}
    
    \begin{as}\label{as:tau} Assume the following on $\tau$, recalling
        $\runtime_1 = \lceil \log^4(n) \rceil$ and $\runtime_2 = \lceil 25\log(3n)
        \rceil$: \newline (Here $C^*_n = 2(\runtime_2)^{1/2}\sqrt{\log(\runtime_2
        n)}$ and $\gamma = (C^*_nC_{\rm Noise})/C_{\rm Reg} + 1$ from
        Lemma \ref{lem:success2})

    For a sufficiently large constant $C_\tau > 0$ (dependent on $r$ and the constant $C_k>0$ that comes from the inequality $k\geq C_k\sqrt{n}$  ),    
    \begin{itemize}
        \item[(a)] $\tau \geq C_\tau \gamma/\sqrt{\runtime_1}$,
        \item[(b)] $\tau \geq C_\tau C_{\rm Noise} 
        \sqrt{\log(\runtime_1)} / C_{\rm Signal}$
        \item[(c)] $\tau \geq C_\tau$.
    \end{itemize}
    Note that all terms in the right hand sides of the above inequalities are of $\plog$ order.
    \end{as}

 We see momentarily that we often need to consider $\Variance$, given in
    Definition \ref{def:AnalysisA}, when $\|S_t\|_0 = n$ and $\|S_t -
    2(S_t)_{p_t} e_{p_t}\|_0 = n$. When this is the case, by using $\|v\|_2^2 = \|v\|_0$ for vectors in $\{-1,1\}^n$, we use the shorthand
    \begin{align}
        V_n &:= \Variance = (\|(S_t - 2e_{p_t})^{\otimes r} - (S_t)^{\otimes r}\|_F^2)^{1/2}\\
        &= (\|S_t -
    2(S_t)_{p_t} e_{p_t}\|^{2r}_2 - 2(\langle S_t - 2(S_t)_{p_t}e_{p_t}, S_t \rangle)^r + \|S_t\|_2^{2r})^{1/2}\\
    &= (2n^r - 2(\langle S_t, S_t \rangle - 2(S_t)_{p_t}\langle e_{p_t}, S_t\rangle)^r)^{1/2}\\
    &= (2(n^r - (n - 2)^r))^{1/2}.
    \end{align}
Now, Theorem \ref{thm:TrinaryExt} (and thus, Theorem \ref{thm:Trinary}) follows
from the following key lemmas. Their proofs are given in Section
\ref{sec:NoiseInjectKeyLemsProofs}. Notably, Lemma \ref{lem:Init} is the only one that requires $r$ to be odd.
   
\begin{lemma}[``Homotopy initialization achieves sufficient overlap with $\theta$'']\label{lem:Init}Let $\lambda = \tau \sqrt{\runtime_1} n^{r/4}$,
    where $\tau$ satisfies Assumption \ref{as:tau} and recall $S_{\rm HOM}$ from
    Definition \ref{def:HOM}. If $r$ is odd, then with probability $1-o(1)$,
    \[\langle S_{\rm HOM}, \theta \rangle \geq n^{1/4}k^{1/2}/5 \qquad \text{and}
    \qquad \|S_{\rm HOM}\|_0 = n.\] 
\end{lemma}

\begin{lemma}[``Algorithm \ref{alg:B} boosts small overlap to large overlap'']\label{lem:boost03} Let $\lambda = \tau \sqrt{\runtime_1} n^{r/4}$,
    where $\tau$ satisfies Assumption \ref{as:tau}. Assume that $S_1 \in
    \{-1,1\}^n$ has $\langle S_1, \theta \rangle \geq n^{1/4}k^{1/2}/5$. 
    
    Then,
    with probability $1 - o(1)$, there exists a time step $t^* \leq \lfloor
    \finalIter / 2 \rfloor$ in Algorithm \ref{alg:B} (with input $(Y, S_1,
    \gamma, \runtime_2)$) such that $\langle S_{t^*}, \theta\rangle \geq 2$ and satisfies 
    \[\frac{2C_{\rm
    Signal}\frac{\lambda}{k^{r/2}}}{V_n\sqrt{\runtime_1}}(\langle S_{t^*},
    \theta \rangle - 2)^{r-1} \geq 2\sqrt{\log(\runtime_1 n)}.\tag{\theequation}\hypertarget{eq:threshold}{}
\stepcounter{equation}\]
\end{lemma}

\begin{lemma}[``Algorithm \ref{alg:B} boosts large overlap to sign recovery'']\label{lem:boost01} Let $\lambda = \tau \sqrt{\runtime_1} n^{r/4}$,
where $\tau$ satisfies Assumption \ref{as:tau}. 

Then, with probability $1-o(1)$,
if there exists a time step $t_0 \leq \lfloor \finalIter /2 \rfloor$ in Algorithm
\ref{alg:B} (with input $(Y, S_1, \gamma, \runtime_2)$) such that $\langle S_{t_0}, \theta \rangle \geq 2$ and
satisfies  (for
$\varepsilon > 0$ sufficiently small),  
\[\fuConsto \geq (2-\varepsilon)\sqrt{\log(\runtime_1 n)},\label{eq:8.9}\] we have that
Algorithm \ref{alg:B} outputs $S \in \{-1,1\}^n$ with $\langle S, \theta \rangle
= k$, i.e., for all $i \in [n]$ if $\theta_i \neq 0,$ it holds $S_i=\theta_i$.
\end{lemma}

\begin{lemma}[``Algorithm \ref{alg:A} boosts sign recovery to full recovery'' ]\label{lem:AlgAWorks} Let $\lambda = \tau \sqrt{\runtime_1}
    n^{r/4}$ and where $\tau$ satisfies Assumption \ref{as:tau}. 
    
    Then, with
    probability $1 - o(1)$, Algorithm \ref{alg:A} with input $(Y, S_1, \gamma,
    \runtime_2)$ (with $\gamma$ from Assumption \ref{as:tau}) outputs $\theta$ whenever $\langle S_1, \theta \rangle = k$,
    $\|S_1\|_0 = n$.
\end{lemma}

\begin{proof}[Proof of Theorem \ref{thm:TrinaryExt} (and thus Theorem
    \ref{thm:Trinary})]
        Consider the following events for $r$ odd:
        \begin{itemize}
            \item[(a)] $\langle S_{\rm HOM}, \theta \rangle \geq
            n^{1/4}k^{1/2}/5$ with $\|S_{\rm HOM}\|_0 = n$.
        \item[(b)] Algorithm \ref{alg:B} with input $(Y,
            S_{1}, \runtime_1)$ contains an iterate $S_{t^*}$ for $t^* \leq \lfloor
            \finalIter / 2 \rfloor$ such that,
                \[\frac{2C_{\rm
                Signal}\frac{\lambda}{k^{r/2}}}{V_n\sqrt{\runtime_1}}(\langle
            S_{t^*}, \theta \rangle - 2)^{r-1} \geq
                2\sqrt{\log(\runtime_1 n)}\]
                and $\langle S_{t^*}, \theta \rangle \geq 2$  when $\langle S_1, \theta
            \rangle \geq n^{1/4}k^{1/2} / 5$ with $\|S_1\|_0 = n$.
        \item[(c)] If Algorithm \ref{alg:B} contains an iterate $S_{t_0}$ with $\frac{2C_{\rm
                Signal}\frac{\lambda}{k^{r/2}}}{V_n\sqrt{\runtime_1}}(\langle
            S_{t_0}, \theta \rangle -2)^{r-1} \geq
            (2-\varepsilon)\sqrt{\log(\runtime_1 n)}$ and $\langle S_{t_0}, \theta \rangle \geq 2$ where $t_0 \leq
            \lfloor T^* / 2 \rfloor$
            then Algorithm \ref{alg:B} outputs $S \in \{-1,1\}^n$ where $\langle S, \theta \rangle = k$ and $\|S\|_0 = n$.
        \item[(d)] Algorithm \ref{alg:A} with input $(Y, S, \gamma, \runtime_2)$, where
            $\gamma$ is from the statement of Lemma \ref{lem:AlgAWorks} and $S$ is from
            (c), outputs $\theta$.
        \end{itemize}

        The definition of Algorithm \ref{alg:B} (specifically how the only
        proposed transitions in Algorithm \ref{alg:B} maintain $\|S_t\|_0 =
        \|S_{t-1}\|_0$), Lemmas \ref{lem:Init},  \ref{lem:boost03},
        \ref{lem:boost01} and Lemma \ref{lem:AlgAWorks} 
        imply that the events (a), (b), (c), and (d) hold jointly with
        probability $1-o(1)$ under Assumption \ref{as:tau} for $\lambda = \tau
        \sqrt{\runtime_1} n^{r/4}$.
        Therefore, Theorem \ref{thm:TrinaryExt} holds by the following argument:

        By (a), we have that $\langle S_{\rm HOM}, \theta \rangle \geq
        n^{1/4}k^{1/2}/5$. Thus, (b) implies that Algorithm \ref{alg:B} with
        input $(Y, S_{\rm HOM}, \runtime_1)$ contains an iterate $S_{t^*}$ for $t^* \leq
        \lfloor \finalIter / 2 \rfloor$ such that \[\frac{2C_{\rm
        Signal}\frac{\lambda}{k^{r/2}}}{V_n\sqrt{\runtime_1}}(\langle
        S_{t^*}, \theta \rangle - 2)^{r-1} \geq 2\sqrt{\log(\runtime_1 n)}.\]
        The iterate $S_{t^*}$ is then boosted by Algorithm \ref{alg:B} to output
        iterate $S$ with $\langle S, \theta \rangle = k$ and $\|S\|_0 = n$ by
        (c). Then, (d) finally implies Theorem \ref{thm:TrinaryExt} holds as (d)
        implies Algorithm \ref{alg:A} with input $(Y, S, \gamma, \runtime_2)$ outputs
        $\theta$. The statement on the run-time follows as both Algorithm \ref{alg:A} and \ref{alg:B} have, at most, $\Theta(n\log^4(n))$ proposals.
    \end{proof}

\subsection{Proofs of Key Lemmas}\label{sec:NoiseInjectKeyLemsProofs}
In this section we prove Lemmas \ref{lem:success2}, \ref{lem:Init}, \ref{lem:boost03}, \ref{lem:boost01} and \ref{lem:AlgAWorks}.

As previously mentioned, a majority of the proofs to key results require bounds
on the terms $\bSignal$, $\bNoise$, $\bReg$. The unfortunately very technical corollary we provide below, which 
is implied by the Lemmas \ref{lem:NoiseEquiv}, \ref{lem:VarLawBound},
\ref{lem:bSignalControl}, \ref{lem:bRegControl} stated in Section
\ref{sec:NoiseInjectSecondaryLemsProofs}, is used of great help and used extensively in this
section. Its proof is deferred to Section \ref{sec:NoiseInjectSecondaryLemsProofs}.

\begin{cor}\label{cor:bounds}
If $S' = S_t + e_{p_t}(q_t - (S_t)_{p_t})$ is the proposed transition, the following statements all hold:
    \begin{itemize}
        \item [(a)] Recall $\bNoise$ from Definition \ref{def:AnalysisA}, let
        $Z$ be a standard Gaussian, we have that $\bNoise \laweq \Variance Z$.
        \item [(b)] Recall $\Variance$ from Definition \ref{def:AnalysisA}, we
        have the upper bound: \[\Variance \leq C_{\rm Noise}|q_t -
            (S_t)_{p_t}|\;\|S_t\|_0^{\frac{r-1}{2}}\label{eq:VU}.\]
        \item [(c)] Let $K_n \in \RR_+$ such that $\PP(|Z| \leq K_n) = 1-o(1)$ for a standard Gaussian $Z$, then with probability $1-o(1)$:
            \[\hspace{-30pt}-K_n \cdot C_{\rm Noise}|q_t -
(S_t)_{p_t}|\;\|S_t\|_0^{\frac{r-1}{2}} \leq \bNoise \leq K_n \cdot C_{\rm
Noise}|q_t - (S_t)_{p_t}|\;\|S_t\|_0^{\frac{r-1}{2}}.\label{eq:NoiseUL}\]
   \item [(d)] Recall $\bSignal$ from Definition \ref{def:AnalysisA}. Assume
       that \[\min(\langle S_t, \theta \rangle, \langle S', \theta
            \rangle) \geq 0.\] If $(q_t - e_{p_t})\theta_{p_t} \geq 0$, then 
\[C_{\rm Signal}\frac{\lambda}{k^{r/2}}\min(\langle S', \theta \rangle, \langle S_t, \theta \rangle)^{r-1}(q_t -
(S_t)_{p_t})\theta_{p_t} \leq \bSignal,\label{eq:SignalL}\] 
            \[\text{and,\;}\bSignal \leq C_{\rm
Signal}\frac{\lambda}{k^{r/2}}\max(\langle S',
\theta \rangle,\langle S_t, \theta \rangle)(q_t -
            (S_t)_{p_t})^{r-1}\theta_{p_t}. \label{eq:SignalU}\] 
            Moreover, if $(q_t - e_{p_t})\theta_{p_t} \leq 0$, then,
\[\hspace{-40pt}C_{\rm Signal}\frac{\lambda}{k^{r/2}}\max(\langle S', \theta \rangle, \langle S_t, \theta \rangle)^{r-1}(q_t -
(S_t)_{p_t})\theta_{p_t} \leq \bSignal,\label{eq:NSignalL}\]
\[\text{and,\;}\bSignal \leq C_{\rm Signal}\frac{\lambda}{k^{r/2}}\min(\langle S', \theta \rangle, \langle S_t, \theta \rangle)
            ^{r-1}(q_t - (S_t)_{p_t})\theta_{p_t}.\label{eq:NSignalU}\]
  \item [(e)] Recall $\bReg$ from Definition \ref{def:AnalysisA}, we have the following upper and lower bounds: 
  \[C_{\rm Reg}\gamma\min(\|S_t\|_0, \|S_t\|_0 + (|q_t| -
            |(S_t)_{p_t}|))^{\frac{r+1}{2}}(|q_t| - |(S_t)_{p_t}|) \leq
\bReg,\label{eq:RegL}\] 
\[\text{and,\;}\bReg \leq C_{\rm Reg}\gamma\max(\|S_t\|_0, \|S_t\|_0 + (|q_t| -
    |(S_t)_{p_t}|))^{\frac{r+1}{2}}(|q_t| - |(S_t)_{p_t}|).\label{eq:RegU}\]\end{itemize}
\end{cor}

We also provide a more specific bounds on $V_n$ used in Lemma \ref{lem:boost01} and Lemma \ref{lem:boost03}    
    \begin{cor}\label{cor:Vn} We have the following upper bounds,
        \[V_n = \|(S_t - 2(S_t)_{p_t}^te_{p_t})^{\otimes r}
    - (S_t)^{\otimes r}\|_F \leq C_{\rm Noise}
    \|S_t\|_0^{\frac{r-1}{2}} = C_{\rm Noise}
    n^{\frac{r-1}{2}}\label{eq:VNU},\]
    \end{cor}
\begin{proof}
    The penultimate inequality is due to \eqref{eq:VU} in Corollary \ref{cor:bounds} and the final inequality uses $\|S_t\|_0 = n$.
\end{proof}

\subsubsection{Proof of Lemma \ref{lem:success2}} To prove Lemma
\ref{lem:success2}, we make use of the following lemma, whose proof is deferred to Section \ref{sec:NoiseInjectSecondaryLemsProofs}.
\begin{lemma}\label{lem:equiv03} Let $\runtime \in \NN$ be arbitrary, $(p_t)_{t
\in \NN}$, $(q_t)_{t \in \NN}$ be a sequence of independent uniform draws from
$[n]$, $\{-1,0,1\}$, respectively. Define the random function,
    \[T^*_\runtime((p_t)_{t \in \NN}) = T^*_\runtime = \inf\left\{t \in \NN:
    \max_{i \in [n]} \sum_{t' \leq t} \1\{p_{t'} = i\} = \runtime + 1\right\}.\label{eq:TstarMdef}\] Almost surely over the value of $T^*_\runtime$, the
    sequence of observations $(D_t)_{t \in [T^*_\runtime]}$ from
    \eqref{eq:DSparse}, has the following equivalence in law,
        \[\hspace{-30pt}\label{eq:DAlignLawOld}(D_t)_{t \in [T^*_\runtime]}
        \laweq \bigg{(} \bSignal + \bReg + \Variance \runtime^{1/2}Z_t\bigg{)}_{t \in
        [T^*_\runtime]},\] where each $Z_t \iidsim \cN(0, 1)$ is independent of
        $(p_t)_{t \in \NN}$ and thus $T^*_\runtime$.
\end{lemma}

We are now ready to prove Lemma \ref{lem:success2}.
    \begin{proof}[Proof of Lemma \ref{lem:success2}]

    For each $t\in[T^*_\runtime]$, let $Z_t \iidsim \cN(0, 1)$. Then by Lemma
    \ref{lem:equiv03}, almost surely with respect to $T^*_{\runtime}$, it holds that
    \[\hspace{-30pt}(D_t)_{t \in [T^*_\runtime]} \laweq \bigg{(} \bSignal +
        \bReg + \Variance \runtime^{1/2}Z_t\bigg{)}_{t \in [T^*_\runtime]},\] where
        $(D_t)_{t\in[T^*_\runtime]}$ are the observations from
        \eqref{eq:DSparse} and each $Z_t$ is independent of $(p_t)_{t \in [T^*_M]}$. By Lemma \ref{lem:Mills} and a union bound, we have
        \[\PP\left(\max_{t \in [T^*_\runtime]} |Z_t| > 2\sqrt{\log(\runtime
        n)}\right) \leq \PP\left(\max_{t \in [nM]} |Z_t| > 2\sqrt{\log(\runtime
        n)}\right) \leq \runtime n \PP(|Z_1| > 2\sqrt{\log(\runtime
        n)}) = o(1),\] 
        and thus, $\max_{t \in [T_M^*]}
        \runtime^{1/2}Z_t \leq 2\runtime^{1/2}\sqrt{\log(\runtime n)} = C^*_n$
        with probability $1-o(1)$.    
    For the remainder of the proof of Lemma \ref{lem:success2}, we will use the shorthand
    $S' = \proposal$ and condition on the above bound for $\max_{t \in [Mn]}|Z_t|$---an event with probability $1-o(1)$. For simplicity, we often abbreviate $\max_{t \in [Mn]}|Z_t|$ by $\max |Z_t|$, leaving maximum over the set of $t \in [nM]$ implicit.
    
    \textbf{We prove the following by induction on $t$}: For each $t \in \{0,1, \dots , T^*_M\}$ and for any transition that is accepted conditional on $\max|Z_t| \leq C_n^*$, we have that the statements (1) and (2) in Lemma \ref{lem:success2} hold at time $t$ and
        \[
        \langle S_{t+1}, \theta \rangle \geq 2 
        \label{eq:overlapatleasttwo}\]
and
        \[
        \left(\frac{\langle S_{t+1}, \theta \rangle - 2}{k^{1/2}(\|S_t\|_0 +
        1)^{1/2}}\right)^{r-1} \geq \left(\frac{\langle S_1, \theta \rangle -
        2}{k^{1/2}(\|S_1\|_0 + 1)^{1/2}}\right)^{r-1} \geq \frac{2\gamma}{C_{\rm
        Signal}} \frac{\sqrt{k}}{\lambda},\label{eq:ind_goal}\] 
         (assuming, for notational simplicity, both the statements (1) and (2) in Lemma \ref{lem:success2} at time $t=0$ and \eqref{eq:overlapatleasttwo}, \eqref{eq:ind_goal} at time $t=T^*_M+1$ are voidly true).

        For the base case $t = 0$ \eqref{eq:overlapatleasttwo}, \eqref{eq:ind_goal} hold trivially by
        assumption on $\langle S_1, \theta \rangle$ and the voidness of the the statements (1) and (2) in Lemma \ref{lem:success2} at time $t=0$. Now, assume the inductive hypothesis holds at time $t$. We also repeat here for reference statements (1) and (2) from Lemma \ref{lem:success2} for convenience.
         \begin{enumerate}
        \item When $p_{t}, q_t$ are such that $\theta_{p_t} = (S_t)_{p_t}$ and 
        $q_t \neq \theta_{p_t}$, then $D_t < 0$, leading to rejection.
    \item When $p_{t}, q_t$ are such that $\theta_{p_t} \ne (S_t)_{p_t}$ and 
        $q_t = \theta_{p_t}$, then $D_t > 0$, leading to acceptance.
    \end{enumerate}

        We now prove the inductive claim by showing that statements (1) and (2) hold for time $t$ and the inequalities \eqref{eq:overlapatleasttwo} and \eqref{eq:ind_goal} hold for time $t+1$. We do this by a case-by-case analysis of the behavior of the algorithm under each proposal. Notice that some possible transitions do not belong in the set of transitions considered in statements (1) or (2). To address all possible transitions we divide our analysis in the following four steps.

        \begin{itemize}
            \item \textbf{Step 1.} Prove the inductive statement for the set of transitions that aren't considered in the statements (1) and (2) in Lemma  \ref{lem:success2} and leave both the dot product of $S_t$ with $\theta$ and the support size of $S_t$ unchanged. In that case the inductive claim boils down to proving \eqref{eq:overlapatleasttwo} and \eqref{eq:ind_goal}.
            \item \textbf{Step 2.} Classify which proposals are going to be accepted or not for all the proposals considered in Table \ref{tb:Tern}, i.e. all proposals that are specified in statements (1) and (2). 
            
            \item \textbf{Step 3.} Classify which proposals are going to be accepted or not for all the proposals considered in Table \ref{tb:tbtwo}, i.e., not specified in Steps 1 and 2. 
            
            \item \textbf{Step 4.} Combine our findings from Steps 3,4 to prove the inductive step. 
        \end{itemize}

        To further assist the reader, in the Tables \ref{tb:Tern} and \ref{tb:tbtwo} we added an extra column, using ``A'' for the cases we will prove $D_t > 0$ (i.e., acceptance) with high probability and 
        for transitions with a ``R'' we will prove $D_t < 0$ (i.e., rejection) with high probability (now, without the conditioned event on $\max |Z_t|$).

        \textbf{Step 1.}
    
        There are two sets of proposals we will consider here. 
        If a proposal has $(S_t)_{p_t} = q_t$, then the
        proposal $S'$ has $\langle S', \theta \rangle = \langle S_t, \theta \rangle$ and $\|S'\|_0 = \|S_t\|_0$, for which the inductive claim of inequality \eqref{eq:ind_goal} holds trivially for $t+1$ whether the algorithm accepts the proposal or not. Moreover, the absence of change in $\langle S', \theta \rangle$ and $\|S'\|_0$ when
        $(\theta_{p_t}, (S_t)_{p_t}, q_t) = (0,1,-1)$ or $(0,-1,1)$, means these two proposals also satisfy the inductive claim of inequality \eqref{eq:ind_goal} trivially for $t+1$ regardless of the proposals acceptance or rejection. 
        Note, neither of these transitions satisfy the statements of (1) or (2) from Lemma \ref{lem:success2} so these transitions can be safely ignored with respect to statements (1) and (2) in the inductive claim.
        
\begin{table}[]

\begin{tabular}{|cccc|cccc|cccc|}
\hline
& \multicolumn{2}{c}{$\theta_{p_t} = -1$} & & & \multicolumn{2}{c}{$\theta_{p_t} = 0$} & & & \multicolumn{2}{c}{$\theta_{p_t} = 1$} & \\
\hline
$(S_t)_{p_t}$ & $q_t$ & Case & A/R 
  & $(S_t)_{p_t}$ & $q_t$ & Case & A/R 
  & $(S_t)_{p_t}$ & $q_t$ & Case & A/R \\
\hline
-1 &  0 & (a)(iii) & R   
  & -1 &  0 & (b)(i) & A   
  &  -1   &  1 & (b)(ii) & A   \\

  -1 & 1 & (a)(ii) & R   
  &  0 & -1 & (a)(i) & R   
  &  0   &  1 & (b)(iii) & A   \\

  0 & -1 & (b)(iii) & A   
  &  0 &  1 & (a)(i) & R   
  &  1 & -1 & (a)(ii) & R   \\

  1   &  -1  &  (b)(ii) &   A    
  &   1  &  0   &  (b)(i)       
  &  A & 1 &  0 & (a)(iii) & R   \\
\hline
\end{tabular}
\vspace{1em}
\caption{The set of all proposal $S' = S_t + e_{p_t}(q_t - (S_t)_{p_t})$
which satisfy conditions (1) or (2) from Lemma \ref{lem:success2}.
The first two columns give a possible transition with respect to the value of $\theta_{p_t}$. The third columns details where we prove that a given proposal is accepted or rejected. The fourth column, A/R (short for accept or reject) designates whether we will prove that the proposal is going to be accepted or not with high probability. }\label{tb:Tern}
\end{table}
\textbf{Step 2.}

 In this step, we analyze the different proposal split into a series of cases. All cases (in both Steps 2,3) are hanled in the same fashion.  First we invoke Lemma \ref{lem:equiv03} according to which we know the following equality in law, \[D_t \laweq \bSignal + \bReg + \Variance M^{1/2}Z_t.\] We then plug
        in the given choices of $ \theta_{p_t},(S_t)_{p_t}, q_t$ of each case, and we employ the
        bounds from Corollary \ref{cor:bounds} (alongside our conditional bound on
        the $\max |Z_t|$) to control the difference $D_t$ and derive the desired acceptance/rejection guarantees. On top of that, we use that via the
        inductive hypothesis, $\langle S_t, \theta \rangle \geq 2$,
        and therefore any transition $S'$ will have $\langle S', \theta \rangle
        \geq 0$ as no proposal from the algorithm can decrease $\langle S_t, \theta \rangle$ by more than two. For this reason, we will also use
        Corollary \ref{cor:bounds} for time $t$ to bound $\bSignal$. 
        \begin{itemize}[leftmargin=0pt]
            \item[(a)] \noindent First, we consider transitions which satisfy condition (1) in the statement of Lemma \ref{lem:success2}. For all we prove the algorithm is going to reject them with high probability.\\ We consider three sub-cases (i), (ii) and (iii): 
            \begin{itemize}
            \item [(i)] If $(\theta_{p_t},(S_t)_{p_t}, q_t) = (0,0,1)$ or $(0,0,-1)$, then 
            \[(q_t - (S_t)_{p_t})\theta_{p_t} = 0, \quad \langle
            S_t, \theta \rangle = \langle S', \theta \rangle,\quad
            |q_t| - |(S_t)_{p_t}| = 1, \quad |q_t - (S_t)_{p_t}| = 1.\] Therefore, using 
            \eqref{eq:NSignalU}, \eqref{eq:RegL}, and \eqref{eq:NoiseUL},
            \begin{align}
                \bSignal &\leq C_{\rm
            Signal}\frac{\lambda}{k^{r/2}}(\langle S_t, \theta \rangle -
            2)^{r-1}(q_t - (S_t)_{p_t})\theta_{p_t} = 0,\\
            \bReg &\geq C_{\rm
            Reg}\gamma\|S_t\|_0^{(r-1)/2},\\
            \text{and, }\quad\bNoise &\leq C^*_n C_{\rm
            Noise}\|S_t\|_0^{(r-1)/2}.
            \end{align}
            Thus, $D_t$ is bounded above by, 
            \[D_t \leq C^*_nC_{\rm Noise}\|S_t\|_0^{(r-1)/2} - C_{\rm
            Reg}\gamma\|S_t\|_0^{(r-1)/2} < 0, \text{ { as } }\gamma >
            C^*_nC_{\rm Noise}/C_{\rm Reg}.\] 
            \item[(ii)] If $(\theta_{p_t}, (S_t)_{p_t}, q_t) = (-1,-1,1)$ or
            $(1,1,-1)$, then 
            \[(q_t - (S_t)_{p_t})\theta_{p_t} = -2,\quad \langle S_t,
            \theta \rangle - 2 = \langle S', \theta \rangle, \quad |q_t| -
            |(S_t)_{p_t}| = 0, \quad |q_t - (S_t)_{p_t}| = 2.\] Therefore, using \eqref{eq:NSignalU} and
            \eqref{eq:NoiseUL},
            \begin{align}
            \bSignal &\leq
            -2C_{\rm Signal}\frac{\lambda}{k^{r/2}}(\langle S_t, \theta
            \rangle-2)^{r-1}\\
            \text{and, }\quad\bNoise &\leq 2C^*_nC_{\rm
            Noise}\|S_t\|_0^{(r-1)/2}.
            \end{align} Thus, $D_t$ is bounded above by,
            \[\begin{split}D_t &\leq 2\left(-C_{\rm Signal}\frac{\lambda}{k^{r/2}}(\langle
            S_t, \theta \rangle - 2)^{r-1} + C^*_nC_{\rm
            Noise}\|S_t\|_0^{(r-1)/2}\right) < 0,\\ 
            &\text{{ as } }\frac{\lambda}{k^{r/2}}\left(\langle S_t,
            \theta \rangle - 2\right)^{r-1} > C^*_nC_{\rm
            Noise}\|S_t\|_0^{(r-1)/2} / C_{\rm Signal}.\end{split}\] 
            \item [(iii)] If
            $(\theta_{p_t}, (S_t)_{p_t}, q_t) = (1,1,0)$ or $(-1,-1,0)$,
            then,
            \[(q_t -(S_t)_{p_t})\theta_{p_t} = -1, \qquad \langle S_t, \theta\rangle - 1 =
            \langle S', \theta \rangle, \qquad |q_t| -
            |(S_t)_{p_t}| = 1, \qquad |q_t - (S_t)_{p_t}| = 1.\] Therefore, using
            \eqref{eq:NSignalU}, \eqref{eq:RegU}, and \eqref{eq:NoiseUL},
            \begin{align}
            \bSignal &\leq
            -C_{\rm Signal}\frac{\lambda}{k^{r/2}}(\langle S_t, \theta
            \rangle-2)^{r-1},\\
            \bReg &\leq C_{\rm Reg}\gamma(\|S_t\|_0 +
            1)^{(r-1)/2},\\
            \text{and, }\quad\bNoise &\leq C^*_n C_{\rm Noise}\|S_t\|_0^{(r-1)/2}.
            \end{align} 
            Thus, $D_t$ is bounded above by,
            \[\begin{split}D_t &\leq -C_{\rm
            Signal}\frac{\lambda}{k^{r/2}}(\langle S_t, \theta \rangle -
            1)^{r-1} + C^*_nC_{\rm Noise}\|S_t\|_0^{(r-1)/2} + \gamma C_{\rm
            Reg}(\|S_t\|_0 + 1)^{(r-1)/2} < 0,\\
            &\text{ { as } } \frac{\lambda}{k^{r/2}}(\langle S_t, \theta
            \rangle - 1)^{r-1} > (C^*_nC_{\rm Noise}\|S_t\|_0^{(r-1)/2} +
            C_{\rm Reg}\gamma(\|S_t\|_0 + 1)^{(r-1)/2})/C_{\rm
            Signal}.\end{split}\]
            \end{itemize}
            \item[(b)] Second, we consider  transitions which satisfy condition (2) in the statement of 
            Lemma \ref{lem:success2}. For all we prove the algorithm is going to accept them with high probability. \\
            \noindent Again, we consider three sub-cases (i), (ii) and (iii): 
            \begin{itemize}
                \item [(i)] If $(\theta_{p_t}, (S_t)_{p_t}, q_t) = (0,1,0)$ or $ (0,-1,0)$, then 
            \[\hspace{-1.5cm}(q_t - (S_t)_{p_t})\theta_{p_t} =0,\qquad \langle S',
            \theta \rangle = \langle S_t, \theta \rangle, \qquad |q_t| -
            |(S_t)_{p_t}| = -1,\qquad  |q_t -(S_t)_{p_t}| = 1.\] Therefore, using \eqref{eq:SignalL},
            \eqref{eq:RegU}, and \eqref{eq:NoiseUL},
            \begin{align}
            \bSignal & \geq C_{\rm
            Signal}\frac{\lambda}{k^{r/2}}(\langle S_t, \theta \rangle)^{r-1}(q_t - (S_t)_{p_t})\theta_{p_t} =
            0,\\ 
            \bReg &\leq -C_{\rm
            Reg}\gamma\|S_t\|_0^{(r-1)/2},\\
            \text{and, }\quad\bNoise &\geq -C^*_nC_{\rm Noise}\|S_t\|_0^{(r-1)/2}.
            \end{align} 
            Thus, $D_t$ is bounded below by,
            \[D_t \geq -C^*_nC_{\rm Noise}\|S_t\|_0^{(r-1)/2} + C_{\rm
            Reg}\gamma\|S_t\|_0^{(r-1)/2} > 0, \text{ { as } }\gamma >
            C^*_nC_{\rm Noise}/C_{\rm Reg}.\]
            \item[(ii)] If $(\theta_{p_t}, (S_t)_{p_t}, q_t) = (-1,1,-1)$ or
            $(1,-1,1)$, we have that \[(q_t - (S_t)_{p_t})\theta_{p_t} =2,
            \quad \langle S_t, \theta \rangle + 2 = \langle S',  \theta
            \rangle,\quad |q_t| - |(S_t)_{p_t}| = 0, \quad |q_t -(S_t)_{p_t}| =2 .\] Therefore, using
            \eqref{eq:SignalL} and \eqref{eq:NoiseUL},
            \begin{align}
            \bSignal &\geq 2C_{\rm
            Signal}\frac{\lambda}{k^{r/2}}(\langle S_t, \theta \rangle -
            2)^{r-1},  \\
            \text{and, }\quad
            \bNoise &\geq -2C^*_nC_{\rm
            Noise}\|S_t\|_0^{(r-1)/2}.
            \end{align} 
            Thus, $D_t$ is bounded below by,
            \begin{align}
             D_t &\geq 2\left(\frac{\lambda}{k^{r/2}}(\langle S_t, \theta
            \rangle-2)^{r-1} - C^*_nC_{\rm
            Noise}\|S_t\|_0^{(r-1)/2}\right) > 0,\\ &\text{ { as }
            } \frac{\lambda}{k^{r/2}}(\langle S_t, \theta \rangle-2)^{r-1} >
            C^*_nC_{\rm
            Noise}\|S_t\|_0^{(r-1)/2}/C_{\rm Signal}.    
            \end{align}
            \item[(iii)] If $(\theta_{p_t}, (S_t)_{p_t}, q_t) = (1,0,1)$ or
            $(-1,0,-1)$ we have that \[(q_t - (S_t)_{p_t})\theta_{p_t} =1,
            \quad \langle S_t, \theta \rangle +1= \langle S', \theta
            \rangle, \quad |q_t| - |(S_t)_{p_t}| = 1, \quad |q_t -(S_t)_{p_t}| = 1.\] Therefore, using
            \eqref{eq:SignalL}, \eqref{eq:RegU}, and \eqref{eq:NoiseUL},
            \begin{align}
            \bSignal &\geq C_{\rm
            Signal}\frac{\lambda}{k^{r/2}}\langle S_t, \theta \rangle^{r-1},\\
            \bNoise &\geq -C^*_nC_{\rm
            Noise}\|S_t\|_0^{(r-1)/2}, \\
            \text{and, }\quad
            \bReg &\leq -C_{\rm
            Reg}(\gamma\|S_t\|_0 + 1)^{(r-1)/2}.
            \end{align} Thus, $D_t$ is bounded below by,
            \begin{align}
            D_t &\geq C_{\rm Aignal}\frac{\lambda}{k^{r/2}}\langle S_t, \theta
            \rangle^{r-1} - C^*_nC_{\rm
            Noise}\|S_t\|_0^{(r-1)/2} +
            C_{\rm Reg}\gamma(\|S_t\|_0 + 1)^{(r-1)/2} > 0,\\
            &\text{ { as }
            }\frac{\lambda}{k^{r/2}}\langle S_t, \theta \rangle^{r-1} >
            (C^*_nC_{\rm
            Noise}\|S_t\|_0^{(r-1)/2} + C_{\rm Reg}\gamma(\|S_t\|_0+1)^{(r-1)/2})/C_{\rm Signal}. 
            \end{align}
            \end{itemize}
\end{itemize}

Under the conditioned event $\max_{t \in [nM]} |Z_t| \leq C^*_n$, statement (1) from Lemma \ref{lem:success2} (at time $t$) involves proposals where $\theta_{p_t} = (S_t)_{p_t}$ and $q_t \ne (S_t)_{p_t}$. This is the exact set of transitions analyzed in case (a) where we proved any proposal in this set will have $D_t < 0$. Moreover, statement (2) involves proposals where $\theta_{p_t} \ne (S_t)_{p_t}$ and $q_t = \theta_{p_t}$. Again, this is the exact set of transition analyzed in case (b) where we proved any proposal in this set will have $D_t > 0$. As the conditioned event has probability $1-o(1)$ uniformly over $t \in [T^*_M]$, we can remove the conditioning at time $t$ to prove statements (1) and (2) hold with high probability uniformly over $t\in [T^*_M] \subseteq [nM]$. This partially confirms the inductive claim. It remains to show that $\langle S_{t+1}, \theta \rangle \geq 2$ and the inequality \eqref{eq:ind_goal} for $t+1$ for the case (b) that leads to acceptance. We shortly complete this step below, with a unified argument across steps, after we prove which cases are accepted or rejected in Step 3 as well. 

\textbf{Step 3.}

Although we have proven the inductive claim on statements (1) and (2) at time $t$ in Lemma \ref{lem:success2} with respect to cases (a) and (b), there remains four possible transitions left to analyze to complete the inductive claim for all possible transitions. Each are given in Table \ref{tb:tbtwo}. 

\begin{table}[t]
    \centering
    \begin{tabular}{|cccc|cccc|}
\hline
& \multicolumn{2}{c}{$\theta_{p_t} = -1$} & & &\multicolumn{2}{c}{$\theta_{p_t} = 1$} & \\
\hline
$(S_t)_{p_t}$ & $q_t$ & Case & A/R 
  & $(S_t)_{p_t}$ & $q_t$ & Case & A/R \\
\hline
  0 & 1 & (c) & R 
  & 0 & -1 & (c) & R \\

   1 & 0 & (d) & A 
  & -1 & 0 & (d) & A \\
\hline
\end{tabular}
\vspace{1em}
    \caption{Transitions where $\langle S' , \theta \rangle \ne \langle S_t, \theta \rangle$ or $\|S'\|_0 \ne \|S_t\|_0$ but have not been~analyzed by cases (a) or (b). These transitions are required to finish proving the inductive claim of the inductive argument, but are not directly relevant to the conclusion of Lemma \ref{lem:success2}. The purpose of each column is identical to Table \ref{tb:Tern}.}
\label{tb:tbtwo}    
\end{table}
\begin{itemize}     
    \item [(c)] If $(\theta_{p_t}, (S_t)_{p_t}, q_t) = (-1,0,1)$ or $ (1,0,-1)$, then
            \[(q_t - (S_t)_{p_t})\theta_{p_t} = -1,\quad \langle S_t,
            \theta \rangle - 2 = \langle S', \theta \rangle, \quad |q_t| -
            |(S_t)_{p_t}| = 1,\quad |q_t -(S_t)_{p_t}| = 1.\]
            Therefore, using \eqref{eq:NSignalU}, \eqref{eq:NoiseUL} \eqref{eq:RegL},
             \begin{align}
            \bSignal &\leq -C_{\rm
            Signal}\frac{\lambda}{k^{r/2}}\left(\langle S_t, \theta \rangle-2\right)^{r-1},\\
            \bNoise &\leq C^*_nC_{\rm
            Noise}\|S_t\|_0^{(r-1)/2}, \\
            \text{and, }\quad
            \bReg &\leq -C_{\rm
            Reg}\gamma \|S_t\|_0^{(r-1)/2}.
            \end{align}
            Thus, $D_t$ is bounded above by, 
            \begin{align}D_t &\leq -C_{\rm
            Signal}\frac{\lambda}{k^{r/2}}\left(\langle S_t, \theta \rangle-2\right)^{r-1}+C^*_nC_{\rm Noise}\|S_t\|_0^{(r-1)/2} - C_{\rm
            Reg}\gamma\|S_t\|_0^{(r-1)/2} < 0,\\ &\text{ { as } }\gamma >
            C^*_nC_{\rm Noise}/C_{\rm Reg}.\end{align}
        \item [(d)] If $(\theta_{p_t}, (S_t)_{p_t}, q_t) = (1,-1,0)$ or $ (-1,1,0)$, then 
            \[(q_t - (S_t)_{p_t})\theta_{p_t} = 1,\quad \hspace{-.1cm} \langle S_t,
            \theta \rangle + 2 = \langle S', \theta \rangle, \quad |q_t| -
            |(S_t)_{p_t}| = -1,\quad |q_t -(S_t)_{p_t}| = 1.\]
           Therefore, using \eqref{eq:NSignalU}, \eqref{eq:NoiseUL} \eqref{eq:RegL},
             \begin{align}
            \bSignal &\geq C_{\rm
            Signal}\frac{\lambda}{k^{r/2}}\left(\langle S_t, \theta \rangle+2\right)^{r-1},\\
            \bNoise &\geq -C^*_nC_{\rm
            Noise}\|S_t\|_0^{(r-1)/2}, \\
            \text{and, }\quad
            \bReg &\geq C_{\rm
            Reg}\gamma \|S_t\|_0^{(r-1)/2}.
            \end{align}
            Thus, $D_t$ is bounded above by,
            \begin{align}D_t &\geq C_{\rm
            Signal}\frac{\lambda}{k^{r/2}}\left(\langle S_t, \theta \rangle+2\right)^{r-1}-C^*_nC_{\rm Noise}\|S_t\|_0^{(r-1)/2} + C_{\rm
            Reg}\gamma\|S_t\|_0^{(r-1)/2} > 0\\
            &\text{ { as }
            }C_{\rm Signal}\frac{\lambda}{k^{r/2}}\left(\langle S_t, \theta \rangle+2\right)^{r-1}+ C_{\rm Reg}\gamma(\|S_t\|_0)^{(r-1)/2} >
            C^*_nC_{\rm
            Noise}\|S_t\|_0^{(r-1)/2} .\end{align}
    \end{itemize}
    This completes the analysis of all possible transitions.

   \textbf{Step 4}. We now prove the inductive statement, that is, we have shown that statements (1) and (2) hold for time $t$ and the inequalities \eqref{eq:overlapatleasttwo} and \eqref{eq:ind_goal} hold for time $t+1$. Notice first that for the second part of the induction in Step 2 we proved in (a) that for any  $p_{t}, q_t$ such that $\theta_{p_t} = (S_t)_{p_t}$ and 
        $q_t \neq \theta_{p_t}$, $D_t < 0$, leading to rejection.
    Additionally, in (b) we proved that for and $p_{t}, q_t$ such that $\theta_{p_t} \ne (S_t)_{p_t}$ and 
        $q_t = \theta_{p_t}$,  $D_t > 0$, leading to acceptance. Combining these two we complete the proof of the inductive claim for the statements in (1) and (2).  We now can complete the induction by proving that inequalities \eqref{eq:overlapatleasttwo} and \eqref{eq:ind_goal} holds for $t+1$ by considering the set of transitions that are accepted (i.e., the cases (b) and
    (d)) and proving they lead to an $S_{t+1}$ that satisfies \eqref{eq:overlapatleasttwo} and \eqref{eq:ind_goal}.

    First, it always holds $\langle S_{t+1}, \theta \rangle \geq \langle S_{t}, \theta \rangle
    \geq 2$ as in both (b), (d) cases it holds $\langle S_{t+1}, \theta \rangle \geq \langle S_{t}, \theta \rangle,$.
    
    In terms of \eqref{eq:ind_goal},  the cases (b)(i), (b)(ii) and (d) this is trivial since
    \[
    \left(\frac{\langle 
    S_{t}, \theta \rangle - 1}{k^{1/2}(\|S_{t}\|_0 + 2)^{1/2}}
    \right)^{r-1}
    \]
    is an increasing function of $\langle 
    S_{t}, \theta \rangle$ and a decreasing function of $\|S_{t}\|_0$.
    For case (b)(iii) where $S'=S_{t+1}$ and it holds $\langle S', \theta\rangle=\langle S_t,
    \theta\rangle+1$ and $\|S'\|_0=\|S_t\|_0+1$, a sufficient condition to prove the inductive claim for \eqref{eq:ind_goal} is,
        \begin{align}
        \left(\frac{\langle 
    S_{t}, \theta \rangle - 1}{k^{1/2}(\|S_{t}\|_0 + 2)^{1/2}}
    \right)^{r-1} &\geq \left(\frac{\langle S_{t},
        \theta \rangle - 2}{k^{1/2}(\|S_t\|_0 + 1)^{1/2}}\right)^{r-1}. \\
    \intertext{Or, equivalently,}
        \frac{\langle 
    S_{t}, \theta \rangle - 1}{(\|S_{t}\|_0 + 2)^{1/2}}
     &\geq \frac{\langle S_{t},
        \theta \rangle - 2}{(\|S_t\|_0 + 1)^{1/2}},\\
        \intertext{which we can rewrite as,}
        1+\frac{1}{\langle 
    S_{t}, \theta \rangle - 1}
     &\geq\left( 1+\frac{1}{\|S_t\|_0}\right)^{1/2}.
        \end{align}And thus, a sufficient condition for the above displayed inequality to hold is,
        \[
     \|S_t\|_0 \geq \langle 
    S_{t}, \theta \rangle - 1
        \]
        which is trivial.
        
    Thus, the inductive argument is complete and the statement of the lemma holds for all $t \in [T^*_M]$.
       \end{proof} 

\subsubsection{Proof of Lemma \ref{lem:Init}}

\begin{proof}[Proof of Lemma \ref{lem:Init}]
    
Recall from Definition
\ref{def:HOM}, when $r$ is odd, the homotopy initialization $S_{\rm HOM}$ is
defined as  
    \[(S_{\rm HOM})_i = 2\cdot \1\left\{\sum_{j_1, \dots, j_{(r-1)/2} = 1}^n
    Y_{i, j_1, j_1, \dots, j_{(r-1)/2}, j_{(r-1)/2}} \geq 0\right\} - 1.\]
    Obviously, we see that $(S_{\rm HOM})_i \in \{-1,1\}$ and therefore
    $\|S_{\rm HOM}\|_0 = n$. Using the definition of $Y$ from \eqref{eq:model},
    we have for each $i \in [n]$ that
\[\sum_{j_1, \dots, j_{(r-1)/2} = 1}^n Y_{i, j_1, j_1, \dots, j_{(r-1)/2},
j_{(r-1)/2}} \laweq \theta_i\frac{\lambda}{k^{r/2}}k^{(r-1)/2} + n^{(r-1)/4}
Z,\] where $Z \sim \cN(0,1)$. Thus, as the elements of $Y$ are independent, we
have that $(S_{\rm HOM})_i = \theta_i$, for $i \in [n]$ where $\theta_i \ne
0$, with probability \[p := \PP\left(\frac{\lambda}{\sqrt{k}n^{(r-1)/4}} + Z >
0\right).\] Moreover, by this independence, we write
\[\langle S_{\rm HOM}, \theta \rangle \laweq \sum_{i \in [k]}R_i,\] where
each $R_i$ is independent and has the law, 
\[R_i \sim \begin{cases} 1 &\qquad \text{with probability }p\\ -1 &\qquad
\text{with probability }1 - p \end{cases}.\] Clearly this summation is
stochastically increasing as $p$ increases (and thus as $\lambda$ increases).
Thus, we
stochastically lower bound $\langle S_{\rm HOM}, \theta \rangle$ by choosing
considering a smaller $\mathring \lambda = n^{r/4}$ where $\lambda \geq
\mathring \lambda$ by Assumption \ref{as:tau} for large enough $C_\tau$.

Substituting in $\mathring \lambda$, we have $\mathring p \leq p$ where
\[\mathring p = \PP\left(\frac{\mathring \lambda}{\sqrt{k}n^{(r-1)/4}} + Z > 0\right) =
\PP\left(\frac{n^{1/4}}{k^{1/2}} + Z > 0\right) \geq \frac{1}{2} +
\frac{3}{10}\frac{n^{1/4}}{k^{1/2}}.\]
The final inequality above is due to $\Phi(x)$ being concave on $x \in [0, +\infty)$,
therefore for any $x \leq C$, we have that $\Phi(x) \geq 1/2 + (\Phi(C) - 1/2)x/C$,
choosing $C = 1$ and recognizing that $\Phi(1) \geq 8/10$ gives the inequality.

Therefore, each $R_i$ is stochastically
lower bounded by a random variable $R'_i$ where we have replaced $p$ with $1/2 +
\frac{3}{10}\frac{n^{1/4}}{\sqrt{k}}$ in the definition of $R_i$. The
summation $\sum_{i = 1}^k R'_i$ has expectation
$\frac{3}{5}\sqrt{k}n^{1/4}$ (as there are $k$ terms in the
sum). Then, using that each $R_i \in [-1,1]$ are independent, Hoeffding's
inequality then gives that,
\begin{align}
    \PP\left(\sum_{i = 1}^k R'_i \leq \sqrt{k} n^{1/4} / 5\right) 
    &= \PP\left(\sum_{i=1}^k R'_i - \frac{3}{5}\frac{\lambda\sqrt{k}}{n^{(r-1)/4}} 
    \leq -\frac{2}{5}\sqrt{k} n^{1/4}\right)\\
&\leq e^{-c\left(\frac{2}{5}\sqrt{k} n^{1/4}\right)^2 / k}
\end{align} 
for some $c > 0$. The above is $1-o(1)$ for large enough $n$. Using
the previously justified $\sum_{i = 1}^k R_i \stlow \sum_{i = 1}^k R_i \laweq
\langle S_{\rm HOM}, \theta \rangle$, then the conclusion of the statement follows.
\end{proof}

\subsubsection{Proofs of Lemma \ref{lem:boost03} and Lemma
\ref{lem:boost01}}

In order to prove Lemma \ref{lem:boost03} and Lemma \ref{lem:boost01}, we
required the following auxiliary lemma, whose proof is deferred to Section
\ref{sec:proofAlgBTrans}.
\begin{lemma}\label{lem:AlgBTrans} Recall $\runtime_1 = \lceil \log^4(n)
\rceil$, sequence of proposals $(p_t)_{t \in [n \runtime_1]}$, and random runtime
$\finalIter((p_t)_{t \in [n \runtime_1]}) = \finalIter$ from Algorithm
\ref{alg:B}.

Then, almost surely over $\finalIter$, with $D'_t$ from \eqref{eq:DAlign}, it holds that 
    \[(D'_t)_{t \in [\finalIter]} \laweq \left(\frac{f_{\theta, \lambda}(S_t, p_t, - (S_t){p_t})}{V_n\sqrt{\runtime_1}} + Z_t\right)_{t \in [\finalIter]}\]
where each $Z_t$ is an independent standard normal random
variables independent of $(p_t)_{t \in \finalIter}$ and thus $\finalIter$. As
a consequence, the event that $\{D'_t > 0\}$ for $t \in [\finalIter]$ is
equivalent to the event that 
\[ \left\{\fConst + Z_t > 0\right\}.\] 
    \end{lemma}

\textbf{A proof sketch for Lemmas \ref{lem:boost03} and \ref{lem:boost01}} Due to the technical complexity of our proofs to Lemmas \ref{lem:boost03} and Lemma \ref{lem:boost01}, we provide a high level sketch of the arguments for the reader's convenience.

\begin{enumerate}
    \item \textbf{Calculating the probabilities of each one-step change of the ``overlap" \boldmath$\langle S_t, \theta \rangle $\unboldmath}. Recalling $D'_t$ is the difference of the Hamiltonian plus injected noise in Algorithm \ref{alg:B}, we use Lemma \ref{lem:AlgBTrans} to find a convenient equivalence in law for the event $D'_t > 0$ with respect to the random variables $(p_t)_{t \in [T^*]}$ and auxiliary Gaussian random variables $(Z_t)_{t \in [T^*]}$.
     This permits an exact calculation of the probability that $\langle S_t, \theta \rangle$ will either increase or decrease by two for each time step $t$.
    \item \textbf{Modeling the overlap as a stochastic process \boldmath$A_t$\unboldmath.} Due to the exact form of these probabilities, we can model the dynamics of $\langle S_t, \theta\rangle$ in Algorithm \ref{alg:B} by a stochastic process $A_t$. The process has position-dependent transition probabilities (i.e., dependent on the current value of $\langle S_t, \theta \rangle$) and hence requires a less standard analysis. Interestingly, we prove that it's evolution undergoes two phases: an oscillatory phase until it reaches a specific level of overlap, and a second phase happening above this critical level where it switches to a monotonic evolution. 
    \item \textbf{Deriving the total runtime to reach a high enough \boldmath$\langle S_t, \theta \rangle$\unboldmath.} Due to the above inhomogenity of $\langle S_t, \theta\rangle$ (which does not allow us to invoke standard random walk results), we calculate the needed runtime for the stochastic process to either double or half the $\langle S_t, \theta \rangle$ for each region up to the threshold where $\langle S_t, \theta \rangle$ satisfies the inequality (\hyperlink{eq:threshold}{8.4}), a bound dependent on the maximal value of any $|Z_t|$. This is the ``oscillatory" phase. For each region, the stochastic process behaves similarly to a biased random walk (where the bias changes per region), which allows us to use standard results on each region.
    
    \item \textbf{Confirming \boldmath$\langle S_t, \theta\rangle$\unboldmath\; is non-decreasing above a specific threshold (\hyperlink{eq:threshold}{8.4}).} Once $\langle S_t, \theta \rangle$ passes a specific threshold, the dynamics switch to a deterministic increase. To prove this, we utilize techniques similar to proving Lemma \ref{lem:success2} to show that beyond that threshold, with high probability over the sequence $(Z_t)_{t \in [T^*]}$, we always reject transitions that decreases the overlap and always accept transitions that increase the overlap. Note (4) is only used in the proof of Lemma \ref{lem:boost03}.

\end{enumerate}

\begin{proof}[Proof of Lemma \ref{lem:boost03}]
We translate the statement to be proved into a statement about an time dependent
stochastic process on the integers which is later stochastically lower bounded
by a family of random walks.

\textbf{Step 1. Calculating the probability of changes for \boldmath$\langle S_t, \theta \rangle $\unboldmath.}

Let $s = \langle S_{t}, \theta \rangle$; note, this overlap may either
increase/decrease by two or remain the same. We then calculate the probability
that $\langle S_{t+1}, \theta \rangle = s + 2$ or $s-2$. For example, $s$ will
increase by two if, given the sampled $p_t$, we have $(S_t)_{p_t} \ne \theta_{p_t},
\theta_{p_t} \ne 0$, fixed $q_t = - (S_t)_{p_t}$ and $D'_t > 0$ (meaning we accept the transition).
The
probability of sampling such a $p_t$ given $s$ is calculated as $\left(
\frac{k}{2n} - \frac{s}{2n}\right)$. Moreover, using Lemma \ref{lem:AlgBTrans}
and recalling $V_n$ from Corollary \ref{cor:Vn}, we have almost surely over
$\finalIter$, that $\{D'_t > 0\}$ is equivalent to the event that
\[\left\{\fConst + Z_t > 0\right\},\] where each $Z_t$ is an independent
standard normal random variable independent of both $(p_t)_{t \in [n
\runtime_1]}$ and $\finalIter$. Such an event has probability
$\Phi(\fConstSmall s^{r-1})$ where $\Phi$ is the Gaussian CDF. As the sample
$p_t$ is independent of $Z_t$, the product of these two events gives that
$\langle S_{t+1}, \theta \rangle = s + 2$ has probability
\[\left(\frac{k}{2n} - \frac{s}{2n}\right)\Phi\left(\fConst\right),\] for any
$S_{t}$ where, $\langle S_{t}, \theta \rangle = s$ and $p_t$ such that
$\theta_{p_t} \ne 0$, $(S_t)_{p_t} \ne \theta_{p_t}$ and fixed $q_t = -(S_t)_{p_t}$. Using
\[f_{\theta,\lambda}(S_{t}, p_t, -(S_t)_{p_t}) =
\frac{\lambda}{k^{r/2}}\left(\langle (S_{t} - 2(S_t)_{p_t}e_{p_t})^{\otimes
r},\,\theta^{\otimes r}\rangle - \langle (S_{t})^{\otimes
r},\,\theta^{\otimes r}\rangle\right),\] we can decompose this value further with 
\begin{align}
    \langle (S_{t} - 2(S_t)_{p_t}e_{p_t})^{\otimes r},\,\theta^{\otimes r}\rangle - \langle (S_{t})^{\otimes r},\,\theta^{\otimes r}\rangle
    &= \langle S_{t} - 2(S_t)_{p_t}e_{p_t}, \theta \rangle^r - \langle S_{t}, \theta \rangle^r\\
    &=\Bigl(\langle S_{t},\,\theta \rangle \;-\;2\,(S_t)_{p_t}\,\langle e_{p_t},\,\theta \rangle\Bigr)^r
    \;-\;\langle S_{t},\,\theta \rangle^r\\
    &=\Bigl(\langle S_{t},\,\theta \rangle - 2\,(S_t)_{p_t}\,\theta_{p_t}\Bigr)^r
    - \langle S_{t},\,\theta \rangle^r\\
    &=\sum_{k=1}^r \binom{r}{k}\,\langle S_{t},\,\theta \rangle^{\,r-k}\,
    \bigl(-2\,(S_t)_{p_t}\,\theta_{p_t}\bigr)^k .
\end{align}
Under our previously mentioned constraints, we then conclude that $\langle S_{t},
\theta \rangle = s + 2$ has probability     
\[\left(\frac{k}{2n} - \frac{s}{2n}\right)\Phi\left(\fPrime{s}{-1}\right),\] where
$\tilde f_{\theta,\lambda}(s, s') := \frac{\lambda}{k^{r/2}}\sum_{k=1}^r \binom{r}{k}\,s^{\,r-k}\,(-2s')^k$.
Symmetry implies that
$\langle S_{t+1}, \theta \rangle = s - 2$ has probability
\[\left( \frac{k}{2n} + \frac{s}{2n}\right)\Phi\left(\fPrime{s}{1}\right).\]
Note, we still have the following bound, due to Corollary \ref{cor:bounds}(d), when $s
\geq 2$, 
\[-2C_{\rm Signal} \frac{\lambda}{k^{r/2}}s^{r-1}\leq \tilde f_{\theta,\lambda}(s, 1) \leq -2C_{\rm Signal} \frac{\lambda}{k^{r/2}}(s-2)^{r-1},\label{eq:boundfpone}\]
and, 
\[2C_{\rm Signal} \frac{\lambda}{k^{r/2}}s^{r-1}\leq \tilde f_{\theta,\lambda}(s, -1) \leq 2C_{\rm Signal} \frac{\lambda}{k^{r/2}}(s+2)^{r-1}.\label{eq:boundfnone}\]

\textbf{Step 2. Modeling the overlap as a stochastic process \boldmath$A_t$\unboldmath.}

Further, Lemma \ref{lem:AlgBTrans} and the above two probabilities gives that
$\langle S_{t}, \theta \rangle$ (where $(S_t)t$ are the iterates of Algorithm \ref{alg:B}) is equivalent in law
to the following stochastic process $A_t$ almost surely over $\finalIter$:
\begin{definition}[Stochastic Process $A_t$ tracking the dynamics of $\langle S_{t}, \theta \rangle$ in Algorithm \ref{alg:B}]\label{def:walk1}
Let $A_1 = \langle S_{1}, \theta \rangle$, then for $t \geq 1$, we randomly
select one of the following actions:
\begin{enumerate}
    \item With probability $\frac{k}{2n} + \frac{A_t}{2n}$ we do the following:
    With probability $\Phi\left(\fPrime{A_t}{1}\right)$ we set $A_{t+1} = A_t
    - 2$, otherwise set $A_{t+1} = A_t$.
    \item With probability $\frac{k}{2n} - \frac{A_t}{2n}$ we do the following:
    With probability $\Phi\left(\fPrime{A_t}{-1}\right)$ we set $A_{t+1} = A_t +
    2$, otherwise set $A_{t+1} = A_t$.
    \item With the remaining probability, set $A_{t+1} = A_t$.
\end{enumerate}
\end{definition}

As each transition in the above process can change $A_t$ to either $A_t - 2$, $A_t + 2$ or remain the same, therefore special care is taken with respect to the parity
of $A_1$. Indeed, the parity of the initialization dictates the
support of $A_t$. In what follows, denote $\parity = \parity(x)$, for $x \in
\NN$, as the function,
\[\parity(x) = \begin{cases} 1 & x \text{ is odd}\\
    0 & x \text{ is even} \end{cases}.\] If no argument is given, we assume
$\parity = \parity(A_1)$.

Recall the statement of Lemma \ref{lem:boost03}, henceforth denoted as event
$\cB$,

``If $S_1$ has $\langle S_1, \theta \rangle \geq n^{1/4}k^{1/2}/5$,
there exists time step $t^* \leq \lfloor
\finalIter / 2 \rfloor$ in Algorithm \ref{alg:B} (with input $(Y, S_1,
\gamma, \runtime_2)$) such that \[\frac{2C_{\rm
Signal}\frac{\lambda}{k^{r/2}}}{V_n\sqrt{\runtime_2}}(\langle S_{t^*},
\theta \rangle - 2)^{r-1} \geq 2\sqrt{\log(\runtime_2 n)}\;\text{''}\label{eq:B}\]

Let $s^* = s^*(n,k,\lambda)$ be the minimal positive value of $\langle S_{t}, \theta
\rangle$ for which \[\fuConst \geq 2\sqrt{\log(\runtime_1 n)}\] holds
and $\parity(s^*) = \parity$. Now, define $T = \inf\{t \in \NN: A_t \in \{s^*,
\parity\}\}$ and the event $\cA$ where $T \leq \lfloor \finalIter / 2
\rfloor$ and $A_T = s^*$. If $\cA$ holds, then so does $\cB$; taking
probabilities, $\PP(\cA) \leq \PP(\cB)$.
Further, the event $\cA$ is of equal probability if the states $\parity \in
\{0,1\}$ and $s^*$ are absorbing as event $\cA$ is only with respect to the stochastic process $A_t$ up to the time $T$ where $A_T$ is in one of these absorbing states. Hence, we restrict the support of the stochastic process $A_t$ to the set $[\parity, s^*] \cap \ZZ$. This allows us to neglect the
natural support of the process $A_t$ in $\pm [k] \cup 0$ and instead
consider the support to be this non-negative set. Thus, the proof of Lemma
\ref{lem:boost03} is reduced to proving $\PP(\cA) = 1-o(1)$.

Next, cover the set $\{i \in [s^*]: i \geq 2\lceil \langle S_1, \theta
\rangle /4 \rceil + \parity\}$ by $m$ sets\footnote{We see momentarily that $m =
O(\log(n))$} $\cI_1, \dots, \cI_m$ defined by
\[a_i = 2\lceil3a_{i-1} / 4\rceil \text{ for }(i-1) \in [m-1], a_1 = \langle
S_1, \theta \rangle \label{eq:ai}\] where
\[\cI_j = \{i \in [s^*]: 2\lceil a_j / 4 \rceil + \parity \leq i \leq 2\lceil
3a_j/4 \rceil + \parity\}.\] Note, for every element $s \in \cI_j$ which is
non-empty, we have that $s \geq n^{1/4}k^{1/2}/10 \geq 2$ for sufficiently large
$n$ (this is relevant at the end of the proof were we show that $\langle S_{t^*}, \theta \rangle \geq 2$). Therefore, we are free to invoke the bounds from \eqref{eq:boundfpone} and
\eqref{eq:boundfnone}. Moreover, every endpoint $x \in \{\partial \cI_1, \dots,
\partial \cI_m\}$ has $\parity(x) = \parity$; and thus, $A_t$ can equal $x$ with
non-zero probability. Additionally, define a sequence of deterministic times $\{T_1, \dots,
T_m\}$, where each $T_i = 10n\log^2(n)$ (as we will see later in the
proof, $T_i$ will serve as an upper bound on the time for process $A_t$ to transition from
$\partial \cI_i$ to $\partial \cI_{i+1}$). As $s^* \leq n$ (trivially), $m$
can be bounded above than the solution to $\left(3/2\right)^{m}a_1 = n$, and as $a_1
\geq 2$ for large $n$, then $m = O(\log(n))$. Notably, $\cI_i$ for larger $i$
may be the empty set; for such $i$ we set $T_i = 0$.

For each $i \in [m]$, define the event $\cE_i$, in which stochastic process
$A_t$, initialized at $A_1 = a_i$ from \eqref{eq:ai}, has $A_{T'_i} = 2\lceil
3a_i/4 \rceil + \parity$ and $T'_i \leq T_i$, where we define $T'_i = \inf\{t
\in \NN: A_t \in \partial \cI_i\}$. If $\cI_i = \varnothing$, then vacuously set
$\PP(\cE_i) = 1$. Further, let $\cE = \bigcap_{i = 1}^m \cE_i$. Notice that under $\cE(\omega)$,
the following holds. By $\cE_1$ we have that $A_t$ will hit $2\lceil 3a_1/4 \rceil +
\parity$ before $2\lceil a_1/4 \rceil + \parity$ within time $T_1$. Moreover,
event $\cE_2$ implies that once $A_t$ hits $2\lceil 3a_1/4 \rceil + \parity =
a_2$ then the process $A_t$ will then hit $2\lceil 3a_2/4 \rceil + \parity$
(notice this term has $a_2$ instead of $a_1$) before $2 \lceil a_2/4 \rceil +
\parity$ within time $T_2$. Therefore, recalling $m = O(\log(n))$, the event $\cE(\omega) \cap \{10n\log^{3.5}(n) \leq \lceil T^*/2\rceil\}$ (as $\sum_{i=1}^m T'_i \leq 10n\log^2(n)m \leq 10n\log^{3.5}(n))$ for large $n$) implies the event $\cA(\omega)$ holds. Hence, taking probabilities
\[\PP(\cE \cap \{10n\log^{3.5}(n) \leq \lceil T^*/2 \rceil\}) \leq \PP(\cA) \leq \PP(\cB).\]
Therefore, if we prove $\PP(\cE \cap \{10n\log^{3.5}(n)\}) \geq 1-o(1)$, we complete the proof. Moreover, by a union bound,
\begin{align}\PP(\cE \cap \{10n\log^{3.5}(n) \leq \lceil T^*/2 \rceil\}) &\geq 1 - \PP(\cE^c) - \PP(\{10n\log^{3.5}(n) \leq \lceil T^* /2\rceil\}^c)\\
&= \PP(\cE) - \PP(\{10n\log^{3.5}(n) > \lceil T^* /2\rceil\}).\label{eq:boundT*}\end{align}
Now, we prove that $\PP(\cE) = 1-o(1)$. We then prove that $\PP(\{10n\log^{3.5}(n) > \lceil T^* /2\rceil\}) = o(1)$, completing the proof.

The choice of covering set $\cI_1, \dots, \cI_m$ and the fact that $\cE$ implies
$\cA$ is visualized in Figure \hyperref[fig:walkSplit]{6}. 

Notice, the statement of $\PP(\cE) \geq 1-o(1)$ follows from proving $\PP(\cE^c)
= o(1)$. Conveniently, it suffices to prove for all $i=1,\ldots,m$ that  $\PP(\cE_i^c) =
O(\log^{-2}(n))$. Indeed, as $m \leq C\log(n)$ for some constant $C > 0$. We
could then conclude,
\[\PP(\cE^c) = \PP\left(\bigcup_{i = 1}^m \cE_i^c\right) \leq \sum_{i = 1}^m
\PP(\cE_i^c) \leq C\log(n) \max_{i \in [m]} \PP(\cE_i^c) = O(\log^{-1}(n)),\] as
desired.

\begin{figure}[h]\label{fig:walkSplit}
    \centering
    \resizebox{\textwidth}{!}{
    \begin{tikzpicture}
    \draw[->, thick, red, opacity = .5] (1, -.1) to[out=-15, in=-165] (13, -.1);
    \node[above=4pt, text=red] at (7, -1.7) {$\cA$};
    
  \draw[->, thick, blue, opacity = .5] (0.5, .1) to[out=30, in=150] (1.5, .1);
  \node[above=4pt, text=blue] at (1, 0.1) {$\cE_{1}$};
 
  \draw[->, thick, blue, opacity = .5] (0.75, .1) to[out=30, in=150] (2.25, .1);
  \node[above=4pt, text=blue] at (1.5, 0.2) {$\cE_{2}$};
 
  \draw[->, thick, blue, opacity = .5] (1.125, .1) to[out=30, in=150] (3.375, .1);
  \node[above=4pt, text=blue] at (2.25, 0.3) {$\cE_{3}$};
 
  \draw[->, thick, blue, opacity = .5] (1.6875, .1) to[out=30, in=150] (5.0625, .1);
  \node[above=4pt, text=blue] at (3.375, 0.45) {$\cE_{4}$};
 
  \draw[->, thick, blue, opacity = .5] (2.53125, .1) to[out=30, in=150] (7.59375, .1);
  \node[above=4pt, text=blue] at (5.0625, .7) {$\cE_{5}$};
 
  \draw[->, thick, blue, opacity = .5] (3.796875, .1) to[out=30, in=150] (11.390625, .1);
  \node[above=4pt, text=blue] at (8.7, 1.1) {$\cE_{6} \text{ and so on...}$};
    
        \draw[thick] (0,0) -- (12,0);
    
        \foreach \x/\label in {0/$0$, 1/$a_1$, 1.5/$a_2$, 2.25/$a_3$, 3.375/$a_4$,
            5.0625/$a_5$, 7.59375/$a_6$, 11.390625/$a_7$} \draw (\x,0.2) --
            (\x,-0.2) node[below] {\label};
    
        \foreach \x/\label in {1/$a_1$, 1.5/$a_2$, 2.25/$a_3$, 3.375/$a_4$,
            5.0625/$a_5$, 7.59375/$a_6$} \draw (\x/2,0.1) -- (\x/2,-0.1);
            
        \foreach \x/\label/\i in {1/$\cI_1$/1, 1.5/$\cI_2$/2, 2.25/$\cI_3$/3, 3.375/$\cI_4$/4,
            5.0625/$\cI_5$/5, 7.59375/$\cI_6$/6} 
            \draw[thick] (\x/2,1 + \i/1.6) -- (3*\x/2,1 + \i/1.6) node[midway,above]
            {{\small \;\;\;\label}};
        \foreach \x/\label/\i in {1/$\cI_1$/1, 1.5/$\cI_2$/2, 2.25/$\cI_3$/3, 3.375/$\cI_4$/4,
            5.0625/$\cI_5$/5, 7.59375/$\cI_6$/6} 
            \draw[thick] (\x/2,1 + \i/1.6 - .1) -- (\x/2,1 + \i/1.6 + .1)
            node[above=-1pt] {{\tiny $2\lceil a_{\i}/4 \rceil$}};
        \foreach \x/\label/\i in {1/$\cI_1$/1, 1.5/$\cI_2$/2, 2.25/$\cI_3$/3, 3.375/$\cI_4$/4,
            5.0625/$\cI_5$/5, 7.59375/$\cI_6$/6} 
            \draw[thick] (3*\x/2,1 + \i/1.6 - .1) -- (3*\x/2,1 + \i/1.6 + .1)
            node[below=3pt] {{\tiny $2\lceil 3a_{\i}/4 \rceil$}};
    
        
        \node[below] at (7,-1.5) {$a_{i+1} = 2\lceil 3a_i/4 \rceil,
        \quad a_1 = \langle \sigma^1, \theta \rangle \geq n^{1/4}\sqrt{k}/5$};

        \node[above] at (12.35, -.24) {$\cdots$};  
        \draw[-|] (12.6, 0) -- (13, 0);
        \node[below] at (13,0) {$s^*$};
 
    \end{tikzpicture}}
    
    \caption{A cartoon visualization of the covering sets $\{\cI_1, \dots,
    \cI_m\}$ which divide the stochastic process into easier to handle ``chunks''. We
    also see how event $\cA$, represented abstractly by the red arrow, is
    implied by the sequence of events $\cE = \cup_i \cE_i$, each with a blue
    arrow. More specifically, each blue line represents how every event $\cE_i$
    implies that the process $A^i_t$ will eventually output the absorbing state
    $A^i_{T'_i} = 2\lceil 3a_i/4 \rceil + \parity$ with $T'_i \leq T_i$ steps.
    Note for visual simplicity we assume that
    $\parity = 0$.}\label{fig:covers}
    \end{figure}

\textbf{Fix an $i \in [m]$ where $\cI_i \neq \varnothing$ for the rest of the proof.} The
probability of event $\cE_i$ can be analyzed using the stochastic process $A^i_t$
initialized with $A^i_1 = a_i$ from \eqref{eq:ai}, with two absorbing states at
$2\lceil a_i/4 \rceil + \parity$ and $2\lceil 3a_i/4 \rceil + \parity$ (i.e. the
elements of $\partial \cI_i$) outputting $A^i_{T'_i} = 2\lceil 3a_i/4 \rceil +
\parity$ for $T'_i \leq T_i$. The stochastic process $A_t^i$ is formally defined by the
following transition probabilities for $A_t^i \in \cI_i \setminus \partial
\cI_i$:

\begin{definition}[Stochastic process $A^i_t$ used to calculate $\PP(\cE_i)$]   
Let $A^i_1 =  a_i$ with absorbing barriers at $2 \lceil a_i/4 \rceil + \parity$
and $\min(2\lceil 3a_i/4 \rceil + \parity, s^*)$. Then, for $t \geq 1$, we
randomly select one of the following actions:
\begin{enumerate}
    \item With probability $\frac{k}{2n} + \frac{A^i_t}{2n}$, we do the following: With
    probability $\Phi\left(\fPrime{A^i_t}{1}\right)$, we set $A^i_{t+1} = A^i_t
    - 2$, otherwise set $A^i_{t+1} = A^i_t$.
    \item With probability $\frac{k}{2n} - \frac{A^i_t}{2n}$ we do the following: With
        probability $\Phi\left(\fPrime{A^i_t}{-1}\right)$ we set $A^i_{t+1} = A^i_t + 2$, otherwise set
        $A^i_{t+1} = A^i_t$.
    \item With the remaining probability, set $A^i_{t+1} = A^i_t$.
\end{enumerate}
\end{definition}

Recall, $A_t$ represents the dynamics of $\langle S^t, \theta \rangle$ and has no absorbing states. We have introduced $A_t^i$ to keep track of the dynamics of $A_t$ within a given interval $\cI_i$ with added absorbing states at $\partial \cI_i$. It is immediate that $A^i_t$ is equivalent in law to $A_t$ if we initialize $A_1 = a_i$ so long as $A_t \in \cI_i$.

\textbf{Step 3(a). Deriving the total runtime to reach a high enough \boldmath$\langle S_t, \theta \rangle$\unboldmath.}

Observe that if
$A^i_{t} = s$, the right and left step transition probabilities $p(s) =
\left(\frac{k}{2n} - \frac{s}{2n}\right)\Phi\left(\fPrime{s}{-1}\right)$ and
$q(s) = \left(\frac{k}{2n} + \frac{s}{2n}\right)\Phi\left(\fPrime{s}{1}\right)$
both change with the value of $s$. To deal with this ``position-dependent" transition probabilities, we construct a
random walk $B^i_t$, with non-varying transition probabilities, that
stochastically lower bounds the original random walk $A^i_t$.

Let $B^{i}_t$ be a random walk where $B^{i}_1 = A^{i}_1$ and $B^{i}_t$ has the
same absorbing states as $A^{i}_t$. Moreover, set $\PP(B^{i}_{t+1} = B^{i}_{t} +
2) = \tilde p$, $\PP(B^{i}_{t+1} = B^{i}_{t} - 2) = \tilde q$, and
$\PP(B^{i}_{t+1} = B^{i}_{t}) = 1 - \tilde p - \tilde q$ for $\tilde p$ and
$\tilde q$ defined momentarily. Notice that for all $s \in \cI_i\setminus \partial
\cI_i$, using \eqref{eq:boundfnone} (as we previously showed that $s = \omega(1)$ for any $s \in \cI_i$),
\begin{align}
p(s) &= \left(\frac{k}{2n} - \frac{s}{2n}\right)\Phi\left(\fPrime{s}{-1}\right)\\ 
&\geq \left(\frac{k}{2n} - \frac{s}{2n}\right)\Phi\left(2C_{\rm Signal}\frac{\lambda}{k^{r/2}V_n\sqrt{\runtime_1}}s^{r-1}\right),\\
&\geq \left(\frac{k}{2n} - \frac{s}{2n}\right)\Phi\left(2C_{\rm Signal}\frac{\lambda}{k^{r/2}V_n\sqrt{\runtime_1}}(s-2)^{r-1}\right)
\end{align}
(as $s \geq 2$ for large $n$), and, using
\eqref{eq:NSignalU}, \begin{align}q(s) &= \left(\frac{k}{2n} +
\frac{s}{2n}\right)\Phi\left(\fPrime{s}{1}\right)\\ &\leq
\left(\frac{k}{2n} + \frac{s}{2n}\right)\Phi\left(-2C_{\rm
Signal}\frac{\lambda}{k^{r/2}V_n \sqrt{\runtime_1}}(s - 2)^{r-1}\right),\end{align} (as this is the acceptance probability for a transition which decreases the overlap). We immediately see that $\frac{k}{2n}{-}\frac{s}{2n}$
$\left(\text{resp. }\frac{k}{2n} + \frac{s}{2n}\right)$ is decreasing (resp. increasing) in $s$, and
$\Phi\left(\prefactor(s-2)^{r-1}\right)$ $\left(\text{resp. }
\Phi\left(-\prefactor(s-2)^{r-1}\right)\right)$ is decreasing (resp. increasing) in
$s$ when $s \geq 2$. Therefore, denoting the minimal and maximal
elements of $\cI_i$ as $s_1 = s_1(i)$ and $s_2 = s_2(i)$ (recalling that the smallest
possible value of $s_1$ is $2\lceil a_1/4 \rceil = \omega(1)$ which is greater than two for large $n$), we can choose
\begin{align}
    \tilde p &= \left(\frac{k}{2n} -
\frac{s_2}{2n}\right)\Phi\left(\prefactor (s_1 - 2)^{r-1}\right)\intertext{ and } 
\tilde q&= \left(\frac{k}{2n} + \frac{s_2}{2n}\right)\Phi\left(-\prefactor 
(s_1-2)^{r-1}\right).
\end{align} Further, notice this choice of $\tilde p$ and $\tilde q$ has
$\tilde p + \tilde q \leq 1$ as $s_1 \in [k]$, seen by calculating,
\begin{align}
\tilde p + \tilde q &\leq \left(\frac{k}{2n} - \frac{s_1}{2n}\right)\Phi\left(\prefactor
(s_{2}-2)^{r-1}\right)\\
&\qquad+ \left(\frac{k}{2n} + \frac{s_1}{2n}\right)\Phi\left(-\prefactor (s_2-2)^{r-1}\right)\\
&\leq \frac{k}{2n} - \frac{s_1}{2n} + \frac{k}{2n} + \frac{s_1}{2n} = \frac{k}{n}
\leq 1.
\end{align} Two observations become apparent: (1) $\min_{s \in \cI_i \setminus
\partial \cI_i}p(s) - \max_{s \in \cI_i \setminus \partial \cI_i}q(s) \geq
\tilde p - \tilde q$. (2) As $\tilde p$ is uniformly smaller than $p(s)$ for all
$s \in \cI_i \setminus \partial \cI_i$, $\tilde q$ is uniformly larger than
$q(s)$ for all $s \in \cI_i \setminus \partial \cI_i$ and $\tilde p + \tilde q
\leq 1$, we have that each step of $B^i_{t}(s)$ stochastically lower bound the
steps of $A^i_{t}(s)$ for all $s \in \cI_i \setminus \partial \cI_i$ by Lemma
\ref{lem:stSing}; further---using Lemma \ref{lem:stSum}---we have that $B^i_t
\stlow A^i_t$ for all $t \geq 1$.

Recalling the stopping time $T'_i$, we can equivalently write \[T'_i
= \inf\{t
\in \NN: A^i_t \in \{2\lceil a_i / 4 \rceil + \parity, 2\lceil 3a_i/4 \rceil +
\parity\}\}.\label{eq:Tprime}\] Thus, observation (1) combined with Theorem \ref{thm:AddDrift} and
$s_2 - s_1 \leq 2k$ (as both $k \geq s_2 \geq s_1 \geq 0$) gives that
\[\label{eq:expTprime}\EE[T'_i] \leq \frac{2k}{\tilde p - \tilde q}.\] We then calculate,
\begin{align}
    \tilde p - \tilde q &= \left(\frac{k}{2n} - \frac{s_2}{2n}\right)\Phi\left(\prefactor
    (s_1-2)^{r-1}\right) \\
    &\qquad- \left(\frac{k}{2n} + \frac{s_2}{2n}\right)\Phi\left(-\prefactor
    (s_1 - 2)^{r-1}\right)\\
    &= \left(\frac{k}{2n} - \frac{s_2}{2n}\right)\Phi\left(\prefactor (s_1-2)^{r-1}\right) \\
    &\qquad+ \left(\frac{k}{2n} + \frac{s_2}{2n}\right)\left(1 - \Phi\left(\prefactor (s_1-2)^{r-1}\right)\right)\\
    &= \frac{k}{2n} + \frac{s_2}{2n}\left(1-2\Phi\left(\prefactor (s_1-2)^{r-1}\right)\right) 
    \geq \frac{1}{2n}(k - s_2) 
    \geq \frac{k}{5n}.\label{eq:part3finalbound1}
\end{align}
The penultimate inequality is due to $\Phi(x) \leq 1$ for any $x \in \RR$, and
the last inequality holds for large enough $n$ when $s_2 \leq k/2 + 2$, proven
by the following argument: Recall that $s_2 \leq s^*$ where 
\[\prefactor (s^* - 2)^{r-1} \leq 2\sqrt{\log(\runtime_1 n)}.\]
Rearranging this inequality then gives,
\[(s^* - 2)^{r-1} \leq \frac{V_n\sqrt{\runtime_1}}{C_{\rm
Signal}}\frac{k^{r/2}}{\lambda}\sqrt{\log(\runtime_1 n)}.\] Now, using
the form of $\lambda = \tau \sqrt{\runtime_1}n^{r/4}$, and the bound
$V_n \leq C_{\rm Noise}n^{(r-1)/2}$ from Corollary \ref{cor:Vn}, 
\[(s^* - 2)^{r-1} \leq \frac{C_{\rm Noise}n^{(r-1)/2}}{C_{\rm
Signal}}\frac{k^{r/2}}{\tau n^{r/4}}\sqrt{\log(\runtime_1 n)} \leq
\frac{C_{\rm Noise}}{C_{\rm
Signal}}\frac{k^{r/2}n^{(r-2)/4}}{\tau}\sqrt{\log(\runtime_1 n)}.\]
Then, as $k=\Omega(\sqrt{n})$ this implies that there exists $C_k>0$ such that $k \geq C_k\sqrt{n}$. Substituting this in we get the inequality,
\[(s^* - 2)^{r-1} \leq \frac{1}{\tau}\frac{C_{\rm Noise}}{C_{\rm
Signal}}C_k^{(r-2)/4}k^{r-1}\sqrt{\log(\runtime_1 n)}\] By Assumption
\ref{as:tau}(b), we have that $\tau \geq C_\tau C_{\rm Noise}
\sqrt{\log(\runtime_1 n)}/C_{\rm Signal}$, and taking the $r-1$ root and
choosing $C_\tau$ sufficiently large gives $s_2 - 2
\leq s^* - 2 \leq k/13 \leq k/2$; rearranging proves the desired inequality for
sufficiently large $n$ (as $k = \omega(1)$). 

Using \eqref{eq:part3finalbound1} in \eqref{eq:expTprime}, we then have that $\EE[T'_i] \leq 10n$ with $T'_i$ from \eqref{eq:Tprime};
Markov's inequality then gives that $\PP(T'_i > 10n\log^2(n)) \leq
\frac{1}{\log(n)^2}$.
We can see by a union bound that $\PP(\cE_i^c) \leq \PP(T'_i > T_i) +
\PP(A^i_{T_i'} = s_1)$. Thus, we are done after showing that $\PP(A^i_{T_i'} =
s_2) = 1-O(\log^{-2}(n))$ as we have just shown that $\PP(T'_i > T_i) =
O(\log^{-2}(n))$. 

\textbf{Step 3(b). Stochastically lower bounding \boldmath$A_t$\unboldmath\; over various regions by fixed probability random walks.}

Observation (2) gave that $B^i_t \stlow A^i_t$ for all $t \geq 1$, this means
that \[\PP(B^i_t = s_2) = \PP(B^i_t \geq s_2) \leq \PP(A^i_t \geq s_2) =
\PP(A^i_t = s_2)\] as both $B^i_t$ and $A^i_t$ have absorbing boundaries at the
points $s_1$ and $s_2$. Indeed, the above displayed equation must hold as $B_t^i$ cannot be larger than $s_2$ and hence $\{B_t^i = s_2\}$ holds if and only if the event $\{B_t^i \geq s_2\}$ holds; a similar argument holds for $A_t^i$.  Meaning, with $\tilde T_i = \inf \{t: B^i_t \in \{s_1,
s_2\}\}$, if we prove that $\PP(B^i_{\tilde T_i} = s_2) = 1 - O(\log^{-2}(n))$
then $\PP(A^i_{\tilde T_i} = s_2) = 1 - O(\log^{-2}(n))$, further implying that
$\PP(A^i_{T_i'} = s_2) = 1 - O(\log^{-2}(n))$. Next, we utilize Lemma
\ref{lem:walkSpeed2} to analyze $B^i_{\tilde T_i}$. As such, we must certify
each condition of this lemma where (in the language of Lemma
\ref{lem:walkSpeed2}) $a = s_1$, $b = s_2$ and $\beta = \prefactor$. 

As previously demonstrated, $s_2 \leq s^* \leq \frac{k}{13} + 2$, then it is immediate that $0 \leq s_2 \leq k/12$ for sufficiently large $n$ (recalling that $k = \omega(1)$). 
Recall that $s_1 = 2\lceil a_i/4 \rceil +
\parity$, $s_2 = 2\lceil 3a_i / 4 \rceil + \parity$ and $B_1^i$ is initialized
with $B_1^i = a_i$ from \eqref{eq:ai}. Moreover, we see in Lemma
\ref{lem:walkSpeed2} the initialization of our walk satisfies the following
(again, in what follows, $a = s_1$, $b = s_2$ and $\beta = \prefactor$ describing the random walk $X_t$): there should exist $a, b \in \NN$ with $a,b = \omega(1)$ of the same parity and $X_1 \in \NN$ fixed with value $X_1 = a + C(b-a)$ for
some $C \in [1/4,1)$ with identical parity as $a$ and $b$.

Indeed, in our case $a_i = \omega(1)$. Thus, for a
sufficiently large $n$ and fixed $\varepsilon > 0$, there exist a constant $C \geq
0.5 - \varepsilon$, such that $B_1^i=a_i = 2\lceil a_i / 4 \rceil + \parity +
C(2\lceil 3a_i / 4 \rceil - 2\lceil a_i / 4 \rceil)$; choosing $\varepsilon = 1/4$
gives that $C \geq 1/4$.

Moreover, using the set values of $s_1$ and $s_2$, we must certify the following
inequalities to apply Lemma \ref{lem:walkSpeed2} for $a = s_1$, $b = s_2$ and
$\beta = \prefactor$. First, inequality
\[2C_{\rm Signal}\betaConst (s_1-2)^{r-1} \geq 5 \frac{s_2}{2k}\] is equivalent
to the inequality 
\[\lambda \geq 5 k^{r/2 - 1}n^{(r - 1)/2} \frac{2\lceil 3a_i / 4
\rceil + \parity}{4C_{\rm Signal}(2\lceil a_i / 4 \rceil + \parity-2)^{r-1}}.\]
Next, recalling $a_i = \omega(1)$ for all $i \in [m]$ we know that $\lceil
ca_i\rceil=(1+o(1))ca_i$  for any $c\in \RR$. Therefore, it suffices to prove,
for $n$ sufficiently large, that $\lambda \geq C \cdot k^{r/2
- 1} n^{(r - 1)/2} / a_i^{r-2}$ (As $C_{\rm Signal} = r \geq 1$) where $C > 0$ is a
large constant dependent on the above displayed inequality. Additionally, as
each $a_i \geq n^{1/4}k^{1/2}/5$ by assumption, we then have
\[
    \frac{k^{r/2 -1}n^{(r - 1)/2}}{a_i^{r-2}} \leq 5^{r-2}\frac{k^{r/2 - 1}n^{(r - 1)/2}}{n^{(r-2)/4}k^{(r-2)/2}} = 5^{r-2}n^{r/4}.
\]
Meaning that we only require $\lambda \geq C 5^{r-2} \cdot 
n^{r/4}$. This holds by Assumption \ref{as:tau}(c) as $\runtime_1
\geq 1$ for large $n$ and $C_\tau$ is considered a sufficiently large constant
dependent on $r$. 

Second, inequality 
\[2C_{\rm Signal}\betaConst (s_1-2)^{r-1} (s_2 - s_1) / 16 \geq
\log(\log^2(n)),\] can be equivalently written as 
\[2C_{\rm Signal}\betaConst \left(2\lceil a_i / 4 \rceil + \parity
-2\right)^{r-1} (2\lceil 3a_i / 4 \rceil - 2\lceil a_i / 4 \rceil) / 16 \geq
\log(\log^2(n)). \label{eq:ineqrand}\] Again, using that $a_i = \omega(1)$ for
all $i \in [m]$, a sufficient condition
for inequality \eqref{eq:ineqrand} is, when $n$ is sufficiently large,
\[\lambda \geq C \frac{\log(\log^2(n))k^{r/2}n^{(r -
1)/2}\sqrt{\runtime_1}}{C_{\rm Signal}a_i^r},\] where, again, $C > 0$ is a large
constant dependent on \eqref{eq:ineqrand}. Recalling that $a_i \geq n^{1/4}
k^{1/2}/5$ and $\lambda = \tau \sqrt{\runtime_1} n^{r/4}$, the above inequality
holds when, 
\[\tau n^{r/4} \geq C \cdot 5^{r-1} n^{(r-1)/2 - r/4}\log(\log^2(n)) = C5^{r-1}n^{r/4 - 1/2}\log(\log^2(n)),\]
which follows from $\tau \geq 1$ for large enough $n$---assumed true by
Assumption \ref{as:tau}(c).

Now that each condition for Lemma \ref{lem:walkSpeed2} is certified for walk
$B^i_t$. Invoking this Lemma gives that $\PP(B^i_{\tilde T_i} = s_2) =
1-O(\log^{-2}(n))$. By our reduction of Lemma \ref{lem:boost03} to proving that
$\PP(\cE_i^c) = O(\log^{-2}(n))$ and using $\PP(\cE_i^c) \leq \PP(T'_i > T_i) +
\PP(A^i_{T_i'} = s_1) = \PP(T_i' > T_i) + (1-\PP(A^i_{T_i'} = s_2))$, we are
done once we prove that $\sum_{i = 1}^m T_i \leq \lfloor \finalIter / 2
\rfloor$.

Now, recalling \eqref{eq:boundT*}, it remains to prove that $\lfloor \finalIter / 2
\rfloor < 10n\log^{3.5}(n)$ with probability $o(1)$. A sufficient condition for this to hold is when 
\[\finalIter((p_t)_{t \in [n\runtime_1]}) = \inf\left\{t:
\max_{i \in [n]}\sum_{t' \leq t} \1\{p_t = i\} = \runtime_1 + 1\right\} =
\omega(n\log^3(n))\] with probability $1-o(1)$.

This follows as $\PP(\finalIter \leq Cn\log^4(n)) \leq n\PP(X_1 >
\runtime_1)$ where $X_i$ is a ${\rm Binomial}(Cn\log^4(n), 1/n)$ random
variable. Thus, as $\EE[X_1] = C \log^4(n)$ and $M_1 = \lceil \log^4(n) \rceil$, we have, for any $\epsilon > 0$, that
\[\PP(\finalIter \leq C\log^4(n)) 
\leq n\PP(X_1 > \frac{(1-\epsilon)}{C}\EE[X_1]),\]
for sufficiently large $n$. Therefore, setting $C = 2/(3(1-\epsilon))$, by a
Chernoff bound (see \cite[Section 4.1]{Motwani_Raghavan_1995}), 
\[\PP(\finalIter \leq Cn\log^4(n)) \leq n\exp\left(-(0.5)^2\left(\frac{2}{3(1-\epsilon)}\log^4(n)\right)/6\right)
= o(1).\]

By the preceding arguments and the definition $\partial \cI_1$ and event $\cE_1$, with probability $1-o(1)$, $\langle S_{t}, \theta \rangle$ must be larger than $a_1/2 = \Omega(k^{1/2} n^{1/4}) = \omega(1)$ for all time $t > 0$ until the inequality (\hyperlink{eq:threshold}{8.4}) holds. Therefore, this immediately implies that $\langle S_{t^*}, \theta \rangle \geq 2$ with probability $1-o(1)$ as well.

\end{proof}

\begin{proof}[Proof of Lemma \ref{lem:boost01}]
Using Lemma \ref{lem:AlgBTrans}, we have almost surely over the value of $\finalIter$,
that the sign of the sign of $D'_{t_0}$ is equal in law to the sign of
\[\fConst + Z_{t_0},\] where $Z_{t_0} \sim \cN(0, 1)$ independently of the sequence $(p_t)_{t \in [n\runtime_1]}$ and $\finalIter$. Then, using Lemma
\ref{lem:Mills} applied to $|Z_{t}|$, for each $t \in [n\runtime_1]$, and a union
bound, we have that $\max_{t \in [\finalIter]}
|Z_{t}| \leq \max_{t \in [n\runtime_1]}
|Z_{t}| \leq 2(1 - \varepsilon)\sqrt{\log(\runtime_1 n)}$ with probability
 $1-o(1)$, for sufficiently small $\varepsilon > 0$. We condition on this bound. Notice, this does not change random variable $(p_t)_{t \in [\finalIter]}$ as it is independent on $Z_t$. 

Additionally, by the assumption $\langle S_{t_0} , \theta \rangle \geq 2$, we can invoke the bounds in Corollary \ref{cor:bounds}(d).

\textbf{Step 4. Confirming \boldmath$\langle S_t, \theta\rangle$\unboldmath\;is non-decreasing above threshold (\hyperlink{eq:threshold}{8.4}).}
 
Notice, the overlap does not
change when the algorithm flips a sign of $(S_t)_{p_t}$ if $\theta_{p_t} = 0$.
Therefore, we only consider $\theta_{p_t} \ne 0$. We then have two cases: 
\begin{itemize}
    \item[(a)] If $(S_{t_0})_{P_{t_0}} \ne \theta_{P_{t_0}}$, then $(Q_{t_0} -
   (S_{t_0})_{P_{t_0}})\theta_{P_{t_0}} = -2(S_{t_0})_{P_{t_0}}\theta_{P_{t_0}}
   = 2\theta_{P_{t_0}}^2 = 2$. Thus, using \eqref{eq:SignalL} from Corollary \ref{cor:bounds},
   \[D'_{t_0} \geq \frac{2C_{\rm
   Signal}\frac{\lambda}{k^{r/2}}}{V_n\sqrt{\runtime_1}}\langle
   S_{t_0}, \theta \rangle^{r-1}- 2(1 - \varepsilon)\sqrt{\log(\runtime_1
   n)},\] (note, we can
   invoke this inequality since we assumed that $\langle S_{t_0}, \theta \rangle
   -2 \geq 0$, and thus, any transition $S_{t_0} + e_{p_t}(q_t -
   (S_{t_0})_{P_{t_0}})$ also has $\langle S_{t_0} + e_{p_t}(q_t -
   (S_{t_0})_{P_{t_0}}), \theta \rangle \geq 0$). And therefore, recalling the statement of the lemma assumed that $\frac{2C_{\rm
   Signal}\frac{\lambda}{k^{r/2}}}{V_n\sqrt{\runtime_1}}\langle
   S_{t_0}, \theta \rangle^{r-1} \geq (2-\epsilon)\sqrt{\log(M_1n)}$ and as $(2-\varepsilon) > 
   2(1-\varepsilon)$ for sufficiently small $\varepsilon > 0$, we have that $D_{t_0} \geq 0$ and the algorithm accepts the transition.
   \item[(b)] If $(S_{t_0})_{P_{t_0}} = \theta_{P_{t_0}}$, then $(Q_{t_0} - (S_{t_0})_{P_{t_0}})\theta_{P_{t_0}} =
   -2(S_{t_0})_{P_{t_0}}\theta_{P_{t_0}} = -2\theta_{P_{t_0}}^2 = -2$. Thus,
   using \eqref{eq:NSignalU} from Corollary \ref{cor:bounds}, 
   \[D'_{t_0} \leq -\frac{2C_{\rm Signal}\frac{\lambda}{k^{r/2}}}{V_n
   \sqrt{\runtime_1}}(\langle S_{t_0}, \theta \rangle -2)^{r-1} +
   2(1 - \varepsilon)\sqrt{\log(\runtime_1 n)}.\] And therefore, $D_t < 0$ and the algorithm rejects this transition by identical reasoning to (a).
\end{itemize}

Thus, dependent on our conditioned bound over each $|Z_t|$, we have that
\begin{enumerate}
    \item $D_{t_0} > 0$ (and thus we accept the transition) when
    $(S_{t_0})_{p_t} \ne \theta_{p_t}$ and $\theta_{p_t} \ne 0$;
    \item $D_{t_0} < 0$ (and thus we reject the transition) when
    $(S_{t_0})_{p_t} = \theta_{p_t}$ and $\theta_{p_t} \ne 0$.
\end{enumerate}

As such, the algorithm always accepts a transition which increases $\langle S_{t_0},
\theta \rangle$ after flipping the sign of $(S_{t_0})_{p_t}$ and never accepts
a transition where $\langle S_{t_0}, \theta \rangle$ decreases after flipping
the sign of $(S_{t_0})_{p_t}$. Thus, $\langle S_{t_0 + 1}, \theta \rangle \geq
\langle S_{t_0}, \theta \rangle$ with probability $1-o(1)$. Recursively applying
this argument to $D'_t$ from \eqref{eq:DAlign} over all times $t \geq t_0$ gives
that each transition has \[\frac{2C_{\rm Signal}\frac{\lambda}{k^{r/2}}}{V_n
\sqrt{\runtime_1}}(\langle S_{t}, \theta \rangle -2)^{r-1} \geq \fuConsto >
(2-\varepsilon)\sqrt{\log(\runtime_1 n)}.\] Hence, the overlap
$\langle S_{t}, \theta \rangle$ is non-decreasing dependent on our bound on each $|Z_t|$ for
every $t \geq t_0$ and is always bounded below by two.

\textbf{Step 6. Checking probabilistic conditions on \boldmath$(p_t)_{t \in [T^*]}$\unboldmath\;.}

Therefore, a sufficient
condition for Algorithm \ref{alg:B} to return $S \in \RR^n$ with $\langle
S, \theta \rangle = k$ is for $\PP(\cA) = 1-o(1)$, where $\cA =
\{\{\finalIter \geq \frac{3}{4}n\runtime_1\} \cap \{\text{each }i \in [n]
\text{ is in }(p_t)_{\frac{1}{2}n\runtime_1 \leq t \leq
\frac{3}{4}n\runtime_1}\}\}$. Indeed, this follows from the assumption that $t
\leq \lfloor \finalIter / 2 \rfloor \leq n\runtime_1/2$. Notice that
\[\PP(\cA^c) \leq \PP\left(\left\{\finalIter < \frac{3}{4}n\runtime_1\right\}\right) +
\PP(\{\text{each }i \in [n] \text{ is in }(p_t)_{\frac{1}{2}n\runtime_1 \leq t
\leq \frac{3}{4}n\runtime_1}\}^c),\] and thus, define the event $\cB_i
= \{i \not\in (p_t)_{\frac{1}{2}n\runtime_1 \leq t \leq
\frac{3}{4}n\runtime_1}\}$. The statement follows if\\ $\PP(\{\finalIter <
\frac{3}{4}n\runtime_1\})$ and $\PP(\cB_i) = o(1/n)$ by a union bound.
Then, with $X_i \sim \mathrm{Binomial}(\frac{3}{4}nM_1, \frac{1}{n})$ independent for $i \in [n]$, we have that
\[
\PP\left(\left\{\finalIter < \frac{3}{4}n\runtime_1\right\}\right) = \PP\left(\max_{i \in {n}} X_i \geq \runtime_1\right)
\leq n\PP(X_1 \geq \runtime_1)
\leq n\PP\left(X_1 \geq \frac{4}{3}\EE[X_1]\right).
\]
Therefore, using a Chernoff bound \cite[Section 4.1]{Motwani_Raghavan_1995} gives that,
\[\PP\left(\left\{\finalIter < \frac{3}{4}n\runtime_1\right\}\right)\leq ne^{-(0.5)^2 \runtime_1 / 6} = o(1/n).\] The final equality comes from $\runtime_1 = \lceil \log^4(n)\rceil$. 
Moreover, we have that
\[\PP(\cB_i) = \left(1-\frac{1}{n}\right)^{n \runtime_1/8} \leq e^{-\runtime_1/8} \leq n^{-\log(n)} = o(1/n),\]
for sufficiently large $n$ as $\runtime_1 = \lceil \log^4(n) \rceil$.

Thus, recalling the above is done under a conditioned event on $(Z_t)_{t \in [\finalIter]}$ which holds with probability $1-o(1)$, independent of
$(p_t)_{t \in [n\runtime_1]}$, we complete the proof. 

\end{proof}

\subsubsection{Proof of Lemma \ref{lem:AlgAWorks}}

\begin{lemma}\label{lem:AlgASuc} Consider Algorithm \ref{alg:A} and set $T =
\lceil \runtime_2 n / 2 \rceil$. Recall $(p_t, q_t)_{t \in [T]}$ is the sequence
of proposals used in Algorithm \ref{alg:A}. Then the following holds with
probability $1-o(1)$:

For the set of proposals $(p_t, q_t)_{t \in [T]}$, each possible $(i,j) \in [n]
 \times \{-1, 0, 1\}$ is proposed at least once and no single $i \in [n]$ is
 proposed by $(p_t)_{t \in [T]}$ strictly more than $\runtime_2$ times.
 \end{lemma}

\begin{proof}[Proof of Lemma \ref{lem:AlgASuc}]
Consider an arbitrary $\runtime \in \NN$, we first
derive probabilistic bounds for the statements of the lemma, denoting $T =
\lceil \runtime n / 2 \rceil$ (slightly abusing notation). Afterwords, we use
the definition of $\runtime_2$ to conclude the proof.

Let $\cA_i$ be the event that $i$ is proposed by $(p_t)_{t \in [T]}$ more
than $\runtime$ times and $\cB_{i,j}$ be the event that $(i,j)$ is not proposed
by $(p_t, q_t)_{t \in [T]}$, then the statement follows by a union bound if
$\PP(\cup_{i \in [n]} \cA_i) = o(1)$ and $\PP(\cup_{(i,j) \in [n] \times \{-1,
0, 1\}} \cB_{i,j}) = o(1)$.    

Define the random variable $X_i = \sum_{t = 1}^T \1\{p_t = i\}$. Marginally,
$X_i \sim {\rm Binomial}(T, \frac{1}{n})$. As such, by a Chernoff bound (see
Section 4.1 of \cite{Motwani_Raghavan_1995}), we have for sufficiently large
$n$, 
\[
\PP(\cup_{i \in [n]} \cA_i) = \PP(\cup_{i \in [n]} \{X_i \geq \runtime\}) \leq n \PP\left(X_i \geq 1.5 \frac{T}{n}\right) \leq n e^{-(0.5)^2 \runtime/6} \leq ne^{-\runtime/24}.
\]
The penultimate inequality is due to $T = \lceil
\runtime n / 2 \rceil \geq \runtime n / 2$.

Further, by the existence of a bijection between $[n] \times \{-1, 0, 1\}$ and
$[3n]$, the event \newline $\cup_{(i,j) \in [n] \times \{-1, 0, 1\}} \cB_{i,j}$
is equivalent to it requiring more than $T$ uniform draws to collect all $3n$
out of $3n$ coupons. Then, Lemma \ref{lem:coupon} implies that the time $T'$ to
collect all $3n$ coupons is bounded by 
\[
\PP(T' > \beta (3n)\log(3n)) \leq (3n)^{-\beta + 1}
\]
where $\beta = T / (3n \log(3n)) \geq \runtime / (6\log(3n))$. As we have
assumed that $\runtime = \runtime_2 = \lceil 25\log(3n) \rceil$, then
both $ne^{-\runtime_2/24} = o(1)$ and $(3n)^{-\beta + 1} \leq (3n)^{-25/6 + 1} =
o(1)$ as $n \to \infty$, giving the proof.
\end{proof}

\begin{proof}[Proof of Lemma \ref{lem:AlgAWorks}]
        By Assumption \ref{as:tau}(a),  recalling we can choose $C_\tau \geq 4C_k^{(r-2)/4}$, where $k \geq C_k\sqrt{n}$ for some constant $C_k>0$, we have that 
        \[\lambda \geq 4C_k^{(r-2)/4}\gamma n^{r/4} = 4C_k^{(r-2)/4}\gamma \frac{n^{(r-1)/2}}{n^{r/4 -
        1/2}} \geq 4\gamma\frac{n^{(r-1)/2}}{k^{(r-2)/2}}.\] Thus, for any
        vector $S_1 \in \{-1,1\}^n$ satisfying the statement of the lemma,
        for large $n$, we have
        \begin{align}\left(\frac{\langle S_1, \theta \rangle - 2}{k^{1/2}(\|S_1\|_0
        + 1)^{1/2}} \right)^{r-1} &= \sqrt{\frac{k-2}{k^{1/2}(n +
        1)^{1/2}}}^{r-1} \geq \frac{1}{2}\sqrt{k/n}^{r-1} \\
        &= \frac{1}{2}\sqrt{k}
        \frac{k^{(r-2)/2}}{n^{(r-1)/2}} \geq 2\gamma \frac{\sqrt{k}}{\lambda} \geq \frac{2\gamma}{C_{\rm Signal}}\frac{\sqrt{k}}{\lambda},\end{align}
        and therefore, the condition $\left(\frac{\langle S_1, \theta
        \rangle - 2}{k^{1/2}(\|S_1\|_0 + 1)^{1/2}} \right)^{r-1} \geq
        \frac{2\gamma}{C_{\mathrm{Signal}}}\frac{\sqrt{k}}{\lambda}$ holds for
        large enough $n$. Moreover, since $\langle S_1 , \theta \rangle =
        k$, then for sufficiently large $n$ we have that $\langle S_t,
        \theta \rangle \geq 2$. And therefore, we satisfy both conditions of
        Lemma \ref{lem:success2}. Combining this result with Lemma
        \ref{lem:AlgASuc}, ensuring that we propose each possible transition once and we can inject noise for the duration of the runtime of Algorithm \ref{alg:A}, proves the statement as both lemmas hold with
        probability $1 - o(1)$.
\end{proof}

\subsection{Proofs of Secondary Lemmas}\label{sec:NoiseInjectSecondaryLemsProofs}
In this section we prove Corollary \ref{cor:bounds} and Lemma \ref{lem:equiv03}, that were used in the previous section to prove the Key Lemmas. The proof of Lemma \ref{lem:AlgBTrans} is deferred to Section \ref{sec:proofAlgBTrans}.

In order to prove Corollary \ref{cor:bounds}, we introduce Lemmas \ref{lem:NoiseEquiv}, \ref{lem:VarLawBound},
\ref{lem:bSignalControl} and \ref{lem:bRegControl}. Their proofs are given in Section \ref{aux:bounds}.
\begin{lemma}\label{lem:NoiseEquiv} Let $\sigma \in \{-1,0,1\}$, $a \in \RR$, $i
    \in [n]$, and $r \in \NN$. Consider the function 
    \[g(\sigma, a, i) = \langle (\sigma + a e_i)^{\otimes r}, W \rangle -
    \langle \sigma^{\otimes r}, W \rangle,\] where $W \in \RR^{n^{\otimes r}}$ has i.i.d. standard Gaussian elements. Then, the following equivalence in law holds:
    \[g(\sigma, i, a) \laweq \|(\sigma + a e_i)^{\otimes r} - \sigma^{\otimes
    r}\|_F Z,\] where $Z$ is a standard
    Gaussian random variable.
\end{lemma}

In order to make this equality in law more practical to apply, we 
provide a non-asymptotic bound on the pre-factor $\Variance 
= \|(\proposal)^{\otimes r} - (S_{t})^{\otimes r}\|_F$ in front of $Z$.

\begin{lemma}\label{lem:VarLawBound} Let $\sigma \in \{-1,0,1\}^n$ and $a \in
    [-2,2]$, $i \in [n]$, and $r \in \NN$. We have the following bound on
    $\|(\sigma + ae_i)^{\otimes r} - \sigma^{\otimes r}\|_F$:
    \[\|(\sigma + ae_i)^{\otimes r} - \sigma^{\otimes r}\|_F \leq C_{\rm
    Noise}|a|\;\|\sigma\|_0^{\frac{r-1}{2}},\] where $C_{\rm Noise} = \left(
    \sum_{j = 1}^r \binom{r}{j}^2 2^{2j}\right)^{1/2}$.
\end{lemma}

It remains to bound both $\bSignal$ and $\bReg$; this is demonstrated by the
next two lemmas.

\begin{lemma}\label{lem:bSignalControl} Let $a \in \RR$, $i \in [n]$, $\sigma
    \in \RR^n$, $\theta \in \RR^n$, $a \in \RR$, and $r \in \NN$ where 
    $\min(\langle \sigma + ae_i, \theta \rangle, \langle \sigma,
    \theta \rangle) \geq 0$. Consider the function \[f(\sigma, \theta, a, i) = \langle
    (\sigma + a e_i)^{\otimes r}, \theta^{\otimes r} \rangle - \langle
    \sigma^{\otimes r}, \theta^{\otimes r} \rangle.\] We then have the following
    upper and lower bounds: When $a\theta_i \geq 0$,
    \[C_{\rm Signal}\min(\langle \sigma + ae_i, \theta \rangle, \langle
    \sigma, \theta \rangle)^{r-1}a\theta_i \leq f(\sigma, \theta, a, i) \leq
    C_{\rm Signal}\max(\langle \sigma + ae_i, \theta \rangle, \langle \sigma,
    \theta \rangle)^{r-1}a\theta_i,\] where $C_{\rm Signal} = r$. Moreover, when
    $a\theta_i \leq 0$,
    \[C_{\rm Signal}\max(\langle \sigma + ae_i, \theta \rangle, \langle \sigma,
    \theta \rangle)^{r-1}a\theta_i \leq f(\sigma, \theta, a, i) \leq C_{\rm
    Signal}\min(\langle \sigma + ae_i, \theta \rangle, \langle \sigma, \theta
    \rangle)^{r-1}a\theta_i.\]
\end{lemma}
\begin{lemma}\label{lem:bRegControl} Let $s, a, \gamma \in \RR$ where $\min(s,
    s+a, \gamma) \geq 0$. Consider the function \[h(s,a) =
    \gamma((s+a)^{\frac{r+1}{2}} - s^{\frac{r+1}{2}}),\] we then have the
    following upper and lower bounds:
    \[C_{\rm Reg}\gamma\min(s, s+a)^{\frac{r+1}{2}}a \leq h(s,a) \leq C_{\rm
    Reg}\gamma\max(s, s+a)^{\frac{r+1}{2}}a,\] where $C_{\rm Reg} =
    \frac{r+1}{2}$.
\end{lemma}

\begin{proof}[Proof of Corollary \ref{cor:bounds}]
(a): Apply Lemma \ref{lem:NoiseEquiv} with the choice of $\sigma = S_{t}$, $i = p_t$ and $a = q_t -(S_t)_{p_t}$ 

(b): Apply Lemma \ref{lem:VarLawBound} with $\sigma = S_{t}$, $i = p_t$ and $a = q_t -(S_t)_{p_t} \in [-2,2]$ 

(c): Combine statements (a) and (b)

(d): Apply Lemma \ref{lem:bSignalControl} with $\sigma = S_{t}$, $i = p_t$, and $a = q_t -(S_t)_{p_t}$. The factor
$\frac{\lambda}{k^{r/2}}$ is due to the pre-factor on $\theta^{\otimes r}$---as
seen in Definition \eqref{eq:model}. 

(e): Apply Lemma \ref{lem:bRegControl} with $s = \|S_{t}\|_0$ and $a = |q_t| - |(S_t)_{p_t}|$. Clearly $s$ and $s + a$ must be larger than zero as no proposal can have zero-norm less than zero.

\end{proof}

\begin{proof}[Proof of Lemma \ref{lem:equiv03}]
    
    For notational convenience, for this proof only, we abbreviate $\Variance$
    as $V$.
    
    Recall \eqref{eq:DSparse}, defining $D_t$ as
    \[D_t = \bSignal + \bNoise - \bReg + \bCorrect\] for a given time $t$.
    Moreover, using Corollary \ref{cor:bounds}(a) and the decomposition of $D_t$ from Remark \ref{rm:AlgA}, we have that
    \[D_t \laweq \bSignal + VZ - \bReg + \bCorrect,\] where each $Z \iidsim
    \cN(0,\Id_n)$. Consider the term $D_t/V = \bSignal/V + Z - \bReg/V +
    \bCorrect/V$. Then (as a slight abuse of notation), define a vector $v = \mu + Z$ where $\mu_{p_t} =
    \bSignal/V - \bReg/V$ for $p_t \in [n]$ and $Z \iidsim \cN(0,\Id_n)$. Thus,
    \[(D_t/V)_{t \in [T^*_\runtime]} = (v_{p_t} + \bCorrect/V)_{t \in
    [T^*_\runtime]},\] and expanding the definition of $\bCorrect$, we
    equivalently have
   \[(D_t/V)_{t \in [T^*_\runtime]} = (v_{p_t} + G_{P_{t}}^{t_{p_t}}- \bar
   G_{p_t})_{t \in [T^*_\runtime]}.\] Recall how Algorithm \ref{alg:A} generated
   each element $G_i^j$ for $i \in [n]$ and $j \in [\runtime]$. This generation leads
   to the right-hand side above to correspond to the output of a $\NIQ(\mu,
   \setlister, \runtime)$ algorithm where $T_\runtime = T^*_\runtime$. And therefore, by Theorem
   \ref{thm:NIQ}, 
   \[(D_t/V)_{t \in T^*_\runtime} \laweq (\bSignal - \bReg)/V + Z'_{t},\] where, almost
  surely over $T^*_\runtime$, the random variable $Z' \in \RR^{T^*_\runtime} \sim \cN(0,
  \runtime \Id_{T^*_\runtime})$, additionally $Z'$ is independent
  of $(p_t)_{t \in \NN}$ and thus $T^*_\runtime((p_t)_{t \in \NN}) = T^*_\runtime$.
  Multiplying the above by $V$ then gives that
   \[(D_t)_{t \in t^*_\runtime} \laweq (\bSignal - \bReg + V Z'_{t})_{t \in
   [T^*_\runtime]}.\] Pulling out a $\runtime^{1/2}$ factor from $Z'_t$
   concludes the proof. 
\end{proof}

\section{Deferred Material}
\subsection{Auxiliary Lemmas}

\begin{lem}[\cite{BLM13}, Theorem 2.5]\label{lem:MaxGaussianMean}
    Let $(Z_i)_{i=1}^N$ be a collection of mean zero, possibly dependent, Gaussian random variables, each with variance at most $\sigma^2$. Then 
    \[\mathbb{E}\left(\max_i Z_i\right)\leq \sqrt{2\sigma^2 \log N}\]
\end{lem}

\begin{lem}[\cite{BLM13}, Theorem 5.8]\label{lem:MaxGaussianDeviation}
    Let $(Z_i)_{i=1}^N$ be a collection of mean zero, possibly dependent, Gaussian random variables, each with variance at most $\sigma^2$. Then $\mathrm{Var}(\max_i Z_i) \leq \sigma^2$, and for all $u>0$, 
    \[\mathbb{P}\left(\left|\max_i Z_i - \mathbb{E}\left(\max_i Z_i\right)\right| \geq u\right) \leq \exp\left(-\frac{u^2}{2\sigma^2}\right)\]
\end{lem}

\begin{thm}\label{thm:AddDrift} Let $(X_t)_{t > 0}$ be a sequence of random
    variables supported on the finite interval $\cI = [a, b] \cap \ZZ$, where
    $a,b \in \ZZ$ with $a \leq b-1$. Let $T = \inf\{t \geq 0 : X_t = b \text{ or
    } X_t = a\}$, if $\delta_t(s) = \EE[X_{t+1} - X_{t} | X_t = s] \geq \delta$
    for all $s \in \cI \setminus \{a,b\}$ and $X_{t+1} = X_t$ almost surely when
    $X_t \in \{a,b\}$, then $\EE[T] \leq (b-a) / \min(\delta, 1)$.
\end{thm}

\begin{proof} 
Let $\tilde X_{t+1}$ be a random variable with $\PP(\tilde X_{t+1}
= a + 1| \tilde X_t = a) = 1$ and all other probabilities are
equivalent to $X_t$. Then, setting $\tilde X_1 = X_1$, we construct a coupling
where we generate $X_t$ given $X_{t-1}$ and then set $\tilde X_t = X_t$ until
$X_{t} = a$, after which draw $\tilde X_{t+1}$ and the subsequent
iterates of $(\tilde X_t)_{t' \geq t+1}$ independently of $(X_t)_{t' \geq t+1}$. Under this coupling we have that
$\tilde X_t = X_t$ almost surely until $X_t = a$.

If $X_t$ hits $b$ before $a$ then $\tilde X_t = X_t = b$. Otherwise, if $X_t$
hits $a$ before $b$ then $\tilde X_t = X_t = a$, and it requires at least one
more step for $\tilde X$ to reach $b$. Therefore, with $\tilde T = \inf\{t \geq 0 : \tilde X_t =
b\}$, we have that $T \leq \tilde T$ almost surely on the coupling's probability
space; meaning $\EE[T] \leq \EE[\tilde T]$. 

Let $Y_t = -\tilde X_t + b$, where $Y_t$ is supported on the set $\cI' = [0,
b-a] \cap \ZZ$. It is trivial that $\EE[Y_t - Y_{t+1}| Y_{t-1} = s] = \EE[\tilde
X_{t+1} - \tilde X_t| \tilde X_{t-1} = b-s] \geq \min(\delta, 1)$ for all $s \in \cI'
\setminus \{0\}$ as this inequality holds for $\tilde X_{t+1}$ when $\tilde
X_{t} \not \in \{a,b\}$, and we constructed $\tilde X_{t+1}$ for this to hold
when $\tilde X_{t} = a$. Now letting $T'$ be the hitting time for $Y_t$ to reach
$0$, by the additive drift theorem \cite[Theorem 1]{lengler2018driftanalysis},
we have that $\EE[T'] \leq \EE[Y_0] / \min(\delta,1) \leq (b-a) / \min(\delta,1)$. As $Y_t = 0$
is equivalent to $X_t = b$ by construction, we in turn have that $\EE[T] \leq
\EE[\tilde T] = \EE[T'] \leq (b-a) / \min(\delta,1)$.

\end{proof}
\begin{remark}\label{rm:st}
    Lemma \ref{lem:stSum} and Lemma \ref{lem:stSing} rely on two equivalent
    definitions of stochastic domination, denoted between two scalar random
    variables $X$, $Y$ by $X \stlow Y$.

    Let $X \sim \mu$ and $Y \sim \nu$, for two probability measures $\mu$ and
    $\nu$. Then, $X \stlow Y$ if either condition below holds:
    \begin{enumerate}
        \item For all $t \in \RR$, we have that $\PP(X \geq t) \leq \PP(Y \geq
        t)$.
        \item There exists a coupling $\gamma$ of $(X, Y)$ such that $X = Y +
        \delta$ for some random variable $\delta$ which is non-positive almost
        surely. Equivalently, also could have $Y = X + \delta$ for some random
        variable $\delta$ which is non-negative almost surely.
    \end{enumerate}

    Notably, this definition of stochastic domination differs slightly from the
    traditional definition as $X \stlow Y$ even when $X \laweq Y$.
\end{remark}

\begin{lemma}\label{lem:stSing} Let $X$ be a random variable with law 
    \[\PP(X = -2) = q, \PP(X = 2) = p, \PP(X = 0) = 1 - q - p.\] If $Y$ is
    supported on the set $\{-2, 0, 2\}$ and $\PP(Y = -2) \geq q$ and $\PP(Y = 2)
    \leq p$ then $Y \stlow X$.
\end{lemma}

\begin{proof}
    The proof follows if we can write that $\PP(Y \geq t) \leq \PP(X \geq t)$
    for $t \in \{-2,0,2\}$ as this is the support of both $Y$ and $X$. This
    follows immediately as $\PP(Y \geq -2) = 1 = \PP(X \geq -2)$, $\PP(Y \geq 0)
    = 1- \PP(Y = -2) \leq 1-q = \PP(X \geq 0)$, and $\PP(Y \geq 2) = p \leq
    \PP(X \geq 2)$. 
\end{proof}

\begin{lemma}\label{lem:stSum} Let $S_t = \sum_{i=1}^t X_i$ be a random walk on
    the set $\cI = [a,b] \cap \ZZ$ where $a,b \in \ZZ$ with each $X_i = X_i(s)$
    dependent on $S_{i-1} = s$. Moreover, if $S_t \in \{a,b\}$ then $X_{t+1} =
    0$ almost surely. In addition, let $S'_t = \sum_{i=1}^t Y_i$ be a random
    walk on the set $\cI$, we similarly assume that $S'_t \in \{a,b\}$ implies
    that $Y_{t+1} = 0$ almost surely. If $Y_i \stlow X_i(s)$ almost surely for
    all $s \in \cI \setminus \{a,b\}$ then $S'_t \stlow S_t$ for all $t \geq 1$. 
\end{lemma}

\begin{proof}
    We prove the statement $S'_t \stlow S_t$ by induction on $t$, as $S'_1 =
    Y_1 \stlow X_1 = S_1$, the base case holds. 

    Thus, assuming that $S'_t \stlow S_t$, there exists a coupling of $(S'_t,
    S_t)$ such that $S_t = S'_t + \delta$ where $\delta \geq 0$ almost surely.
    Therefore, under this coupling, if $S_t = b$ then $S'_t \leq b$ almost
    surely, wherein $S'_{t+1} \leq S_{t+1} = b$ is trivial. If $S'_t = a$ then
    $S_t \geq a$ by definition and, again, $a = S'_{t+1} \leq S_{t+1}$ is
    trivial. Finally, consider $S'_t > a$ and $S_t < b$, by the inductive
    hypothesis we must also have $S_{t} > a$ and $S'_{t}  < b$ almost surely.
    Therefore, $Y_{t+1} \stlow X_{t+1}$ for all possible values of $S_t$ in this
    case. We can then calculate, under the coupling implied by $S'_t \stlow S_t$
    and $Y_{t+1} \stlow X_{t+1}$, that
    \begin{align}
        S_{t+1} &= S_t + X_{t+1}\\
        &= S_t + \delta_1 + Y_{t+1} + \delta_2\\
        &= S'_{t+1} + (\delta_1 + \delta_2),
    \end{align}
    where both $\delta_1 \geq 0$ and $\delta_2 \geq 0$ almost surely; the proof
    follows from Remark \ref{rm:st}.
\end{proof}

\begin{lemma}\label{lem:walkSpeed2} Let $\beta \in \RR_+$. Consider a random
    walk $S_t = \sum_{i=1}^t X_i$ on the set $[a, b] \cap \ZZ$ where $a, b \in
    \NN$ have $a,b = \omega(1)$ of the same parity. Moreover, let $X_1
    \in \NN$ be fixed with value $X_1 = a + C(b-a)$ for some $C \in (0,1)$ and,
    in addition, $X_1$ must have same parity as $a$ and $b$. We define each
    $X_i$, for $i > 1$, as an independent random variable with
    law,
    \[X_i \sim \begin{cases}
        2 &\qquad \text{with probability }p = \left(\frac{k}{2n} -
        \frac{b}{2n}\right)\Phi(\beta (a-2)^{r-1})\\
        -2 &\qquad \text{with probability }q = \left(\frac{k}{2n} +
        \frac{b}{2n}\right)\Phi(-\beta (a-2)^{r-1})\\
        0 &\qquad \text{with probability }1-p-q \end{cases}.\] If $C \geq 1/4$, $0
    \leq b \leq k/12$, $\beta (a-2)^{r-1} \geq \frac{5b}{2k}$, and $\beta
    (a-2)^{r-1}(b-a) / 16 \geq \log(\log^2(n))$ then with $T = \inf\{t \geq 0 : S_t
    = b \text{ or } S_t = a\}$ we have $\PP(S_T \geq b) \geq
    1-\log^{-2}(n)$ for sufficiently large $n$.
\end{lemma}

\begin{proof}
    First we will get a useful bounds on the ratio of $q/p$, 
    \begin{align}
        \frac{q}{p} &= \frac{(\frac{k}{2n} + \frac{b}{2n})\Phi(-\beta (a-2)^{r-1})}{(\frac{k}{2n} - \frac{b}{2n})\Phi(\beta (a-2)^{r-1})}\\
        &= \frac{\frac{k}{2n} + \frac{b}{2n}}{\frac{k}{2n} - \frac{b}{2n}}\frac{\Phi(-\beta (a-2)^{r-1})}{\Phi(\beta (a-2)^{r-1})}.
    \end{align}
    As $\beta (a-2)^{r-1} \geq 0$, Lemma \ref{lem:PhiRatio} gives that $\Phi(-\beta
    (a-2)^{r-1})/\Phi(\beta (a-2)^{r-1}) \leq e^{-\beta a^{r-1}}$. Additionally, as $b \leq k/12$, Lemma \ref{lem:vexupper} gives that $\frac{k/(2n) + b/(2n)}{k/(2n) - b/(2n)} \leq 1 + 5b/(2k)$. Thus,
    \[q/p \leq \left(1 + 5\frac{b}{2k}\right) \cdot e^{-\beta (a-2)^{r-1}},\] for
    $n$ sufficiently large. Taking the logarithm of both sides gives
    \[\log(q/p) \leq \log\left(1 + 5\frac{b}{2k}\right) - \beta (a-2)^{r-1},\] and
    using that $\log(1+x) \leq x$ and $\frac{1}{2}\beta (a-2)^{r-1} \geq
    \frac{5b}{2k}$ (by assumption) results in 
    \[\log(q/p) \leq -\frac{1}{2}\beta (a-2)^{r-1}.\label{eq:boundAux}\] Thus, $q/p \leq 1$. We can now
    calculate the probability that $S_T \geq b$ using that $C \geq 1/4$ and the
    formula \cite[Section 3.9]{grimmett2001probability}, 
    \begin{align}\PP(S_T = b) = \frac{1 - (q/p)^{\frac{1}{2}(S_1 -
    a)}}{1 - (q/p)^{\frac{1}{2}(b-a)}} \geq \frac{1 -
    (q/p)^{\frac{C}{2}(b-a)}}{1 - (q/p)^{\frac{1}{2}(b-a)}} \geq 1 -
    (q/p)^{\frac{C}{2}(b-a)} \geq 1 - (q/p)^{\frac{1}{8}(b-a)}.
\end{align} The
    $\frac{1}{2}$ factor in the exponent of $(q/p)$ is due to the walk $S_t$
    taking steps of size two instead of size one as in
    \cite{grimmett2001probability}, note the formula still holds as we assumed
    that the parity of $a$ and $b$ is equal to the parity of $X_1$, accounting
    for steps of size two. Therefore, using \eqref{eq:boundAux}, our assumed
    lower bound on $\beta(b-a)(a-2)^{r-1}$ and $n$ sufficiently large, we have
    \[(q/p)^{\frac{1}{8}(b-a)} = e^{\frac{1}{8}(b-a)\log(q/p)} \leq
    e^{-\frac{1}{16}(b-a)\beta (a-2)^{r-1}} \leq e^{-\log(\log^2(n))} =
    \log^{-2}(n).\] Which in turn implies that $\PP(S_T = b) \geq 1 -
    \log^{-2}(n)$ for sufficiently large $n$.
\end{proof}

\begin{lemma}\label{lem:PhiRatio}
    Let $\Phi$ be the Gaussian CDF. For each $x \geq 0$, we have that
    $\frac{\Phi(-x)}{\Phi(x)} \leq e^{-x}$.
\end{lemma}

\begin{proof}
    Let $\phi(x) = \frac{1}{\sqrt{2\pi}}e^{-x^2/2}$. Define $f(x) =
    \log\left(\frac{\Phi(-x)}{\Phi(x)} \right) + x$, clearly $f(x) \leq 0$ for a
    fixed $x$ implies that $\frac{\Phi(-x)}{\Phi(x)} \leq e^{-x}$ by rearranging
    and taking the exponential. Since $\Phi(0) = 1/2$, we have that $f(0) = 0$
    and thus, if $f'(x) \leq 0$ for all $x \geq 0$, then $f$ is non-increasing, 
    meaning $f(x) \leq 0$ for all $x \geq 0$. Simple calculus then gives that 
    \[f'(x) = 1 - \phi(x)\left(\frac{1}{\Phi(x)} - \frac{1}{1-\Phi(x)}\right).\]
    Therefore, the condition $f'(x) \leq 0$ holds if 
    \[\phi(x)\left(\frac{1}{\Phi(x)} - \frac{1}{1-\Phi(x)}\right) \geq 1,\] or
    equivalently, $\phi(x) \geq \Phi(x)(1-\Phi(x))$. 
    
    We further define $h(x) = \phi(x) - \Phi(x)(1-\Phi(x))$. We see that $h(0)
    \geq 0$ and, moreover, we calculate $h'(x) = -x\phi(x) + \phi(x)(1-\Phi(x))
    - \Phi(x)(-\phi(x)) = -\phi(x)(x + 1 - 2\Phi(x))$. Thus, using that $\Phi(x)
    \leq \frac{1}{2}(1+ x)$ for $x \geq 0$, we have that $h'(x) \leq 0$ for $x
    \geq 0$. Thus, $h$ is non-decreasing for $x \geq 0$ and hence $h(x) \geq 0$
    for each $x \geq 0$. This implies that $\phi(x) \geq \Phi(x)(1-\Phi(x))$ for
    all $x \geq 0$, which in turn gives the stated lemma.
\end{proof}

\begin{lemma}\label{lem:vexupper}
Let $x \in [0,k/(12n)]$, then $\frac{k/(2n) + x}{k/(2n) -x} \leq 1 + 5\frac{n}{k}x$.
\end{lemma}

\begin{proof}
    Calculate the second derivative of $x \mapsto \frac{k/(2n) + x}{k/(2n) - x}$ as
    $-\frac{2k}{n(2x-\frac{k}{2n})^3}$. This function is positive for any $0 \leq x < k/(2n)$,
    therefore $x \mapsto \frac{k/(2n) + x}{k/(2n) - x}$ is convex on the domain $\{0
    \leq x \leq 1/12\}$. Thus, setting $f(x) = \frac{k/(2n) + x}{k/(2n) - x}$, we have
    that $f(x)$ is bounded above by the line connecting the points $(0,1)$ and
    $(k/(12n), 7/5)$ when $0 \leq x \leq k/(12n)$ by convexity. Basic arithmetic then
    gives the bound $f(x) \leq 1 + \frac{12n}{k}(7/5 - 1)x \leq 1 + 5\frac{n}{k}x$ when $0 \leq x \leq 1/12$, completing the proof.
\end{proof}

\begin{lemma} \label{lem:Mills} Let $Z$ be a standard Gaussian
    random variable, then for all $t \geq 0$, $\PP(Z > t) \leq e^{-t^2/2}$.
\end{lemma}

\begin{proof}
    If $t \geq \sqrt{2\pi}^{-1}$, by the Mills ratio \cite[Formula 7.1.13]{Milton}, we have the bound of $\PP(Z > t) \leq
    \frac{1}{t}\frac{1}{\sqrt{2\pi}}e^{-t^2/2} \leq e^{-t^2/2}$. As
    $\sqrt{2\pi}^{-1} \leq 1$ the proof follows by observing that $\PP(Z > t)
    \leq 1/2$ and $e^{-1^2/2}$ (i.e. our desired upper bound at $t = 1$) being
    greater than $1/2$.
\end{proof}

\begin{lemma}\label{lem:coupon}
    Consider the classical coupon collector problem with $n$ coupons, let $T$ be
    the time required to collect all $n$ coupons, then 
    \[\PP(T > \beta n \log n) \leq n^{-\beta + 1}.\]
\end{lemma}

\begin{proof}
     Let
    $Z_i^r$ be the event that the $i$-th coupon was not picked in the first $r$
    draws. Then $\PP(Z_i^r) = \left(1 - \frac{1}{n}\right)^{r}$ and thus, for $r
    = \beta n \log n$, we have that $\PP(Z_i^r) = n^{-\beta}$ and therefore, a
    union bound implies that 
    \[\PP(T > \beta n \log n) \leq \PP(\cup_{i \in [n]} Z_i^r) \leq n \PP(Z_1^r)
    \leq n^{-\beta + 1}.\]
\end{proof}

\subsection{Proof of Lemma \texorpdfstring{\ref{lem:AlgBTrans}}{lem:AlgBTrans}}\label{sec:proofAlgBTrans}
\begin{proof}[Proof of Lemma \ref{lem:AlgBTrans}]
    First we make the following claim: Let $\runtime \in \NN$ be arbitrary,
    $(p_t)_{t \in \NN}$ be a set of uniform draws from $[n]$ and let
    $T^*_\runtime$ be the random function defined in \eqref{eq:TstarMdef}.
    
    Then, almost surely over the value of $T^*_\runtime$, the sequence of observations $D'_t$ from \eqref{eq:DAlign}, has the
    following law when $(S_t)_{p_t} \ne 0$:
    \[
    (D'_t)_{t \in [T^*_\runtime]} \laweq \left(f_{\theta, \lambda}(S_{t}, 
    p_t, -(S_t)_{p_t}) + \Variance \runtime^{1/2}Z_t\right)_{t \in [T^*_\runtime]},\label{eq:DAlignLaw1}
    \]
    where each $Z_t$ is an independent standard Gaussian
    independent of $(p_t)_{t \in \NN}$ and thus $T^*_\runtime$. 
    
    Now we prove Lemma \ref{lem:AlgBTrans}.
    Notice, for each time $t \in [\finalIter]$, Algorithm \ref{alg:B} proposes
    coordinates $i \in [n]$ which, if accepted, would transition to $S_{t} -
    2(S_t)_ie_i$. Therefore, $\|S_{t}\|_0 = \|S_1\|_0$ for the entire
    run-time of Algorithm \ref{alg:B}. Then, by Lemma\footnote{Note, this is not
    a circular argument as we do not use Lemma \ref{lem:AlgBTrans} in the proof
    of Lemma \ref{lem:Init}.}  \ref{lem:Init}, $S_{\rm HOM}$ has
    $\|S_{\rm HOM}\|_0 = n$, giving that $\|S_{t}\|_0 = n$ for all $t \in
    [\finalIter]$. Thus, using the claim with $\runtime = \runtime_1$,
    we have, almost surely over $\finalIter$, that
    \[(D'_t)_{t \in [\finalIter]} \laweq \left(f_{\theta, \lambda}(S_{t}, p_t,
        -(S_t)_{p_t}) + V_n \sqrt{\runtime_1}Z_t\right)_{t \in
        [\finalIter]}, \] where each $Z_t$ is an independent standard normal
        random variable independent of $(p_t)_{t \in [\finalIter]}$ and thus
        $\finalIter$. Dividing both sides of the above displayed equation by
        $V_n\sqrt{\runtime_1}$ gives the result.

    It remains to prove the claim. Recall the difference of $H_{\frac{r+1}{2}, \gamma}(\proposal) -
    H_{\frac{r+1}{2}, \gamma}(S_{t})$ from Remark \ref{rm:AlgA}. We can see that
    the difference $D'_t$ is equivalent to this difference plus the random threshold term $\bCorrect$. Therefore, using the results of Lemma
    \ref{lem:equiv03} for $D'_t$ (i.e. for $D_t$ with $q_t = -(S_t)_{p_t}$ and
    $\gamma = 0$), we have, almost surely over $T^*_\runtime$, that 
    \[\left(D'_t\right)_{t \in [T^*_\runtime]} \laweq \left(f_{\theta,
    \lambda}(S_{t}, p_t, -(S_t)_{p_t}) + \Variance \runtime^{1/2}Z_t\right)_{t \in
    [T^*_\runtime]},\] where each $Z_t$ is an independent standard Gaussian
    independent of $(p_t)_{t \in \NN}$ and thus $T^*_\runtime$. 
    \end{proof}

\subsection{Proofs of Lemmas\texorpdfstring{ 
\ref{lem:NoiseEquiv}}{lem:NoiseEquiv},
\texorpdfstring{\ref{lem:VarLawBound}}{lem:VarLawBound},  \texorpdfstring{\ref{lem:bSignalControl}}{lem:bSignalControl} And \texorpdfstring{\ref{lem:bRegControl}}{lem:bRegControl}}\label{aux:bounds}

\begin{proof}[Proof of Lemma \ref{lem:NoiseEquiv}]
    We have $g(\sigma, i, a) = \langle (\sigma + a e_i)^{\otimes r} -
    \sigma^{\otimes r}, W \rangle$. As each element of $W$ is standard Gaussian,
    we have the law equivalence 
    \[g(\sigma, a, i) \laweq \cN(0, \|(\sigma + a e_i)^{\otimes r} -
    \sigma^{\otimes r}\|_F^2).\]
    This proved the first equality in law for $g(\sigma, a, i)$.
    And thus, the statement is proven.
\end{proof}

\begin{proof}[Proof of Lemma \ref{lem:VarLawBound}]
    We have that $(\sigma + ae_i)^{\otimes r} - \sigma^{\otimes r} =
    \sum_{k=1}^r \binom{r}{k} a^k e_i^{\otimes k} \otimes \sigma^{\otimes r-k}$.
    As any two $k$ summands have differing support, then we have that 
    \[\|(\sigma + ae_i)^{\otimes r} - \sigma^{\otimes r}\|^2_F = \sum_{k=1}^r
    \|\binom{r}{k} a^k e_i^{\otimes k} \otimes \sigma^{\otimes r-k}\|^2_F =
    \sum_{k=1}^r \binom{r}{k}^2 a^{2k} \|\sigma\|^{r-k}_0.\] 
    As each element of the summation is positive on the right-hand side above,
    we have automatically the following upper bound: 
    \[\|(\sigma + ae_i)^{\otimes r} -
    \sigma^{\otimes r}\|^2_F = \sum_{k=1}^r \binom{r}{k}^2 a^{2k}
    \|\sigma\|^{r-k}_0 \leq \|\sigma\|_0^{r-1} \sum_{k=1}^r \binom{r}{k}^2
    a^{2k}.\]
    The proof follows from taking the square root and realizing the right-hand
    side above is maximized when $a = 2$.

\end{proof}

\begin{proof}[Proof of Lemma \ref{lem:bSignalControl}]
Define the function $g(x) = x^r$. As both inner products in $f(\sigma, \theta, a,
i)$ have rank one structure, we equivalently have 
\[f(\sigma, \theta, a, i) = g(\langle \sigma + ae_i, \theta \rangle) - g(\langle
\sigma, \theta \rangle).\]
Therefore, by the mean value theorem, we have that 
\[f(\sigma, \theta, a, i) = g'(c) a\theta_i,\] where $c \in [\min(\langle \sigma
+ ae_i, \theta \rangle, \langle \sigma, \theta \rangle), \max(\langle \sigma +
ae_i, \theta \rangle, \langle \sigma, \theta \rangle)]$.
We then consider two cases: 

First, when $a\theta_i \geq 0$, we have the bound 
\[C_{\rm Signal}\min(\langle \sigma + ae_i, \theta \rangle^{r-1}, \langle
\sigma, \theta \rangle^{r-1})a\theta_i \leq f(\sigma, \theta, a, i) \leq C_{\rm
Signal}\max(\langle \sigma + ae_i, \theta \rangle, \langle \sigma, \theta
\rangle)^{r-1}a\theta_i,\] as the assumption on the non-negativity of the inner 
products ensures that $g'(c) = r c^{r-1}$ is non-decreasing the interval $[\min(\langle
\sigma + ae_i, \theta \rangle, \langle \sigma, \theta \rangle)^{r-1},
\max(\langle \sigma + ae_i, \theta \rangle, \langle \sigma, \theta
\rangle)^{r-1}]$.

Second, when $a\theta_i \leq 0$, we reverse the above bound. This leads to 
\[C_{\rm Signal}\max(\langle \sigma + ae_i, \theta \rangle^{r-1}, \langle \sigma,
\theta \rangle^{r-1})a\theta_i \leq f(\sigma, \theta, a, i) \leq C_{\rm
Signal}\min(\langle \sigma + ae_i, \theta \rangle, \langle \sigma, \theta
\rangle)^{r-1}a\theta_i.\]
\end{proof}

\begin{proof}[Proof of Lemma \ref{lem:bRegControl}]
    Let $g(x) = x^{\frac{r+1}{2}}$. We have that 
    \[h(s,a) = \gamma (g(s+a) - g(s)).\]
    Therefore, by the mean value theorem, we have that 
    \[h(s,a) = \gamma g'(c) a,\] where $c \in [\min(s, s+a), \max(s, s+a)]$.
    Since $g'(\gamma c) = \gamma \frac{r+1}{2} c^{\frac{r-1}{2}}$ is increasing
    in c over the interval $[\min(s, s+a), \max(s, s+a)]$ as each of $s$ and
    $s+a$ are non-negative, the result follows.
\end{proof}

\newpage

\bibliographystyle{alpha}
\bibliography{ref}

\end{document}